\documentclass[10pt]{amsart}

\makeatletter
\def\@settitle{\begin{center}%
  \baselineskip14\p@\relax
    \normalfont\LARGE

  \@title
  \end{center}%
}
\makeatother

\usepackage[dvipsnames]{xcolor}
\usepackage[hyperfootnotes=false]{hyperref}
\usepackage{enumerate}
\usepackage{tikz}
\usepackage{tikz-cd}
\usepackage{hyperref}
\usepackage{amsthm,amsfonts,amsmath,amscd,amssymb}
\usepackage{xcolor,float}
\usepackage[all]{xy}
\usepackage{mathrsfs}
\usepackage{graphics,graphicx}
\usepackage{caption}
\usepackage{comment}
\usepackage{array}
\newcolumntype{P}[1]{>{\centering\arraybackslash}p{#1}}
\newcolumntype{M}[1]{>{\centering\arraybackslash}m{#1}}

\usepackage{epigraph}
\let\oldmarginpar\marginpar
\renewcommand\marginpar[1]{\-\oldmarginpar[\raggedleft\footnotesize #1]%
	{\raggedright\footnotesize #1}}
\theoremstyle{plain}
\newtheorem{thm}{Theorem}[section]
\newtheorem{lemma}[thm]{Lemma}
\newtheorem*{theorem*}{Theorem}
\newtheorem*{corollary*}{Corollary}
\newtheorem{example}[thm]{Example}
\newtheorem{prop}[thm]{Proposition}

\newtheorem{cor}[thm]{Corollary}

\theoremstyle{definition}

\newtheorem{definition}[thm]{Definition}

\newtheorem{remark}[thm]{Remark}

\numberwithin{equation}{section}

\newcommand{\m}{\mathfrak{m}}
\renewcommand{\S}{\mathbb{S}}

\renewcommand{\L}{\mathbb{L}}
\newcommand{\I}{\underline{\mathbb{I}}}

\newcommand{\N}{\mathbb{N}}
\newcommand{\Z}{\mathbb{Z}}

\newcommand{\R}{\mathbb{R}}
\newcommand{\C}{\mathbb{C}}
\newcommand{\SA}{\mathcal{A}}
\newcommand{\SO}{\mathcal{O}}

\newcommand{\fM}{\mathfrak{M}}

\newcommand{\SF}{\mathscr{F}}
\newcommand{\SG}{\mathscr{G}}

\newcommand{\SX}{\mathcal{X}}
\newcommand{\SL}{\mathscr{L}}
\newcommand{\SR}{\mathscr{R}}

\newcommand{\La}{\Lambda}

\newcommand{\la}{\lambda}

\renewcommand{\a}{\alpha}

\newcommand{\dd}{\partial}

\newcommand{\sse}{\subset}
\newcommand{\lr}{\longrightarrow}

\newcommand{\Br}{\operatorname{Br}}
\newcommand{\ssup}{\operatorname{\mu\mbox{supp}}}
\newcommand{\msh}{\operatorname{\mu\mbox{Sh}}}

\newcommand{\GL}{\operatorname{GL}}
\newcommand{\PGL}{\operatorname{PGL}}

\newcommand{\Sh}{\operatorname{Sh}}

\newcommand{\wt}{\widetilde}

\newcommand{\st}{\text{st}}

\newcounter{daggerfootnote}

\newcommand{\bC}{\mathbb{C}}

\newcommand{\bG}{\mathbb{G}}

\newcommand{\bL}{\mathbb{L}}

\newcommand{\bS}{\mathbb{S}}
\newcommand{\bR}{\mathbb{R}}

\newcommand{\T}{\mathrm{T}}

\newcommand{\cA}{\mathcal{A}}

\newcommand{\cM}{\mathcal{M}}

\newcommand{\cX}{\mathcal{X}}

\newcommand{\w}{\mathtt{w}}

\newcommand{\n}{\mathfrak{n}}
\renewcommand{\l}{\mathfrak{l}}
\renewcommand{\i}{\mathfrak{i}}
\renewcommand{\c}{\mathfrak{c}}

\newcommand{\ww}{\mathfrak{w}}
\newcommand{\fB}{\mathfrak{B}}
\newcommand{\fS}{\mathfrak{S}}

\newcommand{\M}{\mathcal{M}}

\newcommand{\img}{\mathop{\mathrm{img}}}
\newcommand{\Loc}{\mathrm{Loc}}

\newcommand{\rank}{\mathrm{rk}}
\newcommand{\codim}{\mathop{\mathrm{codim}}}

\newcommand{\FM}{\mathfrak{M}}

\newcommand{\Span}{\mathrm{Span}}
\newcommand{\Dem}{\mathrm{Dem}}

\usepackage{dsfont}
\allowdisplaybreaks
\newcommand{\coker}{\mathop{\mathrm{coker}}}
\newcommand{\id}{\mathrm{id}}
\def \vertbar [#1](#2,#3,#4){
    \draw [#1] (#2,#3) -- (#2,#4);
    \draw [fill=white] (#2,#3) circle [radius=0.1];
    \draw [fill=black] (#2,#4) circle [radius=0.1];
}

\usepackage{mathdots}
\usepackage{afterpage}
\makeatletter
\providecommand{\leftsquigarrow}{%
  \mathrel{\mathpalette\reflect@squig\relax}%
}
\newcommand{\reflect@squig}[2]{%
  \reflectbox{$\m@th#1\rightsquigarrow$}%
}
\makeatother
\newcommand{\inprod}[2]{\left\langle#1,#2\right\rangle}
\makeatletter
\def\Ddots{\mathinner{\mkern1mu\raise\p@
\vbox{\kern7\p@\hbox{.}}\mkern2mu
\raise4\p@\hbox{.}\mkern2mu\raise7\p@\hbox{.}\mkern1mu}}
\makeatother

\def \horline [#1](#2,#3,#4){
    \draw [#1] (#2,#4) -- (#3,#4);
    \draw [fill=white] (#2,#4) circle [radius=0.1];
    \draw [fill=black] (#3,#4) circle [radius=0.1];
}

\def \crossing (#1,#2)(#3,#4){
\draw (#1,#2) -- (#3,#4);
\draw (#1,#4) -- (#3,#2);
}

\newcommand{\DT}{\mathrm{DT}}

\DeclareFontFamily{U}{mathb}{}
\DeclareFontShape{U}{mathb}{m}{n}{
  <-5.5> mathb5
  <5.5-6.5> mathb6
  <6.5-7.5> mathb7
  <7.5-8.5> mathb8
  <8.5-9.5> mathb9
  <9.5-11.5> mathb10
  <11.5-> mathbb12
}{}

\usetikzlibrary{decorations.markings, arrows}
\tikzset{tangent/.style={decoration={markings,mark=at position #1 with {
      \coordinate (tangent point-\pgfkeysvalueof{/pgf/decoration/mark info/sequence number}) at (0pt,0pt);
      \coordinate (tangent unit vector-\pgfkeysvalueof{/pgf/decoration/mark info/sequence number}) at (1,0pt);
      \coordinate (tangent orthogonal unit vector-\pgfkeysvalueof{/pgf/decoration/mark info/sequence number}) at (0pt,1);
      }},postaction=decorate},
    use tangent/.style={
        shift=(tangent point-#1),
        x=(tangent unit vector-#1),
        y=(tangent orthogonal unit vector-#1)
    },
    use tangent/.default=1
    }

\begin{document}

	\title{Microlocal Theory of Legendrian Links and Cluster Algebras}
	
	
	\author{Roger Casals}
	\address{University of California Davis, Dept. of Mathematics, USA}
	\email{casals@math.ucdavis.edu}
	
	\author{Daping Weng}
	\address{University of California Davis, Dept. of Mathematics, USA}
	\email{dweng@ucdavis.edu}
	
\maketitle
\vspace{-1.2cm}
\begin{abstract} We show the existence of quasi-cluster $\mathcal{A}$-structures and cluster Poisson structures on moduli stacks of sheaves with singular support in the alternating strand diagram of grid plabic graphs by studying the microlocal parallel transport of sheaf quantizations of Lagrangian fillings of Legendrian links. The construction is in terms of contact and symplectic topology, showing that there exists an initial seed associated to a canonical relative Lagrangian skeleton. In particular, mutable cluster $\mathcal{A}$-variables are intrinsically characterized via the symplectic topology of Lagrangian fillings in terms of dually $\mathbb{L}$-compressible cycles. New ingredients are introduced throughout this work, including the initial weave associated to a grid plabic graph, cluster mutation along non-square faces of a plabic graph, possibly including lollipops, the concept of sugar-free hull, and the notion of microlocal merodromy. Finally, we prove the existence of the cluster DT-transformation for shuffle graphs, constructing a contact geometric realization and an explicit reddening sequence, and establish cluster duality for the cluster ensembles.
\end{abstract}




\section{Introduction}\label{sec:intro}

The object of this article will be to show the existence of intrinsically symplectic quasi-cluster $\mathrm{K}_2$-structures and quasi-cluster Poisson structures on moduli stacks of sheaves with singular support in the alternating strand diagram of a complete grid plabic graph. The construction of such quasi-cluster structures is achieved via contact and symplectic topology, based on the recently developed machinery of Legendrian weaves, and we show that there exists a canonical initial quasi-cluster seed associated to a relative Lagrangian skeleton. This is the first manuscript proving the existence of such cluster structures for these general moduli stacks, and entirely in symplectic geometric terms, as well as introducing the first symplectic topological definition of cluster $\mathcal{A}$-variables associated to Lagrangian fillings of Legendrian links. In particular, our constructions admit natural contact and symplectic invariance and functoriality properties, and the cluster variables can be named and computed after performing Hamiltonian isotopies.

Several new ingredients are introduced for this purpose, among them are the initial weave of a grid plabic graph, cluster mutations along non-square faces, possibly with lollipops, the concept of sugar-free hulls, and the notion of microlocal merodromy. Microlocal merodromies capture microlocal parallel transport along a relative cycle and they are crucial in defining a set of initial cluster $\mathcal{A}$-variables. From a contact geometry viewpoint, embedded Lagrangian disks whose boundaries lie on embedded exact Lagrangian fillings have a central role. This allows for geometric characterizations of mutable and frozen vertices, which arise from relative homology groups of triples, and naturally explains the appearance of quasi-cluster structures.


\begin{center}
	\begin{figure}[H]
		\centering
		\includegraphics[scale=1.4]{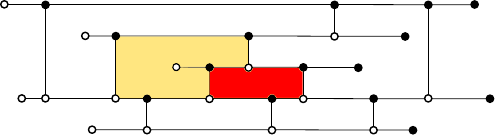}
		\caption{The quasi-cluster $\mathrm{K}_2$-structure we construct for this grid plabic graph is on the coordinate ring of the moduli of decorated sheaves on $\R^2$ with singular support in a max-tb Legendrian representative of the $m(9_6)$ knot.}
		\label{fig:ElementaryRegions_ExampleIntro}
	\end{figure}
\end{center}

\subsection{Scientific Context}\label{ssec:context} Cluster algebras, first introduced by S.~Fomin and A.~Zelevinsky \cite{FominZelevinsky_ClusterI,FominZelevinsky_ClusterII,BFZ05} in the context of Lie theory, are commutative rings endowed with a set of distinguished generators that have remarkable combinatorial structures. Cluster varieties, a geometric enrichment of cluster algebras introduced by V.~Fock and A.~Goncharov \cite{FockGoncharov_ModuliLocSys,FockGoncharovII}, are affine schemes equipped with an atlas of torus charts whose transition maps obey certain combinatorial rules. Cluster varieties come in dual pairs consisting of a cluster $\mathrm{K}_2$-variety, also known as a cluster $\mathcal{A}$-variety, and a cluster Poisson variety, also known as a cluster $\mathcal{X}$-variety. In particular, the coordinate ring of a cluster $\mathcal{A}$-variety coincides with an upper cluster algebra, see Berenstein-Fomin-Zelevinsky \cite{BFZ05}.

Since their introduction, cluster algebras and cluster varieties have appeared in many contexts, such as Teichm\"uller theory \cite{FockGoncharov_ModuliLocSys,FST08,GSV05}, birational geometry \cite{GHK15,GHKK,HK18}, the Riemann-Hilbert correspondence \cite{Allegretti,GMN_Cluster,GMN_Wallcrossing}, exact WKB analysis \cite{IN1,IN2}, and the study of positroid and Richardson varieties \cite{LamGalashin_Positroids,SSBW}, among others. The first appearance of cluster mutations in symplectic geometry occurred in the study of wall-crossing formulas, following the work of D.~Auroux, K.~Fukaya, M.~Kontsevich, P.~Seidel, Y.~Soibelman and others, e.g.~see \cite{Auroux_Wallcrossing2,Auroux_Wallcrossing,PascaleffTonkonog20} and references therein. We also thank A.~Goncharov for pointing out to us his recent work with M.~Kontsevich \cite{GoncharovKontsevich21_Noncommutative} focusing on non-commutative clusters, which also aligns well with the developments we present here. The first hint that cluster $\mathcal{X}$-structures might naturally exist in the symplectic study of Legendrian knots was provided in \cite{STWZ}, where it was computed how certain absolute monodromies around a square plabic face change under a square move in a plabic fence. See also the generalization presented in \cite{STW}. In conjunction, \cite{STWZ,STW} should imply the existence of partial $\SX$-structures for certain moduli stacks of sheaves singularly supported in the Legendrian lifts of the alternating strand diagrams of plabic fences. Nevertheless, they do not imply the existence of the full cluster $\SX$-structures, nor the full cluster $\SA$-structures and certainly not the fact that the rings of regular functions are cluster algebras. (See Subsection \ref{ssec:clarificationAX}.) These stronger statements are proven in the present manuscript.

There are two obstacles to prove the existence of a cluster $\mathcal{A}$-structure. First, many plabic faces are typically not square and may contain lollipops; thus, one needs a new construction that both associates a cluster variable to them and allows for a geometric mutation to be performed. Second, more fundamental, the regularity problem: even if all faces are square, the absolute monodromies are {\it not} global regular functions, and it is not possible to deduce the existence of a cluster structure purely from these microlocal monodromies. These obstacles are unavoidable if one is either restricted to plabic graphs or absolute cycles, both of which are limiting constraints in that approach.
\color{black}

Our new approach uses Legendrian weaves, which are more versatile than plabic graphs, and actually builds cluster $\mathcal{A}$-variables from relative cycles, which is stronger than the absolute analogue; see Subsection \ref{ssec:clarificationAX}. In particular, we overcome both obstacles above, resolving the regularity problem, and finally prove the existence of cluster $\mathcal{A}$--structures and, consequently, cluster $\mathcal{X}$-structures in entirely symplectic topological terms. Some of our previous work has been using ideas from the theory of cluster algebras for new applications to contact and symplectic geometry, see e.g.~ \cite{CasalsLagSkel,CGGS,CGGS2,CasalsZaslow,GSW}, including the discovery of infinitely many Lagrangian fillings for many Legendrian links \cite{CasalsHonghao}. This article builds in the converse direction, using contact and symplectic topology to construct (upper) cluster algebras, and symplectic topological results to deduce algebraic properties. In fact, we also know $\mathcal{A}=\mathcal{U}$ by recent work of the first author \cite{CGGLSS} which builds on the present manuscript.

Note that what can be deduced from our previous works \cite{CGGS,CasalsHonghao,CasalsZaslow,GSW,GSW2} is that certain moduli spaces that appear in contact topology are sometimes abstractly isomorphic to certain affine varieties, which themselves can independently be endowed\footnote{Explicitly, double Bott-Samelson cells for \cite{GSW}, and positroids for \cite{CasalsHonghao,STWZ}. These instances are, in any case, particular cases of the moduli stacks that we associate to grid plabic graphs.} with cluster structures, but currently there does not exist any symplectic construction or characterization of cluster $\mathcal{A}$-variables or general cluster $\mathcal{X}$-variables, nor a symplectic geometric proof of the existence of cluster structures on these moduli spaces, nor even a geometric understanding of frozen variables. In particular, none of these previous constructions is known to have any Hamiltonian or Legendrian invariance properties, which are crucial in contact and symplectic topology. In fact, in all previous constructions even the initial seeds cannot be named after a Hamiltonian isotopy (e.g. even after a Reidemeister I or II move) and no symplectic computation or interpretation of cluster $\mathcal{A}$-variables existed. The present work finally resolves this matter and, as we shall see, interesting symplectic features appear with regards to both mutable and frozen variables.


\subsection{Main Results}\label{ssec:main_geometry} Let $\La\sse (T^*_\infty\R^2,\xi_{st})$ be a Legendrian link in the ideal contact boundary of the cotangent bundle of the plane $\R^2$, and $T\sse\La$ a set of marked points. The precise details and definitions for these contact geometric objects are provided in Section \ref{sec:gridplabic}. Let $L\sse(T^*\R^2,\la_\st)$ be an embedded exact Lagrangian filling of $\La$. By definition, an embedded closed curve $\gamma\sse L$ is said to be $\bL$-compressible if there exists a properly embedded Lagrangian 2-disk $D\sse (T^*\R^2\setminus L)$ such that $\dd\overline{D}\cap L=\gamma\sse\R^4$. A collection $\{\gamma_1,\ldots,\gamma_\ell\}$ of such curves, with a choice of $\mathbb{L}$-compressing disk for each curve, is said to be an $\L$-compressing system for $L$ if the curves form a maximal linearly independent subset in $H_1(L)$. In line with this, we will use Lagrangian disk surgeries, as defined in \cite{Polterovich_Surgery,MLYau17}.

Consider also the moduli stack $\fM(\La,T)$ of decorated microlocal rank-one constructible sheaves on $\R^2$ with singular support contained in $\La$, as defined in Subsection \ref{sssec:sheaf_decorated}, following \cite{KashiwaraSchapira_Book,Sheaves1}, which is invariant under contact isotopies. Let $\bG\sse\R^2$ be a complete grid plabic graph and $\La=\La(\bG)\sse T^*_\infty\R^2$ its associated Legendrian link, as defined in Section \ref{sec:gridplabic}. See Subsection \ref{ssec:hulls} for the definition of the sugar-free hull $\S_f$ of a face $f$ in $\bG$ and Subsection \ref{ssec:completenessGP} for completeness. Note that the concept of sugar-free hulls, and whether a region is sugar-free, only depends on the behaviour at non-convex corners, see Definition \ref{def:sugarfree}.

The main result of the article, stated in Theorem \ref{thm:main}, is the existence and explicit symplectic construction of a full quasi-cluster $\mathcal{A}$-structure on $\fM(\La,T)$. In particular, the cluster $\mathcal{A}$-variables of the initial seed and all the once-mutated seeds are obtained by a new microlocal parallel transport along certain relative cycles on exact Lagrangian fillings of $\La$. This microlocal parallel transport is associated to a sheaf quantization of each exact Lagrangian filling, following \cite{Sheaves1,CasalsZaslow}, and we refer to it as a {\it microlocal merodromy}, see Section \ref{sec:cluster}. The central result of the manuscript is stated as follows.

\begin{thm}[Main Result]\label{thm:main} Let $\bG\sse\R^2$ be a complete grid plabic graph, $\La=\La(\bG)\sse (\R^3,\xi_\st)$ its associated Legendrian link, $T\sse\La$ a set of marked points, with at least one marked point per component of $\La$, and $\mathfrak{M}(\La,T)$ the stack of decorated microlocal rank-one constructible sheaves on $\R^2$ with singular support contained in $\La$.

Then, there exists a canonical embedded exact Lagrangian filling $L=L(\bG)\sse(\R^4,\omega_\st)$ of $\La$ and a canonical $\L$-compressing system $\mathfrak{S}=\{\gamma_1,\ldots,\gamma_\ell\}$ for $L$, indexed by the sugar-free hulls of $\bG$, such that for any completion of $\mathfrak{S}$ into a basis $\fB$ of $H_1(L,T)$ the following hold: 

\begin{itemize}
    \item[(i)] 
    The microlocal merodromies
    $A_{\eta_i}$, defined on $($and by using$)$ the open chart $(\C^\times)^{b_1(L,T)}\sse \fM(\La,T)$ associated to $L$, extend to global regular functions $$A_{\eta_i}:\mathfrak{M}(\La,T)\lr\C,\quad \mbox{i.e. } A_{\eta_i}\in\mathcal{O}(\fM(\La,T)),$$
    \noindent where $\fB^\vee=\{\eta_1,\dots, \eta_s\}$ is the dual basis in $H_1(L\setminus T,\Lambda\setminus T)$.\\
    
    \item[(ii)] The microlocal merodromies $\{A_{\eta_1},\ldots,A_{\eta_\ell}\}$ associated to the relative cycles that are dual to an $\bL$-compressible absolute cycle in $\mathfrak{S}$ are irreducible functions in $\mathcal{O}(\fM(\La,T))$, whereas the merodromies $\{A_{\eta_{\ell+1}},\ldots,A_{\eta_{b_1(L,T)}}\}$ are non-vanishing functions, i.e.~ units in $\mathcal{O}(\fM(\La,T))$.\\

    \item[(iii)] Let $L_k'\sse(\R^4,\omega_\st)$ be the Lagrangian filling obtained via Lagrangian disk surgery on $L$ at the $\mathbb{L}$-compressing disk for $\gamma_k\in\mathfrak{S}$, and $\eta_k'\in H_1(L_k'\setminus T,\La\setminus T)$ the image of $\eta_k$ under the surgery. Then the  merodromy $A_{\eta_k'}$ extends to a global regular function
    $$A_{\eta_k'}:\mathfrak{M}(\La,T)\lr\C,\quad \mbox{i.e. } A_{\eta_k'}\in\mathcal{O}(\fM(\La,T)),$$
    
    \noindent and satisfies the cluster $\mathcal{A}$-mutation formula $$A_{\eta_k'}A_{\eta_k}=\prod_{\eta_i\to\eta_k}A_{\eta_i}+\prod_{\eta_k\to\eta_j}A_{\eta_j}$$ with respect to the intersection quiver $Q(\fB)$ of the basis elements $\fB\subset H_1(L,T)$.
\end{itemize}

\noindent Finally, the moduli variety $\fM(\La,T)$ admits a cluster $\mathcal{A}$-structure with quiver $Q(\fB)$ in the initial seed associated to the Lagrangian filling $L$, where the mutable vertices (dually) correspond to the absolute cycles in the $\L$-compressing system $\mathfrak{S}$ for $L$. Furthermore, different choices of completion of $\fS$ into a basis $\fB$ give rise to quasi-equivalent cluster $\mathcal{A}$-structures.
\end{thm}

The grid plabic graph $\bG$ actually provides several natural completions of the $\mathbb{L}$-compressing system $\mathfrak{S}$ to a basis $\mathfrak{B}$, as explained in Section \ref{sec:weaves}. The canonical exact Lagrangian filling $L=L(\bG)$ associated with $\bG$ is obtained as the Lagrangian projection of the Legendrian surface whose front is given by the weave $\ww(\bG)$ associated with $\bG$, which is constructed in Section \ref{sec:weaves}. The weave $\ww(\bG)$ is used crucially in the argument so as to obtain a sheaf quantization of $L(\bG)$ and prove Items (i) through (iii) as required. In addition to the existence of the cluster $\mathcal{A}$-structures on $\fM(\La,T)$, another upshot of Theorem \ref{thm:main} is that the initial and the once-mutated cluster $\mathcal{A}$-variables can be named entirely in terms of symplectic topology, in an intrinsic and geometric manner. The resulting quasi-cluster $\mathcal{A}$-structure and these $\mathcal{A}$-variables can be equally considered and computed after a Hamiltonian isotopy.

In terms of the dichotomy between geometry and algebra, Theorem \ref{thm:main} shows that the ring $\mathcal{O}(\fM(\La,T))$ behaves {\it as if} it were always possible to perform an arbitrary sequence of Lagrangian disk surgeries starting at $L(\mathbb{G})$ with the curve configuration from the $\mathbb{L}$-compressing system $\mathfrak{S}$. It is known that geometric obstructions to further surger the Lagrangian skeleton can arise as one performs a series of Lagrangian surgeries (geometric mutations), e.g.~ through the appearance of immersed curves, or algebraic intersection numbers differing from geometric ones, and yet the existence of the cluster $\mathcal{A}$-structure built in Theorem \ref{thm:main} shows that it is not possible to detect such obstructions by studying $\mathcal{O}(\fM(\La,T))$. The following table schematically relates different ingredients involved in the proof of Theorem \ref{thm:main}:

\begin{center}
\begin{tabular}{ |P{5cm}|P{6cm}|P{5cm}|  }
	\hline
	\multicolumn{3}{|c|}{ {\bf Ingredients in the symplectic construction of upper cluster algebra for $\mathcal{O}(\mathfrak{M}(\Lambda,T))$}} \\
	\hline
	{\bf Grid Plabic Graph $\mathbb{G}$} & {\bf Symplectic Topology in $T^*\mathbb{R}^2$} & {\bf Cluster Theory}\\
	\hline
	Alternating strand diagram  & Legendrian Link $\Lambda\subseteq T^*_\infty\mathbb{R}^2$  & $D^-$-stack $\mathfrak{M}(\Lambda,T)$\\
	(with marked points $T$) & (with marked points $T$) & from dg-category $\mbox{Sh}_\Lambda(\mathbb{R}^2)$\\
	\hline
	Goncharov-Kenyon conjugate surface associated to $\mathbb{G}$&  Weave for Lagrangian filling $L$ ($\Longrightarrow$ Sheaf quantization $\mathcal{F}(L))$ & Open toric chart $T_L=(\mathbb{C}^\times)^{b_1(L,T)}\subseteq \mathfrak{M}(\Lambda,T)$\\
	\hline
	 & {\bf \textcolor{blue}{$\mathbb{L}$-compressible curve}} $\gamma\subseteq L$ & $T_L$-coordinate that extends to \\
	
	Sugar-free hull of $\mathbb{G}$	& with dual relative cycle $[\eta]\in N=H_1(L\setminus T,\Lambda\setminus T)$ &  a \textcolor{purple}{{\bf global regular}} function $A_\eta:\mathfrak{M}(\Lambda,T)\longrightarrow\mathbb{C}$\\
	\hline
	Set $S$ of sugar-free hulls & Mutable sublattice $\mathbb{Z}^{|S|}\subseteq N$ & Mutable variables $\{A_\eta\}$ in $T_L$\\
	\hline
	Non sugar-free region of $\mathbb{G}$ (e.g.~a non sugar-free face) & Immersed curve $\vartheta\subseteq L$ with dual relative cycle $\phi$ in $N$ ($\vartheta$ represented by immersed $\sf Y$-tree in weave) & $T_L$-coordinate extending to {\it non-vanishing} global regular function $A_\phi:\mathfrak{M}(\Lambda,T)\longrightarrow\mathbb{C}$\\
	\hline
	Subset of non s.-f. regions chosen via Hasse diagram& Sublattice $\mathbb{Z}^{b_1(L)-|S|}\subseteq N$ complement to sublattice $\mathbb{Z}^{|S|}$ & Frozen variables $\{A_\phi\}$ in $T_L$ (quasi-cluster equivalent)\\

		(different choices allowed) & (different complements) & \\
	\hline
	Intersection form on absolute $H_1$ of conjugate surface & Intersection form on $M=H_1(L,T)$ (and thus on dual $N=M^*$) & Quiver $Q(\{A_\eta\},\{A_\phi\})$ for $T_L$ (different from naive $Q(\mathbb{G})$)\\
	\hline
	``Mutation'' at sugar-free hull
	(not necessarily a square face, result often not plabic graph but represented by weave) & {\bf \textcolor{blue}{Lagrangian surgery}} $L'=\mu_\gamma(L)$ and relative cycle $\eta'=\mu_{\gamma}(\eta)$ in $L'$. Sheaf quantization $\mathcal{F}(L')$ via {\bf \textcolor{blue}{weave mutation}} at $\sf Y$-tree for $\gamma$. & $T_{L'}$-coordinate extending to \textcolor{purple}{{\bf global regular}} function $A_{\eta'}:\mathfrak{M}(\Lambda,T)\longrightarrow\mathbb{C}$ given by {\bf \textcolor{purple}{cluster $\mathcal{A}$-mutation}} at $\eta$\\
	\hline
\end{tabular}
\end{center}

There are several items from Theorem \ref{thm:main} that can be helpful to unpack. First, by a modification of the Guillermou-Jin-Treumann map \cite{JinTreumann}, the Lagrangian filling $L$ yields an open toric chart $(\C^\times)^{b_1}\sse \fM(\La,T)$, where $b_1=\mbox{rk}(H_1(L\setminus T,\La\setminus T))=\mbox{rk}(H_1(L,T)).$ The group $H^1(L;\C^\times)=Hom(H_1(L;\Z),\mbox{GL}_1(\C))$ accounts for the $\C^\times$-local systems on $L(\bG)$, and the modification accounts for the relative piece given by the marked points $T$; see Section \ref{ssec:sheaves} for details. By construction, microlocal merodromies are {\it a priori} functions on this particular chart $(\C^\times)^{b_1}$, and they visibly depend on $L$. In fact, in many cases they are (restrictions of) rational functions with non-trivial denominators and do not extend to global regular functions. Nevertheless, Theorem \ref{thm:main} shows that, remarkably, there is a particular set of such functions, indexed by a basis completion of the $\L$-compressing system $\mathfrak{S}$, whose elements extend to regular functions from $(\C^\times)^{b_1}$ to the entire moduli $\fM(\La,T)$. \\
    
\noindent Second, the frozen cluster $\mathcal{A}$-variables in Theorem \ref{thm:main} have two geometric, markedly distinct, origins: absolute cycles in $H_1(L)$, and relative cycles with endpoints in $T$ which are themselves not dual to any absolute cycle. The appearance of the former type of frozen variables, associated to absolute cycles, is an entirely new phenomenon, starting the study of $\bL$-(in)compressible curves in Lagrangian fillings. (E.g.~we show that a Chekanov $m(5_2)$ already displays such features.) At least to date, all known instances of frozen variables of geometric origin were related to marked points, in line with the latter type of frozens. The existence of a cluster structure on $\fM(\La,T)$ with a particular quiver $Q$ has neat applications to symplectic geometry, e.g. studying the possible relative Lagrangian skeleta containing $L$ for the Weinstein relative pair $(\C^2,\La)$; see below for more.\\
    
\noindent Third, Item (iii) in Theorem \ref{thm:main} is geometrically keeping track of certain relative cycles before and after a Lagrangian surgery: the data being analyzed is the change of a specific local system along that relative cycle (which itself changes topologically). This local system is obtained by applying the microlocal functor, with the target being the Kashiwara-Schapira stack $\msh_\La$, to a sheaf quantization of $L$. In our proof of Theorem \ref{thm:main}, the sheaf quantization is obtained thanks to the construction of the weave $\ww(\bG)$, which represents a (front of the) Legendrian lift of $L$. In fact, Section \ref{sec:weaves} will provide a diagrammatic method to draw those relative cycles before and after a weave mutation, and Section \ref{sec:cluster} provides a Lie-theoretic procedure to compute with such (microlocal) local systems. Note also that the geometric mutations are associated with sugar-free hulls, which are not necessarily square faces and might include lollipops: the fact that the calculus of weaves allows for these general mutations is crucial so as to conclude that the coordinate ring of $\fM(\La,T)$ is an upper cluster algebra.\\

\noindent Finally, the symplectic geometry perspective naturally leads to a quasi-cluster $\mathcal{A}$-structure, rather than a cluster $\mathcal{A}$-structure. Indeed, the weave $\ww(\bG)$ canonically gives the $\L$-compressing system $\mathfrak{S}$, which yields a linearly independent subset of $H_1(L\setminus T,\La\setminus T)$. Nonetheless, there are cases in which this subset does not span and a choice of basis completion is precisely what introduces the quasi-cluster ambiguity. In particular cases, such as $\bG$ being a plabic fence, the $\L$-compressing system already gives basis and hence $\fM(\La(\bG),T)$ carries a natural cluster $\mathcal{A}$-structure, but for a generic grid plabic graph $\bG$ there is no a priori reason for that to be the case; the natural algebraic structure arising from symplectic geometry is only unique up to quasi-cluster equivalence.\\

Theorem \ref{thm:main} also implies a series of new computations and results in 3-dimensional contact topology. Indeed, in many interesting cases, such as those where the cluster algebra equals the upper cluster algebra \cite{Muller13,Muller14,CGGLSS}, the existence of a full cluster $\mathcal{A}$-structure on the moduli space\footnote{If not made explicitly, the set of marked points $T$ is taken to have one marked point per component of $\La$.} $\mathfrak{M}=\mathfrak{M}(\La,T)$, as proven in Theorem \ref{thm:main}, leads to:

\begin{itemize}
    \item[(1)] The computation of its deRham cohomology ring $H^*(\mathfrak{M},\C)$, including the refinement of its mixed Hodge structure. These computations are done in \cite{LamSpeyer22} for the locally acyclic cases.\\
    
    \item[(2)] The existence of a holomorphic (pre)symplectic structure for the moduli space $\mathfrak{M}$.
    This allows for many classical techniques, such as quantization, to be applied to the coordinate ring $\mathcal{O}(\fM)$, see \cite{GSV_Book}. We emphasize that the cluster $\mathcal{A}$-variables associated to a seed are exponential Darboux coordinates for the symplectic 2-form. Note also that a holomorphic symplectic structure on the augmentation variety was recently constructed in \cite{CGGS} by different means (using the Cartan 3-form and Bott-Shulman forms) and see work of P.~Boalch \cite{Boalch1,Boalch2}. Upcoming work with our collaborators will show that these holomorphic symplectic structures coincide whenever they can be compared.\\
    
    \item[(3)] In the Louise case \cite{LamSpeyer22}, it is possible to compute the eigenvalues of the Frobenius automorphism on $\ell$-adic cohomology and perform finite point counts $\#\mathfrak{M}(\mathbb{F}_q)$ over finite fields $\mathbb{F}_q$, $q=p^k$ and $p$ large enough. These ought to be compared with the contact and symplectic results in \cite{HenryRutherford15,NRSS17}.
\end{itemize}

Another byproduct of our result, thinking in terms of cluster ensembles \cite{FockGoncharovII}, is that there also exists a (full) cluster $\mathcal{X}$-structure. Let $\mathcal{M}_1(\La)$ be the undecorated stack associated to $\mathfrak{M}(\La,T)$ and $\mathcal{M}_1(\La,T)$ its enhancement with framing data at $T$. Theorem \ref{thm:main} implies the following result:

\begin{cor}\label{cor:Xstructure}
Let $\bG\sse\R^2$ be a complete grid plabic graph, $\La=\La(\bG)\sse (\R^3,\xi_\st)$ its associated Legendrian link and $T\sse\La$ marked points. Then there exists a quasi-cluster $\mathcal{X}$-structure on $\mathcal{M}_1(\La,T)$.

\noindent In fact, each completion of the $\bL$-compressing system $\fS$ to a basis $\fB$ of $H_1(L,T)$, gives a cluster $\mathcal{X}$-structure on $\cM_1(\La,T)$. The initial quiver $Q$ is defined by the intersections in $\fB$ and the initial cluster $\mathcal{X}$-variables are microlocal monodromies associated with elements of $\fB$. In addition, the mutable cluster $\mathcal{X}$-variables are those associated with curves in the $\bL$-compressing system $\fS$ and different choices of completion of $\fS$ to a basis $\fB$ give quasi-equivalent cluster $\mathcal{X}$-structures on $\cM_1(\La,T)$. 
\end{cor}

\noindent Corollary \ref{cor:Xstructure} is a new result and establishes the existence of a (full) cluster $\mathcal{X}$-structure.
It is crucial to understand that there is currently no proof of Corollary \ref{cor:Xstructure} on its own. Namely, we are only able to deduce the existence of a cluster $\mathcal{X}$-structure once we have proven the existence of a full cluster $\mathcal{A}$-structure in Theorem \ref{thm:main}: two mathematical reasons are that the results used from \cite{BFZ05} are only applicable to cluster $\mathcal{A}$-structures and that the codimension-2 arguments in Section \ref{sec:cluster} require an explicit understanding of the $\mathcal{A}$-variables, including their irreducibility; see also Subsection \ref{ssec:clarificationAX}.


The two moduli $\fM(\Lambda(\bG),T)$ and $\cM_1(\Lambda(\bG), T)$ in Theorem \ref{thm:main} and Corollary \ref{cor:Xstructure} form a cluster ensemble. In Section \ref{sec:DT}, we focus on shuffle grid plabic graphs and prove that these cluster varietes always admit a Donaldson-Thomas (DT) transformation. See \cite{KontsevichSoibelman_MotivicDT,GoncharovLinhui_DT} for the necessary preliminaries on DT-transformations. In fact, we realize this cluster automorphism geometrically, as a composition of a Legendrian isotopy of $\La(\bG)$ and the strict contactomorphism $t:(x,y,z)\longmapsto(-x,y,-z)$ of $(\R^3,\ker\{dz-ydx\})$. In particular, we conclude the following result:

\begin{cor}\label{cor:DT} Let $\bG$ be a shuffle grid plabic graph. Consider the contactomorphism $t$ and the half K\'alm\'an loop Legendrian isotopy $K^{1/2}$. Then the composition $t\circ K^{1/2}$ induces the (unique) cluster Donaldson-Thomas transformation of $\cM_1(\Lambda(\bG))$.

\noindent In particular, the cluster duality conjecture holds for the cluster ensemble $(\fM(\Lambda(\bG),T),\cM_1(\Lambda(\bG), T))$. 
\end{cor}

\noindent The explicit sequence of mutations realizing the DT-transformation is presented in Section \ref{sec:DT}. We show it is a reddening sequence. Examples prove that it is not necessarily a maximal green sequence.

Finally, the contact and symplectic geometric results and techniques we use and develop to prove Theorem \ref{thm:main} are invariant under Hamiltonian isotopies, not necessarily compactly supported. Given that the cluster coordinates in Theorem \ref{thm:main} and Corollary \ref{cor:Xstructure} are all intrinsically named through symplectic geometric means, they can be named, and computed, after a compactly supported Hamiltonian isotopy is applied to $L(\bG)$ or a contact isotopy is applied to $\La(\bG)$. This is a distinctive crucial feature which had been missing in our previous works \cite{CGGS,CGGS2,CasalsZaslow,GSW}, where even the initial seed could not typically be defined (nor computed) after a Legendrian isotopy.\footnote{The pull-back structures from \cite[Section 3]{STWZ} had the same issue.}


{\bf Acknowledgements}. We are grateful to C.~Fraser, H.~Gao, A.~Goncharov, E.~Gorsky, M.~Gorsky, I.~Le, W.~Li, J.~Simental-Rodriguez, L.~Shen, M.~Sherman-Bennett and E.~Zaslow for their interest, comments on the draft and useful conversations. We also thank the referee for their valuable comments. R.~Casals is supported by the NSF CAREER DMS-1942363 and a Sloan Research Fellowship of the Alfred P. Sloan Foundation.

{\bf Notation}. We denote by $[a,b]$ the discrete interval $[a,b]:=\{k\in\N: a\leq k\leq b\}$, if $a\leq b$, $a,b\in\N$. In this article, $S_n$ denotes the group of permutations of $n$ elements, $n\in\N$, and $s_i$ its $i$-th simple transposition, $i\in[1,n-1]$. We abbreviate $s_{[b,a]}:=s_bs_{b-1}\ldots s_{a+1}s_a$ and $s_{[b,a]}^{-1}:=s_as_{a+1}\ldots s_{b-1}s_b$, for $a<b$, $a,b\in\N$, and $s_{[b,a]}$ and $s_{[b,a]}^{-1}$ are empty if $b<a$. Let $w_{0,n}\in S_n$ be the longest word in the symmetric group $S_n$, we will sometimes write $w_0\in S_n$ if $n$ is clear by context. The standard word $\w_{0,n}$ for $w_{0,n}$ is defined to be the reduced expression $\w_{0,n}:=s_{[1,1]}s_{[2,1]}s_{[3,1]}\ldots s_{[n-1,1]}$.

{\bf Extended version}. The present manuscript is a condensed account of the arXiv submission 2204.13244. The arXiv submission, also available in our research websites, is an extended version and contains a series of additional examples and figures, as well as more detail in some of the proofs and motivation and context in parts of the construction. The interested reader might benefit from the more inviting extended version, as it is more comprehensive and builds the proofs in a more self-contained manner. That said, we believe experts will also appreciate this streamlined version, where only the logically necessary steps for our main results are included.

\section{Grid Plabic Graphs and Legendrian Links}\label{sec:gridplabic}

In this section we introduce the starting characters in the manuscript. On the combinatorial side, we introduce the notion of a grid plabic graph $\bG$, in Subsection \ref{ssec:GPgraph}, and that of sugar-free hulls in Subsection \ref{ssec:hulls}. On the geometric side, we introduce a front for the Legendrian link $\La(\bG)$ associated to the alternating strand diagram of a grid plabic graph $\bG$, in Subsection \ref{ssec:legendrianlink}, and set up the necessary moduli spaces from the microlocal theory of sheaves in Subsection \ref{ssec:sheaves}. Several explicit examples are provided in Subsection \ref{sssec:examplesGPgraph}.

\subsection{Grid Plabic Graphs}\label{ssec:GPgraph} The input object in our results is the following type of graphs.

\begin{definition}\label{def:GPgraph}
An embedded planar bicolored graph $\bG\sse\R^2$ is said to be a \emph{grid plabic graph} (or \emph{GP-graph} for short) if it satisfies the following conditions.
\begin{itemize}\setlength\itemsep{0.5em}
    \item[(i)] The vertices of $\bG\sse\R^2$ belong to the standard integral lattice $\Z^2\sse\R^2$, and they are colored in either black or white.
    \item[(ii)] The edges of $\bG\sse\R^2$ belong to the standard integral grid $(\Z\times\R)\cup (\R\times\Z)\sse\R^2$. Edges that are contained in $\Z\times\R$ are said to be \emph{vertical}, and edges that are contained in $\R\times\Z$ are said to be \emph{horizontal}. 
    \item[(iii)] A maximal connected union of horizontal edges is called a \emph{horizontal line}. Each horizontal line must end at a univalent white vertex on the left and a univalent black vertex on the right. These univalent vertices are called \emph{lollipops}.
    \item[(iv)] Each vertical edge must end at trivalent vertices of opposite colors, and the end points of a vertical edge must be contained in the interior of a horizontal line.
\end{itemize}
\end{definition}

\noindent In Definition \ref{def:GPgraph}, it is fine to allow for bivalent vertices. The Legendrian isotopy type of the zig-zag diagram, as introduced in Subsection \ref{ssec:legendrianlink}, does not change when inserting such vertices, nor does the Hamiltonian isotopy type of the Lagrangian filling associated to the conjugate surface.

\subsection{Column types and associated transpositions}\label{ssec:} The intersection of a GP-graph $\bG\sse\R^2$ with a subset of the form $\{(x,y)\in\R^2: l< x< r\}\sse\R^2$, for some $l,r\in\R$, $l<r$, is said to be a column of $\bG$. Any GP-graph $\bG$ is composed by the horizontal concatenation of three types of non-empty columns called \emph{elementary columns}. These three types of elementary columns are depicted in Figure \ref{fig:ElementaryRegions} and can be described as follows.

\begin{itemize}\setlength\itemsep{0.5em}
    \item[-] Type 1: a column is said to be Type 1 if it solely consists of parallel horizontal lines, i.e. it contains no vertices.
    \item[-] Type 2: a column is said to be Type 2, or a crossing, if it contains exactly two oppositely colored vertices of $\bG$ and a (unique) vertical edge between them. 
    \item[-] Type 3: a column is said to be Type 3, or a lollipop, if it contains exactly one lollipop. Note that the lollipop can be either white or black. 
\end{itemize}

\begin{center}
	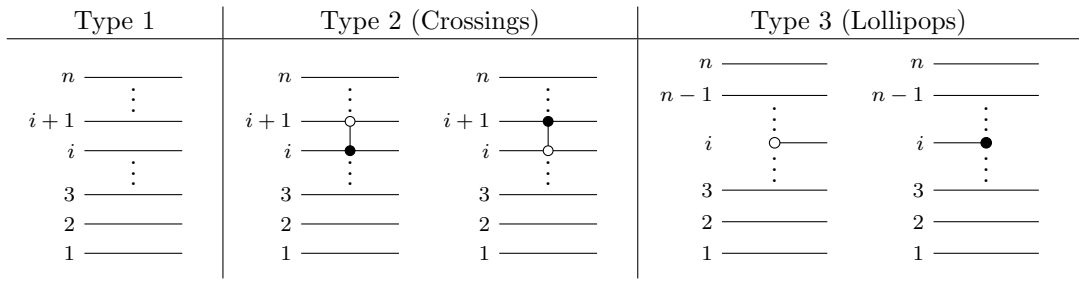
\begin{figure}[H]
		\centering
		\begin{tabular}{c|c|c}
		Type 1 & Type 2 (Crossings) & Type 3 (Lollipops) \\ \hline 
		\begin{tikzpicture}[scale=0.65]
		\foreach \i in {1,2,3} 
		{
		\draw (0,\i*0.6) node [left] {\footnotesize{$\i$}} -- (2,\i*0.6) node [right] {$\ $};
		}
		\foreach \i in {4.5,5.5,7} 
		{
		\draw (0,\i*0.6) -- (2,\i*0.6);
		}
		\node at (0,4.5*0.6) [left] {\footnotesize{$i$}};
		\node at (0,5.5*0.6) [left] {\footnotesize{$i+1$}};
		\node at (0,7*0.6) [left] {\footnotesize{$n$}};
		\node at (0,7.5*0.6) [] {$\quad$};
		\node at (1,4*0.6) [] {$\vdots$};
		\node at (1,6.5*0.6) [] {$\vdots$};
		\end{tikzpicture} & 
		\begin{tikzpicture}[scale=0.65]
		\foreach \i in {1,2,3} 
		{
		\draw (0,\i*0.6) node [left] {\footnotesize{$\i$}} -- (2,\i*0.6) node [right] {$\ $};
		}
		\foreach \i in {4.5,5.5,7} 
		{
		\draw (0,\i*0.6) -- (2,\i*0.6);
		}
		\node at (0,4.5*0.6) [left] {\footnotesize{$i$}};
		\node at (0,5.5*0.6) [left] {\footnotesize{$i+1$}};
		\node at (0,7*0.6) [left] {\footnotesize{$n$}};
		\node at (0,7.5*0.6) [] {$\quad$};
		\node at (1,4*0.6) [] {$\vdots$};
		\node at (1,6.5*0.6) [] {$\vdots$};
		\vertbar[](1,5.5*0.6,4.5*0.6);
		\end{tikzpicture}
		\begin{tikzpicture}[scale=0.65]
		\foreach \i in {1,2,3} 
		{
		\draw (0,\i*0.6) node [left] {\footnotesize{$\i$}} -- (2,\i*0.6) node [right] {$\ $};
		}
		\foreach \i in {4.5,5.5,7} 
		{
		\draw (0,\i*0.6) -- (2,\i*0.6);
		}
		\node at (0,4.5*0.6) [left] {\footnotesize{$i$}};
		\node at (0,5.5*0.6) [left] {\footnotesize{$i+1$}};
		\node at (0,7*0.6) [left] {\footnotesize{$n$}};
		\node at (0,7.5*0.6) [] {$\quad$};
		\node at (1,4*0.6) [] {$\vdots$};
		\node at (1,6.5*0.6) [] {$\vdots$};
		\vertbar[](1,4.5*0.6,5.5*0.6);
		\end{tikzpicture} &
		\begin{tikzpicture}[scale=0.7]
		\foreach \i in {1,2,3} 
		{
		\draw (0,\i*0.6) node [left] {\footnotesize{$\i$}} -- (2,\i*0.6) node [right] {$\ $};
		}
		\foreach \i in {6,7} 
		{
		\draw (0,\i*0.6) -- (2,\i*0.6);
		}
		\node at (0,4.5*0.6) [left] {\footnotesize{$i$}};
		\node at (0,6*0.6) [left] {\footnotesize{$n-1$}};
		\node at (0,7*0.6) [left] {\footnotesize{$n$}};
		\node at (0,7.5*0.6) [] {$\quad$};
		\node at (1,3.9*0.6) [] {$\vdots$};
		\node at (1,5.5*0.6) [] {$\vdots$};
		\draw (1,4.5*0.6) -- (2,4.5*0.6);
		\draw [fill=white] (1,4.5*0.6) circle [radius=0.1];
		\end{tikzpicture}
		\begin{tikzpicture}[scale=0.7]
		\foreach \i in {1,2,3} 
		{
		\draw (0,\i*0.6) node [left] {\footnotesize{$\i$}} -- (2,\i*0.6) node [right] {$\ $};
		}
		\foreach \i in {6,7} 
		{
		\draw (0,\i*0.6) -- (2,\i*0.6);
		}
		\node at (0,4.5*0.6) [left] {\footnotesize{$i$}};
		\node at (0,6*0.6) [left] {\footnotesize{$n-1$}};
		\node at (0,7*0.6) [left] {\footnotesize{$n$}};
		\node at (0,7.5*0.6) [] {$\quad$};
		\node at (1,3.9*0.6) [] {$\vdots$};
		\node at (1,5.5*0.6) [] {$\vdots$};
		\draw (1,4.5*0.6) -- (0,4.5*0.6);
		\draw [fill=black] (1,4.5*0.6) circle [radius=0.1];
		\end{tikzpicture}
		\end{tabular}
		\caption{The three types of elementary columns in a GP-graph.}
		\label{fig:ElementaryRegions}
	\end{figure}
\end{center}

\noindent We label the horizontal $\bG$-edges in Type 1 and 2 columns by consecutively increasing natural numbers from bottom to top. The horizontal lines of a Type 3 column are labeled in a similar way, but using the right side of the column in the case of a white lollipop and using the left side of the column in the case of a black lollipop. Without loss of generality, we always assume that there is a Type 1 column on each side of a column of Type 2 or 3.

Let $S_\N$ be the (infinite) group of permutations on the set $\N$. It is generated by simple transpositions $s_i=(i,i+1)$, $i\in\N$. Within $S_\N$, we define $S_{[a,b]}\cong S_{b-a+1}$ to be subgroup consisting of bijections that map $i$ back to itself for all $i\notin [a,b]$. As we scan from left to right across the elementary columns of $\bG$, we associate a copy of $S_{[a,b]}$ for some $[a,b]$ with each column of Type 1 or 2 via these rules: 
\begin{itemize}\setlength\itemsep{0.5em}
    \item[-] We start with the empty set before the leftmost white lollipop, and we associate $S_{[1,1]}$ with the Type 1 column right after the leftmost white lollipop.
    \item[-] The symmetric group $S_{[a,b]}$ does not change as we scan through a Type 1 or 2 column.
    \item[-] If the symmetric group is $S_{[a,b]}$ before a Type 3 column with a white lollipop, then the symmetric group after this Type 3 column is $S_{[a,b+1]}$.
    \item[-] If the symmetric group is $S_{[a,b]}$ before a Type 3 column with a black lollipop, then the symmetric group after this Type 3 column is $S_{[a+1,b]}$.
\end{itemize}

In summary, when passing through a white lollipop we move from a copy of $S_k$ to a copy of $S_{k+1}$ by adding a simple transposition at the end (with a larger subindex), and when passing through a black lollipop we move from a copy of $S_{k+1}$ to a copy of $S_{k}$ by dropping the first transposition (with smaller subindex).

\subsection{Sugar-free Hulls}\label{ssec:hulls} By definition, a {\it face} of a GP-graph $\bG$ is any bounded connected component of $\R^2\setminus\bG$. A face is said to contain a lollipop if its closure in $\R^2$ contains a univalent vertex of $\bG$. A {\it region} of a GP-graph $\bG$ is a union of faces whose closure in $\R^2$ is connected; in particular, a face is a region and the union of any pair of adjacent faces is a region. For instance, the yellow and red area depicted in Figure \ref{fig:ElementaryRegions_ExampleIntro} are both faces and the yellow face contains a lollipop; their union is a region (which will be the sugar-free hull of the yellow face).

The {\it boundary} $\dd R$ of a region $R$ is the topological (PL-smooth) boundary of its closure $\overline{R}\sse\R^2$. The boundary $\dd R$ of a region necessarily consists of straight line segments meeting at corners that have either $90^\circ$ or $270^\circ$ angles. By definition, a $270^\circ$ corner is said to be \emph{left-pointing} if it is of the form $\begin{tikzpicture}
\draw (0,0) -- (0,0.5) -- (0.5,0.5);
\end{tikzpicture}$ or $\begin{tikzpicture}
\draw (0,0.5) -- (0,0) -- (0.5,0);
\end{tikzpicture}$, and a $270^\circ$ corner is said to be \emph{right-pointing} if it is of the form $\begin{tikzpicture}
\draw (0,0.5) -- (0.5,0.5) -- (0.5,0);
\end{tikzpicture}$ or $\begin{tikzpicture}
\draw (0.5,0.5) -- (0.5,0) -- (0,0);
\end{tikzpicture}$. Equipped with this terminology, we introduce the following notion:

\begin{definition}\label{def:sugarfree} Given a grid plabic graph $\bG$, a region $R$ is said to be \emph{sugar-free} if all left-pointing $270^\circ$ corners along $\partial R$ are white and all right-pointing $270^\circ$ corners along $\partial R$ are black. See Figure \ref{fig:SugarFree_Definition} for a picture with the allowed (and disallowed) corners. The \emph{sugar-free hull} $\S(f)$ of a face $f$ of a plabic graph $\bG$ is defined to be the intersection of all sugar-free regions $R$ containing $f$. In particular, sugar-free hulls are sugar-free regions. 
\end{definition}

\begin{center}
	\begin{figure}[H]
		\centering
		\includegraphics[scale=0.7]{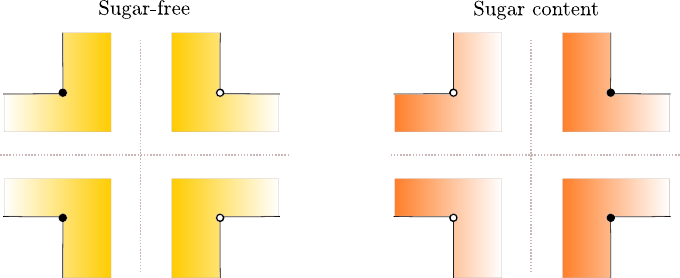}
		\caption{The four corners depicted on the left, in yellow, are allowed in a sugar-free region. The four corners depicted on the right, in orange, are not allowed in a sugar-free region, they have sugar content.}
		\label{fig:SugarFree_Definition}
	\end{figure}
\end{center}

\noindent The boundary of a sugar-free region has the following characterization, which follows immediately from the fact that all vertical bars must be of different colors at the two ends:

\begin{lemma}\label{lemma:staircase} Let $R$ be a sugar-free region in a GP-graph $\bG$. Then $\partial R$ must be decomposed as a concatenation of staircases of the following four types:
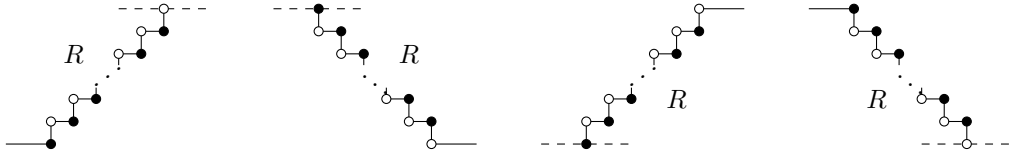
\begin{figure}[H]
\begin{tikzpicture}[scale=0.6]
\draw (-0.5,0) -- (0.5,0) -- (0.5,0.5) -- (1,0.5) -- (1,1) -- (1.5,1) -- (1.5,1.25);
\node at (1.75,1.5) [] {$\Ddots$};
\draw (2,1.75) -- (2,2) -- (2.5,2) -- (2.5,2.5) -- (3,2.5) -- (3,3);
\draw [dashed] (2,3) -- (4,3);
\foreach \i in {1,2,3,5,6}
{
    \draw [fill=black] (\i*0.5,\i*0.5-0.5) circle [radius=0.1];
}
\foreach \i in {1,2,4,5,6}
{
    \draw [fill=white] (\i*0.5,\i*0.5) circle [radius=0.1];
}
\node at (1,2) [] {$R$};
\end{tikzpicture} \quad \quad 
\begin{tikzpicture}[scale=0.6]
\draw (0.5,0) -- (-0.5,0) -- (-0.5,0.5) -- (-1,0.5) -- (-1,1) -- (-1.5,1) -- (-1.5,1.25);
\node at (-1.75,1.5) [] {$\ddots$};
\draw (-2,1.75) -- (-2,2) -- (-2.5,2) -- (-2.5,2.5) -- (-3,2.5) -- (-3,3);
\draw [dashed] (-4,3) -- (-2,3);
\foreach \i in {1,2,3,5,6}
{
    \draw [fill=white] (-\i*0.5,\i*0.5-0.5) circle [radius=0.1];
}
\foreach \i in {1,2,4,5,6}
{
    \draw [fill=black] (-\i*0.5,\i*0.5) circle [radius=0.1];
}
\node at (-1,2) [] {$R$};
\end{tikzpicture} \quad \quad
\begin{tikzpicture}[scale=0.6]
\draw (0.5,0) -- (0.5,0.5) -- (1,0.5) -- (1,1) -- (1.5,1) -- (1.5,1.25);
\node at (1.75,1.5) [] {$\Ddots$};
\draw (2,1.75) -- (2,2) -- (2.5,2) -- (2.5,2.5) -- (3,2.5) -- (3,3)--(4,3);
\draw [dashed] (-0.5,0) -- (1.5,0);
\foreach \i in {1,2,3,5,6}
{
    \draw [fill=black] (\i*0.5,\i*0.5-0.5) circle [radius=0.1];
}
\foreach \i in {1,2,4,5,6}
{
    \draw [fill=white] (\i*0.5,\i*0.5) circle [radius=0.1];
}
\node at (2.5,1) [] {$R$};
\end{tikzpicture} \quad \quad
\begin{tikzpicture}[scale=0.6]
\draw (-0.5,0) -- (-0.5,0.5) -- (-1,0.5) -- (-1,1) -- (-1.5,1) -- (-1.5,1.25);
\node at (-1.75,1.5) [] {$\ddots$};
\draw (-2,1.75) -- (-2,2) -- (-2.5,2) -- (-2.5,2.5) -- (-3,2.5) -- (-3,3)--(-4,3);
\draw [dashed] (-1.5,0) -- (0.5,0);
\foreach \i in {1,2,3,5,6}
{
    \draw [fill=white] (-\i*0.5,\i*0.5-0.5) circle [radius=0.1];
}
\foreach \i in {1,2,4,5,6}
{
    \draw [fill=black] (-\i*0.5,\i*0.5) circle [radius=0.1];
}
\node at (-2.5,1) [] {$R$};
\end{tikzpicture}
\caption{Four types of staircase building blocks for the boundary $\dd R$ of a sugar-free region $R\sse\bG$. In each instance, the letter $R$ marks the location of the region in the plane. The dashed lines indicate that $\partial R$ can continue in either of the two branches}
\label{fig:staircase}
\end{figure}
\end{lemma}

\begin{lemma}\label{lemma:single component} Let $\bG$ be a GP-graph, $R\sse\bG$ be a sugar-free region and $C$ a column in $\bG$ of any type. Then the intersection $R\cap C$ has at most one connected component.
\end{lemma}
\begin{proof} By definition, the region $R$ is connected. Thus, in order for the intersection $R\cap C$ to have more than one connected component, $R$ needs to make a (horizontal) U-turn at some point and $\partial R$ must contain a part that is of the shape ``$R \ ($'' or ``$) \ R$'', where the parentheses indicate the U-turn and the letter $R$ indicates the side of the region. However, such a shape cannot be built using the four types of staircases in Lemma \ref{lemma:staircase} and therefore $R\cap C$ can have at most one connected component.
\end{proof}

Note that if a region $R$ were not simply-connected, then there must exist a column $C$ such that $R\cap C$ has more than one connected component. Thus, Lemma \ref{lemma:single component} has the following consequence, despite the fact that there may exist non-simply connected faces in the GP-graph.

\begin{cor}\label{cor:sugar-free hull simply-connected} Sugar-free regions are simply-connected.
\end{cor}

\subsection{Legendrian Links}\label{ssec:legendrianlink} In this subsection we introduce the Legendrian link $\La(\bG)\sse(\R^3,\xi_\st)$ associated to a GP-graph $\bG\sse\R^2$ and explain how to algorithmically draw a specific front by scanning $\bG$ left to right. Let us begin with the concise definition of $\La(\bG)$, which reads as follows:

\begin{definition}\label{def:Legendrianlink}
Let $\bG\sse\R^2$ be a GP-graph. The Legendrian link $\La(\bG)\sse(\R^3,\xi_\st)$ is the Legendrian lift of the alternating strand diagram of $\bG$, understood as a co-oriented front in $\R^2$, considered inside a Darboux ball in $(T^*_\infty\R^2,\xi_\st)$.
\end{definition}

\noindent Alternating strand diagrams were introduced in \cite[Definition 14.1]{Postnikov} for a reduced plabic graph. In general, we associate such diagrams to a GP-graph $\bG\sse\R^2$ according to the two following local models, where the hairs indicate the co-orientation:
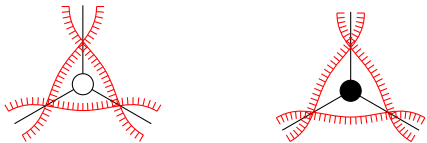
\begin{figure}[H]
    \centering
    \begin{tikzpicture}[scale=0.7]
    \foreach \i in {0,1,2}
    {
        \draw (0,0) -- (90+\i*120:1.5);
        \def\PTH{(80+\i*120:1.5) to [out=-90+\i*120,in=60+\i*120] (150+\i*120:0.5) to [out=-120+\i*120,in=30+\i*120] (-140+\i*120:1.5)}
        \def\tans{30}
        \draw[red] \PTH;
        \foreach \j in {0,...,\tans}
        {
        \path[tangent=\j/\tans] \PTH;
        \draw[red, use tangent=1] (0,0) -- (0,0.1);
        }
    }
    \draw[fill=white] (0,0) circle [radius=0.2];
    \end{tikzpicture}
    \quad \quad \quad \quad
    \begin{tikzpicture}[scale=0.7]
    \foreach \i in {0,1,2}
    {
        \draw (0,0) -- (90+\i*120:1.5);
        \def\PTH{(80+\i*120:1.5) to [out=-90+\i*120,in=60+\i*120] (150+\i*120:0.5) to [out=-120+\i*120,in=30+\i*120] (-140+\i*120:1.5)}
        \def\tans{30}
        \draw[red] \PTH;
        \foreach \j in {0,...,\tans}
        {
        \path[tangent=\j/\tans] \PTH;
        \draw[red, use tangent=1] (0,0) -- (0,-0.1);
        }
    }
    \draw[fill=black] (0,0) circle [radius=0.2];
    \end{tikzpicture}
    \caption{The local models for an alternating strand diagram associated to a GP-graph $\bG$. The small hairs indicate the co-orienting direction, which is needed to specify a Legendrian lift.}
    \label{fig:conormal orientation}
\end{figure}
\noindent The alternating strand diagram near a lollipop (or a bivalent vertex) is the same as in \cite{Postnikov}, and the co-orientation in these pieces is implied by the co-orientations above.

By definition, the Legendrian lift of a co-oriented immersed curve on the plane $\R^2$ is a Legendrian link inside the ideal contact boundary $(T_\infty^*\R^2,\xi_\st)$. The contact structure is the kernel of the restriction of the Liouville 1-form on $T^*\R^2$ to this hypersurface. In general, such Legendrian links cannot be contained in a Darboux ball, but for a GP-graph $\bG$, the Legendrian lift $\La(\bG)$ is naturally contained in a Darboux ball, as we now explain. Let us choose Cartesian coordinates $(u,v)\in\R^2$. Then, the contact structure on $T^*_\infty\bR^2_{u,v}$ can be identified as the kernel of the contact 1-form $\alpha_\st:=\cos\theta du+\sin \theta dv,$
where $\theta\in[0,2\pi)$ is the angle between a given covector $adu+bdv$ and $du$, $a,b\in\R$ and $a^2+b^2\neq0$. Note that $T^*_\infty\R^2$ is diffeomorphic to $\R^2\times S^1$ and $\theta\in S^1$ records that circle coordinate. In fact, we can consider the 1-jet space $(J^1S^1,\xi_\st)$ with its standard contact structure $\ker\{\beta_\st\}$, $\beta_\st:=dz-yd\theta$, where $y\in\mathbb{R}$ is the coordinate along the cotangent fiber and $z\in\R$ the Reeb coordinate, as $J^1S^1:=T^*S^1\times\R$. Then, there exists a strict contactomorphism $\varphi:(T^\infty\bR^2,\a_\st)\rightarrow (J^1S^1,\beta_\st)$ given by
\[
\varphi^*(\theta)=\theta, \quad \quad \varphi^*(y)=-u\sin\theta+v\cos \theta, \quad \quad \varphi^*(z)=u\cos \theta+v\sin \theta.
\]
For any open interval $I\sse S^1$, $(J^1 I,\xi_\st)$ is contactomorphic to a standard Darboux ball $(\R^3,\xi_\st)$. In consequence, if a co-oriented immersed curve $\mathfrak{f}\sse\R^2$ in $\R^2$ has a Gauss map that misses one given angle $\theta_0$, the Legendrian lift of $\mathfrak{f}\sse\R^2$ is contained in $(J^1(S^1\setminus\theta_0),\xi_\st)$, which is contactomorphic to a Darboux ball. This happens for the alternating strand diagram of a GP-graph $\bG\sse\R^2$ and thus $\La(\bG)$ naturally lives inside a Darboux ball.

Let us now construct a particular type of (wave)front for the Legendrian link $\La(\bG)$ which is useful to describe our moduli spaces in Lie-theoretic terms. For that, we consider the front $\mathfrak{f}(\bG)\sse\R^2$ obtained by dividing the GP-graph $\bG$ into elementary columns and then use the assignments as depicted in Figures \ref{fig:RulesFronts1} and \ref{fig:RulesFronts2}. Namely, to an elementary column of Type 1 with $n$ strands, we assign a front consisting of $2n$ parallel horizontal strands. For an elementary column of Type 2 with $n$ strands and a vertical bar at the $i$th position, we assign a front consisting of $2n$ parallel horizontal strands with a crossing at the $i$th position either at the top $n$ strands or the bottom $n$ strands, depending on whether the vertical bar had a white vertex at the top or at the bottom. Figure \ref{fig:RulesFronts1} depicts these three cases, with Types 1 and 2. The case of an elementary column of Type 3 involves inserting a right (resp.~left) cusp at the $i$th position (plus some additional crossings) if there is a black (resp.~white) lollipop inserted at the $i$th position. Figure \ref{fig:RulesFronts2} depicts the two possible cases for a Type 3 column.

\begin{center}
	\begin{figure}[h!]
		\centering
		\includegraphics[width=\textwidth]{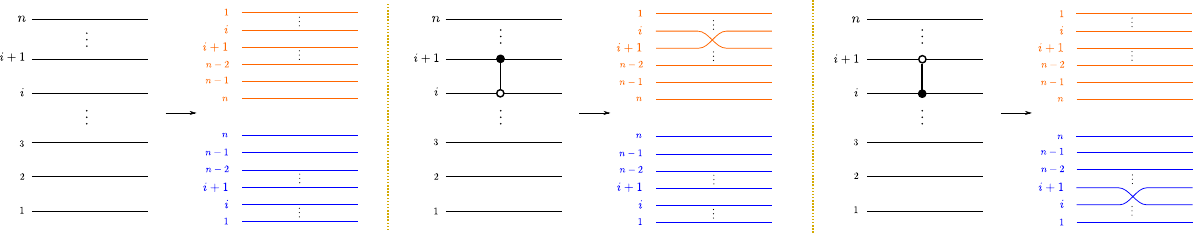}
		\caption{The rules to construct the front $\mathfrak{f}(\bG)$ from the elementary columns of a GP-graph $\bG$. In this case, Type 1 and Type 2 columns are depicted, with the GP-graph $\bG$ on the left and the front $\mathfrak{f}(\bG)$ on the right. We have colored the top $n$-strands of the front in orange and the bottom $n$-strands of the front in blue for clarification purposes.}
		\label{fig:RulesFronts1}
	\end{figure}
\end{center}

\begin{center}
	\begin{figure}[h!]
		\centering
		\includegraphics[scale=0.9]{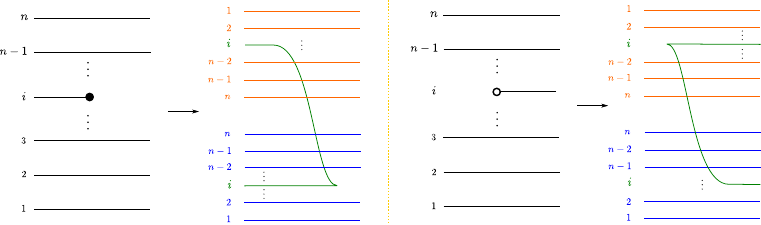}
		\caption{The rules to construct the front $\mathfrak{f}(\bG)$ from the elementary columns of a GP-graph $\bG$. In this case, the two kinds of Type 3 columns are depicted, with the GP-graph $\bG$ on the left and the front $\mathfrak{f}(\bG)$ on the right. We have colored the top $n$-strands of the front in orange, the bottom $n$-strands of the front in blue, and the newly inserted strand with a cusp in green, to help visualize the front.}
		\label{fig:RulesFronts2}
	\end{figure}
\end{center}

\noindent Note that the $2n$ strands in the front are labeled in a specific manner in Figures \ref{fig:RulesFronts1} and \ref{fig:RulesFronts2}, starting the count from the outer strand and increasing towards the middle. This choice of labeling is the appropriate one: in this way, when only left cusps have appeared, which is always the case at the beginning if we read $\bG$ left to right, the $i$th top strand (in orange) and the $i$th bottom strand (in blue) coincide. Now, the front $\mathfrak{f}(\bG)\sse\R^2$ lifts to a Legendrian link $\La(\mathfrak{f}(\bG))\sse(\R^3,\xi_\st)$. We observe that in this case, the lift can be considered directly into $\R^3$, as the front is co-oriented upwards and there are no vertical tangencies. The following proposition follows by applying the above contactomorphism $\varphi:(T^\infty\bR^2,\a_\st)\rightarrow (J^1S^1,\beta_\st)$:

\begin{prop}\label{prop:sameLegendrian}
Let $\bG\sse\R^2$ be a GP-graph. Then the two Legendrian links $\La(\bG)\sse(\R^3,\xi_\st)$ and $\La(\mathfrak{f}(\bG))\sse(\R^3,\xi_\st)$ are Legendrian isotopic.
\end{prop}


\subsection{Instances of GP-graphs \texorpdfstring{$\bG$ and their Legendrian Links $\La(\bG)$}{}}\label{sssec:examplesGPgraph} In this subsection we discuss a few examples of GP-graphs $\bG$ that lead to particularly interesting and well-studied Legendrian links.\\

\begin{center}
	\begin{figure}[h!]
		\centering
		\includegraphics[scale=0.7]{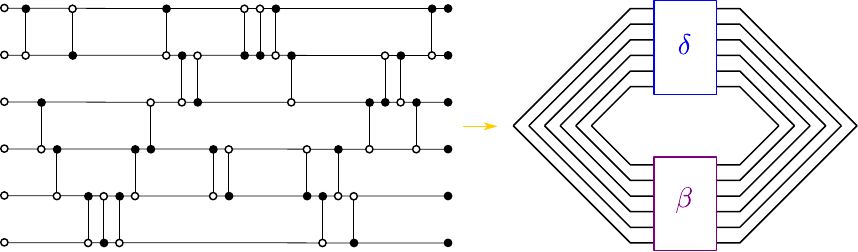}
		\caption{A front for the Legendrian link associated to the GP-graph on the left is drawn on the right, where $\beta,\delta\in\mbox{Br}^+_6$ are the positive braid words $\beta=\sigma_5\sigma_1\sigma_3\sigma_4\sigma_3\sigma_5^2\sigma_2\sigma_1\sigma_4$ and $\delta=\sigma_1\sigma_3\sigma_4\sigma_5^2\sigma_4\sigma_1\sigma_2\sigma_4\sigma_1\sigma_2\sigma_5\sigma_4\sigma_3\sigma_2\sigma_3\sigma_1$.}
		\label{fig:PlabicFence_Example2}
	\end{figure}
\end{center}

{\bf Plabic fences.} Consider a GP-graph $\bG\sse\R^2$ whose white lollipops all belong to the line $\{-1\}\times\R$, and all black lollipops belong to the line $\{1\}\times\R$. Figures \ref{fig:PlabicFence_Example2} depict instances of such GP-graphs. These GP-graphs are called plabic fences in \cite[Section 12]{FPST}, following L. Rudolph's fence terminology. It follows from Proposition \ref{prop:sameLegendrian}, and the rules from Figures \ref{fig:RulesFronts1} and \ref{fig:RulesFronts2}, that the Legendrian link $\La(\bG)$ associated to a plabic fence $\bG\sse\R^2$ is Legendrian isotopic to the (Legendrian lift of the) rainbow closure of a positive braid. In fact, given such a plabic fence $\bG\sse\R^2$ with $n$ horizontal lines, consider the positive braid word $\beta\in\mbox{Br}_n^+$ whose $k$th crossing is $\sigma_j$ if and only if the $k$th vertical edge {\it with black on bottom} of $\bG$ (starting from the left) is between the $j$th and $(j+1)$st horizontal strands. Similarly, consider the positive braid word $\delta$ whose $m$th crossing is $\sigma_{n-j}$ if and only if the $m$th vertical edge {\it with white on bottom} of $\bG$ is between the $j$th and $(j+1)$st horizontal strands. Then, Figure \ref{fig:PlabicFence_Example2} (right) depicts a front for the Legendrian link $\La(\bG)$, which is readily homotopic to the rainbow closure of the positive braid word $\beta\delta^\circ$ (or equivalently $\delta^\circ\beta$), where $\delta^\circ$ denotes the reverse positive braid of $\delta$.

\begin{center}
	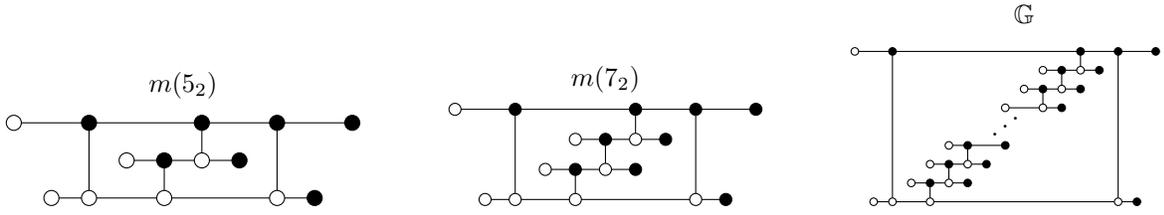
\begin{figure}[h!]
		\centering
		\begin{tikzpicture}
		\horline[](0,3.5,0);
		\horline[](-0.5,4,1);
		\horline[](1,2.5,0.5);
		\vertbar[](0.5,0,1);
		\vertbar[](1.5,0,0.5);
		\vertbar[](2,0.5,1);
		\vertbar[](3,0,1);
		\node at (1.75,1.5) [] {$m(5_2)$};
		\end{tikzpicture} \quad \quad \quad \begin{tikzpicture}[scale=0.8]
		\horline[](0,4,0);
		\horline[](-0.5,4.5,1.5);
		\horline[](1,2.5,0.5);
		\horline[](1.5,3,1);
		\vertbar[](0.5,0,1.5);
		\vertbar[](1.5,0,0.5);
		\vertbar[](2,0.5,1);
		\vertbar[](2.5,1,1.5);
		\vertbar[](3.5,0,1.5);
		\node at (2,2) [] {$m(7_2)$};
		\end{tikzpicture} \quad \quad \quad
		\begin{tikzpicture}[scale=0.5]
		\horline[](0,7,0);
		\horline[](-0.5,7.5,4);
		\horline[](1,2.5,0.5);
		\horline[](1.5,3,1);
		\horline[](2,3.5,1.5);
		\node at (3.5,2) [] {$\Ddots$};
		\horline[](3.5,5,2.5);
		\horline[](4,5.5,3);
		\horline[](4.5,6,3.5);
		\vertbar[](0.5,0,4);
		\vertbar[](1.5,0,0.5);
		\vertbar[](2,0.5,1);
		\vertbar[](2.5,1,1.5);
		\vertbar[](4.5,2.5,3);
		\vertbar[](5,3,3.5);
		\vertbar[](5.5,3.5,4);
		\vertbar[](6.5,0,4);
		\node at (4,5) [] {$\bG$};
		\end{tikzpicture}
		\caption{GP-graphs whose alternating strand diagrams have the smooth type of (mirrors of) twist knots. The GP-graph on the left (resp. middle) yields a max-tb Legendrian representative of $m(5_2)$ (resp. $m(7_2)$). The right picture illustrates the general case, where a GP-graph is built by iteratively inserting -- in a staircase manner -- the local piece inside the GP-graph in the upper left.}
		\label{fig:Example_Twist}
	\end{figure}
\end{center}

{\bf Legendrian twist knots.} Let us consider the family of GP-graphs $\bG_n$, indexed by $n\in\N$, that we have depicted in Figure \ref{fig:Example_Twist}. Each GP-graph $\bG_n$ has two long horizontal bars and it is obtained by inserting a staircase with $n$ steps between two vertical bars, themselves located at the leftmost and rightmost position. Figure \ref{fig:Example_Twist} draws $\bG_1$ (left) and $\bG_2$ (middle). By using Figures \ref{fig:RulesFronts1} and \ref{fig:RulesFronts2}, fronts for the associated Legendrian knots $\La(\bG_n)$ are readily drawn: Figure \ref{fig:Example_TwistFront} depicts fronts for $\La(\bG_1)$ and $\La(\bG_2)$. In general, we conclude that $\La(\bG_n)$ is a max-tb Legendrian representative of a twist knot, with zero rotation number. Note that Legendrian twist knots are classified in \cite{EtnyreNgVertesi13}. In particular, the Legendrian knot $\La(\bG_1)$ associated to the GP-graph depicted in the left picture of Figure \ref{fig:Example_Twist} is the unique max-tb Legendrian representative of $m(5_2)$ with a binary Maslov index. This is one half of the well-known Chekanov pair.\footnote{The other max-tb representative of $m(5_2)$ is not isotopic to $\La(\bG)$ for any GP-plabic graph $\bG$.}

\begin{center}
	\begin{figure}[h!]
		\centering
		\includegraphics[scale=0.9]{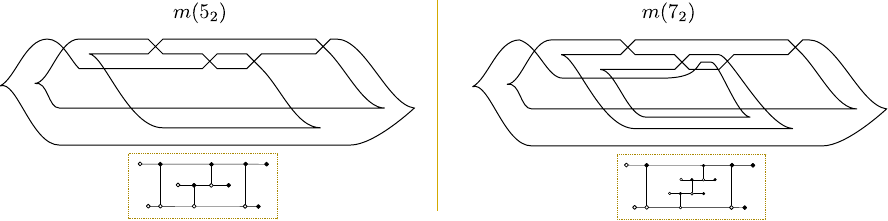}
		\caption{The two Legendrian fronts obtained from the left and middle GP-graphs of Figure \ref{fig:Example_Twist}, according to our recipe translating from GP-graphs to Legendrian front diagrams. The corresponding GP-graphs are drawn in a small yellow box below each front. The front diagram for the general case (right picture in Figure \ref{fig:Example_Twist}) is readily inferred from these two pictures: a knotted spiralling pattern is iteratively added to the center region of the front.}
		\label{fig:Example_TwistFront}
	\end{figure}
\end{center}

{\bf Shuffle graphs.} Let us introduce a class of GP-graphs which leads to interesting examples.

\begin{definition}\label{def:shufflegraph}
A GP-graph $\bG\sse\R^2$ with $n$ horizontal lines is said to be a \emph{shuffle graph} if
\begin{enumerate}
    \item $\exists M\in\N,\sigma\in S_n$ such that each horizontal line goes from $(-M\sigma(i),i)$ to $(M\sigma(i),i)$;
    \item Vertical edges are all of the same pattern, i.e. they either all have a black vertex on top or they all have a white vertex on top.
\end{enumerate}
\end{definition}

The two families above, plabic fences and the GP-graphs in Figure \ref{fig:Example_Twist} for Legendrian twist knots, are instances of shuffle graphs. Shuffle graphs $\bG\sse\R^2$ have the property that the Legendrian $\La(\bG)\sse(\R^3,\xi_\st)$ is Legendrian isotopic to the Legendrian lift of the $(-1)$-closure of a positive braid of the form $\beta\Delta$, where $\Delta\in \Br_n^+$ is the half-twist and $\beta\in\Br_n^+$ has Demazure product $\Dem(\beta)=w_0=w_{0,n}\in S_n$. (By \cite{CGGS,CasalsNg} the condition $\Dem(\beta)=w_0\in S_n$ is necessary.) E.g.~it is a simple exercise to verify that any $(-1)$-closure of a 3-stranded $\beta\Delta$, where $\beta,\Delta\in \Br_3^+$ and $\Dem(\beta)=w_0\in S_3$, arises as $\La(\bG)$ for some shuffle graph $\bG$. In view of this and \cite{CGGS,CGGS2,CasalsNg}, we refer to a positive braid $\beta\in \mbox{Br}_n^+$ as \emph{$\Delta$-complete} if it is cyclically equivalent to a positive braid of the form $\Delta \gamma$ where $\Delta$ is the half twist and $\Dem(\gamma)=w_{0,n}$.


Let us point out two properties of the Legendrian links $\La(\bG)$ that are useful. First, as we will explain, the Legendrian links $\La(\bG)$ always bound an orientable exact embedded Lagrangian filling in the symplectization of $(\R^3,\xi_\st)$, and thus in the standard symplectic Darboux 4-ball. In particular, their Thurston-Bennequin invariant is always maximal and their rotation number vanishes. Second, it follows from the discussion in Subsection \ref{ssec:legendrianlink}, especially Figures \ref{fig:RulesFronts1} and \ref{fig:RulesFronts2}, that $\La(\bG)$ admits a binary Maslov index and that the smooth type of $\La(\bG)$ is that of the $(-1)$-closure of a positive braid. The former is particularly useful for us, as this implies that complexes of sheaves with singular support in $\La(\bG)$ are quasi-isomorphic to sheaves (concentrated in degree 0) and it is the possible to parametrize the moduli of objects of the appropriate dg-category by an affine variety (or algebraic quotient thereof). Subsection \ref{ssec:sheaves} sets up the necessary ingredients on the microlocal theory of sheaves as it relates to these Legendrian links $\La(\bG)$.

\subsection{Lollipop Chain Reaction}\label{ssec:lollipop chain reaction} In this subsection we introduce an algorithmic procedure, called a \emph{lollipop chain reaction}, which aims to select faces for a sugar-free hull. The lollipop chain reaction initiates at a face $f$, and produces a collection of faces that are guaranteed to be inside the sugar-free hull $\S_f$. In many interesting cases of $\bG$, such as shuffle graphs, this procedure yields the entire sugar-free hull $\S_f$. These combinatorial tools are used in Subsection \ref{ssec:relative_position_flags}, in the proof of Proposition \ref{prop:relative position scanning}. Let us start with the definition of a single \emph{lollipop reaction}:

\begin{definition} \label{def:lollipop reaction} Let $w$ be a white lollipop in a GP-graph $\bG$, and let $h_1$ and $h_2$ be the two adjacent horizontal $\bG$-edges to the immediate left of $w$ (in between which the lollipop appears). A vertical line segment between $h_1$ and $h_2$ is said to be a \emph{wall}. By definition, the \emph{lollipop reaction} initiated from the lollipop $w$ pushes this wall to the right along $\bG$ with the following rules: the wall shrinks or expands according to the following five pictures and otherwise the wall stays between the same $\bG$-edges.
\[
{\color{black}
\begin{tikzpicture}[scale=0.7]
\draw (0,0) node [left] {$h_1$} -- (2,0);
\draw (0,2) node [left] {$h_2$}-- (2,2);
\draw (1.2,1) --(2,1);
\draw [fill=white] (1.2,1) circle [radius =0.1];
\node at (1.2,1) [left] {$w$};
\draw [<->,red] (0.5,0) -- (0.5,2);
\node at (1,-1) [] {wall starts};
\node [red] at (1,1.5) [] {$\rightsquigarrow$};
\end{tikzpicture}
}
\quad \quad \quad 
\begin{tikzpicture}[scale=0.7]
\draw (0,2) -- (2,2);
\draw (0,1) -- (2,1);
\draw (0,0) -- (2,0);
\vertbar[](1,1,2);
\draw[<->, red] (0.5,-0.5) -- (0.5,2);
\draw[<->, red] (1.5,-0.5) -- (1.5,1);
\node[red] at (1,-0.5) [] {$\rightsquigarrow$};
\node at (1,-1) [] {wall shrinks};
\end{tikzpicture}\quad \quad \quad 
\begin{tikzpicture}[scale=0.7]
\draw (0,2) -- (2,2);
\draw (0,1) -- (2,1);
\draw (0,0) -- (2,0);
\vertbar[](1,1,0);
\draw[<->, red] (0.5,0) -- (0.5,2.5);
\draw[<->, red] (1.5,1) -- (1.5,2.5);
\node[red] at (1,2.5) [] {$\rightsquigarrow$};
\node at (1,-1) [] {wall shrinks};
\end{tikzpicture}\quad \quad \quad
\begin{tikzpicture}[scale=0.7]
\draw (0,2) -- (2,2);
\draw (0,1) -- (1,1);
\draw (0,0) -- (2,0);
\draw [fill=black] (1,1) circle [radius=0.1];
\draw [<->, red] (0.5,-0.5) -- (0.5,1);
\draw [<->, red] (1.5,-0.5) -- (1.5,2);
\node[red] at (1,-0.5) [] {$\rightsquigarrow$};
\node at (1,-1) [] {wall expands};
\end{tikzpicture}\quad \quad \quad
\begin{tikzpicture}[scale=0.7]
\draw (0,2) -- (2,2);
\draw (0,1) -- (1,1);
\draw (0,0) -- (2,0);
\draw [fill=black] (1,1) circle [radius=0.1];
\draw [<->, red] (0.5,2.5) -- (0.5,1);
\draw [<->, red] (1.5,2.5) -- (1.5,0);
\node[red] at (1,2.5) [] {$\rightsquigarrow$};
\node at (1,-1) [] {wall expands};
\end{tikzpicture}
\]
A single lollipop reaction initiated from a black lollipop $b$ is defined in a symmetric fashion: start with a wall going between the two adjacent horizontal lines to the right of $b$, consider a wall between them and scan to the left. For a black lollipop, the wall shrinks or expands as it moves left according to the following five pictures and otherwise stays between the same $\bG$-edges.
\[
{\color{black}
\begin{tikzpicture}[scale=0.7]
\draw (0,0) node [left] {$h_1$} -- (2,0);
\draw (0,2) node [left] {$h_2$}-- (2,2);
\draw (0.8,1) --(0,1);
\draw [fill=black] (0.8,1) circle [radius =0.1];
\node at (0.8,1) [right] {$b$};
\draw [<->,red] (1.5,0) -- (1.5,2);
\node at (1,-1) [] {wall starts};
\node [red] at (1,1.5) [] {$\leftsquigarrow$};
\end{tikzpicture}
}
\quad \quad \quad 
\begin{tikzpicture}[scale=0.7]
\draw (0,2) -- (2,2);
\draw (0,1) -- (2,1);
\draw (0,0) -- (2,0);
\vertbar[](1,2,1);
\draw[<->, red] (0.5,-0.5) -- (0.5,1);
\draw[<->, red] (1.5,-0.5) -- (1.5,2);
\node[red] at (1,-0.5) [] {$\leftsquigarrow$};
\node at (1,-1) [] {wall shrinks};
\end{tikzpicture}\quad \quad \quad 
\begin{tikzpicture}[scale=0.7]
\draw (0,2) -- (2,2);
\draw (0,1) -- (2,1);
\draw (0,0) -- (2,0);
\vertbar[](1,0,1);
\draw[<->, red] (0.5,1) -- (0.5,2.5);
\draw[<->, red] (1.5,0) -- (1.5,2.5);
\node[red] at (1,2.5) [] {$\leftsquigarrow$};
\node at (1,-1) [] {wall shrinks};
\end{tikzpicture}\quad \quad \quad
\begin{tikzpicture}[scale=0.7]
\draw (0,2) -- (2,2);
\draw (2,1) -- (1,1);
\draw (0,0) -- (2,0);
\draw [fill=white] (1,1) circle [radius=0.1];
\draw [<->, red] (1.5,-0.5) -- (1.5,1);
\draw [<->, red] (0.5,-0.5) -- (0.5,2);
\node[red] at (1,-0.5) [] {$\leftsquigarrow$};
\node at (1,-1) [] {wall expands};
\end{tikzpicture}\quad \quad \quad
\begin{tikzpicture}[scale=0.7]
\draw (0,2) -- (2,2);
\draw (2,1) -- (1,1);
\draw (0,0) -- (2,0);
\draw [fill=white] (1,1) circle [radius=0.1];
\draw [<->, red] (1.5,2.5) -- (1.5,1);
\draw [<->, red] (0.5,2.5) -- (0.5,0);
\node[red] at (1,2.5) [] {$\leftsquigarrow$};
\node at (1,-1) [] {wall expands};
\end{tikzpicture}
\]
As the wall moves to the right (for a white lollipop) or to the left (for a black lollipop), we select all the faces that this wall scans through. By definition, a lollipop reaction \emph{completes} when the length of the wall becomes zero. The output of a lollipop reaction is selection of faces of the GP-graph which it has scanned through. If the length of the wall becomes infinite (i.e.~going to the unbounded region), then the lollipop reaction is said to be \emph{incomplete}, and it outputs nothing.
\end{definition}
\noindent In order to be effective, these lollipop reactions in general need to be iterated as follows:

\begin{definition} Let $f\sse\bG$ be a face of a GP-graph $\bG$. A \emph{lollipop chain reaction} initiated at $f$ is the recursive face selection procedure obtained as follows. First, select the face $f$. Then, for each of the newly selected faces and each inward-pointing lollipop of this face, run a single lollipop reaction and select new faces (if any).

Since the number of faces in $\bG$ is finite, this process terminates either when no new faces are selected, for which we say that the lollipop chain reaction is \emph{complete}, or when one of the single chain reactions is incomplete, for which we say that the whole lollipop chain reaction is \emph{incomplete}.
\end{definition}

\noindent Note that a single lollipop reaction selects faces that are \emph{minimally needed} to avoid sugar-content corners on the immediate right of a white lollipop or the immediate left of a black lollipop. Therefore the outcome of the lollipop chain reaction initiated from a face $f$ must be contained in the sugar-free hull $\S_f$. In other words, if the lollipop chain reaction initiated from $f$ is incomplete, then $\S_f$ does not exist. On the other hand, when sugar-free hulls $\S_f$ exist, lollipop chain reactions do produce sugar-free hulls for a large family of GP-graphs:

\begin{prop}\label{prop:lollipop chain reaction for shuffles} Let $\bG$ be a shuffle graph and $f$ a face of $\bG$ for which $\S_f$ is non-empty. Then the lollipop chain reaction initiated from $f$ is complete and $\S_f$ coincides with the outcome of this lollipop chain reaction.
\end{prop}
\begin{proof} We observe that in a shuffle graph there cannot be any black lollipop on the left side of a white lollipop, nor can there be any white lollipop on the right side of a black lollipop. Therefore, if there is a  white lollipop inside a face $f$, then the part of the boundary $\partial f$ straightly left of the white lollipop can only consists of a single vertical bar. Similarly, if there is a black lollipop inside a face $f$, then the part of the boundary $\partial f$ straightly right of the black lollipop can only consists of a single vertical bar as well. Now, if the face $f$ does not have any lollipops, then $\S_f=f$, which is also equal to the outcome of the lollipop chain reaction, as required. It remains to consider faces that do have lollipops inside. Without loss of generality let us suppose that only vertical $\bG$-edge of the type $\begin{tikzpicture}[baseline=5]
\vertbar[](0,0.5,0);
\end{tikzpicture}$ appear in $\bG$ and suppose that $f$ contains a white lollipop. Consider the left most white lollipop $w$ of $f$. Then, at the starting point, the wall for this lollipop goes between two adjacent horizontal lines $h_1$ and $h_2$, and $\partial f$ only has a single vertical bar to the left of this wall. Then the lollipop reaction starts moving the wall to the right, and one of the following two situation must occur:
\begin{itemize}
    \item {\it The wall never expands}. If this is the case, the wall must be shrinking towards the top as shown below. The result of the lollipop reaction is sugar-free.
    \[
    \begin{tikzpicture}[scale=0.7]
    \draw (0,0) -- (4,0);
    \draw (0,2) -- (6,2);
    \vertbar[](0.5,2,0);
    \draw (1.5,1) -- (2,1);
    \draw [fill=white] (1.5,1) circle [radius=0.1];
    \node at (1.5,1) [above] {$w$};
    \node at (3,1) [] {$\cdots$};
    \draw (3.5,0.5) -- (4,0.5);
    \vertbar[](3.5,0.5,0);
    \vertbar[](4,1,0.5);
    \node at (4.5,1.25) [] {$\iddots$};
    \draw (5,1.5) -- (5.5,1.5);
    \vertbar[](5,1.5,1);
    \vertbar[](5.5,2,1.5);
    \draw [red,<->] (1,0) -- (1,2);
    \draw [red,<->] (3.25,0) -- (3.25,2);
    \draw [red,<->] (3.75,0.5) -- (3.75,2);
    \draw [red,<->] (5.25,1.5) -- (5.25,2);
    \end{tikzpicture}
    \]
   \item {\it The wall expands at some point}. Note that the wall only expands when it passes through a black lollipop $b$. Let $g$ be the face containing $b$. Then the part of $\partial g$ straightly right of $b$ only consists of a single vertical edge $e$, and hence the rightward scanning must end at $e$. Note that in this case there can be a concavity below the selected faces right before the expansion of the wall.
    \[
    \begin{tikzpicture}[scale=0.8,baseline=20]
    \draw (0,0) -- (4,0);
    \draw (0,2) -- (6,2);
    \vertbar[](0.5,2,0);
    \draw (1.5,1) -- (2,1);
    \draw [fill=white] (1.5,1) circle [radius=0.1];
    \node at (1.5,1) [above] {$w$};
    \node at (3,1) [] {$\cdots$};
    \draw (3.5,0.5) -- (4,0.5);
    \vertbar[](3.5,0.5,0);
    \vertbar[](4,1,0.5);
    \node at (4.5,1.25) [] {$\iddots$};
    \draw (5,1.5) -- (5.5,1.5);
    \vertbar[](5,1.5,1);
    \draw [fill=black] (5.5,1.5) circle [radius=0.1];
    \draw [red,<->] (1,0) -- (1,2);
    \draw [red,<->] (3.25,0) -- (3.25,2);
    \draw [red,<->] (3.75,0.5) -- (3.75,2);
    \draw [red,<->] (5.25,1.5) -- (5.25,2);
    \draw (6,2.5) -- (7.5,2.5);
    \draw (5.5,-0.5) -- (7.5,-0.5);
    \vertbar[](7,2.5,-0.5);
    \draw [red,<->] (6.5,-0.5) -- (6.5,2.5);
    \node at (7,1) [right] {$e$};
    \end{tikzpicture} \quad \quad \text{or} \quad \quad 
    \begin{tikzpicture}[scale=0.8,baseline=20]
    \draw (0,0) -- (4,0);
    \draw (0,2) -- (7.5,2);
    \vertbar[](0.5,2,0);
    \draw (1.5,1) -- (2,1);
    \draw [fill=white] (1.5,1) circle [radius=0.1];
    \node at (1.5,1) [above] {$w$};
    \node at (3,1) [] {$\cdots$};
    \draw (3.5,0.5) -- (4,0.5);
    \vertbar[](3.5,0.5,0);
    \vertbar[](4,1,0.5);
    \node at (4.5,1.25) [] {$\iddots$};
    \draw (5,1.5) -- (5.5,1.5);
    \vertbar[](5,1.5,1);
    \draw [fill=black] (5.5,1.5) circle [radius=0.1];
    \draw [red,<->] (1,0) -- (1,2);
    \draw [red,<->] (3.25,0) -- (3.25,2);
    \draw [red,<->] (3.75,0.5) -- (3.75,2);
    \draw [red,<->] (5.25,1.5) -- (5.25,2);
    \draw (5.5,1) -- (7.5,1);
    \vertbar[](7,2,1);
    \draw [red,<->] (6.5,1) -- (6.5,2);
    \node at (7,1.5) [right] {$e$};
    \end{tikzpicture}
    \]
    Now we start scanning leftward from the right most black lollipop in the face $g$. Note that for a leftward scanning, the wall should be shrinking towards the bottom. Also, due to the ``last in first out'' order on the horizontal lines, the top vertex of the vertical edge $e$ cannot be below the top horizontal line of the previous scanning. If the top vertex of $e$ is above the previous horizontal line, then the bottom vertex of $e$ must be below the previous bottom horizontal line, and the leftward scanning will not stop until it goes back all the way to the beginning point of the previous scanning. On the other hand, if the top vertex of $e$ is on the previous horizontal line, then the bottom vertex of $e$ can be above the previous bottom horizontal line. But then the horizontal line at the bottom of the vertical edge $e$ must extend to the left and meets the staircase of the previous rightward scanning, and that is where the leftward scanning stops. In consequence, the lollipop reaction from the rightmost black lollipop of the face $g$ must fill in the lower concavity of the previous rightward scanning. 
\end{itemize}

\noindent Note that in the second case above, the leftward scanning can also end in two ways, but we can conclude by induction that in the end all concavities will be filled and hence the resulting union must be sugar-free, as required.
\end{proof}

\noindent There are many non-shuffle GP-graphs for which lollipop chain reactions also yield sugar-free hulls, and thus the hypothesis in Proposition \ref{prop:lollipop chain reaction for shuffles} is sufficient but not necessary.


\subsection{Legendrian invariants from the microlocal theory of sheaves}\label{ssec:sheaves} In this subsection we lay out the necessary ingredients of the microlocal theory of constructible sheaves that we shall use in our contact geometric framework. We describe the general setup in Subsection \ref{sssec:sheaf_generalsetup}, based on \cite{KashiwaraSchapira,KashiwaraSchapira_Book,Sheaves1,Sheaves2,STZ_ConstrSheaves}.\footnote{See also S.~Guillermou's notes for his lecture series at the conference ``Symplectic topology, sheaves and mirror symmetry'' at the IMJ-PRG in Paris (2016) and \cite{STWZ}.} Subsection \ref{sssec:sheaf_concrete} discusses the specific simplifications that occur for the Legendrian links $\La(\bG)$ and Subsection \ref{sssec:sheaf_decorated} introduces the necessary decorated version of the moduli stacks being discussed. Let $\mathds{k}$ be a commutative coefficient ring, for us either $\mathds{k}=\Z$ or $\mathds{k}=\C$. Consider a smooth manifold $M$, $\pi_M:T^*M\lr M$ its cotangent bundle and $T^{*}_{\infty} M\lr M$ its ideal contact boundary; we will only need $M=\R^2$ and $\R^3$.

\subsubsection{The general setup}\label{sssec:sheaf_generalsetup} The general results on the microlocal theory of constructible sheaves were pioneered by M.~Kashiwara and P.~Schapira in \cite{KashiwaraSchapira_Book} and, more recently, in collaboration with S.~Guillermou in \cite{Sheaves1}. The first category that we need is defined as follows:

\begin{definition}
The category $\I(\mathds{k}_M)$  is the full dg-subcategory of the dg-category of locally bounded complexes of sheaves of $\mathds{k}$-modules on $M$ which consist of $h$-injective complexes of injective sheaves. The homotopy category of $\I(\mathds{k}_M)$ is denoted by $[\I(\mathds{k}_M)]$.
\end{definition}

The dg-category $\I(\mathds{k}_M)$ is a strongly pretriangulated dg-category and the six functor formalism lifts to this dg-enhancement $\I(\mathds{k}_M)$, see \cite{Schnurer18}. The homotopy category $[\I(\mathds{k}_M)]$ is triangulated equivalent to the locally bounded derived category of sheaves on $M$, often denoted by $D^{lb}(\mathds{k}_M)$. For an object $F\in \I(\mathds{k}_M)$, we denote by $\ssup(F)\sse T^*M$ its singular support understood as an object in $[\I(\mathds{k}_M)]\simeq D^{lb}(\mathds{k}_M)$. The notion of singular support leads to defining the following dg-categories:

\begin{definition}
Let $S\sse T^*M$ be a subset. The category $\I_S(\mathds{k}_M)$  is the subcategory of $\I(\mathds{k}_M)$ consisting of objects $F\in\I(\mathds{k}_M)$ such that $\ssup(F)\sse S$. The category $\I_{(S)}(\mathds{k}_M)$  is the subcategory $\I(\mathds{k}_M)$ consisting of objects $F\in\I(\mathds{k}_M)$ for which there exists an open neighborhood $\Omega$ such that $\ssup(F)\cap\Omega\sse S$.

\noindent Let $\La\sse T^*_\infty M$ be a Legendrian submanifold. We denote by $\I_\La(\mathds{k}_M)$ and $\I_{(\La)}(\mathds{k}_M)$ the categories as above with the choice of subset $S$ being the Lagrangian cone of $\Lambda$ union the zero section $M\subset T^*M$.
\end{definition}

The assignment $U\longmapsto \I(\mathds{k}_U)$ to each open subset $U\sse M$ is a stack of dg-categories. Similarly, the pre-stack $\I_{(\La)}$ defined by
$$\I_{(\La)}(U):=\I_{(T^*U\cap\La)}(\mathds{k}_U),\quad U\sse M \mbox{ open subset},$$
where $U$ is an open subset, is a stack. This is an advantage of using the injective dg-enhancements instead of derived categories. A central result in symplectic topology \cite{Sheaves1} is that the stack $\I_{(\La)}$ on $M$ is a Legendrian isotopy invariant of the Legendrian $\La\sse T^*_\infty M$. There are two constructions associated to the stack $\I_{(\La)}$, as follows.

\begin{itemize}
    \item[(i)] {\bf The microlocal functor $\m_\La$}. The Kashiwara-Schapira stack $\msh(\mathds{k}_\La)$ is the stack on $\La$ associated to the pre-stack
    $$V\longmapsto \I_{(V)}(\mathds{k}_M;V),\quad V\sse \La\mbox{ open subset},$$
    where $\I_{(V)}(\mathds{k}_M;V)$ is the Drinfeld dg-quotient of $\I_{(V)}(\mathds{k}_M)$ by $\I_{T^*M\setminus(V\cup M)}(\mathds{k}_M)$. See \cite{KashiwaraSchapira_Book,Sheaves1}. The quotient functor gives a functor of stacks
    $$\m_\La:\I_{(\La)}\lr \left(\pi_M|_\La\right)_*(\msh(\mathds{k}_\La)).$$
    The use of this functor in our article is twofold: it is used in the case that $\La\sse T^*_\infty\R^2$ is a Legendrian link, and in the case where $\La\sse T^*_\infty\R^3$ is a Legendrian surface obtained as the lift of an exact Lagrangian filling of a Legendrian link.\\

    \item[(ii)] {\bf The moduli stack $\mathcal{M}_{\I_{(\La)}(M)}$}. By \cite[Theorem 3.21]{Nadler16} the global sections $\I_{(\La)}(M)$ is a dg-category equivalent to the category of (pseudo)perfect modules of a finite type category, namely of the category of wrapped constructible sheaves $\Sh^w_\La(M)$ defined in \cite[Definition 3.17]{Nadler16}. Then, the main result of \cite{ToenVaquie07} implies that there exists a locally geometric $D^-$-stack $\mathcal{M}_{\I_{(\La)}(M)}$, locally of finite presentation, which acts as the moduli stack of objects in the dg-category $\I_{(\La)}(M)$.
\end{itemize}

Finally, as explained in \cite[Section 1.7]{JinTreumann} given an embedded exact Lagrangian filling $L\sse T^*M$ of a Legendrian submanifold $\La\sse T^*_\infty M$, the microlocal functor $\m_{\overline{L}}$ applied to the Legendrian lift $\overline{L}\sse J^1(M)$ yields an equivalence of categories between (a subcategory of) ${\I_{(\overline{L})}(M)}$ and the dg-derived category of local systems on $L$. This induces an open inclusion $\iota_L:\R \Loc(L)\lr\mathcal{M}_{\I_{(\La)}(M)}$
where $\R \Loc(L)$ denotes the derived moduli space of local systems on $L$.

\subsubsection{The concrete models.}\label{sssec:sheaf_concrete} For a Legendrian link $\La\sse T^*_\infty\R^2$ with vanishing rotation number, the category of global sections of the Kashiwara-Schapira stack $\msh(k_\La)$ admits a simple object $\Xi$ by \cite[Part 10]{Guillermou19_SheafSummary}. In addition, the functor $\mu\mbox{hom}(\Xi,\cdot)$ is an explicit equivalence between $\msh(k_\La)$ and the (twisted) stack $\Loc_\La$ of (twisted) local systems on $\La$.\footnote{A choice of spin structure on $\La$, and correspondingly choices of spin structures for the Lagrangian fillings we consider, allow a further identification to actual (untwisted) local systems. We implicit have these choices on the background and translate them combinatorially in Section \ref{sec:cluster}, through sign curves, when they are needed to assign signs.} In consequence, the microlocal functor $\m_\La$ described above yields a functor
$$\m_{\La,\Xi}:\I_{(\La)}(M)\lr Loc_\La(\La),$$
where we have considered global sections and identified the codomain of $\m_\La$ with $\Loc_\La$ via $\mu\mbox{hom}(\Xi,\cdot)$ and a choice of spin structure. In addition, given a Legendrian link $\La\sse T^*_\infty\R^2$, we only need to consider the moduli substack $\mathcal{M}_1(\La)$ of $\mathcal{M}_{\I_{(\La)}(\R^2)}$ which is associated to the subcategory of objects in $\I_{(\La)}(\R^2)$ whose image under $\m_\La$ is a local system (on $\La$) of locally free $\mathds{k}$-modules of rank one supported in degree zero. In the case that $\La$ admits a binary Maslov index, the stack $\mathcal{M}_1(\La)$ is equal to its truncation $t_0(\mathcal{M}_1(\La))$, which is an Artin stack.

Now, given an embedded exact Lagrangian filling $L\sse T^*M$ of $\La$, the derived stack $\Loc_{L}$ of local systems on $L$ is also equivalent to its truncation and the open inclusion $\iota_L$ described gives an inclusion
$\iota_L:\Loc_1(L)\lr\mathcal{M}_1(\La)$
of Artin stacks, where $\Loc_1(L)$ are local systems (on $L$) of locally free $\mathds{k}$-modules of rank one supported in degree zero. Since Abelian local systems on $L$ can be parametrized by $H^1(L,\mathds{k}^\times)$, the inclusion $\iota_L$ provides a toric chart $\iota_L(Loc_1(L))$ in the moduli stack $\mathcal{M}_1(\La)$. In the course of the article, we typically consider the ground ring $\mathds{k}=\C$. If we are given a Legendrian link for which the stabilizers of $\mathcal{M}_1(\La)$ are trivial and $\mathcal{M}_1(\La)$ is smooth, then $\mathcal{M}_1(\La)$ is (represented by) a smooth affine variety and an embedded exact Lagrangian filling $L$ of $\La$ yields a toric chart
$\iota_L:(\C^\times)^{2g(L)}\lr\mathcal{M}_1(\La),$
where $g(L)$ is the topological genus of the surface $L$.

Finally, both the inclusions $\iota_L:\Loc_1(L)\lr\mathcal{M}_1(\La)$ and the microlocal functors $\m_\La:\mathcal{M}_1(\La)\lr \Loc_1(\La)$ can be computed explicitly from the front via cones of maps between stalks (of the sheaves parametrized by $\mathcal{M}_1(\La)$. Indeed, we shall use the combinatorial model in \cite[Section 3.3]{STZ_ConstrSheaves}, where the points of $\mathcal{M}_1(\La)$ are parametrized by functors from the poset category associated to the stratification induced by the Legendrian front to the Abelian category of $\mathds{k}$-modules (modulo acyclic complexes). Note that in the case of Legendrian weaves, this combinatorial model is explained in \cite[Section 5]{CasalsZaslow}.

\subsubsection{A decorated moduli space.}\label{sssec:sheaf_decorated} Let $T=\{t_1,\ldots,t_s\}$, $t_i\sse\La$, $i\in[1,s]$, be a set of distinct points in a Legendrian link $\La\sse T^*_\infty\R^2$. The elements of $T$ will be referred to as \emph{marked points}. The moduli stack $\mathcal{M}_1(\La)$ can be decorated with additional trivializing information once a set of marked points $T$ has been fixed, as follows.

\begin{definition}\label{def:decoratedsheaves}
Let $\La\sse T^*_\infty\R^2$ be a Legendrian link with a fixed choice of Maslov potential and spin structure. Consider a set of marked points  $T=\{t_1,\ldots,t_s\}$ and label the components of $\La\setminus T$ by $\La_i$, $i\in\pi_0(\La\setminus T)$. The moduli stack $\fM(\La,T)$ is
$$\fM(\La,T):=\{\left(F;\phi_1,\ldots,\phi_{|\pi_0(\La\setminus T)|}\right): F\in\mathcal{M}_1(\La),\phi_i\mbox{  trivialization of }\m_\La(F)\mbox{ on }\La_i\}.$$
Note that an Abelian local system can always be trivialized over $\La_i$ if $\mathds{k}=\C^*$. For a general ground ring $\mathds{k}$, we require that there exist at least one marked point per component of $\La$.
\end{definition}

In Definition \ref{def:decoratedsheaves}, the identification of (global sections of) the codomain of $\m_\La$ with the stack of local systems is fixed by the choice of Maslov potential and spin structure on $\La$. There are at least two advantages to decorating the moduli stack of sheaves $\mathcal{M}_1(\La)$ to $\fM(\La,T)$. First, introducing the data of the trivializations in $\fM(\La,T)$ often results in a smooth affine variety, even if $\mathcal{M}_1(\La)$ was singular; this is similar to the classical setup with character varieties \cite{FockGoncharov_ModuliLocSys}. Second, the trivializations in $\fM(\La,T)$ can be used to define global regular functions. In fact, we will show that $\fM(\La,T)$ admits a cluster $\mathcal{A}$-structure, and our construction of the cluster $\mathcal{A}$-variables crucially relies on the existence of these decorations. Finally, the moduli space $\mathcal{M}_1(\La,T)$ is defined similarly, by considering sheaves in $\mathcal{M}_1(\La)$ with the additional data of trivializations of the stalks of the associated microlocal local systems at each of the marked points in $T$.

\subsection{A clarification on the notion of cluster structures}\label{ssec:clarificationAX} In the literature, the sentence ``{\it a space $Y$ has a cluster structure}'' has different meanings. We record here the precise definitions that have been used, implicitly or explicitly, and clarify the type of results we obtain. Let $Q$ be a quiver, or more generally a skew-symmetrizable matrix. Consider the following concepts:

\begin{itemize}

\item[-] The cluster algebra $\mathbb{A}_Q$. This is a commutative algebra and it comes endowed with a (typically infinite) system of generators $A_i\in \mathbb{A}_Q$, called the {\it cluster variables}. The vertices of the quiver give some of these cluster variables, and the other cluster variables are produced by the process of mutation. Cluster algebras were first introduced and studied by S.~Fomin and A.~Zelevinsky in the series of works \cite{FominZelevinsky_DoubleBruhat,FominZelevinsky_ClusterI,FominZelevinsky_ClusterII}. The affine scheme associated to $\mathbb{A}_Q$ is $\mbox{Spec}(\mathbb{A}_Q)$.\\

\item[-] The space $\mathcal{A}_Q$, called the cluster $\mathcal{A}$-space or cluster $\mbox{K}_2$-space, is a scheme obtained by birationally gluing certain tori according to $Q$. The ring of regular functions $\mathcal{O}(\mathcal{A}_Q)=\Gamma(\mathcal{A}_Q,\mathcal{O}_{\mathcal{A}_Q})$ is often referred to as the upper cluster algebra, due to its connection to \cite{BFZ05}. This scheme is typically not finitely generated but it is separated by \cite[Theorem 3.14]{GHK15}.\\

\item[-] The space $\mathcal{X}_Q$, called the cluster $\mathcal{X}$-space or cluster Poisson space, is also a scheme obtained by birationally gluing certain tori according to $Q$; the gluing maps are different than for $\mathcal{A}_Q$ above. The ring of regular functions $\mathcal{O}(\mathcal{X}_Q)=\Gamma(\mathcal{X}_Q,\mathcal{O}_{\mathcal{X}_Q})$ does not have a name. This scheme is typically not separated, see \cite[Remark 2.6]{GHK15}.\\

\item[-] The subset $\mu_{\leq 1}^\mathcal{A}(Q)\sse \mathcal{A}_Q$ consisting of the union of $\SA$-tori associated to the initial seed for $Q$ and its adjacent seeds, i.e.~those obtained by performing {\it one} cluster $\SA$-mutation. The ring of functions $\SO(\mu_{\leq 1}^\mathcal{A}(Q))$ is known as the upper bound.\\

\item[-] The subset $\mu_{\leq 1}^\mathcal{X}(Q)\sse \mathcal{X}_Q$ consisting of the union of $\SX$-tori associated to the initial seed for $Q$ and its adjacent seeds, i.e.~those obtained by performing {\it one} cluster $\SX$-mutation.\\
\end{itemize}

\noindent The $\mathcal{A}$ and $\mathcal{X}$-schemes were first introduced and studied by V.~Fock and A.~Goncharov
\cite{FockGoncharov_ModuliLocSys,FockGoncharovII} and subsequently featured in \cite{GHK15,GHKK}. There is also the notion of a partial $\SX$-structure (and partial $\SA$-structure), as introduced in \cite[Definition 5.11]{STWZ}, where only some tori in $\mu_{\leq 1}^\mathcal{X}(Q)$ are considered. After studying the literature and discussing with experts, our conclusion is that ``{\it a space $Y$ has a cluster structure}'' might mean that $Y$ is equal to either of the (often quite different) spaces above, or even that it is equal up to codimension 2, i.e.~ $\SO(Y)$ equals any of the (often quite different) rings of functions above. In certain cases, such as \cite{GoncharovKontsevich21_Noncommutative}, it might also mean having a partial $\SA$- or partial $\SX$-structure for what the authors referred to as a non-commutative stack.


\begin{remark}
It is crucial to have a rigorous definition of the ``space'' $Y$, as well as its ``ring of functions'' $\SO(Y)$ so as to give precise meaning to the notion of admitting a cluster structure. If $Y$ is a scheme, the sheaf of regular functions is well-understood \cite{Hartshorne77_AGBook}. In our case, $\mathfrak{M}(\La,T)$ are always affine schemes and $\M_1(\La)$ and $\M_1(\La,T)$ are always algebraic quotients of affine schemes.\hfill$\Box$ 
\end{remark}

\noindent Now, the spaces $\mbox{Spec}(\mathbb{A}_Q)$, $\SA_Q$, $\SX_Q$, $\mu_{\leq 1}^\mathcal{A}(Q)$ and $\mu_{\leq 1}^\mathcal{X}(Q)$ are often quite different from each other, but the following facts hold:

\begin{enumerate}
    \item The inclusion $\mathbb{A}_Q\sse\SO(\SA_Q)$ always holds. This is a non-trivial fact known as the Laurent phenomenon. In particular, {\it all} cluster $\SA$-variables $A_i\in\mathbb{A}_Q$ belong to $\SO(\SA_Q)$. In fact, $\mathbb{A}_Q$ can be defined to be the subalgebra of $\SO(\SA_Q)$ generated by the cluster $\SA$-variables. In stark contrast, the cluster $\SX$-variables $X_i$ are almost never elements of $\SO(\SX_Q)$.\\

    \item The inclusion $\mathbb{A}_Q\sse \SO(\SA_Q)$ of the cluster algebra into its upper cluster algebra may or may not be an equality, depending on the case. This is referred to as the $\SA=\mathcal{U}$ problem, see e.g.~\cite{BFZ05,Muller14}. The inclusion $\SO(\SA_Q)\sse\SO(\mu_{\leq 1}^\mathcal{A}(Q))$ of the upper cluster algebra into its upper bound may or may not be an equality. It is known to be an equality for the case of full rank. In general, it is known that the equality $\SO(\SX_Q)=\SO(\mu_{\leq 1}^\mathcal{X}(Q))$ always holds.\\
\end{enumerate}

In this manuscript, the main spaces $Y$ we study are the affine schemes $\mathfrak{M}(\La,T)$. The results we prove imply that $\SO(\mathfrak{M}(\La,T))$ equals $\SO(\SA_Q)$, where $Q$  is the quiver geometrically defined in Section \ref{sec:weaves}. That is, we construct an inclusion $\SO(\mathfrak{M}(\La,T))\subseteq\SO(\SA_Q)$ and show that it is an equality. We also provide symplectic geometric meaning to the $\SA$-variables in Section \ref{sec:cluster}. In conjunction with other work of the first author \cite{CGGLSS} -- which logically depends on previous work by the second author \cite{ShenWeng} and the present manuscript -- we know that $\SO(\SA_Q)=\mathbb{A}_Q$. Therefore $\mathfrak{M}(\La,T)$ admits a cluster structure in the strongest possible sense: it is an affine scheme whose ring of regular functions {\it equals} the upper cluster algebra $\SO(\SA_Q)$, and also the cluster algebra $\mathbb{A}_Q$.

\begin{remark}
Similarly, a consequence of our results is that $\SO(\M_1(\La))$ equals $\SO(\SX_Q)$, which was also an open question. That is, we show that the inclusion $\SO(\M_1(\La))\subseteq\SO(\SX_Q)$ is an equality. The results of \cite{STWZ}, when combined with their later work \cite{STW}, would likely imply that for $\La$ associated to plabic fence (no lollipops) one has the inclusion $\SO(\M_1(\La))\sse\SO(\SX_Q)$. That said, even \cite{STW,STWZ} combined do not prove the equality in these cases.\hfill$\Box$
\end{remark}

Finally, we emphasize that the geometric description of the cluster $\SA$-variables $A_i\in \mathbb{A}_Q$ and the particular algebraic geometric description of $\mathfrak{M}(\La,T)$ is what allows for the equalities $\SO(\mathfrak{M}(\La,T))=\SO(\SA_Q)$ to be proven in Section \ref{sec:cluster}. In particular, the fact that $\SO\mathfrak(\mathfrak{M}(\La,T))$ is a unique factorization domain (Subsection \ref{ssec:delta-complete}) and the fact that the (candidate) cluster $\SA$-variables are irreducible in $\SO\mathfrak(\mathfrak{M}(\La,T))$ (Subsection \ref{ssec:codim 2 argument}), are key to deduce $\SO(\mathfrak{M}(\La,T))=\SO(\SA_Q)$.



\section{Diagrammatic Weave Calculus and Initial Cycles}\label{sec:weaves}

The new machinery from contact topology that allows us to construct cluster structures is the study of Legendrian weaves, as initiated in \cite{CasalsZaslow}. In this manuscript, we continue to develop techniques for Legendrian weaves so as to prove Theorem \ref{thm:main}. These new weave techniques now relate to GP-graphs $\bG$ and their associated Legendrian links $\La(\bG)$. Among many central facts, the construction of a weave $\ww(\bG)$ associated to $\bG$ yields a canonical embedded exact Lagrangian filling for $\La(\bG)$, a sheaf quantization, and the flag moduli of the weaves $\ww(\bG)$ shall provide the initial seeds for our cluster structures. In addition, as explained in Section \ref{sec:cluster}, the weave $\ww(\bG)$ is used also to carry the explicit computations necessary for the study of cluster $\mathcal{A}$-variables and the proof of Theorem \ref{thm:main}.


\subsection{Preliminaries on weaves}\label{ssec:Weave} The reader is referred to \cite{CasalsZaslow} for the details and background on Legendrian weaves, but we provide here a quick primer on the basics. Let $J,K\sse\R^2$ be two trivalent planar graphs having an isolated intersection point at a common vertex $v\in J\cap K$. By definition, the intersection $v$ is said to be {\it hexagonal} if the six half-edges in $C$ incident to $v$ interlace, i.e. alternately belong to $J$ and $K$. Figure \ref{fig:VertexType} (right) depicts such a hexagonal vertex.

\begin{center}
	\begin{figure}[h!]
		\centering
		\includegraphics[scale=0.4]{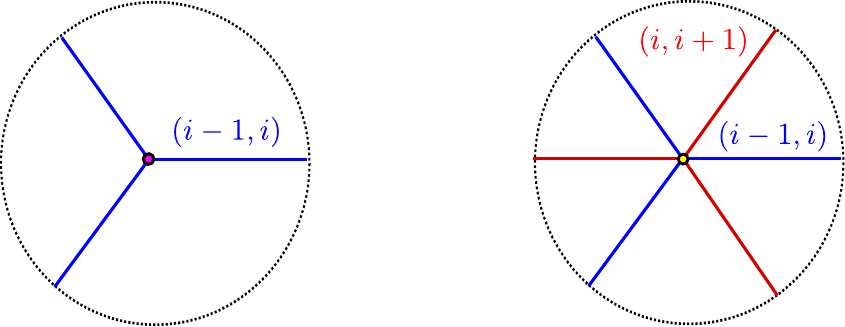}
		\caption{Trivalent vertex (left) and Hexagonal Point (right).}
		\label{fig:VertexType}
	\end{figure}
\end{center}

\vspace{-0.7cm}

\begin{definition}\label{def:Ngraph}
Given $n\in\N$, an $N$-weave $\ww\sse\R^2$ is a set $\ww=\{G_i\}_{1\leq i\leq N-1}$ of $(N-1)$ embedded trivalent planar graphs $G_i\sse \R^2$, possibly empty or disconnected, such that $G_i$ is allowed to intersect $G_{i+1}$ only at hexagonal points, $1\leq i\leq N-2$. By definition, a weave $\ww\sse\R^2$ is an $N$-weave for some $N\in\N$.\hfill$\Box$
\end{definition}

We also refer to the image of a weave in the plane as a weave $\ww$, as no confusion arises. The edges of the graphs that constitute a weave $\ww$ are often referred to as weave lines. We note that two graphs $G_i,G_j\sse \R^2$ are allowed to intersect (anywhere) as long as $j\neq i,i\pm 1$, and we always assume that the intersection is transverse. Through its image, we also think of an $N$-weave as an immersed graph in $\R^2$ with colored (or labeled) edges, the color $i$ corresponding to the graph $G_i$, $1\leq i\leq N-1$.  Edges labeled by numbers differing by two or more may pass through one another (hence the immersed property, which is met generically), but not at a vertex. As a graph in the plane, and $N$-weave has trivalent, tetravalent and hexagonal vertices. 

Let $\{s_i\}_{i=1}^{N-1}$ be the set of Coxeter generators of the symmetric group $S_N$. Instead of colors, we can equivalently label the edges of an $N$-weave $\ww=\{G_i\}$ which belong to the graph $G_i$ with the transposition $s_i$: these labeled edges will also be referred to as $s_i$-edges, or $i$-edges. The theory of weaves as developed in \cite{CasalsZaslow} is grounded on the theory of Legendrian surfaces in $(\R^5,\xi_\st)$ and their spatial wavefronts. In brief, a weave $\ww\sse\R^2$ gives rise to a spatial Legendrian wavefront $\Sigma(\ww)\sse\R^3$, which itself lifts to an embedded Legendrian surface $\La(\ww)$ in $(\R^5,\xi_\st)$. In the present manuscript, the main property of the surface $\La(\ww)$ that we use is that its image $L(\ww):=\pi(\La(\ww))\sse(\R^4,\omega_\st)$ is an exact Lagrangian surface in the standard symplectic Darboux ball, where $\pi:(\R^5,\xi_\st)\lr (\R^4,\omega_\st)$ is the projection along the $\alpha_\st$-Reeb direction.

In this article, unless it is stated otherwise, all the weaves that we construct are free, see \cite[Section 7.1.2]{CasalsZaslow}, which translates into the fact that $L(\ww)\sse(\R^4,\omega_\st)$ will always be an embedded exact Lagrangian surface, and not just immersed. In particular, this implies that $L(\ww)$ must have boundary, which it always will. Moreover, when $\ww$ is free, the Lagrangian projection map $\pi$ is a \emph{homeomorphism}, and hence $H_1(L(\ww))\cong H_1(\Lambda(\ww))$. The underlying contact geometry dictates that certain weaves ought to be considered equivalent. This leads to the following.

\begin{definition}\label{def:EquivalentWeaves}
The moves depicted in Figure \ref{fig:ReidemeisterWeave} are referred to as {\it weave equivalences}. By definition, two weaves $G,G'\sse\R^2$ are said to be equivalent if they differ by a sequence of weave equivalences or diffeomorphisms of the plane. We interchangeably refer to a weave and its weave equivalence class when the context permits.
\end{definition}

\begin{center}
	\begin{figure}[H]
		\centering
		\includegraphics[width=\textwidth]{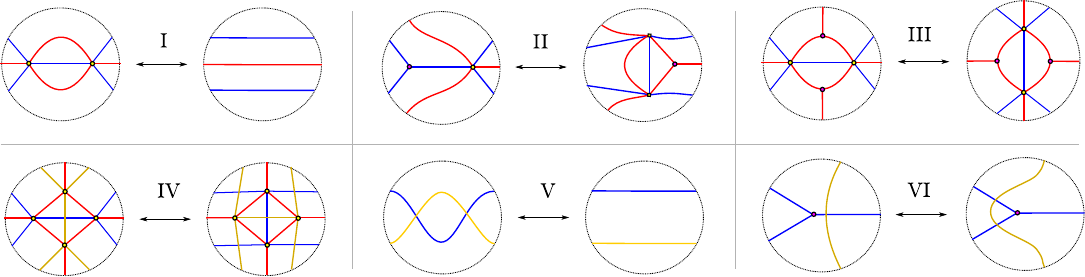}
		\caption{Six weave equivalences.}
		\label{fig:ReidemeisterWeave}
	\end{figure}
\end{center}

\vspace{-0.7cm}

\begin{remark}\label{rmk:non-independence of weave equivalences} As noted in \cite[Section 4]{CasalsZaslow}, these moves are not entirely independent, and Move III can be deduced from Move I and Move II. It is nevertheless useful to underscore Move III when working with weaves. The results in \cite{CasalsZaslow}, using the underlying contact geometry, imply that all the constructions that we associate to a weave are invariant under weave equivalences. (It would to be possible to verify this combinatorially as well, see for instance the computations in \cite{CGGS}.)
\end{remark}

A first goal is constructing a weave for each GP-graph $\bG\sse\R^2$: it is achieved in Subsection \ref{ssec:InitialWeave} once we have reviewed the necessary material on $\sf Y$-cycles and weave mutation.

\subsection{\texorpdfstring{$\sf Y$-cycles and weave mutation}{}}\label{sssec:YCyclesInWeaveAndMutation} Let $\ww\sse\R^2$ be a free weave. The homology group $H_1(L(\ww))\cong H_1(\Lambda(\ww))$ has a central role in our article, as it is a sub-lattice of the defining lattice for the initial seed. In \cite[Section 2]{CasalsZaslow} we devised a method to describe absolute cycles in $L(\ww)$ in terms of $\ww\sse\R^2$. The main concept that is relevant for our purposes is that of a $\sf Y$-cycle on a weave $\ww$, which is defined as follows.

\begin{definition}\label{def:Ycycle} Let $\ww\sse\R^2$ be a weave. An absolute 1-cycle $\gamma\sse \Lambda(\ww)$ is said to be \emph{$\sf Y$-cycle} if its projection onto $\R^2$ consists of weave lines, i.e.~it is contained in $\ww$. A $\sf Y$-cycle is said to be a \emph{$\sf Y$-tree} if its projection image is a tree, considered as a planar embedded graph in $\bR^2$. A $\sf Y$-tree is a \emph{$\sf I$-cycle} if its projection onto $\bR^2$ does not have any trivalent vertices. Finally, an $\sf I$-cycle is \emph{short} if it does not pass through any hexagonal vertex of the weave $\ww$. Figure \ref{fig:local picture of Y-cycles} depicts the four possible local models for a $\sf Y$-tree near a trivalent, tetravalent, and hexagonal vertex of the weave.
\end{definition}
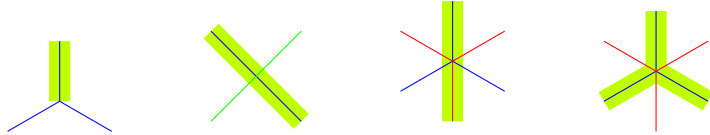
\begin{figure}[H]
    \centering
    \begin{tikzpicture}[scale=0.4]
    \draw [lime, line width = 8] (0,0) -- (0,2);
    \foreach \i in {0,1,2}
    {
        \draw [blue] (0,0) -- (90+\i*120:2);
    }
    \end{tikzpicture} \quad \quad  \quad 
    \begin{tikzpicture}[scale=0.4]
    \draw [lime, line width = 8] (-1.5,1.5) -- (1.5,-1.5);
    \draw [blue] (-1.5,1.5) -- (1.5,-1.5);
    \draw [green] (-1.5,-1.5) -- (1.5,1.5);
    \end{tikzpicture} \quad \quad \quad
    \begin{tikzpicture}[scale=0.4]
    \draw [lime, line width = 8] (0,2) -- (0,-2);
    \foreach \i in {0,1,2}
    {
        \draw [blue] (0,0) -- (90+\i*120:2);
        \draw [red] (0,0) -- (30+\i*120:2);
    }
    \end{tikzpicture} \quad \quad \quad 
    \begin{tikzpicture}[scale=0.4]
    \foreach \i in {0,1,2}
    {
        \draw [lime, line width = 8] (0,0) -- (90+\i*120:2);
        \draw [blue] (0,0) -- (90+\i*120:2);
    }
    \foreach \i in {0,1,2}
    {
        \draw [red] (0,0) -- (30+\i*120:2);
    }
    \end{tikzpicture}
    \caption{The local models for a $\sf Y$-tree. The $\sf Y$-tree is highlighted in light green.}
    \label{fig:local picture of Y-cycles}
\end{figure}
Definition \ref{def:Ycycle} allows us to associate a unique absolute cycle on $\Lambda(\ww)$ (and hence on $L(\ww)$) to each $\sf Y$-tree in a weave $\ww$, as explained in \cite[Section 2]{CasalsZaslow}. (There are two conventions regarding orientations and choice of sheet at which to lift, but once these conventions are fixed, the absolute cycle is defined uniquely.) Note that a $\sf Y$-cycle can stack multiple copies of the above patterns at the same vertex: when this happens at a trivalent or hexagonal vertex, the stacking creates self-intersections of the absolute cycle it represents. The distinction between {\it embedded} and {\it immersed} representatives of absolute homology classes is at the core of the distinction between {\it mutable} and {\it frozen} variables for the cluster structures we construct. The outstanding role of $\sf Y$-trees is justified by the following fact.

\begin{prop}\label{prop:Ytreebounds} Let $\ww\sse\R^2$ be a free weave and $\delta$ be an absolute 1-cycle representing a homology class in $H_1(L(\ww))\cong H_1(\Lambda(\ww))$ which is obtained from a $\sf Y$-tree in $\ww$. Then, there exists a weave equivalence $\ww\sim \ww'$ such that the cycle $\delta\sse\ww$ becomes a short $\sf I$-cycle in $\ww'$. In consequence, any homology class in $H_1(L(\ww))$ represented by a $\sf Y$-tree admits an embedded representative $\gamma\sse L(\ww)$ which bounds an embedded exact Lagrangian disk in $\R^4\setminus L(\ww)$.
\end{prop}

\begin{proof} This readily follows from \cite{CasalsZaslow}, by applying the equivalence moves in Figure \ref{fig:ReidemeisterWeave} and keeping track of the change of a $\sf Y$-tree under these moves. In fact, it suffices to use of the following two local weave equivalences:
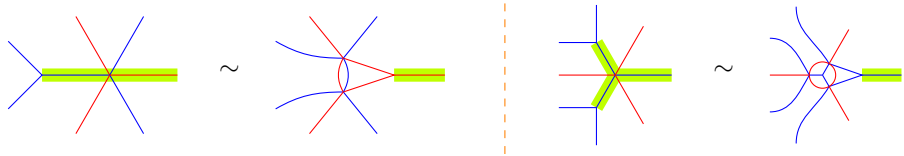
\begin{figure}[H]
    \centering
    \begin{tikzpicture}[baseline=0,scale=0.9]
    \draw [lime, line width=5] (-1,0) -- (1,0);
    \draw [blue] (-1.5,-0.5) -- (-1,0) -- (-1.5,0.5);
    \foreach \i in {0,1,2}
    {
        \draw [blue] (0,0) -- (180+\i*120:1);
        \draw [red] (0,0) -- (120+\i*120:1);
    }
    \end{tikzpicture} \quad $\sim$ \quad 
    \begin{tikzpicture}[baseline=0,scale=0.9]
    \draw [lime, line width=5] (0.75,0) -- (1.5,0);
    \draw [blue] (-1,-0.5) to [out=30, in=180] (0,-0.25) -- (-60:1);
    \draw [blue] (-1,0.5) to [out=-30,in=180] (0,0.25) -- (60:1);
    \draw [red] (120:1) -- (0,0.25) to [out=-120,in=120] (0,-0.25) -- (-120:1);
    \draw [blue] (0,0.25) to [out=-60,in=60] (0,-0.25);
    \draw [red] (0,0.25) -- (0.75,0) -- (0,-0.25);
    \draw [red] (0.75,0) -- (1.5,0);
    \end{tikzpicture}\quad \quad \begin{tikzpicture}[baseline=0,scale=0.7]
    \draw [dashed, orange] (0,-1.5) -- (0,1.5);
    \end{tikzpicture}\quad \quad
    \begin{tikzpicture}[baseline=0]
    \draw [lime, line width=5] (120:0.5) -- (0,0) -- (-120:0.5);
    \draw [lime,line width=5] (0,0) -- (0.75,0);
    \draw [blue] (0,0) -- (120:0.5) -- +(0,0.5);
    \draw [blue] (120:0.5) -- +(-0.5,0);
    \draw [blue] (0,0) -- (-120:0.5) -- +(0,-0.5);
    \draw [blue] (-120:0.5) -- +(-0.5,0);
    \draw [blue] (0,0) -- (0.75,0);
    \foreach \i in {0,1,2}
    {
        \draw [red] (0,0) -- (180+120*\i:0.75);
    }
    \end{tikzpicture} 
    \quad $\sim$ \quad
    \begin{tikzpicture}[baseline=0,scale=0.7]
    \draw [lime, line width=5] (0.75,0) -- (1.5,0);
    \draw [blue] (-1,0.7) to [out=0,in=120] (-0.25,0);
    \draw [blue] (-0.5,1.3) to [out=-90,in=120] (60:0.25);
    \draw [blue] (-1,-0.7) to [out=0,in=-120] (-0.25,0);
    \draw [blue] (-0.5,-1.3) to [out=90,in=-120] (-60:0.25);
    \draw [blue] (60:0.25) -- (0.75,0) -- (-60:0.25);
    \draw [red] (0,0) circle [radius=0.25];
    \foreach \i in {0,1,2}
    {
        \draw [red] (180+120*\i:0.25) -- (180+120*\i:1);
        \draw [blue] (0,0) -- (180+120*\i:0.25);
    }
    \draw [blue] (0.75,0) -- (1.5,0);
    \end{tikzpicture}
    \caption{Local weave equivalences to turn a $\sf Y$-tree into a short $\sf I$-cycle. Note that the first row is just Move II from Figure \ref{fig:ReidemeisterWeave}, where we kept track of the $\sf Y$-tree -- highlighted in light green -- before and after the equivalence.}\label{fig:YtoIequivalences}
\end{figure}

\noindent By using the two weave equivalences in Figure \ref{fig:YtoIequivalences}, we can work outside in on the $\sf Y$-tree $\delta$ and replace each weave line of $\delta$ with a double track, and shorten $\delta$ to a short $\sf I$-cycle somewhere along the original $\sf Y$-tree. The double tracks that appear in this shortening process are schematically depicted in Figure \ref{fig:YtoIequivalence2}. The second half of the proposition follows from the description of a short $\sf I$-cycle in \cite{CasalsZaslow}.
\begin{figure}[H]
    \centering
    \begin{tikzpicture}[baseline=5,scale=0.9]
    \draw [lime, line width=5] (0,0) -- (1.5,0) -- (2,0.5) -- (2.5,0.5) -- (3,0) -- (4,0) -- (4.5,-0.5);
    \draw [lime, line width=5](0,0.5) -- (0.5,1) -- (1.5,1) -- (2,0.5);
    \draw [lime, line width=5](0,1.5) -- (0.5,1);
    \draw [lime, line width=5](2.5,0.5) -- (3,1) -- (3,1.5);
    \draw [lime, line width=5](3,1) -- (4.5,1);
    \draw [lime, line width=5](3,0) -- (3,-0.5);
    \draw [lime, line width=5](4,0) -- (4.5,0.5);
    \end{tikzpicture} \quad \quad $\sim$ \quad \quad 
    \begin{tikzpicture}[baseline=5,scale=0.9]
    \draw [line width=6] (0,0) -- (1.5,0) -- (2,0.5) -- (2.5,0.5) -- (3,0) -- (4,0) -- (4.5,-0.5);
    \draw [line width=6](0,0.5) -- (0.5,1) -- (1.5,1) -- (2,0.5);
    \draw [line width=6](0,1.5) -- (0.5,1);
    \draw [line width=6](2.5,0.5) -- (3,1) -- (3,1.5);
    \draw [line width=6](3,1) -- (4.5,1);
    \draw [line width=6](3,0) -- (3,-0.5);
    \draw [line width=6](4,0) -- (4.5,0.5);
    \draw [white,line width=5] (-0.1,0) -- (1.5,0) -- (2,0.5) -- (2.5,0.5) -- (3,0) -- (4,0) -- (4.6,-0.6);
    \draw [white, line width=5](-0.1,0.4) -- (0.5,1) -- (1.5,1) -- (2,0.5);
    \draw [white, line width=5](-0.1,1.6) -- (0.5,1);
    \draw [white, line width=5](2.5,0.5) -- (3,1) -- (3,1.6);
    \draw [white, line width=5](3,1) -- (4.6,1);
    \draw [white, line width=5](3,0) -- (3,-0.6);
    \draw [white, line width=5](4,0) -- (4.6,0.6);
    \draw [white, line width = 8] (0.25,0) -- (1,0);
    \draw [lime, line width=5] (0.5,0) -- (0.75,0);
    \draw (0.25,0.1) -- (0.5,0);
    \draw (0.25,-0.1) -- (0.5,0);
    \draw (0.5,0) -- (0.75,0);
    \draw (0.75,0) -- (1,0.1);
    \draw (0.75,0) -- (1,-0.1);
    \end{tikzpicture}
    \caption{(Left) A $\sf Y$-tree cycle highlighted in light green. (Right) The double tracks that remain on the (equivalent) weave after the shortening process, where the $\sf Y$-cycle has now become the short $\sf I$-cycle drawn in light green.}\label{fig:YtoIequivalence2}
\end{figure}
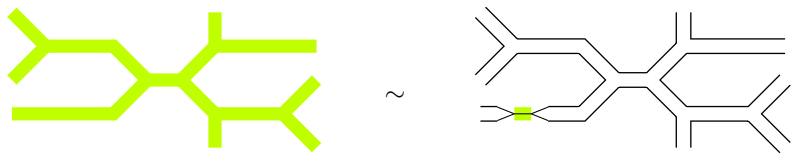
\end{proof}

\begin{remark}
Proposition \ref{prop:Ytreebounds} implies that any one $\sf Y$-tree can be turned into a short $\sf I$-cycle after weave equivalences. It is not the case that a $\sf Y$-cycle, which is not necessarily a $\sf Y$-tree, can always be turned into a short $\sf I$-cycle. It is also not the case that Proposition \ref{prop:Ytreebounds} works for more than one $\sf Y$-tree at once, in the following sense. If two $\sf Y$-trees $Y_1,Y_2\sse \ww$ are given in a weave $\ww$, then there exists a sequence of weave equivalences from $\ww$ to a weave $\ww_1$ such that $Y_1$ becomes a short $\sf I$-cycle in $\ww_1$.There is no guarantee that $Y_2$ will be a short $\sf I$-cycle in $\ww_1$; $Y_2$ will be a short $\sf I$-cycle in another weave $\ww_2$ equivalent to $\ww$, a priori different from $\ww_1$. More generally, there are collections of $\sf Y$-trees in a weave $\ww$ such that there is no sequence of weave equivalences that would turn at once all those $\sf Y$-trees in $\ww$ into short $\sf I$-cycles in any one weave $\ww'$ equivalent to $\ww$.\hfill$\Box$
\end{remark}

The existence of the embedded Lagrangian disks from Proposition \ref{prop:Ytreebounds}, i.e.~$\mathbb{L}$-compressible curves in $L(\ww)$, allows us to perform Lagrangian disk surgeries along $\sf Y$-trees and produce new exact Lagrangian fillings. In \cite[Section 4.8]{CasalsZaslow} we proved that it is possible to describe this symplectic geometric operation via a diagrammatic change in a piece of the weave called ``weave mutations''. This leads to the following definition.

\begin{definition}\label{def:WeaveMutation} Let $\gamma\sse\ww$ be the short $\sf I$-cycle -- a monochromatic blue edge -- depicted at the left of Figure \ref{fig:MutationWeave} (left). Then, the local move illustrated in Figure \ref{fig:MutationWeave} (left) is said to be a \emph{weave mutation} at the short $\sf I$-cycle $\gamma$. This is the standard Whitehead move for trivalent graphs, dual to a flip in a triangulation. Note that a weave mutation replaces the short $\sf I$-cycle $\gamma$ with a new short $\sf I$-cycle, which we often denote by $\gamma'$; we call $\gamma'$ the \emph{image} of $\gamma$ under the weave mutation.
\end{definition}

\begin{definition} For a $\sf Y$-tree in general, one can apply Proposition \ref{prop:Ytreebounds} to turn it into a short $\sf I$-cycle, perform a weave mutation, and then apply some other weave equivalences. Thus, a \emph{weave mutation} at a $\sf Y$-tree $\gamma$ in general means a weave mutation at its short $\sf I$-cycle counterpart conjugated by sequences of weave equivalences. By following the sequences of weave equivalences, the weave mutation replaces $\gamma$ by its \emph{image}, which is a new $\sf Y$-tree $\gamma'$. 
\end{definition}

\begin{center}
	\begin{figure}[h!]
		\centering
		\includegraphics[scale=0.7]{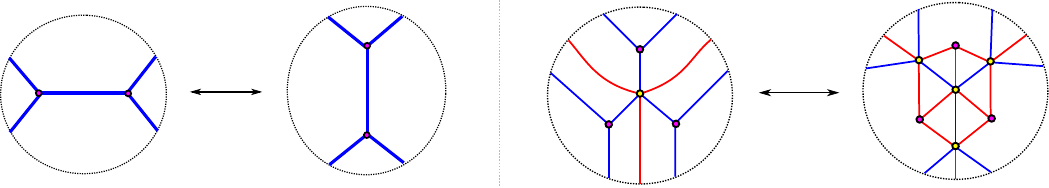}
		\caption{(Left) Weave mutation along the short $\sf I$-cycle given by the blue edge. (Right) A weave mutation along a monochromatic $\sf Y$-tree.}
		\label{fig:MutationWeave}
	\end{figure}
\end{center}

\begin{definition} Two weaves $\ww,\ww'$ are said to be (weave) \emph{mutation equivalent} if they can be connected by a sequence of moves consisting of weave equivalences and weave mutations.
\end{definition}

\noindent Finally, we emphasize that weave mutations will allow us to mutate at $\sf Y$-trees of $\ww(\bG)$ corresponding to faces and regions of $\bG$, even if they might not be square. The resulting weave, typically, will not be of the form $\ww(\bG')$ for any GP-graph $\bG'$, but all the diagrammatic and symplectic geometric results developed in this article and \cite{CasalsZaslow} can still be applied.


\subsection{Initial Weave for a GP-Graph}\label{ssec:InitialWeave} In this section we construct the initial weave $\ww(\bG)$ associated to a GP-graph $\bG$. This is done by breaking $\bG$ into elementary columns and assigning a local weave associated to each such column. Recall the standard reduced word $\w_{0,n}=s_{[1,1]}s_{[2,1]}s_{[3,1]}\ldots s_{[n-1,1]}$ for the longest element $w_{0,n}\in S_n$. Let us define $\ell=\ell(\w_{0,n})=n(n-1)/2$. The first three local weaves $\mathfrak{n}(w),\c^{\uparrow}(w)$ and $\mathfrak{c}^{\downarrow}(w)$ are defined as follows.
 
\begin{definition}
Let $w=s_{i_1}\cdots s_{i_\ell}$ be a reduced expression for $w_{0,n}\in S_n$. By definition, the weave $\mathfrak{n}(w)$ is given by $n$ horizontal parallel weave lines such that the $j$-th strand, counting from the bottom, is labeled by the transposition $s_{i_j}$, $j\in[1,\ell]$.

\noindent The weave $\c^{\uparrow}(w)$ is given by the weave $\mathfrak{n}(w)$ where a trivalent vertex is added at the top strand -- labelled by $s_{i_\ell}$ -- such that the third leg of this trivalent vertex is a vertical ray starting at the top strand and continuing upwards. 

\noindent Similarly, the weave $\mathfrak{c}^{\downarrow}(w)$ is given by the weave $\mathfrak{n}(w)$ where a trivalent vertex is added at the bottom strand -- labelled by $s_{i_1}$ -- such that the third leg of this trivalent vertex is a vertical ray starting at the bottom strand and continuing downwards.\hfill$\Box$
\end{definition}

As explained above, the weave $\ww(\bG)$ associated to $\bG$ is built by horizontally concatenating weaves local models: each local model is associated to one of the three types of elementary columns. 
The corresponding weaves for each of these occurrences are described as follows. 


\subsubsection{Local weaves for Type 1 Columns}\label{sssec:Type1weave} 

\begin{definition}[Weave for Type 1]\label{def:weaveType1}
The weave associated to a Type 1 column of a GP-graph $\bG$, which consists of $n$ parallel horizontal lines, is $\n(\w_{0,n})$, where $\w_{0,n}$ is the standard reduced expression for the longest element in a symmetric group $S_k$.
\end{definition}

It is important to note that the $s_i$-transpositions labeling the strands of $\n(\w_{0,n})$ depend on the simple transpositions that generate the symmetric group $S_{[a,b]}$ associated to that specific region (see Subsection \ref{ssec:GPgraph}). Due to the appearance of lollipops in the GP-graph, the different symmetric groups $S_{[a,b]}$ that we encounter (as we read the GP-graph left to right) have varying discrete intervals $[a,b]$.


\subsubsection{Local weaves for Type 2 Columns} \label{sssec:crossingweave}

It is a well-known property of the symmetric group, see for instance \cite[Section 3.3]{BjornerBrenti}, that any two reduced word expressions for the same element can be transformed into each other via finite sequences of the following two moves:
\begin{itemize}
    \item[-] $s_is_j\sim s_js_i$, if $|i-j|>1$;
    \item[-] $s_is_js_i\sim s_js_is_j$, if $|i-j|=1$.
\end{itemize}

Now consider the weave $\n(\w_{0,n})$ and the $s_1$-strand labelled by the $i$-th appearance of $s_1$ in the standard reduced expression $\w_{0,n}$. In order to construct the weave for a Type 2 column, in Definition \ref{def:weaveType2} below, we need the following auxiliary local weaves:

\begin{definition}\label{def:weavenormalize}
The weave $\n_i^{\uparrow}(\w_{0,n})$ is the unique horizontal weave that coincides with $\n(\w_{0,n})$ at the left, contains only tetravalent and hexagonal vertices, and brings the $i$-th $s_1$-strand of $\w_{0,n}$ to the top level, using a minimal number of weave vertices.

\noindent Similarly, the weave $\n_i^{\downarrow}(\w_{0,n})$ is the unique weave that coincides with $\n(\w_{0,n})$ at the left, contains only tetravalent and hexagonal vertices, and brings the $i$-th $s_1$-strand of $\w_{0,n}$ to the bottom level, using a minimal number of weave vertices.

\noindent Finally, we denote by $\n_i^{\uparrow}(\w_{0,n})^{op}$ and $\n_i^{\downarrow}(\w_{0,n})^{op}$ the weaves obtained by reflecting $\n_i^{\uparrow}(\w_{0,n})$ and $\n_i^{\downarrow}(\w_{0,n})$ along a (disjoint) vertical axis.\hfill$\Box$
\end{definition}

\noindent In Definition \ref{def:weavenormalize}, bringing the $i$-th $s_1$-strand of $\w_{0,n}$ to the top level means considering a horizontal weave which starts at $\n(\w_{0,n})$ on the left hand side and contains a sequence of tetravalent and hexagonal vertices (no trivalent vertices) such that following the $i$-th $s_1$-strand of $\w_{0,n}$ under these vertices (passing through them straight) ends up at the top strand at the right hand side. There are many weaves that verify this property, but by the Zamolodchikov relation proven in \cite{CasalsZaslow}, they are all equivalent and we might as well take the one with a minimal number of vertices. Figure \ref{fig:RulesWeaves_CrossingExamples} illustrates several examples of the weaves $\n_i^{\uparrow}(\w_{0,n})$ and $\n_i^{\downarrow}(\w_{0,n})$ in Definition \ref{def:weavenormalize} in the cases of $n=3,4$. Note that $\n_{n-1}^{\uparrow}(\w_{0,n})=\n_{n-1}^{\uparrow}(\w_{0,n})^{op}=\n_1^{\downarrow}(\w_{0,n})=\n_1^{\downarrow}(\w_{0,n})^{op}=\n(\w_{0,n})$ for any $n\in\N$.

\begin{center}
	\begin{figure}[h!]
		\centering
		\includegraphics[width=\textwidth]{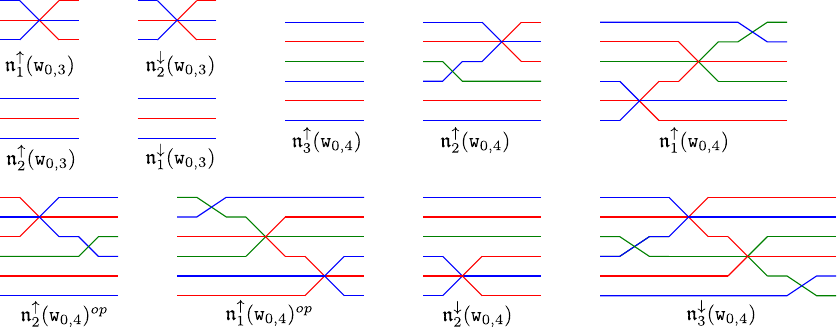}
		\caption{Instances of the weaves $\n_i^{\uparrow}(\w_{0,n})$ and $\n_i^{\downarrow}(\w_{0,n})$, and their opposites, from Definition \ref{def:weavenormalize}, in some of the cases for $n=3,4$.}
		\label{fig:RulesWeaves_CrossingExamples}
	\end{figure}
\end{center}

\vspace{-0.7cm}

\begin{definition}\label{def:weavecrossing}
The weave $\c_i^{\uparrow}(\w_{0,n})$ is the weave obtained by horizontally concatenating the three weaves $\n_i^{\uparrow}(\w_{0,n})$, $\c^{\uparrow}(w_i)$, and $\n_i^{\uparrow}(\w_{0,n})^{op}$, left to right, where $w_i$ denotes the reduced expression for $\w_{0,n}$ found at the right of the weave $\n_i^{\uparrow}(\w_{0,n})$.
\end{definition}
\begin{definition} \label{def:weavecrossing2} Similarly, the weave $\c_i^{\downarrow}(\w_{0,n})$ is the weave obtained by horizontally concatenating the three weaves $\n_i^{\downarrow}(\w_{0,n})$, $\c^{\downarrow}(w_i)$ and $\n_i^{\downarrow}(\w_{0,n})^{op}$, left to right, where $w_i$ denotes the reduced expression for $\w_{0,n}$ found at the right of the weave $\n_i^{\downarrow}(\w_{0,n})$.
\end{definition}

\noindent Figures \ref{fig:RulesWeaves_CrossingExamples2}  illustrates examples of the weaves $\c_i^{\uparrow}(\w_{0,n})$ and $\c_i^{\downarrow}(\w_{0,n})$ in Definitions \ref{def:weavecrossing} and \ref{def:weavecrossing2} for $n=3$. For the next definition, recall that we always label the horizontal lines, in a Type 1 and Type 2 column of the GP-graph $\bG$, with consecutive natural numbers, from bottom to top. For the moment, let us assume that these labels are in $[1,n]$.

\begin{definition}[Weave for Type 2]\label{def:weaveType2}
For $i\in [1,n-1]$, the weave associated to a Type 2 column of a GP-graph $\bG$ whose vertical edge has a white vertex at the $i$th horizontal line and a black vertex at the $(i+1)$st horizontal line is the weave $\c_i^{\uparrow}(\w_{0,n})$, and the weave associated to a Type 2 column of a GP-graph $\bG$ whose vertical edge has a black vertex at the $i$th horizontal line and a white vertex at the $(i+1)$st horizontal line is the weave $\c_{n-i}^{\downarrow}(\w_{0,n})$.
\end{definition}

\begin{center}
	\begin{figure}
		\centering
		\includegraphics[scale=0.8]{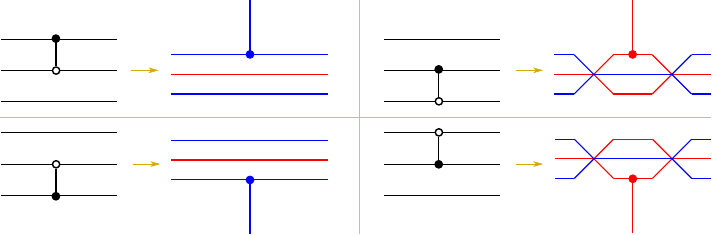}
		\caption{All possible Type 2 columns for $n=3$ and their corresponding weaves. In detail, $\c_{2}^{\uparrow}(\w_{0,3})$ (upper left), $\c_{1}^{\uparrow}(\w_{0,3})$ (upper right), $\c_{1}^{\downarrow}(\w_{0,3})$ (lower left) and $\c_{2}^{\downarrow}(\w_{0,3})$ (lower right).}
		\label{fig:RulesWeaves_CrossingExamples2}
	\end{figure}
\end{center}


\subsubsection{Local weaves for Type 3 Columns} \label{sssec:lollipopweave}
Let us consider a column of Type 3 with labels $1,2,\ldots,n$ for the horizontal lines on the right (counting from bottom to top) and with a white lollipop attached to the $i$th horizontal line with $i\in [1,n]$. The case of a black lollipop is similar, and discussed later.

\noindent By construction, the weaves associated with the two Type 1 columns sandwiching this Type 3 column are $\n(\w_{0,n-1})$ and $\n(\w_{0,n})$, respectively. Hence, the weave we associate to such a Type 3 column must have these boundary conditions. Let us start with the following weave:

\begin{center}
	\begin{figure}[h!]
		\centering
		\includegraphics[width=\textwidth]{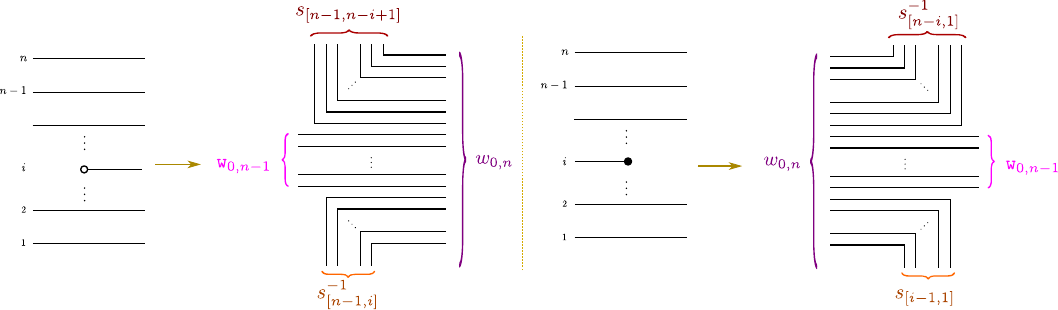}
		\caption{The weaves $\i_i^w$ (left) and $\i_i^b$ (right) from Definitions \ref{def:incomingweave} and \ref{def:outgoingweave}.}
		\label{fig:LollipopWeave}
	\end{figure}
\end{center}

\begin{definition}\label{def:incomingweave}
The weave $\i^w_i$ is the unique weave with no weave vertices, satisfying:
\begin{itemize}\setlength{\itemsep}{0.5em}
    \item[(i)] At its left, $\i^w_i$ coincides with the horizontal weave $\n(\w_{0,n-1})$, and at its right, $\i^w_i$ coincides with the horizontal weave $\n(s_{[n-1,i]}^{-1}\w_{0,n-1}s_{[n-1,n-i+1]})$.
    
    \item[(ii)] The weave lines in $\n(s_{[n-1,i]}^{-1}\w_{0,n-1}s_{[n-1,n-i+1]})$ labeled by the transpositions in the reduced expression $s_{[n-1,n-i+1]}$ diverge upwards to vertical rays.
    
    \item[(iii)] The weave lines in $\n(s_{[n-1,i]}^{-1}\w_{0,n-1}s_{[n-1,n-i+1]})$ labeled by the transpositions in the reduced expression $s_{[n-1,i]}^{-1}$ diverge downwards to vertical rays.
\end{itemize}
See Figure \ref{fig:LollipopWeave} (left) for a depiction of $\i^w_i$, illustrating what is meant by diverging upwards and downwards to vertical rays. 
\end{definition}

\noindent Note that the word $\n(s_{[n-1,i]}^{-1}\w_{0,n-1}s_{[n-1,n-i+1]})$ in Definition \ref{def:incomingweave} is a reduced expression for the half-twist $w_{0,n}$. Now, the weaves $\i^w_i$ in Definition \ref{def:incomingweave} cannot quite be the weaves for the Type 3 column yet because the labeling on the right hand side is not $\w_{0,n}$, but rather $\n(s_{[n-1,i]}^{-1}\w_{0,n-1}s_{[n-1,n-i+1]})$. To fix this, let $\n^w_i$ be any horizontal weave that coincides with $\n(s_{[n-1,i]}^{-1}\w_{0,n-1}s_{[n-1,n-i+1]})$ on the left, coincides with $\n(\w_{0,n})$ on the right, and with no trivalent weave vertices in the middle. Any choice of $\n(\w_{0,n})$ would yield an equivalent weave.

\begin{definition}[Weave for White Lollipop]\label{def:weaveType3white}
The weave $\l^w_i$ associated to a Type 3 column with a white lollipop at the $i$th horizontal line is the horizontal concatenation of $\i^w_i$ and $\n^w_i$.
\end{definition}

\noindent Figures \ref{fig:RulesWeaves2_ExampleN3} and \ref{fig:RulesWeaves2_ExampleN3_Part2} illustrate the weaves $\l^w_1,\l^w_2,\l^w_3$, and $\l^w_4$ for $n=4$, with the coloring convention that $s_1$ is blue, $s_2$ is red and $s_3$ is green. The pink boxes in the figures contain the $\i^w_i$ pieces, and the yellow boxes contain the $\n^w_i$ pieces. The figures also draw the corresponding pieces of the fronts $\mathfrak{f}(\bG)$, explaining the contact geometric origin of these weaves.

\begin{center}
	\begin{figure}[H]
		\centering
		\includegraphics{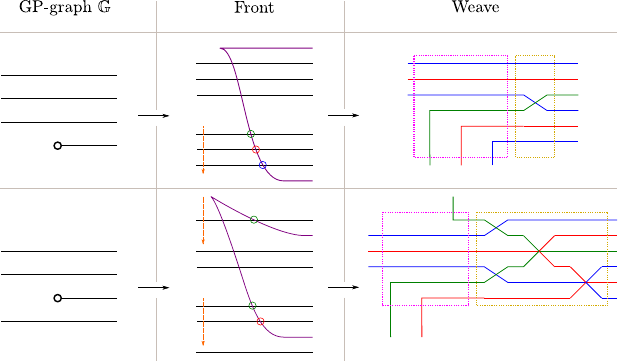}
		\caption{Weaves $\l^w_i$ associated to a white lollipop with $n=4$, as in Definition \ref{def:weaveType3white}. The first row depicts the case $i=1$ and the second row the case $i=2$. The weaves $\n^w_i$ are drawn within the yellow boxes. The weaves $\i^w_i$, with the incoming weave strands, are depicted within the pink boxes.}
		\label{fig:RulesWeaves2_ExampleN3}
	\end{figure}
\end{center}

\begin{center}
	\begin{figure}[H]
		\centering
		\includegraphics{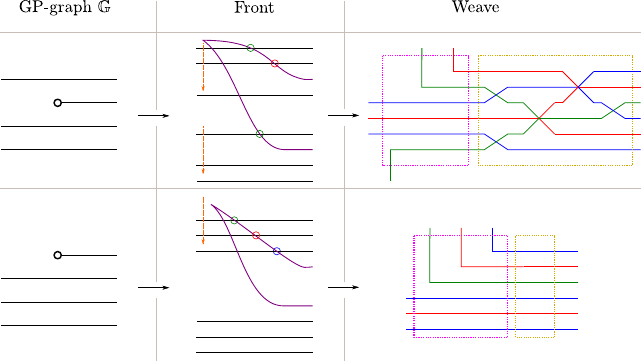}
		\caption{Weaves $\l^w_i$ associated to a white lollipop with $n=4$. The first row depicts the case $i=3$ and the second row the case $i=4$. The weaves $\n^w_i$ are drawn within the yellow boxes, and the weaves $\i^w_i$ are depicted in the pink boxes.}
		\label{fig:RulesWeaves2_ExampleN3_Part2}
	\end{figure}
\end{center}
\vspace{-0.8cm}

\begin{center}
	\begin{figure}[H]
		\centering
		\includegraphics{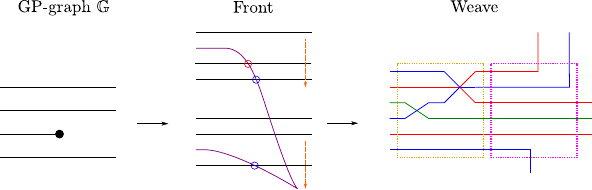}
		\caption{The weave $\l^b_2$ associated to a black lollipop when $n=4$, in accordance with Definition \ref{def:weaveType3black}. The weave $\n^b_2$ is drawn in the yellow box, and $\i^b_2$ in the pink box. Note that the departing strands of the weave (in the pink box) are in bijection with the crossings of the front, as indicated.}
		\label{fig:RulesWeaves2_ExampleN3_Part3}
	\end{figure}
\end{center}

The case of a column of Type 3 with labels $1,2,\ldots,n$ for the horizontal lines on the right, and a black lollipop at the $i$th horizontal line is similar. The necessary definitions are as follows.

\begin{definition}\label{def:outgoingweave}
The weave $\i^b_i$ is the unique weave with no weave vertices, satisfying:

\begin{itemize}
    \item[(i)] At its left, $\i^b_i$ coincides with the horizontal weave $\n(s_{[i-1,1]}\w_{0,n-1}s^{-1}_{[n-i,1]})$, and at its right, $\i^b_i$ coincides with the horizontal weave $\n(\w_{0,n-1})$.
    
    \item[(ii)] The weave lines in $\n(s_{[i-1,1]}\w_{0,n-1}s^{-1}_{[n-i,1]})$ labeled by the transpositions in the reduced expression $s_{[i-1,1]}$ diverge downwards to vertical rays.
    
    \item[(iii)] The weave lines in $\n(s_{[i-1,1]}\w_{0,n-1}s^{-1}_{[n-i,1]})$ labeled by the transpositions in the reduced expression $s_{[n-i,1]}^{-1}$ diverge upwards to vertical rays.
\end{itemize}

\end{definition}

\noindent  See Figure \ref{fig:LollipopWeave} (right) for a depiction of $\i^b_i$. Similarly, we denote by $\n^b_i$ any horizontal weave which coincides with the horizontal weave $\n(s_{[i-1,1]}\w_{0,n-1}s^{-1}_{[n-i,1]})$ on the right, coincides with $\n(\w_{0,n})$ on the left, and with no trivalent weave vertices in the middle.

\begin{definition} [Weave for Black Lollipop]\label{def:weaveType3black} 
The weave $\mathfrak{l}^b_i$ associated to a Type 3 column with a black lollipop at the $i$th horizontal line is the horizontal concatenation of $\mathfrak{n}^b_i$ and $\mathfrak{i}^b_i$.
\end{definition}

\noindent See Figure \ref{fig:RulesWeaves2_ExampleN3_Part3} for an example of $\l^b_i$ in the case $i=2$ and $n=4$.


\subsubsection{The initial weave}\label{sssec:initialweave} Let $\bG$ be a GP-graph. The previous three Subsections \ref{sssec:Type1weave}, \ref{sssec:crossingweave}, \ref{sssec:lollipopweave} have explained how to obtain a weave from each of the three types of elementary columns. Note that each these weaves coincides with $\n(\w_{0,k})$ and with $\n(\w_{0,m})$ for some $k$ and $m$ at its two ends. For Type 1 and Type 2, $k=m$, and for Type 3, $|k-m|=1$. Note that, if we consider two adjacent elementary columns in $\bG$, the associated weaves will coincide at the common side, and thus can be horizontally concatenated.

\begin{definition}[Initial weave]\label{def:initialweave} Let $\bG$ be a GP-graph. The initial weave $\mathfrak{w}(\bG)$ associated weave $\bG$ is the weave obtained by subdividing $\bG$ into elementary columns and then concatenating the weaves associated with each elementary column, in the order dictated by the columns.
\end{definition}



\subsection{Topology of the Initial Weave}\label{ssec:topologyweave} Let $\bG$ be a GP-graph and $\mathfrak{w}(\bG)\sse\R^2$ its initial weave. In this subsection we show how to obtain a Legendrian link $\La(\bG)\sse(\R^3,\xi_\st)$ and an embedded exact Lagrangian filling $L(\ww)\sse(\R^4,\la_\st)$ from the weave $\ww=\ww(\bG)$.

\subsubsection{The braid of an initial weave}\label{sssec:braidweave} Suppose $\ww(\bG)$ is an $N$-weave. Let $K\sse\R^2$ be a compact subset such that $(\R^2\setminus K)\cap\mathfrak{w}(\bG)$ contains no weave vertices. Note that the number of weave lines in $(\R^2\setminus K)\cap\mathfrak{w}(\bG)$ and their labeling is independent of any $K$ with such property. The weave lines are labeled by simple transpositions $s_i\in S_N$, which can be lifted to unique positive generators $\sigma_i$ in the Artin braid group $\Br_N$.

\begin{definition}\label{def:braidweave} Let $\beta(\bG)$ be the positive braid word obtained by reading the positive braid generators associated with the weave lines of $(\R^2\setminus K)\cap\mathfrak{w}(\bG)$ in a counterclockwise manner, starting at the unique strand the corresponds to the left-most white lollipop in $\bG$.
\end{definition}

Part of the usefulness of Definition \ref{def:braidweave} is the following simple lemma:

\begin{lemma}\label{lem:linkweave}
Let $\bG$ be a GP-graph, $\mathfrak{w}(\bG)\sse\R^2$ its initial weave and $\beta(\bG)$ its positive braid word. Then the $(-1)$-framed closure of $\beta(\ww(\bG))$ is a front for the Legendrian link $\La(\bG)$.
\end{lemma}

\noindent Lemma \ref{lem:linkweave} can be phrased as follows. Consider the Legendrian link $\La(\beta(\bG))\sse(\R^3,\xi_\st)$ whose front is the $(-1)$-framed closure of the braid word $\beta(\bG)$. Then the Legendrian links $\La(\beta(\bG))$ and $\La(\bG)$ are Legendrian isotopic in $(\R^3,\xi_\st)$. 

\subsubsection{The surface of the initial weave}\label{sssec:surfaceweave} Let $\ww\sse\R^2$ be a weave and $\Lambda(\ww)\sse(\R^5,\xi_\st)$ the Legendrian represented by its front. By definition, the Lagrangian $L(\ww)\sse(\R^4,\la_\st)$ is the Lagrangian projection of $\Lambda(\ww)$. We refer to \cite[Section 7.1]{CasalsZaslow} for details on how weaves yield exact Lagrangian fillings of Legendrian links in $(\R^3,\xi_\st)$, and recall that $\ww$ is said to be free if $L(\ww)\sse(\R^4,\la_\st)$ is embedded. The following lemma is readily proven:

\begin{lemma}\label{lem:Eulerchar}
Let $\bG$ be a GP-graph and suppose its initial weave $\ww=\mathfrak{w}(\bG)$ is an $N$-weave. Then $L=L(\ww)\sse(\R^4,\la_\st)$ is an embedded exact Lagrangian filling of $\La(\bG)$ with Euler characteristic
$$\chi(L)=N-\#(\mbox{trivalent vertices of }\mathfrak{w})=\#(\text{horizontal lines in $\bG$})-\#(\text{vertical edges in $\bG$}).$$
\end{lemma}

The number of boundary components of $L(\ww(\bG))$ is readily computed from $\beta(\bG)$: it is given by the number of cycles in the cycle decomposition of the Coxeter projection of $\beta(\bG)$. Finally, a central feature of weaves is the following: it is possible to draw many weaves which coincide with $\ww(\bG)$, outside a large enough compact set $K\sse\R^2$, and which represent embedded exact Lagrangian fillings of $\La(\bG)$. In fact, as explained in Subsection \ref{sssec:YCyclesInWeaveAndMutation}, there are some local modifications that we can perform to the weave -- {\it weave mutations} -- such that the smooth embedded class of the associated (Lagrangian) surface in $(\R^4,\la_\st)$ remains the same but the Hamiltonian isotopy class typically changes. Square face mutation of a GP-graph $\bG$ is recovered by weave mutations but, importantly, weave mutations allow for more general mutations, including mutations at non-square faces of $\bG$ and sugar-free regions. The result of such {\it weave mutations} applied to $\ww=\ww(\bG)$ is again another weave $\mu(\ww)$: it may no longer be of the form $\ww(\bG')$ for any GP-graph $\bG'$, but it is a weave and thus, using the calculus in \cite{CasalsZaslow}, we can manipulate it efficiently and use it to prove the results in this article. In this process, we need explicit geometric cycles representing generators of the absolute homology $H_1(L(\ww))$. In fact, such geometric cycles lead to the quiver for the initial seed. Thus, we now gear towards understanding how to construct geometric representatives of homology classes using weaves.


\subsection{Naive Absolute Cycles in \texorpdfstring{$L(\ww(\bG))$}{}}\label{ssec:NaiveAbsoluteCycles} In this section we explain how to find a set of geometric (absolute) cycles on $L=L(\ww(\bG))$ which generate $H_1(L)$.

Since the genus of an embedded exact Lagrangian filling is determined by the (maximal) Thurston-Bennequin invariant of its Legendrian boundary, all embedded exact Lagrangian fillings of a given Legendrian link are topologically equivalent as abstract surfaces, i.e.~ they have the same genus. In the case of a Legendrian link $\La(\bG)$ associated with a GP-graph, it is readily seen that this is the same abstract topological type as that of the Goncharov-Kenyon conjugate surface $S=S(\bG)$ \cite{GonKen}. Since the conjugate surface $S$ deformation retracts back to the GP-graph $\bG$, it follows that the boundaries of the faces of $\bG$ form a basis for the absolute homology groups $H_1(\bG)\cong H_1(S)\cong H_1(L)$. This basis, indexed by the faces of $\bG$, will be referred to as the \emph{naive basis} of $H_1(L)$. 

\begin{remark} We emphasize that this set of generating absolute cycles is {\it not} good enough in order to construct cluster structures, nor its intersection quiver gives the correct initial quiver. Thus, these cycles will be referred to as the set of {\it naive} absolute cycles, and we will perform the necessary corrections in Subsection \ref{ssec:InitialCycles} below. 
\end{remark}

In order to proceed geometrically, we would like identify the naive basis elements of $H_1(L)$ as lifts of a specific collection of absolute cycles on the weave front $\Sigma=\Sigma(\ww(\bG))$, ideally a collection of $\sf Y$-cycles on $\ww(\bG)$ (Definition \ref{def:Ycycle}). Since any GP-graph $\bG$ can be decomposed into elementary columns, we can try to build these absolute cycles by concatenating appropriate relative cycles associated with each elementary column. 

\subsubsection{Local representatives of naive absolute cycles in a Type 1 Column}\label{sssec:WeaveType1Column} In an elementary Type 1 column of $\bG$ with $n$ horizontal lines, there are $n-1$ faces, i.e.~gaps, between these $n$ horizontal lines. For each of these $n-1$ gaps, we identify a unique weave line as follows. First, we observe that a cross-section of the weave front $\Sigma$ associated to a Type 1 column is, by construction, the reduced expression $\w_{0,n}$ of $w_{0,n}$. In this reduced expression, the lowest Coxeter generator ($s_i$ with the smallest $i$) appears exactly $(n-1)$ times. Second, there is a geometric bijection between these $(n-1)$ faces and the $(n-1)$ appearances of the lowest Coxeter generator in the reduced expression $\w_{0,k}$. Indeed, for a face $f$ at a Type 1 column, the intersection of $\partial f$ with a Type 1 column has two connected components, which go along two neighboring horizontal lines, say the $j$th and the $(j+1)$st. Since each horizontal line is the deformation retract of a sheet in the weave front $\Sigma$, a natural choice of the local weave line representative will be the intersection of the two corresponding sheets in $\Sigma$, which in turn corresponds to the $j$th appearance of the lowest Coxeter generator. Figure \ref{fig:type_1_abs_cyc} illustrates a cross-section of the weave front for $n=4$. Figure \ref{fig:RulesCycles_Type1} illustrates all the possible cases for $n=3,4$.

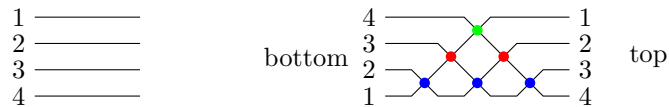
\begin{figure}[H]
    \centering
    \begin{tikzpicture}[scale=0.7]
    \foreach \i in {1,...,4}
    {
        \draw (0,-\i/2) node [left] {$\i$}  -- (2,-\i/2);
    }
 
    \end{tikzpicture} \quad \quad \quad \quad 
    \begin{tikzpicture}[scale=0.7]
    \draw (0,0) node [left] {$1$} -- (0.5,0) -- (2,1.5) -- (3.5,1.5) node [right] {$1$};
    \draw (0,0.5) node [left] {$2$} -- (0.5,0.5) -- (1,0) --(1.5,0) -- (2.5,1) -- (3.5,1) node [right] {$2$};
    \draw (0,1) node [left] {$3$} -- (1,1) -- (2,0) -- (2.5,0) -- (3,0.5) -- (3.5,0.5) node [right] {$3$};
    \draw (0,1.5) node [left] {$4$} -- (1.5,1.5) -- (3,0) -- (3.5,0) node [right] {$4$};
    \node at (-1.5,0.75) [] {bottom};
    \node at (5,0.75) [] {top};
 ;
    \foreach \i in {0,1,2}
    {
        \path [fill=blue] (0.75+\i,0.25) circle [radius =0.1];
    }
    \foreach \i in {0,1}
    {
        \path [fill=red] (1.25+\i,0.75) circle [radius=0.1];
    }
    \path [fill=green] (1.75,1.25) circle [radius=0.1];
    \end{tikzpicture}
    \caption{(Left) An elementary column with four horizontal lines. (Right) The corresponding cross-section for its associated weave surface $\Sigma(\bG)$.}
    \label{fig:type_1_abs_cyc}
\end{figure}

\begin{center}
	\begin{figure}[h!]
		\centering
		\includegraphics{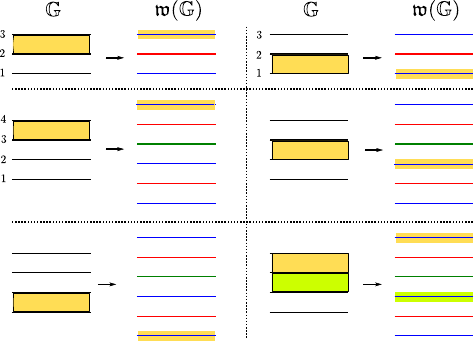}
		\caption{Associating a (piece of a) cycle in the weave for Type 1 columns. The first row depicts the two cases for $n=3$ strands, and the second and third rows depict the three cases for $n=4$, and an example with a union, in the lower right corner. In all cases, the face $f\sse\bG$ is highlighted in yellow, and the associated $s_1$-edge in the weave $\ww(\bG)$ is also highlighted in the same color. In the last case of the union, in the lower right, one of the faces and its cycle are highligthed in green.}
		\label{fig:RulesCycles_Type1}
	\end{figure}
\end{center}

\subsubsection{Local representatives of naive absolute cycles in a Type 2 column}\label{sssec:WeaveType2Column} Consider a Type 2 column with $n$ horizontal lines and a single vertical edge between the $j$th and $(j+1)$st horizontal lines. First, for the faces bounded by any other pair of consecutive horizontal lines, say $k$th and $(k+1)$st with $k\neq j$, the associated naive absolute cycle in the weave is the {\it unique} long $\sf I$-cycle connecting the corresponding weave cycles on the two adjacent Type 1 columns. In other words, one starts at the $k$th appearance (counting from below) of the lowest Coxeter generator on the left (see Subsection \ref{sssec:WeaveType1Column}) and follows that weave line straight through any hexagonal vertices. By the construction of the weave $\ww$ in Subsection \ref{sssec:crossingweave}, this process will go through the weave until it reaches its right hand side at the $k$th appearance of the lowest Coxeter generator. Figure \ref{fig:RulesCycles_Type2} depicts examples of such faces in purple. Note that, as depicted on the right of the second row in that figure, the $\sf I$-cycle might go through hexagonal vertices but shall always have the $k$th lowest Coxeter generator in $\w_{0,n}$ at the two ends.

Second, for the two faces that involve the unique vertical edge, the associated absolute cycle in $\ww$ is the unique $\sf I$-cycle that starts with the $j$th appearance of the lowest Coxeter generator at its boundary end (see Subsection \ref{sssec:WeaveType1Column}) and has the other end at the unique trivalent vertex of $\ww$. 
Figure \ref{fig:RulesCycles_Type2} depicts examples of such faces in yellow and green. Observe that, in general, these $\sf I$-cycles will also go through hexagonal vertices but always have the $j$th lowest Coxeter generator at its boundary end.

\begin{center}
	\begin{figure}[h!]
		\centering
		\includegraphics[scale=0.7]{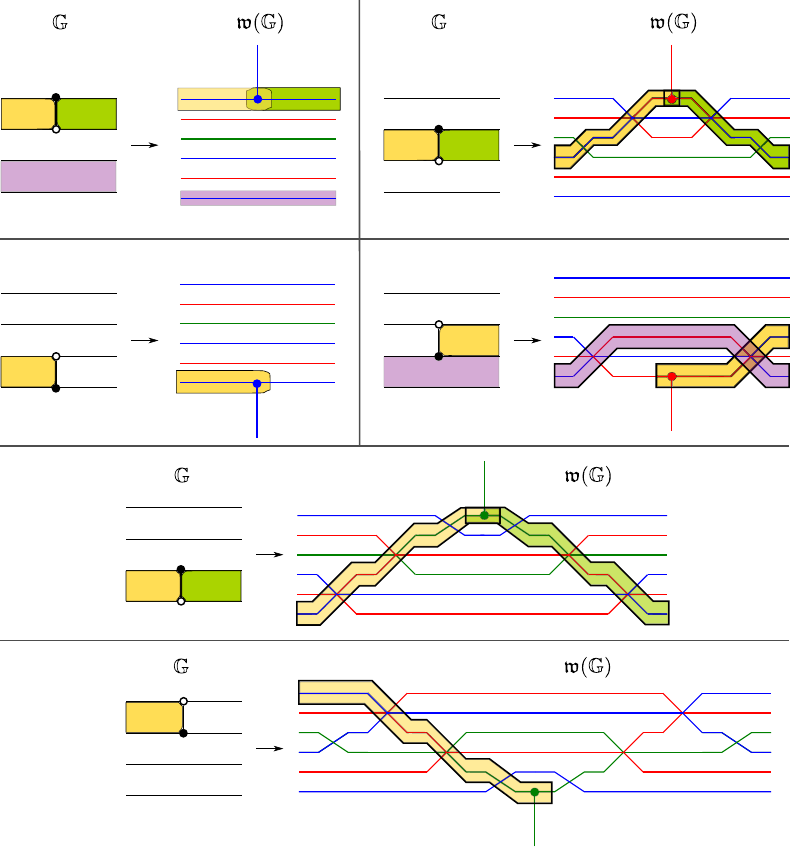}
		\caption{Several examples of faces in Type 2 elementary columns with $n=4$ $\bG$-strands, and their associated $\sf I$-cycle in the weaves $\ww(\bG)$.}
		\label{fig:RulesCycles_Type2}
	\end{figure}
\end{center}

\subsubsection{Local representatives of naive absolute cycles in a Type 3 column}\label{sssec:WeaveType3Column} In a Type 3 column, the majority of faces are similar to a face in a Type 1 column. In the weave, their boundaries are represented by weave lines going across the weave as the unique long $\sf I$-cycles with the correct boundary conditions. The only exceptional face in a Type 3 column is the face $f$ which contains a lollipop, 
which we will discuss in detail in this subsection.

Let us first consider the case of a white lollipop attaching to the $j$th horizontal line on the right. For simplicity let us assume that the horizontal lines on the right are indexed by $1,2,\dots, n$ starting from the bottom. Following Subsection \ref{sssec:WeaveType1Column}, the leftmost and rightmost ends of the cycle $\gamma_f$ are determined by the Type 1 rules. Namely, given that the face $f$ restricts to one gap on the left and two gaps on the right, the ends of the cycle $\gamma_f$ in $\ww$ must be the unique $\sf I$-cycle associated to those gaps. Thus, the cycle $\gamma_f$ will start at a blue $s_1$ edge of the weave on the left and finish at two blue $s_1$ edges on the right. Now, in general, there does not exist a $\sf I$-cycle (nor a $\sf Y$-cycle) with these boundary conditions in $\ww$. This requires introducing a {\it bident}, as follows. Consider the middle slice of $\ww$ where all the newly emerged weave lines have become horizontal, i.e.~the right boundary of the weave building block $\i_j^w$ (Definition \ref{def:incomingweave}). Reading the weave lines from bottom to top at this slice yields an expression for the half-twist $w_{0,n}\in S_n$ (note that it is not $\w_{0,n}$). Let us draw the weave slice as the positive braid $s_{[n-1,i]}^{-1}\w_{0,n-1}s_{[n-1,n-i+1]}$ and mark the $s_1$ edge in $\w_{0,n-1}$ that corresponds to the gap on the left within which the white lollipop emerges. Then, starting at this marked $s_1$-edge, we go along the upper-left and upper-right strands until we reach the highest (and last) possible crossing in each of the strands. These two crossings correspond to two weave lines on the right boundary of $\i_j^w$. These two weave lines are said to be obtained from the (left) $s_1$-edge by a \emph{bifurcation}.

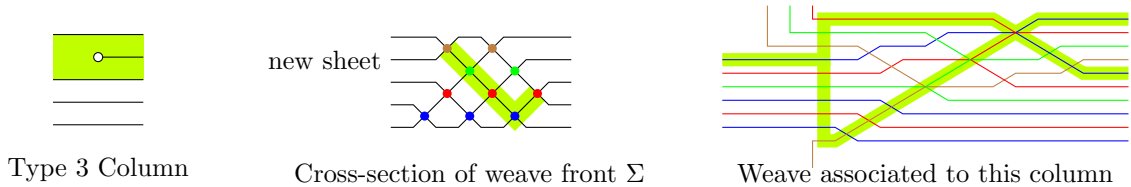
\begin{figure}[H]
    \centering
    \begin{tikzpicture}[scale=0.6]
    \path [fill=lime] (0,1) rectangle (2,2);
    \draw (0,0) -- (2,0);
    \draw (0,0.5) -- (2,0.5);
    \draw (1,1.5) -- (2,1.5);
    \draw (0,1) -- (2,1);
    \draw (0,2) -- (2,2);
    \draw [fill=white] (1,1.5) circle [radius=0.1];
    \node at (1,-1) [] {Type 3 Column};
    \end{tikzpicture} \quad \quad 
    \begin{tikzpicture}[scale=0.6]
    \draw [lime, line width = 8] (1.25,1.75) -- (2.75,0.25) -- (3.25,0.75);
    \draw (0,0) -- (0.5,0) -- (2.5,2) -- (4,2);
    \draw (0,0.5) -- (0.5,0.5) -- (1,0) --(1.5,0) -- (3,1.5) -- (4,1.5);
    \draw (0,1) -- (1,1) -- (2,0) -- (2.5,0) -- (3.5,1) -- (4,1);
    \draw (0,1.5) node [left] {new sheet} -- (1,1.5) -- (1.5,2) -- (2,2) -- (3.5,0.5) -- (4,0.5);
    \draw (0,2) -- (1,2) -- (3,0) -- (4,0);
    \node at (1.75,-1) [] {Cross-section of weave front $\Sigma$};
    \foreach \i in {0,1,2}
    {
        \path [fill=blue] (0.75+\i,0.25) circle [radius =0.1];
        \path [fill=red] (1.25+\i,0.75) circle [radius=0.1];
    }
    \foreach \i in {0,1}
    {
        \path [fill=green] (1.75+\i,1.25) circle [radius=0.1];
        \path [fill=brown] (1.25+\i,1.75) circle [radius=0.1];
    }
    \end{tikzpicture}
    \quad \quad 
    \begin{tikzpicture}[scale=0.6]
    \node at (3.5,3) [] {};
    \draw [lime, line width = 5] (-1,1.5) -- (1.25,1.5) -- (1.25,2.4) -- (5,2.4) -- (7,1.2) -- (8,1.2);
    \draw [lime, line width = 5] (1.25,1.5) -- (1.25,-0.3) -- (1.5,-0.3) -- (6,2.4) -- (8,2.4);
    \draw [blue] (-1,0) -- (1.5,0) -- (2,-0.3) -- (8,-0.3);
    \draw [red] (-1,0.3) -- (2,0.3) -- (2.5,0) -- (8,0);
    \draw [blue] (-1,0.6) -- (2.5,0.6) -- (3,0.3) -- (8,0.3);
    \draw [green] (-1,0.9) -- (3.5,0.9) -- (4,0.6) -- (8,0.6);
    \draw [brown] (1,-0.9) -- (1,-0.3) -- (1.5,-0.3) -- (3.5,0.9) -- (5,0.9) -- (5.5,1.2) -- (6.5,1.2) -- (7,1.5) -- (8,1.5);
    \draw [red] (-1,1.2) -- (2.5,1.2) -- (3,1.5) -- (4.5,1.5) -- (5.5,0.9) -- (8,0.9);
    \draw [blue] (-1,1.5) -- (2,1.5) -- (2.5,1.8) -- (3.5,1.8) -- (4,2.1) -- (5.5,2.1) -- (6,2.4) -- (8,2.4);
    \draw [brown] (0,2.7) -- (0,1.8) -- (2,1.8) -- (3.5,0.9);
    \draw [green] (0.5,2.7) -- (0.5,2.1) -- (3.5,2.1) -- (4.5,1.5) -- (3.5,0.9);
    \draw [green] (4.5,1.5) -- (6,1.5) -- (6.5,1.8) -- (8,1.8);
    \draw [red] (1,2.7) -- (1,2.4) -- (5,2.4) -- (5.5,2.1) -- (8,2.1);
    \draw [red] (4.5,1.5) -- (5.5,2.1);
    \draw [blue] (5.5,2.1) -- (7,1.2) -- (8,1.2);
    \node at (3.5,-1) [] {Weave associated to this column};
    \end{tikzpicture}
    \caption{The left is an elementary Type 3 column with a white lollipop. Its associated weave $\ww$ is depicted on the right. In the middle, a vertical slice of the weave front highlighting the two directions of bifurcation that come up from the $s_1$-crossing (in blue) at the bottom.}
    \label{fig:type_3_abs_cyc_white}
\end{figure}

\begin{definition}
A {\it bident} is a PL-embedding of a $T$-shape domain into the plane containing the weave such that on the left it coincides with an $s_1$-edge and on the right it coincides with the two crossings obtained by bifurcation on this $s_1$-edge.
\end{definition}

\begin{center}
	\begin{figure}[H]
		\centering
		\includegraphics[scale=0.8]{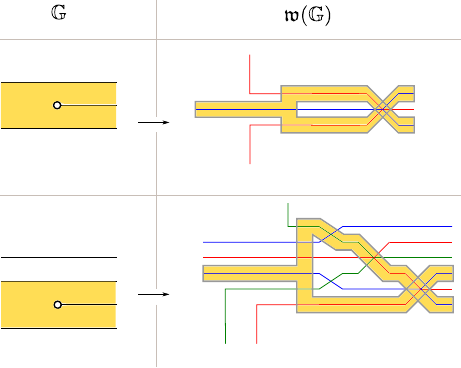}
		\caption{(Left) Examples of faces in Type 3 elementary columns with $n=3,4$ $\bG$-strands. (Right) The associated naive cycles, with the bidents, in the corresponding weaves $\ww(\bG)$.}
		\label{fig:RulesCycles_Type3}
	\end{figure}
\end{center}

Finally, the case of an elementary column of Type 3 with a black lollipop is treated in exactly the same manner as for a white lollipop, with the roles vertically reversed. For a face $f$ containing a black lollipop in a Type 3 column, the boundary conditions being $\sf I$-cycles on the weave and having a unique bident determine the cycle $\gamma_f\sse\R^2$ in the same manner as in the white lollipop case, except now the bident is left-pointing.

\begin{figure}[H]
    \centering
    \begin{tikzpicture}[scale=0.7]
    \path [fill=lime] (0,0.5) rectangle (2,1.5);
    \foreach \i in {0,1,3}
    {
        \draw (0,\i*0.5)--(2,\i*0.5);
    }
    \draw (0,1) -- (1,1);
    \draw [fill=black] (1,1) circle [radius=0.1];
    \node at (1,-1) [] {A Type 3 column $\bG$};
    \end{tikzpicture} \quad \quad \quad \quad
    \begin{tikzpicture}[scale=0.7]
    \draw [lime, line width = 5] (0,0.6) -- (1,0.6) -- (2,0) -- (2.25,0) -- (2.25,1.2);
    \draw [lime, line width = 5] (0,1.5) -- (2.25,1.5) -- (2.25,1.2) -- (5,1.2);
    \draw [blue] (0,0) -- (1,0) -- (1.5,0.3) -- (3,0.3) -- (3,-0.6);
    \draw [red] (0,0.3) -- (1.5,0.3) -- (2,0) -- (2.5,0) -- (2.5,-0.6);
    \draw [blue] (0,0.6) -- (1,0.6) -- (1.5,0.3);
    \draw [red] (1.5,0.3) -- (2,0.6) -- (5,0.6);
    \draw [green] (0,0.9) -- (5,0.9);
    \draw [red] (0,1.2) -- (5,1.2);
    \draw [blue] (0,1.5) -- (2.5,1.5) -- (2.5,2.1);
    \node at (2.5,-1) [] {Associated weave $\ww(\bG)$};
    \end{tikzpicture}
    \caption{A case for the naive absolute cycle in a Type 3 column with a black lollipop. The face $f\in\bG$ and its associated cycle $\gamma_f\sse\R^2$ are both highlighted in light green.}
    \label{fig:type_3_abs_cyc_black}
\end{figure}
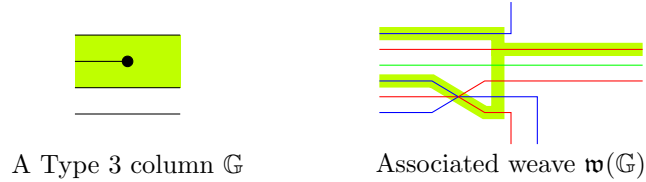

Due to the possible existence of bidents, it is hard to tell whether a naive absolute cycle has self intersections (and hence it is not an embedded absolute cycle or $\mathbb{L}$-compressible) or not. It is easier if we can represent the naive absolute cycles are $\sf Y$-cycles (Definition \ref{def:Ycycle}). Thus, we prove the following:

\begin{prop}\label{all cycles are Y-cycles} Let $\bG\sse\R^2$ be a GP-graph and $\ww=\ww(\bG)$ its associated weave. Then, there exist a weave $\ww'$ and an equivalence $\ww'\sim \ww$ such that, under the isotopy\footnote{This isotopy naturally induces a Hamiltonian isotopy between $L(\ww)$ and $L(\ww')$.} between $\Sigma(\ww)$ and $\Sigma(\ww')$, the image of each naive absolute cycle on $L(\ww)$ is homologous to a $\sf Y$-cycle on $L(\ww')$.
\end{prop}
\begin{proof} For Type 1 and Type 2 elementary columns, the associated (pieces of) naive absolute cycles are already $\sf I$-cycles, and hence $\sf Y$-cycles. In particular, for $\bG$ with no (internal) lollipops, we can take $\ww'=\ww$. It thus suffices to study the case of a Type 3 column, where a bident appears: it suffices to show that there exists a weave equivalence that allows us to replace a bident by a $\sf Y$-cycle.

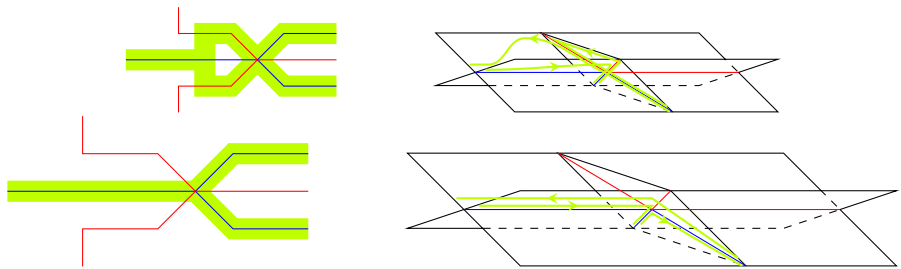
\begin{figure}[H]
    \centering
    \begin{tikzpicture}[scale=0.7]
    \draw [lime, line width = 8] (0,1) -- (1.5,1);
    \draw [lime, line width = 8] (4,0.5) -- (3,0.5) -- (2,1.5) -- (1.5,1.5) -- (1.5,0.5) -- (2,0.5) -- (3,1.5) -- (4,1.5);
    \draw [blue] (0,1) -- (2.5,1);
    \draw [red] (1,0) -- (1,0.5) -- (2,0.5) -- (2.5,1);
    \draw [red] (1,2) -- (1,1.5) -- (2,1.5) -- (2.5,1);
    \draw [blue] (2.5,1) -- (3,1.5) -- (4,1.5);
    \draw [blue] (2.5,1) -- (3,0.5) -- (4,0.5);
    \draw [red] (2.5,1) -- (4,1);
    \end{tikzpicture} \quad \quad \quad 
    \begin{tikzpicture}[scale=0.7]
    \draw (1,0.5) -- (0,0.5) -- (1.5,1) -- (6.5,1) -- (5.75,0.75) -- (6.5,0) -- (1.5,0) -- (0,1.5) -- (5,1.5) -- (5.5,1);
    \draw (2.5,1) -- (2,1.5) -- (3.5,1) -- (4.5,0);
    \draw [dashed] (1,0.5) -- (5,0.5) -- (5.75,0.75) -- (5.5,1);
    \draw [dashed] (2.5,1) -- (3,0.5) -- (4.5,0);
    \draw [blue] (0.75,0.75) -- (3.25,0.75);
    \draw [blue] (3.25,0.75) -- (4.5,0);
    \draw [blue] (3.25,0.75) -- (3,0.5);
    \draw [red] (3.5,1) -- (3.25,0.75);
    \draw [red] (2,1.5) -- (3.25,0.75);
    \draw [red] (3.25,0.75) -- (5.75,0.75);
    \draw [thick,lime, decoration={markings,mark=at position 0.5 with {\arrow{stealth}}},postaction={decorate}] (0.95,0.8) to [out=0,in=180] (3.3,0.9) to [out=0,in=45] (3.2,0.75) -- (3,0.55);
    \draw [thick,lime,decoration={markings,mark=at position 0.7 with {\arrow{stealth}}},postaction={decorate}] (4.45,0) -- (3.2,0.75) to [out=155,in=0](1.7,1.4) to [out=180,in=0] (1,0.9)--(0.65,0.9);
    \draw [thick,lime,smooth,decoration={markings,mark=at position 0.3 with {\arrow{stealth}}},postaction={decorate}] (3.1,0.55) -- (3.5,0.95) -- (2.3,1.35) -- (4.4,0.1);
    \end{tikzpicture} 
\\
    \begin{tikzpicture}
    \draw [lime, line width = 8] (0,1) -- (2.5,1);
    \draw [lime, line width = 8] (4,0.5) -- (3,0.5) -- (2.5,1) -- (3,1.5) -- (4,1.5);
    \draw [blue] (0,1) -- (2.5,1);
    \draw [red] (1,0) -- (1,0.5) -- (2,0.5) -- (2.5,1);
    \draw [red] (1,2) -- (1,1.5) -- (2,1.5) -- (2.5,1);
    \draw [blue] (2.5,1) -- (3,1.5) -- (4,1.5);
    \draw [blue] (2.5,1) -- (3,0.5) -- (4,0.5);
    \draw [red] (2.5,1) -- (4,1);
    \end{tikzpicture} \quad \quad \quad 
    \begin{tikzpicture}
    \draw (1,0.5) -- (0,0.5) -- (1.5,1) -- (6.5,1) -- (5.75,0.75) -- (6.5,0) -- (1.5,0) -- (0,1.5) -- (5,1.5) -- (5.5,1);
    \draw (2.5,1) -- (2,1.5) -- (3.5,1) -- (4.5,0);
    \draw [dashed] (1,0.5) -- (5,0.5) -- (5.75,0.75) -- (5.5,1);
    \draw [dashed] (2.5,1) -- (3,0.5) -- (4.5,0);
    \draw [blue] (0.75,0.75) -- (3.25,0.75);
    \draw [blue] (3.25,0.75) -- (4.5,0);
    \draw [blue] (3.25,0.75) -- (3,0.5);
    \draw [red] (3.5,1) -- (3.25,0.75);
    \draw [red] (2,1.5) -- (3.25,0.75);
    \draw [red] (3.25,0.75) -- (5.75,0.75);
    \draw [thick,lime, decoration={markings,mark=at position 0.5 with {\arrow{stealth}}},postaction={decorate}] (0.95,0.8) -- (3.25,0.8) -- (3,0.55);
    \draw [thick,lime,decoration={markings,mark=at position 0.7 with {\arrow{stealth}}},postaction={decorate}](4.4,0.1)  -- (3.25,0.9) --(0.65,0.9);
    \draw [thick,lime,smooth,decoration={markings,mark=at position 0.3 with {\arrow{stealth}}},postaction={decorate}] (3.1,0.55) -- (3.25,0.7) -- (4.45,0) ;
    \end{tikzpicture} 
    \caption{Two homologous cycles $\gamma_f$ and $\wt\gamma_f$ depicted on the left. The upper left cycle $\gamma_f$ contains a bident and is not a $\sf Y$-cycle, the lower left cycle $\wt\gamma_f$ is $\sf Y$-cycle. The right hand side of each row depicts these cycles in their spatial Legendrian fronts.}
    \label{fig:homotopy of a bident}
\end{figure}

For a Type 3 column with a face $f$ containing a white lollipop, this is done as follows. Consider the two weave lines to the right of the bident where the cycle $\gamma_f$ propagates. By construction, these two weave lines intersect (to the right) at a unique hexagonal weave vertex of $\ww(\bG)$. In addition, the horizontal weave line entering from the left at this hexagonal vertex connects with an $\sf I$-cycle to the $s_1$-edge on the left of the bident where $\gamma_f$ starts. Therefore, we can consider the $\sf Y$-cycle $\wt\gamma_f$ which starts with this $s_1$-edge at the left, propagates to the left (as an $\sf I$-cycle) until the hexagonal vertex and then contains a unique $\sf Y$-vertex at the hexagonal vertex. Figure \ref{fig:homotopy of a bident} depicts both cycles $\gamma_f$, at the left of the first row, and $\wt\gamma_f$, at the left of the second row, in the case of the Type 3 elementary column draw in Figure \ref{fig:RulesCycles_Type3} (upper left). By considering the description of $\gamma_f$ via vertical slices, it is readily seen that $\gamma_f$ is homologous to $\wt \gamma_f$. In consequence, in the case of a white lollipop we can consider the same weave $\ww'=\ww$ and have the naive absolute cycle $\gamma_f$ with a bident be homologous to the $\sf Y$-cycle $\wt \gamma_f$.

For a Type 3 column $\bG$ with a black lollipop, the situation is similar, with the exception that a hexagonal vertex might not exist in $\ww(\bG)$ and thus the bident cannot readily be substituted by a $\sf Y$-cycle. Nevertheless, we can insert two consecutive hexagonal vertices with a candy twist -- Move I in Figure \ref{fig:ReidemeisterWeave} -- and then apply the same argument as above.
\end{proof}


\subsection{Naive Relative Cycles in \texorpdfstring{$L(\ww(\bG))$}{}}\label{ssec:naiverelativecyles} 
Let $L=L(\ww(\bG))$ be the initial filling of the GP-link $\Lambda=\Lambda(\bG)$. Subsection \ref{ssec:NaiveAbsoluteCycles} constructed an explicit set of generators for a basis of $H_1(L)$ in terms of $\sf Y$-cycles. Nevertheless, in order to construct cluster $\mathcal{A}$-variables, we also need access to the lattice given by the relative homology group $H_1(L, \Lambda)=H_1(L,\partial L)$. Recall that, by Poincar\'{e} duality, there exists a non-degenerate pairing between the absolute homology group $H_1(L)$ and the relative homology group $H_1(L,\Lambda)$:
$$\inprod{\cdot}{\cdot}:H_1(L)\otimes H_1(L,\Lambda)\lr\Z.$$
Let $\{\gamma_f\}$ be the basis of naive absolute cycles constructed in Subsection \ref{ssec:NaiveAbsoluteCycles}, where the index $f$ runs over all faces of $\bG$. Consider the Poincar\'e dual basis $\{\eta_f\}$ on $H_1(L,\Lambda)$.\footnote{The dual of an absolute cycle $\gamma_f$ is constructed from the entire basis of naive absolute cycles, not just $\gamma_f$.} In order to perform computations in the moduli stack of sheaves, we also want to describe the relative cycles in $\{\eta_f\}$ combinatorially in terms of the weave $\ww=\ww(\bG)$. This is done according to the following discussion.

In general, given an $N$-weave $\ww\sse\bR^2$, we can consider an (unoriented) curve $\kappa\sse\bR^2$ that ends at unbounded regions in the complement of $\ww\sse\R^2$ and intersect weave lines of $\ww$ transversely and generically -- in particular, away from the weave vertices. There are $N$ natural ways to lift $\kappa$ to the weave front $\Sigma(\ww)$, which in turn correspond to $N$ unoriented curves on $L$. In consequence, any subset of these $N$ lifts, together with any orientation we choose for each of element of such a subset, defines a relative homology cycle $\eta\in H_1(L,\Lambda)$. Figure \ref{fig:intersection numbers between weave lines and dashed curves} (left) depicts two possible oriented lifts of the (dashed) yellow curve $\kappa$ drawn to its right.

\begin{figure}[H]
    \centering
    \begin{tikzpicture}[scale=0.7,baseline=-15]
    \draw [yellow, decoration={markings,mark=at position 0.5 with {\arrow{stealth}}},postaction={decorate}] (3.25,0.75) -- (2.75,0.75) -- (2,0) -- (1.25,0) -- (0.5,0.75) -- (0,0.75);
    \draw [yellow, decoration={markings,mark=at position 0.4 with {\arrow{stealth}}},postaction={decorate}] (3.25,0) -- (2.75,0) -- (1.25,1.5) -- (0,1.5);
    \path [fill=blue] (0.875,0.375) circle [radius=0.1];
    \path [fill=blue] (2.375,0.375) circle [radius=0.1];
    \path [fill=red] (1.625,1.125) circle [radius=0.1];
    \node at (-0.5,0.75) [] {bottom};
    \node at (3.75,0.75) [] {top};
    \end{tikzpicture} \quad \quad \quad\quad \quad\quad
    \begin{tikzpicture}[scale=0.7]
    \draw [blue] (0,0) -- (2,0);
    \draw [red] (0,0.75) -- (2,0.75);
    \draw [blue] (0,1.5) -- (2,1.5);
    \draw [dashed, yellow, thick] (1,-0.5) -- (1,2);
    \node at (1,0) [below right] {$1$};
    \node at (1,0.75) [below right] {$1$};
    \node at (1,1.5) [below right] {$0$};
    \end{tikzpicture}
    \caption{(Left) In yellow, two lifts of the $\kappa$ dashed curve depicted on the right. (Right) A 3-weave with a labeled dashed curve $\kappa$: the labels $0,1,1$ indicate the intersection number with each of the (pieces of cycles associated to) the weave lines.}
    \label{fig:intersection numbers between weave lines and dashed curves}
\end{figure}
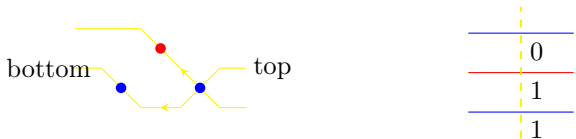

Back to the case with $\ww=\ww(\bG)$, where we have the basis of naive absolute cycles available, we can compute the intersection numbers between such relative cycles $\eta$ and the naive absolute cycles. This leads to a tuple of integers $I(\eta):=(\inprod{\eta}{\gamma_f})_{\text{faces $f$}}$. By the non-degeneracy of the Poincar\'e pairing $\inprod{\cdot}{\cdot}$ and the fact that $\{\gamma_f\}_{\text{faces $f$}}$ is a basis, the tuple of intersection numbers $I(\eta)$ uniquely determines the relative homology class of $\eta$. In fact, given that all naive absolute cycles that $\eta$ intersects non-trivially must pass through weave lines, in order to describe the relative homology class of $\eta$ it suffices to draw the unoriented curve $\kappa\sse\R^2$ and record the collection of the intersection number of its lift $\eta$ with each of the weave lines. In order to distinguish such curves from weave lines, we will use dashed lines to depict such a curve $\kappa$.

\begin{definition} A dashed curve $\kappa\sse\R^2$ as above, with the data of intersection numbers for each weave line it crosses, is called a \emph{labeled dashed curve}.
\end{definition} 

Figure \ref{fig:intersection numbers between weave lines and dashed curves} (right) depicts a labeled dashed curve $\kappa$. From a diagrammatic perspective, it is desirable to be manipulate labeled dashed curves in a weave diagram in the same manner that \cite{CasalsZaslow} explained how to combinatorially manipulate absolute cycles. For that, we have depicted in Figure \ref{fig:equivalence for labeled dashed curves} the key moves on labeled dashed curves: these are all equivalences, in that these moves do not change the relative homology classes that the labeled dashed curves represent.

\begin{figure}[H]
    \centering
    \begin{tikzpicture}[scale=0.7,baseline=0]
    \draw [blue] (0,-1) -- (0,1);
    \draw [dashed, thick, yellow] (-1,0) -- (0.5,0) -- (0.5,-1);
    \node at (0,0) [above left] {any number};
    \node at (0,-1) [below] {exterior};
    \end{tikzpicture} \quad = \quad \begin{tikzpicture}[scale=0.7,baseline=0]
    \draw [blue] (0,-1) -- (0,1);
    \draw [dashed, thick, yellow] (-1,0) -- (-0.5,0) -- (-0.5,-1);
    \node at (0,-1) [below] {exterior};
    \end{tikzpicture}\quad \quad \quad \quad
    \begin{tikzpicture}[scale=0.7,baseline=0]
    \draw [blue] (0,-1) -- (0,1);
    \draw [dashed, thick, yellow] (-1,0.7) -- (0.5,0.7) -- (0.5,-0.7) -- (-1,-0.7);
    \node at (0,0.7) [above right] {$n$};
    \node at (0,-0.7) [below right] {$-n$};
    \end{tikzpicture} \quad = \quad \begin{tikzpicture}[scale=0.7,baseline=0]
    \draw [blue] (0,-1) -- (0,1);
    \draw [dashed, thick, yellow] (-1,0.7) -- (-0.5,0.7) -- (-0.5,-0.7) -- (-1,-0.7);
    \end{tikzpicture}\\
    \begin{tikzpicture}[baseline=0]
    \foreach \i in {0,1,2}
    {
        \draw [blue] (0,0) -- (90+\i*120:1.5);
    }
    \draw [dashed, thick, yellow] (0,-0.5) -- (-0.75,0.5) -- (0.75,0.5);
    \node at (-120:0.5) [left] {$0$};
    \node at (0,0.5) [above right] {$0$};
    \end{tikzpicture} \quad =\quad \begin{tikzpicture}[scale=0.7,baseline=0]
    \foreach \i in {0,1,2}
    {
        \draw [blue] (0,0) -- (90+\i*120:1.5);
    }
    \draw [dashed, thick, yellow] (0,-0.5) -- (0.75,0.5);
    \node at (-30:0.5) [below] {$0$};
    \end{tikzpicture} \quad \quad \quad \quad
    \begin{tikzpicture}[scale=0.7,baseline=0]
    \foreach \i in {0,1,2}
    {
    \draw [red] (0,0) -- (\i*120:1.5);
    \draw [blue] (0,0) -- (60+\i*120:1.5);
    }
    \draw [dashed, thick, yellow] (-0.5,-1.25) -- (-0.5,1.25);
    \node at (-0.5,1) [right] {$i$};
    \node at (-0.5,0) [above left] {$j$};
    \node at (-0.5,-1) [right] {$k$};
    \end{tikzpicture} \quad = \quad \begin{tikzpicture}[scale=0.7,baseline=0]
    \foreach \i in {0,1,2}
    {
    \draw [red] (0,0) -- (\i*120:1.5);
    \draw [blue] (0,0) -- (60+\i*120:1.5);
    }
    \draw [dashed, thick, yellow] (0.5,-1.25) -- (0.5,1.25);
    \node at (0.5,1) [left] {$k$};
    \node at (0.5,0) [above right] {$j$};
    \node at (0.5,-1) [left] {$i$};
    \end{tikzpicture}
    \caption{Four equivalence moves for labeled dashed curves.}
    \label{fig:equivalence for labeled dashed curves}
\end{figure}
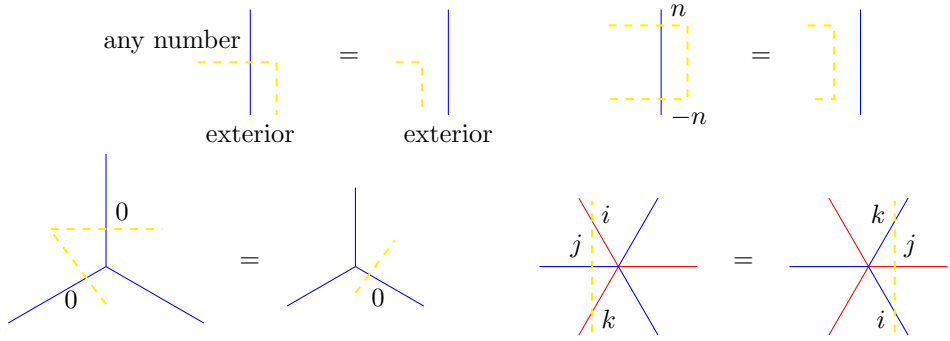

Finally, we can now diagrammatically describe a collection of labeled dashed curves that is a basis of the relative homology group $H_1(\Sigma, \Lambda)$, dual to the naive basis of $H_1(\Sigma)$ built in Subsection \ref{ssec:NaiveAbsoluteCycles}, as follows. First, for each face $f\sse\bG$, we select a Type 1 elementary column inside of $f$. Second, consider the piece of the weave $\ww(\bG)$ associated to this column -- weave lines arranged according to $\w_{0,n}$ -- and draw a vertical dashed curve $\kappa_f$ transverse to this piece of the weave. It now suffices to specify the correct labels encoding the intersection between the cycle for $\kappa$ and the naive absolute cycle associated to this face $f\in\bG$. For that, consider the braid given by slicing the weave front $\Sigma$ along $\kappa$. Recall that there is a natural bijection between the appearances of the lowest Coxeter generator $s_i$ in $\w_{0,n}$ and the gaps in this Type 1 column (see Subsection \ref{sssec:WeaveType1Column}). Locate the appearance of $s_i$ that corresponds to a gap belonging to $f$ and consider set of crossings in this braid which are contained within the rhomboid diamond whose unique lowest vertex is at this appearance of $s_i$. Figure \ref{fig:naive relative cycle} (center) draws an example of such a diamond for the face $f\sse\bG$ depicted to its left. Then we label the curve $\kappa$, to a labeled curve $\kappa_f$, by assigning the intersection number $1$ for all the weave lines in $\ww(\bG)$ which are associated to crossings in the braid {\it inside} of the diamond, and by assigning the intersection number $0$ for all the remaining weave lines. Figure \ref{fig:naive relative cycle} (right) depicts the corresponding curve $\kappa_f$ with its intersection labels for the face $f\sse\bG$.

\begin{figure}[H]
    \centering
    \begin{tikzpicture}
    \foreach \i in {0,...,4}
    {
    \draw (0,\i*0.6) -- (2,\i*0.6);
    }
    \node at (1,0.9) [] {$f$};
    \end{tikzpicture} \quad \quad \quad \quad \quad 
    \begin{tikzpicture}
    \draw [yellow,  decoration={markings,mark=at position 0.4 with {\arrow{stealth}}},postaction={decorate}] (0,0) -- (0.3,0) -- (1.5,2)-- (2.7,2);
    \draw [yellow,  decoration={markings,mark=at position 0.5 with {\arrow{stealth}}},postaction={decorate}] (0,0.5) -- (0.3,0.5) -- (0.6,0) -- (0.9,0) -- (1.8,1.5) -- (2.7,1.5);
    \foreach \i in {0,...,3}
    {
    \path [fill=blue] (0.45+\i*0.6,0.25) circle [radius=0.1];
    }
    \foreach \i in {0,1,2}
    {
    \path [fill=red] (0.75+\i*0.6,0.75) circle [radius=0.1];
    }
    \foreach \i in {0,1}
    {
    \path [fill=green] (1.05+\i*0.6,1.25) circle [radius=0.1];
    }
    \path [fill=brown] (1.35,1.75) circle [radius=0.1];
    \draw [dotted] (1.05,-0.25) -- (1.95,1.25) -- (1.35,2.25) -- (0.45,0.75) -- cycle;
    \end{tikzpicture} \quad \quad \quad \quad \quad
    \begin{tikzpicture}[scale=0.5]
    \draw [blue] (0,0) -- (4,0);
    \draw [red] (0,0.5) -- (4,0.5);
    \draw [blue] (0,1) -- (4,1);
    \draw [green] (0,1.5) -- (4,1.5);
    \draw [red] (0,2) -- (4,2);
    \draw [blue] (0,2.5) -- (4,2.5);
    \draw [brown] (0,3) -- (4,3);
    \draw [green] (0,3.5) -- (4,3.5);
    \draw [red] (0,4) -- (4,4);
    \draw [blue] (0,4.5) -- (4,4.5);
    \draw [thick, dashed, yellow] (2,-0.5) -- (2,5);
    \foreach \i in {1,2,3,4,6,7}
    {
    \node at (2,\i*0.5) [right] {\footnotesize{$1$}};
    }
    \foreach \i in {0,5,8,9}
    {
    \node at (2,\i*0.5) [right] {\footnotesize{$0$}};
    }
    \end{tikzpicture}
    \caption{(Left) A face $f\in\bG$ chosen at an elementary Type 1 column. (Center) A slice of the weave $\ww(\bG)$ with the rhomboid diamond associated to the $s_1$-crossing, in blue, for the face $f\in\bG$. (Right) The labeled dashed curve $\kappa_f$, in yellow, in the weave $\ww(\bG)$, with $1$ in the weave lines for the crossings inside the diamond, and $0$ otherwise.}
    \label{fig:naive relative cycle}
\end{figure}
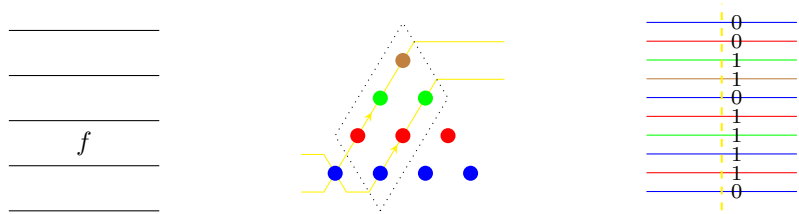

By construction, the intersection pairing between the labeled dashed curve $\kappa_f$ and the naive absolute cycles is given by $\langle \eta_f,\gamma_f\rangle=1$ and $\langle \eta_f,\gamma_g\rangle=0$ for $g\neq f$. Thus, the collection of relative cycles $\{\eta_f\}$, associated to these particular labeled dashed curves $\kappa_f$, are representatives of a dual naive basis of $H_1(L,\Lambda)$. We call this collection of relative cycles the \emph{naive basis of relative cycles} of the relative homology group $H_1(L,\Lambda)$.


\subsection{Initial Absolute Cycles and Initial Relative Cycles in \texorpdfstring{$L(\ww(\bG))$}{}}\label{ssec:InitialCycles} Subsections \ref{ssec:NaiveAbsoluteCycles} and \ref{ssec:naiverelativecyles} above explain the construction of the naive basis of absolute cycles and the corresponding naive basis of relative cycles. The generators of these basis are not geometrically appropriate: despite being $\sf Y$-cycles in the weave (or dual to them), they are often represented by {\it immersed} cycles and it is a priori unclear whether it is possible to mutate at them.\footnote{In any sense of the word mutation: geometrically, through a Lagrangian surgery, diagrammatically, via a weave mutation, or cluster-theoretically, mutating at the naive vertex representing them in the naive quiver.} A key idea in this manuscript is the consideration and study of {\it sugar-free} hulls, as introduced in Subsection \ref{ssec:hulls}. In this section, these two parts, sugar-free hulls and the study of homology cycles compatible with the weave $\ww(\bG)$, converge together: we show that it is possible to associate a $\sf Y$-tree absolute cycle on $\ww(\bG)$ to every sugar-free hull of $\bG$. In consequence, given that $\sf Y$-trees are {\it embedded}, it will be possible to perform a weave mutation at every sugar-free hull of $\bG$. This leads to the notion of {\it initial basis}, which eventually give rise to the initial seeds for our cluster structures.


\noindent In the study of sugar-free hulls, we must consider cycles which are associated to regions of $\bG$ -- namely, the sugar-free hulls -- and not just faces $f\sse\bG$. The simplest case is that of a region in an elementary column of Type 1, which is considered in the following simple lemma, where we use the weave $\n(\w_{0,n})$ introduced in Definition \ref{def:weaveType1}.

\begin{lemma}\label{lemma: gap & weave line correspondence} Let $\bG$ be a GP-graph and $C\sse\bG$ a Type 1 elementary column. Consider the region $R\sse C$ given by the union of $k$ consecutive gaps in $C$. Then the boundary of $\dd R$ is homologous to the lift of a unique weave line on the $k$th level.
\end{lemma}
\begin{proof} The boundary of a single gap has two connected components, and each of them is a deformation retract of a sheet of the spatial wavefront $\Sigma(C)$. For a single gap, the two sheets associated with its boundary intersect at a unique weave line at the bottom level. For a union $R$ of $k$ consecutive gaps, the two sheets associated with $\partial R$ intersect at the $k$th level. The conclusion follows.
\end{proof}

Figure \ref{fig:RulesCycles_Type1Gaps} depicts four cases illustrating how to associate a cycle on a weave line for a region on a Type 1 column. The case of arbitrary strands can be readily imagined by examining these few cases. 

\begin{center}
	\begin{figure}[h!]
		\centering
		\includegraphics{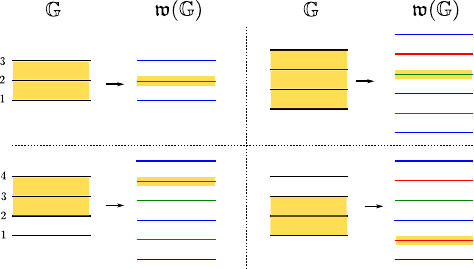}
		\caption{Associating a (piece of a) cycle in the weave for regions on a Type 1 column. The first row depicts the case where the region is the entire column, for $n=3$ and $4$ strands. The second row depicts the remaining two cases for $n=4$ strands.}
		\label{fig:RulesCycles_Type1Gaps}
	\end{figure}
\end{center}

\noindent Proposition \ref{all cycles are Y-cycles} showed that it is possible, up to possibly performing a weave equivalence, to represent the naive absolute cycles with $\sf Y$-cycles. Nevertheless, these are typically immersed: it is not always possible, in general, to find embedded $\sf Y$-cycles representing these homology classes. Now, the following result, which we refer to as the $\sf Y$-representability Lemma, shows that it {\it is} possible to represent the boundary cycle $\dd R$ by a $\sf Y$-tree if the region $R$ is sugar-free.

\begin{lemma}[$\sf Y$-Representability Lemma]\label{lem: Y representability} Let $\bG$ be a GP-graph and $R\sse\bG$ a sugar-free region. Then the boundary cycle $\partial R$ is homologous to a $\sf Y$-tree on $\ww(\bG)$, up possibly performing a weave equivalence that adds $\n_k^\uparrow(\w_{n,0})\n_k^\uparrow(\w_{n,0})^{op}$ and $\n_k^\downarrow(\w_{n,0})\n_k^\downarrow(\w_{n,0})^{op}$ to $\ww(\bG)$.
\end{lemma}
\begin{proof} Let $C$ be an elementary column of $\bG$. By Lemma \ref{lemma:single component}, the intersection $R\cap C$ has at most one connected component. Therefore, if $R$ intersects $C$ non-trivially, $R$ must be a union of consecutive gaps in the column $C$. By Lemma \ref{lemma: gap & weave line correspondence}, for a Type 1 or a Type 3 elementary column $C$, we can represent the boundary $\partial (R\cap C)$ by a single $\sf I$-cycle weave line going from left to right; note that the weave line color may change within a Type 3 column. It thus remains to treat the cases of elementary columns of Type 2. By Lemma \ref{lemma:staircase} applied to a Type 2 column, we conclude that the following four cases -- in correspondence with the four staircase patterns -- are to be analyzed.
\[
\begin{tikzpicture}[scale=0.7],
\draw (0,0) -- (1,0) -- (1,1) -- (2,1);
\node at (0,1) [] {$R$};
\draw [fill=white] (1,1) circle [radius=0.1];
\draw [fill=black] (1,0) circle [radius=0.1];
\end{tikzpicture} \quad \quad 
\begin{tikzpicture}[scale=0.7],
\draw (0,0) -- (-1,0) -- (-1,1) -- (-2,1);
\node at (0,1) [] {$R$};
\draw [fill=black] (-1,1) circle [radius=0.1];
\draw [fill=white] (-1,0) circle [radius=0.1];
\end{tikzpicture}\quad \quad
\begin{tikzpicture}[scale=0.7],
\draw (0,1) -- (1,1) -- (1,2) -- (2,2);
\node at (2,1) [] {$R$};
\draw [fill=black] (1,1) circle [radius=0.1];
\draw [fill=white] (1,2) circle [radius=0.1];
\end{tikzpicture}\quad \quad
\begin{tikzpicture}[scale=0.7],
\draw (0,2) -- (1,2) -- (1,1) -- (2,1);
\node at (0,1) [] {$R$};
\draw [fill=white] (1,1) circle [radius=0.1];
\draw [fill=black] (1,2) circle [radius=0.1];
\end{tikzpicture}
\]

First, let us consider the staircase pattern which is second from the left. The corresponding local weave pattern is $\c_k^\uparrow(\w_{0,n})$, as introduced in Definition \ref{def:weavecrossing}. In the construction of this weave pattern, we first bring the $k$th strand in the bottom level upward, using the weave pattern $\n_k^\uparrow(\w_{0,n})$, subsequently insert a trivalent weave vertex at the top strand, and then insert the weave pattern $\n_k^\uparrow(\w_{0,n})^{op}$. Figure \ref{fig:Example of the second staircase} depicts a case with $4$ horizontal lines and $k=2$. (See also Figures \ref{fig:RulesWeaves_CrossingExamples} and \ref{fig:RulesWeaves_CrossingExamples2}.) Since the $k$th horizontal line is the bottom boundary of $\partial R$, there must be a unique hexavalent weave vertex in the $\n_k^\uparrow(\w_{0,n})^{op}$ that connects to the weave line corresponding to the union of all gaps in $R$ at the right boundary. Therefore, we can create a $\sf Y$-tree in this region, with the required boundary conditions, by inserting a tripod leg that connects to the trivalent weave vertex (at the top) and continues to the left towards whichever weave line is required by the boundary condition of $R$. This resulting $\sf Y$-tree, in the shape of a tripod, then represents $\partial R$ locally, as desired. Figure \ref{fig:Example of the second staircase} (right) depicts this $\sf Y$-tree, highlighted in light green, for the region $R$ on the left, also drawn in the same color. This concludes the second case among the four staircase patterns; the third case, which also contains the region $R$ to the right of the crossing, can be resolved analogously.

\begin{figure}[H]
    \centering
    \begin{tikzpicture}[scale=0.7]
    \path [fill=lime] (0,1.4) -- (1.5,1.4) -- (1.5,0.7) -- (3,0.7) -- (3,2.1) -- (0,2.1) -- cycle;
    \foreach \i in {0,...,3}
    {
        \draw (0,\i*0.7) -- (3,\i*0.7);
    }
    \draw (1.5,0.7) -- (1.5,1.4);
    \draw [fill=black] (1.5,1.4) circle [radius=0.15];
    \draw [fill=white] (1.5,0.7) circle [radius=0.15];
    \end{tikzpicture} \quad \quad \quad\quad \quad \quad 
    \begin{tikzpicture}[scale=0.7]
    \draw [lime, line width = 5] (2,2) -- (2.5,1.6) -- (4,1.6);
    \draw [lime, line width=5] (0,2) -- (1,2) -- (2,1.2) -- (2.5,1.6);
    \draw [blue] (0,0) -- (4,0);
    \draw [red] (0,0.4) -- (4,0.4);
    \draw [blue] (0,0.8) -- (0.5,0.8) -- (1.5,1.6) -- (2.5,1.6) -- (3.5,0.8) -- (4,0.8);
    \draw [green] (0,1.2) -- (0.5,1.2) -- (1,0.8) -- (3,0.8) -- (3.5,1.2) -- (4,1.2);
    \draw [red] (0,1.6) -- (1.5,1.6) -- (2,1.2) -- (2.5,1.6) -- (4,1.6);
    \draw [blue] (0,2) -- (1,2) -- (1.5,1.6);
    \draw [red] (1.5,1.6) -- (2,2) -- (2.5,1.6);
    \draw [blue] (2.5,1.6) -- (3,2) -- (4,2);
    \draw [red] (2,2) -- (2,3);
    \end{tikzpicture}
    \caption{(Left) Type 2 column with region $R$ highlighted in light green. (Right) The associated local weave and the $\sf Y$-tree, in the shape of a tripod.}\label{fig:Example of the second staircase}
\end{figure}
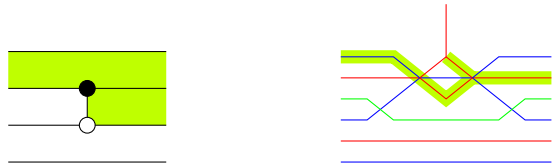

Next let us study the first (leftmost) of the four staircase patterns. In this case, the chosen weave pattern $\c_k^\downarrow(\w_{0,n})$, as assigned in Subsection \ref{ssec:InitialWeave}, does not already have a hexavalent weave vertex that meets our need. Nevertheless, we can create it by concatenating the local weave pieces $\n_k^\uparrow(\w_{0,n})$ and $\n_k^\uparrow(\w_{0,n})^{op}$ on the left first. Note that this can be achieved by inserting a series of Moves I and V, and thus the equivalence class of the weave remains the same. Now, inside of the weave piece $\n_k^\uparrow(\w_{0,n})^{op}$, there exists a unique hexavalent weave vertex that connects to both the weave line representative of $\partial R$, to the left, and the trivalent weave vertex in $\c_k^\downarrow(\w_{0,n})$, to the right. The $\sf Y$-tree, again in a tripod shape, represents $\partial R$ locally, as desired. Figure \ref{fig:Example of the first staircase} depicts an example of such tripod with $n=4$ and $k=2$. An analogous argument also resolves the case of the fourth (rightmost) staircase pattern.
\begin{figure}[H]
    \centering
    \begin{tikzpicture}[scale=0.7]
    \path [fill=lime] (0,0.7) -- (1.5,0.7) -- (1.5,1.4) -- (3,1.4) -- (3,2.1) -- (0,2.1) -- cycle;
    \foreach \i in {0,...,3}
    {
        \draw (0,\i*0.7) -- (3,\i*0.7);
    }
    \draw (1.5,0.7) -- (1.5,1.4);
    \draw [fill=white] (1.5,1.4) circle [radius=0.15];
    \draw [fill=black] (1.5,0.7) circle [radius=0.15];
    \end{tikzpicture}\quad \quad \quad\quad \quad \quad
    \begin{tikzpicture}[scale=0.7]
    \draw [lime, line width = 5] (0,1.6) -- (2.5,1.6) -- (3,2) -- (6.5,2);
    \draw [lime, line width=5] (2.5,1.6) -- (3.5,0.8) -- (4,0.8) -- (5,0);
    \draw [blue] (0,0) -- (4,0) -- (4.5,0.4) -- (5.5,0.4) -- (6,0) -- (6.5,0);
    \draw [red] (0,0.4) -- (4.5,0.4) -- (5,0) -- (5.5,0.4) -- (6.5,0.4);
    \draw [blue] (0,0.8) -- (0.5,0.8) -- (1.5,1.6) -- (2.5,1.6) -- (3.5,0.8) -- (4,0.8) -- (4.5,0.4);
    \draw [blue] (5.5,0.4) -- (6,0.8) -- (6.5,0.8);
    \draw [red] (5,0) -- (5,-1);
    \draw [red] (4.5,0.4) -- (5,0.8) -- (5.5,0.4);
    \draw [green] (0,1.2) -- (0.5,1.2) -- (1,0.8) -- (3,0.8) -- (3.5,1.2) -- (6.5,1.2);
    \draw [red] (0,1.6) -- (1.5,1.6) -- (2,1.2) -- (2.5,1.6) -- (6.5,1.6);
    \draw [blue] (0,2) -- (1,2) -- (1.5,1.6);
    \draw [red] (1.5,1.6) -- (2,2) -- (2.5,1.6);
    \draw [blue] (2.5,1.6) -- (3,2) -- (6.5,2);
    \end{tikzpicture}
    \caption{Example of the first staircase.}\label{fig:Example of the first staircase}
\end{figure}
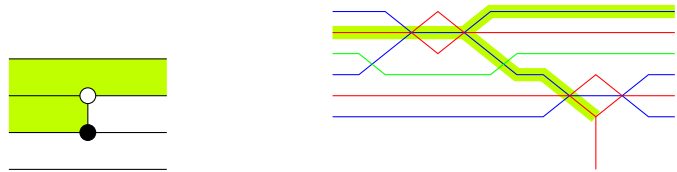
Finally, by combining the local pictures for all three types of columns, we conclude that there exists a representative for $\partial R$ which is a $\sf Y$-tree.\end{proof}

The $\sf Y$-representability Lemma allows us to introduce the following definition.

\begin{definition}
Let $\bG$ be a GP-graph. The set of \emph{initial absolute cycles} $\fS(\bG)$ is the set of all $\sf Y$-trees on $L(\ww(\bG))$ which are associated to the (non-empty) sugar-free hulls in $\bG$.
\end{definition}

The inclusion relation between sugar-free hulls naturally puts a partial order on the set $\fS(\bG)$: by definition, $\partial R\leq \partial R'$ if $R\subset R'$ as sugar-free hulls. Also, given a face $f\sse\bG$, we note that a face $g\in \S_f$ in its sugar-free hall, must satisfy $\S_g\subset \S_f$ and hence $\partial \S_g\leq \partial \S_f$. Finally, note that multiple faces in $\bG$ can share the same sugar-free hull, and there may exist faces with an empty sugar-free hull. Thus, in general, the number of sugar-free hulls may be smaller than the number of faces in $\bG$. Nevertheless, we can prove that the set of initial absolute cycles $\fS(\bG)$ is always linearly independent.

\begin{prop}\label{prop:replacement construction} Let $\bG$ be a GP-graph and let $\fS(\bG)$ be its set of initial absolute cycles. Then $\fS(\bG)$ is a linearly independent subset of $H_1(L(\ww(\bG)))$. In addition, it is possible to add naive absolute cycles to $\fS(\bG)$ to complete $\fS(\bG)$ into a basis of $H_1(L(\ww(\bG)))$.
\end{prop}
\begin{proof} Let $L=L(\ww(\bG))$ be the initial filling. Consider the naive basis $\{\gamma_f\}_{\text{faces $f$}}$ of $H_1(L)$. It suffices to show that we can replace $|\fS(\bG)|$ many naive basis elements with elements in $\fS(\bG)$ while maintain the spanning property of the set. This is can be done by using the partial order on $\fS(\bG)$ as follows. 

Let us start with the minimal elements in $\fS(\bG)$ and work our way up, replacing the appropriate naive basis elements in $\{\gamma_f\}$ with elements in $\fS(\bG)$. At each turn, we select a face $f\sse\bG$ whose sugar-free hull $\S_f$ defines a $\sf Y$-tree $\partial \S_f\in \fS(\bG)$ and then replace the naive absolute cycle $\gamma_f$ with the $\sf Y$-tree $\partial \S_f$. Note that it is always possible to find such a face $f\sse\bG$ in this process: if no such face $f$ were available, any faces within the sugar-free hull would not have this particular sugar-free region as their sugar-free hull, which is tautologically absurd.

In the process of implementing each of these replacements, we must argue that the resulting set is still spanning the homology group $H_1(L)$. For that, we observe that if a replacement of $\gamma_g$ by $\partial \S_g$ is done before the replacement of $\gamma_f$ by $\S_f$, the sugar-free hull $\S_g$ must not contain the face $f$ inside. This follows from the fact that $f\in \S_g$ implies $\S_f\subset \S_g$. Therefore, we can use the set $\{\partial \S_g \mid g\in \S_f\}$, which is already in the basis set due to the partial order, and $\partial \S_f$ to recover $\gamma_f=\partial f$. This shows that the resulting set after the replacement of $\gamma_f$ by $\partial \S_f$ as above, is still a basis for $H_1(L)$. 
\end{proof}

The choice of completion of $\fS(\bG)$ to a basis in $H_1(L(\ww(\bG))$ is a neat instance of the natural appearance of quasi-cluster structures in (symplectic) geometry. There is no particular canonical manner by which we can typically choose a basis for this complement but, as we shall explain, the different choices all lead to the same cluster structure {\it up to} monomials in the frozen variables, i.e.~a quasi-cluster structure. The dualization process, via the Poincar\'e pairing, requires a choice of basis. Therefore, there is no canonical choice of {\it initial relative cycles} associated to the set $\fS(\bG)$ unless the later spans $H_1(L(\ww(\bG))$. In a general situation, we can at least consider the following concept.

\begin{definition} Let $\bG$ be a GP-graph, $L=L(\ww(\bG))$, $\Lambda=\Lambda(\bG)$ and $\fS(\bG)$ its set of initial absolute cycles on $L$. Let $\fB$ be a basis of $H_1(L)$ which is obtained by adding naive absolute cycles to the set $\fS(\bG)$. Then the \emph{initial relative cycles} associated to $\fB$ is the collection $\fB^\vee$ of linear combinations of naive relative cycles whose relative homology classes form a basis of $H_1(L,\Lambda)$ dual to the basis $\fB$ of $H_1(L)$.\end{definition}

We emphasize that the basis $\fB^\vee$ depends not only on $\fS(\bG)$, but also on the chosen basis completion $\fB$. Note that a few simple choices of $\fB$ are available as a result of the replacement construction in the proof of Proposition \ref{prop:replacement construction}. Namely, we start with the naive absolute basis $\{\gamma_f\}$, and then swaps some of the basis elements $\gamma_f$ with their corresponding initial absolute cycles $\partial \S_f$. In case that there are multiple faces sharing the same sugar-free hull, only one of the naive absolute cycles gets replaced, and the rest remain in the basis, which will become frozen basis elements.

Furthermore, if the basis $\fB$ is chosen via such a basis replacement process, then the corresponding dual relative cycle basis $\fB^\vee$ can be described with respect to the partial order on sugar-free hulls as well. Indeed, suppose that we have an equality $\S_f=\S_g=\cdots$ of sugar-free hulls, for some faces $f,g\sse\bG$, among others, and $\gamma_f$ was chosen to be replaced by $\partial \S_f$ in the replacement process. Then, the naive relative cycle $\eta_g$ needs to be replaced by $\eta_g-\eta_f$ for each $g$ with $\S_g=\S_f$. Similarly, $\eta_f$ would need to be replaced by $\eta_f+N$, where the $N$ summand is a linear combination of the naive relative cycles $\eta_h$ associated to the chosen faces $h$ with $\S_f\subsetneq \S_h$, such that the pairing of $(\eta_f+N)$ with $\partial \S_h$ vanishes for all such $h$.


\subsection{The naive quiver of a GP-graph}\label{ssec:QuiverGP} Let us start by emphasizing that the quiver of the initial seed for the cluster structure we construct is {\it not} always the dual quiver of the GP-graph $\bG$. Nevertheless, that dual quiver is useful in order to construct the actual quiver of the initial seed because it can be used to compute the intersection form on the initial filling, which is needed to define the initial quiver. Let us provide the details.

\noindent Following \cite{GHKK}, in order to define a cluster structure, we first fix an integer lattice with a skew-symmetric form on it. In the context of a GP-graph $\bG$, the integer lattice is $H_1(L(\ww(\bG)))$, and a natural skew-symmetric form on it is given by the intersection pairing between absolute homology classes. By using the {\it naive} basis $\{\gamma_f\}$ of naive absolute cycles and the GP-graph $\bG$, we can describe the intersection pairing form $\{\cdot,\cdot\}$ combinatorially using a quiver.

\begin{definition} Let $\bG\sse\R^2$ be a GP-graph. The \emph{naive quiver} $Q_0(\bG)$, or {\it dual quiver}, associated to $\bG$ is the quiver constructed as follows:
\begin{enumerate}
    \item A quiver vertex is associated to each face $f\sse\bG$.
    \item For every bipartite edge in $\bG$, we draw an arrow according to $\begin{tikzpicture}[scale=0.5,baseline=0]
    \draw (0,-0.5) -- (0,0.5);
    \draw [fill=black] (0,0.5) circle[radius=0.1];
    \draw [fill=white] (0,-0.5) circle[radius=0.1];
    \node [blue] (l) at (-1,0) [] {$\bullet$};
    \node[ blue](r) at (1,0) [] {$\bullet$};
    \draw [->,blue] (l) -- (r);
    \end{tikzpicture}$.
    \item For each pair of quiver vertices, sum up the arrows between them.
\end{enumerate}
Note that in the third step (3) there might be cancellations.
\end{definition}


\begin{lemma}[{\cite[Definition 8.2 and Proposition 8.3]{GonKen}}]\label{lemma:skew-symmetric pairing and quiver} Let $\epsilon_{fg}$ be the exchange matrix of the quiver $Q_0(\bG)$. Then the intersection pairing between $\gamma_f$ and $\gamma_g$ is given by $\{\gamma_f,\gamma_g\}=\epsilon_{fg}$.
\end{lemma}

Since $\{\gamma_f\}_{\text{faces $f$}}$ is a basis of $H_1(L(\ww(\bG)))$, Lemma \ref{lemma:skew-symmetric pairing and quiver} uniquely determines the intersection skew-symmetric form. This intersection form, i.e.~the quiver $Q_0(\bG)$, is then used to compute the correct quiver $Q(\bG)$ for the initial seed. The unfrozen vertices of the correct initial quiver $Q(\bG)$ will be indexed by $\fS(\bG)$, the set of initial absolute cycles, and the remaining frozen vertices are determined by the choice of completion of $\fS(\bG)$ to a basis of $H_1(L)$. Since we will elaborate more on this in Section \ref{sec:cluster}, we conclude this discussion for now and revisit $Q(\bG)$ then.

\begin{remark} In the case of a plabic fence $\bG$, all naive absolute cycles are $\sf I$-cycles and thus they also are initial absolute cycles. Thus, for a plabic fence, $Q(\bG)$ coincides with $Q_0(\bG)$.
\end{remark}


\subsection{Bases and homology lattices in the presence of marked points}\label{ssec:marked_points} The construction of the cluster structures in Theorem \ref{thm:main}, and the definition of the moduli space $\fM(\Lambda, T)$, in general require an additional piece of data:
a set $T$ of marked points on $\La(\bG)$. 

\begin{definition}\label{def:markedpoints} Let $\bG$ be a GP-graph and let $\Lambda=\Lambda(\bG)$ be its GP-link. A set of \emph{marked points} $T\subset \Lambda$ is a subset of distinct points in $\La$, where we require that there is at least one marked point on each link component of $\Lambda$ and, without loss of generality, the set $T$ is disjoint from all crossings and all cusps in the front $\mathfrak{f}(\bG)$ of $\Lambda$.
\end{definition}

All prior statements in Section \ref{sec:weaves} remain unchanged by the addition of marked points, as they do not affect the associated weaves or the Hamiltonian isotopy class of exact Lagrangian fillings. Therefore, we can still consider the initial embedded exact filling $L=L(\ww(\bG))$ of the GP-link $\Lambda$. As before, we select the collection of initial absolute cycles $\fS(\bG)$ associated with sugar-free hulls, and they form a linearly independent subset of $H_1(L)$. The addition of marked points affects only the cluster-theoretic constructions: we need to replace the lattice of absolute homology $H_1(L)$ by the lattice of relative homology $H_1(L,T)$. The natural inclusion $H_1(L)\subset H_1(L,T)$, induced by the inclusions $T\sse\La=\dd L\sse L$, allows us to include the initial absolute cycles $\fS(\bG)$ as a linearly independent subset of $H_1(L,T)$. The only difference is that, in order to fix a cluster structure, we must expand $\fS(\bG)$ further to a basis $\fB$ of $H_1(L,T)$. This expansion can be done in two steps: we first expand $\fS(\bG)$ to a basis of $H_1(L)$, as done via the replacement process in Subsection \ref{ssec:InitialCycles}, and then expand this basis of $H_1(L)$ to a basis of $H_1(L,T)$. 

As was the case for $H_1(L)$ and its dual $H_1(L,\La)$, we shall need a dual space of $H_1(L,T)$ together with a basis dual to a chosen basis $\fB$ of $H_1(L,T)$. In fact, there is a natural intersection pairing
$$\inprod{\cdot}{\cdot}:H_1(L,T)\otimes H_1(L\setminus T,\Lambda\setminus T)\lr\Z$$
obtained by algebraically counting geometric intersections of relative cycles in generic position. In the same manner that Poincar\'e duality was used in Subsection \ref{ssec:InitialCycles}, a duality also exists in the setting with marked points. We record the precise statement in the following.

\begin{prop}\label{prop:pairing_markedpoints}
Let $L$ be a connected smooth surface with boundary $\La=\dd L$, and $i:T\lr\La$ an inclusion of a set of marked points with $\pi_0(i)$ being surjective. Then $$\rank (H_1(L,T))=\rank (H_1(L\setminus T,\Lambda\setminus T)),$$
and the intersection pairing $\inprod{\cdot}{\cdot}$ is non-degenerate.
\end{prop}

\noindent It is possible to consider intermediate lattices $M$ and $N$ in between the lattices discussed above. Namely, we can consider sublattices $N$ of $H_1(L,T)$ which include $H_1(L)$, and dually quotients $M$ of $H_1(L\setminus T,\La\setminus T)$, as in the following diagram, where all horizontal arrows are dual lattices:
$$
\begin{tikzcd}
 H_1(L,T) \arrow[r,<->] & H_1(L\setminus T,\La\setminus T) \arrow[d, two heads] \\
N \arrow[u,hook] \arrow[r,<->] & M \arrow[d, two heads] \\
H_1(L)  \arrow[u,hook]\arrow[r,<->] & H_1(L,\La) 
\end{tikzcd}$$


\section{Construction of Quasi-Cluster Structures on Sheaf Moduli}\label{sec:cluster}

In this section we develop the necessary results to study the geometry of the moduli stack $\FM(\La,T)$ associated to $\La=\La(\bG)$ and prove Theorem \ref{thm:main}. In particular, we introduce microlocal merodromies in Subsection \ref{ssec:clusterAvars} which, as we will prove, become the cluster $\mathcal{A}$-variables. The construction of the cluster structures is obtained purely by symplectic geometric means, using the results for Legendrian weaves from Section \ref{sec:weaves} above and \cite{CasalsZaslow} and the microlocal theory of sheaves \cite{Sheaves1,KashiwaraSchapira_Book,STZ_ConstrSheaves}. Let us review what we have developed in Sections \ref{sec:gridplabic} and \ref{sec:weaves} thus far. Given a GP-graph $\bG$, we constructed the following list of objects:
\begin{itemize}\setlength\itemsep{0.5em}
    \item[(i)] A Legendrian link $\La=\La(\bG)$, which is a $(-1)$ closure of a positive braid $\beta(\bG)$.
    \item[(ii)] An exact Lagrangian filling $L=L(\ww)$ of $\La$ called the \emph{initial filling}. This exact Lagrangian filling $L$ is obtained as the Lagrangian projection of the Legendrian lift associated with the spatial front defined by the \emph{initial weave} $\ww=\ww(\bG)$.
    \item[(iii)] A collection of \emph{initial absolute cycles} $\fS(\bG)$, which form an $\mathbb{L}$-compressing system for $L$ and can be described by $\sf Y$-trees on $\ww$.
    \item[(iv)] A skew-symmetric intersection pairing on the lattice $H_1(L)$. This intersection pairing can be computed directly from the GP-graph $\bG$.
\end{itemize}

By specifying an additional generic set of \emph{marked points} $T\subset \La$ with at least one marked point per component, we also obtain the lattice $H_1(L,T)$, which contains $H_1(L)$ and hence the linearly independent subset $\fS(\bG)$. The skew-symmetric pairing on $H_1(L)$ extends naturally to a skew-symmetric pairing on $H_1(L,T)$. By Poincar\'e duality, we can identify the dual lattice of $H_1(L,T)$ with the relative homology $H_1(L\setminus T, \Lambda\setminus T)$. Any completion of $\fS(\bG)$ to a basis $\fB$ of $H_1(L,T)$ gives rise to a unique dual basis $\fB^\vee$ of $H_1(L\setminus T,\Lambda\setminus T)$.

The outline for this section is as follows. First, we give working definitions of the moduli space $\fM(\La,T)$, which allows us to draw connections to Lie-theortical moduli spaces and also deduce the factoriality of its ring of regular functions $\mathcal{O}(\fM(\La,T))$. Next, on the moduli space $\fM(\La,T)$, we construct a new family of rational functions called \emph{microlocal merodromies}, which are associated with relative cycles in $H_1(L\setminus T,\Lambda\setminus T)$. Although the definition of microlocal merodromies depends on the initial filling $L$, we show that for elements in the dual basis $\fB^\vee$, their microlocal merodromies actually extend to $\mathbb{C}$-valued regular functions on the entire moduli space $\fM(\La,T)$. Moreover, we also prove that within these special microlocal merodromies, those dual to $\fS(\bG)$ can be mutated according to the cluster $\mathcal{A}$-mutation formula as the initial weave $\ww$ undergoes weave mutation, corresponding to a Lagrangian disk surgery on $L(\ww)$. Then, we show that the codimension-2 argument in cluster varieties can be applied by studying immersed Lagrangian fillings represented by non-free weaves. These results together with \cite{BFZ05} allow us to conclude the existence of a cluster $\mathcal{A}$-structure on $\fM(\La,T)$, where the initial and adjacent seeds are constructed via the Lagrangian filling $L(\bG)$, its Lagrangian surgeries and the associated microlocal merodromies.


\subsection{Descriptions of sheaves with singular support on the Legendrian \texorpdfstring{$\La(\bG)$}{}}\label{ssec:moduli_Legendrian} Let $\La\sse(\R^3\xi_\st)$ be a Legendrian link, $T\sse\La$ a set of marked points, and consider the moduli stacks $\M_1(\Lambda)$ and $\FM(\Lambda, T)$ discussed in Section \ref{ssec:sheaves}. These stacks classify (complexes of) constructible sheaves on $\R^2$ with a singular support condition. In this section, we provide Lie-theoretical descriptions for $\M_1(\Lambda)$ and $\FM(\Lambda, T)$ which are suited for our computations, using \cite{KashiwaraSchapira_Book} and closely following \cite[Section 3.3]{STZ_ConstrSheaves} and Section 5 in ibid.\footnote{An expert in the results of \cite{KashiwaraSchapira_Book,STZ_ConstrSheaves} might be able to quickly move forward to Subsection \ref{ssec:delta-complete}.} These are more combinatorial presentations of these stacks, as the constructible and microlocal aspects of the original definition are translated into explicit quiver representations satisfying certain conditions.

Given a co-oriented front projection $\pi_F(\La)$, consider the following quiver $Q_F(\Lambda)$: 
\begin{enumerate}
    \item[-] A vertex of $Q_F(\La)$ is placed at each connected component of the $\mathbb{R}^2\setminus \pi_F(\Lambda)$,
    \item[-] For each (1-dimensional) connected component of $\pi_F(\Lambda)\setminus S_0$, where $S_0$ denotes the set of crossings and cusps in $\pi_F(\Lambda)$, draw an arrow connecting the two vertices associated to the two adjacent 2-dimensional cell (that contain that stratum in their closure). The direction of the arrow is opposite to the co-orientation of the front $\pi_F(\La)$.
\end{enumerate}  

\noindent The following is then proven in \cite[Section 3]{STZ_ConstrSheaves}:

\begin{prop}\label{def: working def for M1}  Let $\La\sse(\R^3,\xi)$ be a Legendrian and $\pi_F(\La)\sse\R^2$ a front, with a binary Maslov potential, such that $(\R^2,\pi_F(\La))$ is a regular stratification. Consider the stack $\mathcal{M}(Q_F(\La))$ classifying linear representations of the quiver $Q_F(\Lambda)$ that satisfy the following conditions:
\begin{enumerate}
    \item The vector space associated with the unbounded region in $\mathbb{R}^2$ is $0$.
    \item Any two vector spaces associated with neighboring vertices differ in dimension by $1$.
    \item At each cusp, the composition depicted in Figure \ref{fig: id condition and exact condition} (left) is the identity map.
    \item At each crossing, the four linear maps involved form a commuting square which is exact, as precised in Figure \ref{fig: id condition and exact condition} (right).
    
    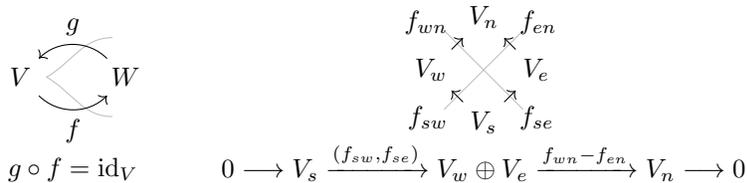
\begin{figure}[H]
        \begin{tikzpicture}[scale=0.7]
    \draw [lightgray] (0.75,0.75) to [out=180,in=30] (-0.5,0) to [out=-30,in=180] (0.75,-0.75);
    \node (v) at (-1,0) [] {$V$};
    \node (w) at (1,0) [] {$W$};
    \draw [->] (v) to [out=-45,in=-135] node [below] {$f$} (w);
    \draw [->] (w) to [out=135,in=45] node [above] {$g$} (v);
    \node at (0,-1.75) [] {$g\circ f=\mathrm{id}_V$};
    \end{tikzpicture}
   \quad \quad  \begin{tikzpicture}[scale=0.7]
    \draw [lightgray] (-0.75,-0.75) -- (0.75,0.75);
    \draw [lightgray] (-0.75,0.75) -- (0.75,-0.75);
    \node (e) at (1,0) [] {$V_e$};
    \node (s) at (0,-1) [] {$V_s$};
    \node (w) at (-1,0) [] {$V_w$};
    \node (n) at (0,1) [] {$V_n$};
    \draw [->] (s) -- node[below right] {$f_{se}$} (e);
    \draw [->] (s) -- node [below left] {$f_{sw}$} (w);
    \draw [->] (e) -- node [above right] {$f_{en}$} (n);
    \draw [->] (w) -- node [above left] {$f_{wn}$} (n);
    \node at (0,-1.75) [] {$0\longrightarrow V_s\xrightarrow{(f_{sw},f_{se})} V_w\oplus V_e\xrightarrow{f_{wn}-f_{en}} V_n\longrightarrow 0$};
    \end{tikzpicture}
        \caption{Identity condition at cusps and exactness condition at crossings.}
        \label{fig: id condition and exact condition}
    \end{figure}
\end{enumerate}
Then the stack $\M_1(\Lambda)$ is isomorphic to $\mathcal{M}(Q_F(\La))$.
\end{prop}

Proposition \ref{def: working def for M1} describes $\mathcal{M}_1(\La)$. Now we gear towards the decorated moduli $\mathfrak{M}(\La,T)$. First, we need a description of microlocal monodromy in terms of these quiver representations, which is provided in \cite[Section 5]{STZ_ConstrSheaves} and we briefly summarize as follows. Let $S$ be the set of singular points (crossings and cusps) in $\pi_F(\Lambda)$ and note that each connected components of $\pi_F(\Lambda)\setminus S$ is associated with a 1-dimensional kernel or a 1-dimensional cokernel. These kernels and cokernels can be glued together along strands of $\Lambda$ using the identity condition at cusps and the exactness condition at crossings, as follows:
\begin{itemize}
    \item[-] At a cusp, as in Figure \ref{fig: id condition and exact condition} (left), the condition $g\circ f=\id_V$ forces the composition $\ker g\hookrightarrow W\twoheadrightarrow \coker f$ to be an isomorphism. By definition, we glue $\ker g$ and $\coker f$ using this isomorphism.\\
    \item[-] At a crossing, as in Figure \ref{fig: id condition and exact condition} (right), there are three cases depending, on the injectivity or surjectivity of the four maps: the four maps can be all injective, all surjective, or two injective with two surjective. In each of the three cases, we have the following isomorphisms from the exactness condition:
    \begin{align}
    \begin{tikzpicture}[scale=0.7,baseline=0]
    \draw [lightgray] (-0.75,-0.75) -- (0.75,0.75);
    \draw [lightgray] (-0.75,0.75) -- (0.75,-0.75);
    \node (e) at (1,0) [] {$V_e$};
    \node (s) at (0,-1) [] {$V_s$};
    \node (w) at (-1,0) [] {$V_w$};
    \node (n) at (0,1) [] {$V_n$};
    \draw [right hook->] (s) -- node[below right] {$f_{se}$} (e);
    \draw [left hook->] (s) -- node [below left] {$f_{sw}$} (w);
    \draw [left hook->] (e) -- node [above right] {$f_{en}$} (n);
    \draw [right hook->] (w) -- node [above left] {$f_{wn}$} (n);
    \end{tikzpicture} \quad & \quad \quad \begin{array}{l}
    \coker f_{sw} \hookrightarrow\dfrac{V_n}{V_s} \twoheadrightarrow \coker f_{en}       \\ \\ 
    \coker f_{wn} \twoheadleftarrow\dfrac{V_n}{V_s} \hookleftarrow \coker f_{se}
    \end{array} \label{eq:gluing with inclusion maps}
    \\
    \begin{tikzpicture}[scale=0.7,baseline=0]
    \draw [lightgray] (-0.75,-0.75) -- (0.75,0.75);
    \draw [lightgray] (-0.75,0.75) -- (0.75,-0.75);
    \node (e) at (1,0) [] {$V_e$};
    \node (s) at (0,-1) [] {$V_s$};
    \node (w) at (-1,0) [] {$V_w$};
    \node (n) at (0,1) [] {$V_n$};
    \draw [->>] (s) -- node[below right] {$f_{se}$} (e);
    \draw [->>] (s) -- node [below left] {$f_{sw}$} (w);
    \draw [->>] (e) -- node [above right] {$f_{en}$} (n);
    \draw [->>] (w) -- node [above left] {$f_{wn}$} (n);
    \end{tikzpicture} \quad & \quad \quad \begin{array}{l}
    \ker f_{sw} \hookrightarrow \ker(f_{wn}\circ f_{sw})=\ker(f_{en}\circ f_{se}) \twoheadrightarrow \ker f_{en}       \\ \\ 
    \ker f_{wn} \twoheadleftarrow\ker(f_{wn}\circ f_{sw})=\ker(f_{en}\circ f_{se}) \hookleftarrow \ker f_{se}
    \end{array}
    \\
    \begin{tikzpicture}[scale=0.7,baseline=0]
    \draw [lightgray] (-0.75,-0.75) -- (0.75,0.75);
    \draw [lightgray] (-0.75,0.75) -- (0.75,-0.75);
    \node (e) at (1,0) [] {$V_e$};
    \node (s) at (0,-1) [] {$V_s$};
    \node (w) at (-1,0) [] {$V_w$};
    \node (n) at (0,1) [] {$V_n$};
    \draw [->>] (s) -- node[below right] {$f_{se}$} (e);
    \draw [left hook->] (s) -- node [below left] {$f_{sw}$} (w);
    \draw [left hook->] (e) -- node [above right] {$f_{en}$} (n);
    \draw [->>] (w) -- node [above left] {$f_{wn}$} (n);
    \end{tikzpicture} \quad & \quad \quad \begin{array}{l}
    \coker f_{sw} \xrightarrow{f_{wn}} \coker f_{en}       \\ \\ 
    \ker f_{wn} \xleftarrow{f_{sw}} \ker f_{se}
    \end{array}
    \end{align}
\end{itemize}
The result of gluing these 1-dimensional vector spaces is a rank-1 local system $\Phi$ on $\Lambda$. In fact, it coincides with the microlocal monodromy functor, see \cite[Section 5.1]{STZ_ConstrSheaves} for more details. Given the set $T$ of marked points on $\Lambda$, with at least one marked point per link component, $\Lambda\setminus T$ is a collection of open intervals. Thus, along each such open interval $I$, we can trivialize the rank-1 local system $\Phi$ by specifying an isomorphism $\phi_I:I\times\mathbb{C}\xrightarrow{\cong}\Phi|_I$. By definition, a collection of such maps $\{\phi_I\}$ are said to be a \emph{framing} for the local system $\Phi$.

In conclusion, a point in the decorated moduli space $\FM(\Lambda,T)$, as defined in Subsection \ref{sssec:sheaf_decorated}, is a point in $\M_1(\Lambda)$, which is combinatorialized via Proposition \ref{def: working def for M1}, together with a framing for the local system $\Phi$, i.e.~ a trivialization of the (trivial) local system $\Phi|_{\Lambda\setminus T}$. Here two framings are considered \emph{equivalent} if they differ by a global scaling $\mathbb{C}^\times$-factor, and thus $\dim \FM(\Lambda,T)=\dim \M_1(\Lambda)+|T|-1$.

\subsubsection{Description for a GP-graph $\bG$.}\label{sssec:flag_description} In the case that $\La=\La(\bG)$ comes from a GP-graph $\bG$, Section \ref{ssec:legendrianlink} provides a specific front $\mathfrak{f}(\bG)\sse\R^2$. For this front, the description from Proposition \ref{def: working def for M1} can be translated in terms of configurations of flags, as follows. The front projection $\pi_F(\Lambda)=\mathfrak{f}(\bG)$ can be sliced into the three types of elementary columns.\\

\noindent For a Type 1 column, there are $n$ strands in the bottom region, with Maslov potential 0, and $n$ strands in the top region, with Maslov potential 1. By Proposition \ref{def: working def for M1}, the vector space associated with the central region must be $\mathbb{C}^n$. From the quiver representation data, we can construct the following pair of flags in $\mathbb{C}^n$. Since all of the linear maps in the bottom region are injective, their images in the middle $\mathbb{C}^n$ naturally form a first flag. Similarly, since all of the linear maps in the top region are surjective, their kernels in the middle vector space $\mathbb{C}^n$ form a second flag. See Figure \ref{fig:Flags_Type1} for a depiction of the front in Type 1 and its associated pair of flags. We adopt the convention of indexing flags from the bottom region with a subscript, and indexing flags from the top region with a superscript, so as to distinguish between these two types of two flags.
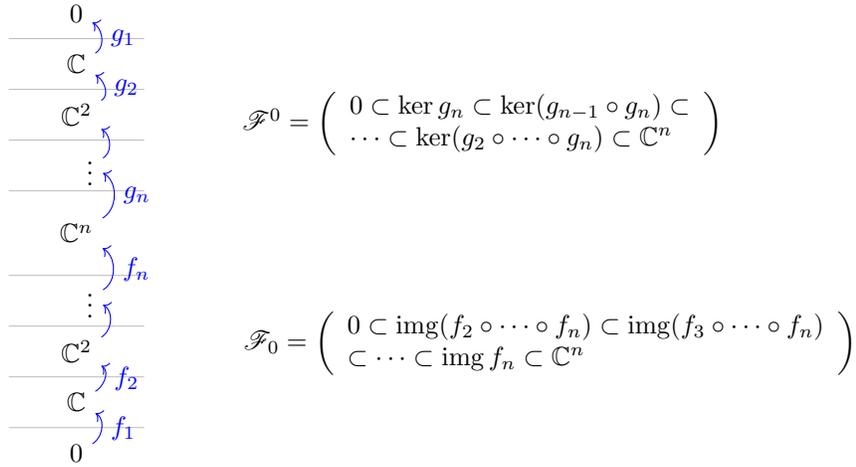
\begin{figure}[H]
\centering
    \begin{tikzpicture}[scale=0.9]
    \foreach \i in {0,1,2,3}
    {
        \draw[lightgray] (0,\i*0.75) -- (2,\i*0.75);
        \draw[lightgray] (0,\i*0.75+3.5) -- (2,\i*0.75+3.5);
    }
    \foreach \i in {2}
    {
        \node (b\i) at (1,\i*0.75-0.375) [] {$\mathbb{C}^\i$};
        \node (t\i) at (1,3*0.75+3.5+0.375-\i*0.75) [] {$\mathbb{C}^\i$};
    }
    \node (b1) at (1,0.375) [] {$\mathbb{C}$};
    \node (t1) at (1,3*0.75+3.5+0.375-0.75) [] {$\mathbb{C}$};
    \node (m)  at (1,2.875) [] {$\mathbb{C}^n$};
    \node (b3) at (1.2,1.925) [] {$\vdots$};
    \node (t3) at (1.2,3.875) [] {$\vdots$};
    \node (b0) at (1,-0.375) [] {$0$};
    \node (t0) at (1,3.5+3*0.75+0.375) [] {$0$};
    \draw [blue, ->] (b0) to [out=30,in=-30] node [right] {$f_1$} (b1);
    \draw [blue, ->] (b1) to [out=30,in=-30] node [right] {$f_2$} (b2);
    \draw [blue, ->] (b2) to [out=30,in=-30] (b3);
    \draw [blue, ->] (b3) to [out=30,in=-30] node [right] {$f_n$} (m);
    \draw [blue, ->] (m) to [out=30,in=-30] node [right] {$g_n$} (t3);
    \draw [blue, ->] (t3) to [out=30,in=-30] (t2);
    \draw [blue, ->] (t2) to [out=30,in=-30] node [right] {$g_2$} (t1);
    \draw [blue, ->] (t1) to [out=30,in=-30] node [right] {$g_1$} (t0);
    \node at (8,1.25) [] {$\SF_0=\left(\begin{array}{l} 0\subset \img (f_2\circ \cdots \circ f_n) \subset \img(f_3\circ  \cdots \circ f_n) \\ \subset \cdots\subset \img f_n\subset \mathbb{C}^n\end{array}\right)$};
    \node at (7,4.5) [] {$\SF^0=\left(\begin{array}{l} 0\subset \ker g_n\subset \ker (g_{n-1}\circ g_n)  \subset \\ \cdots \subset \ker(g_2\circ \cdots \circ g_n)\subset \mathbb{C}^n\end{array}\right)$};
    \end{tikzpicture}
    \caption{The pair of flags associated to a Type 1 column.}\label{fig:Flags_Type1}
\end{figure}

Before discussing Type 2 and 3 columns, we recall that the relative positions relations between two flags in $\mathbb{C}^n$ are classified by elements of the symmetric group $S_n$, which is a Coxeter group with Coxeter generators $\{s_i\}_{i=1}^{n-1}$. For two flags $\SF=(0\subset \SF_1\subset \cdots \subset \mathbb{C}^n)$ and $\SF'=(0\subset \SF'_1\subset \cdots\subset \mathbb{C}^n)$, we denote
\begin{itemize}
    \item $\SF\stackrel{s_i}{\sim}\SF'$ if $F_i\neq F'_i$ but $F_j=F'_j$ for all $j\neq i$;
    \item $\SF\stackrel{w}{\sim}\SF'$ if there exists a sequence of flags $\SG_0, \SG_1,\dots, \SG_l$ such that 
    \[
    \SF=\SG_0\stackrel{s_{i_1}}{\sim} \SG_1\stackrel{s_{i_2}}{\sim}\SG_2\stackrel{s_{i_3}}{\sim}\cdots \stackrel{s_{i_l}}{\sim}\SG_l=\SF'
    \] 
    and $s_{i_1}s_{i_2}\cdots s_{i_l}$ is a reduced word of $w$.
\end{itemize}
This classification can be identified with the Tits distance obtained from a Bruhat decomposition of $\GL_n$. In particular, being in $w$ relative position does not depend on the choice of reduced word of $w$. If $\SF\stackrel{w}{\sim}\SF'$, then for each choice of reduced word $(i_1,\dots, i_l)$ for $w$, there exists a unique sequence of flags $(\SG_k)_{k=0}^l$ that relate the two flags $\SF$ and $\SF'$.

We can now translate the local quiver representation data associated with a Type 2 column into relative position relations between flags. Suppose the pair of flags to the left of a Type 2 column is $(\SL_0,\SL^0)$, and the pair of flags to the right of a Type 2 column is $(\SR_0,\SR^0)$. If there is a crossing in the bottom region at the $i$th gap, counting from the bottom in both the front and in the GP-graph, then from the exactness condition at the crossing we obtain the constraints
\begin{equation}\label{eq: type 2 bottom relative position}
\SL_0\overset{s_i}{\sim}\SR_0 \quad \quad \text{and} \quad \quad \SL^0=\SR^0.
\end{equation}
Similarly, if there is a crossing in the top region at the $i$th gap, counting from the top in the front projection or counting from the bottom in the GP-graph, then the exactness condition at the crossing yields
\begin{equation}\label{eq: type 2 top relative position}
\SL_0=\SR_0 \quad \quad \text{and} \quad \quad \SL^0\overset{s_{n-i}}{\sim}\SR^0.
\end{equation}

\noindent Since crossings in Type 2 columns correspond to vertical edges in the GP-graph, we can infer the relative position relation between pairs of flags from the GP-graph as well.

For a Type 3 column, the pairs of flags on the (Type 1 column on the) left and on the (Type 1 column on the) right are not in the same ambient vector space, as the dimensions of the two vector spaces differ by one. Instead, there is a linear map {\it from} the ambient vector space for the pair of flags on the left {\it to} the ambient vector space for the pair of flags on the right. This linear map is injective if the lollipop is white and it is surjective if the lollipop is black. Let us investigate how the two pairs of flags are related.

\noindent Suppose first that the lollipop is white, so that the linear map $h:\mathbb{C}^{n-1}\rightarrow \mathbb{C}^{n}$ between the two (middle) adjacent ambient vector spaces is injective. Given any flag $\SF=(0\subset \SF_1\subset \cdots\subset \SF_{n-1}=\mathbb{C}^{n-1})$ in $\mathbb{C}^{n-1}$, we can use $h$ to naturally extend it to a flag $h(\SF)$ in $\mathbb{C}^{n}$ by defining
\[
h(\SF):=(0\subset h(\SF_1)\subset h(\SF_2)\subset \cdots\subset h(\SF_{n-1})\subset \mathbb{C}^{n}).
\]

This extension from $(\SL_0,\SL^0)$ to $(h(\SL_0),h(\SL^0))$ can be achieved geometrically by a sequence of RII moves that pulls the left cusp upward in the front projection. Indeed, consider the following local example:
\begin{figure}[H]
\centering
\begin{tikzpicture}[scale=0.9]
\foreach \i in {0,1,3}
{
\draw [lightgray] (0,\i*0.7) -- (3,\i*0.7);
\draw [lightgray] (0,6-\i*0.7) -- (3,6-\i*0.7);
}
\draw [lightgray] (3,4.6) -- (1,4.6) -- (2,1.4) -- (3,1.4);
\node (bl0) at (0.5,-0.35) [] {$0$};
\node (bl1) at (0.5,0.35) [] {$\mathbb{C}$};
\node (bl2) at (0.5,1.4) [] {$\mathbb{C}^2$};
\node (l3) at (0.5,3) [] {$\mathbb{C}^3$};
\node (tl2) at (0.5,4.6) [] {$\mathbb{C}^2$};
\node (tl1) at (0.5,5.65) [] {$\mathbb{C}$};
\node (tl0) at (0.5,6.35) [] {$0$};
\node (br0) at (2.5,-0.35) [] {$0$};
\node (br1) at (2.5,0.35) [] {$\mathbb{C}$};
\node (br2) at (2.5,1.05) [] {$\mathbb{C}^2$};
\node (br3) at (2.5,1.8) [] {$\mathbb{C}^3$};
\node (r4) at (2.5,3) [] {$\mathbb{C}^4$};
\node (tr3) at (2.5,4.25) [] {$\mathbb{C}^3$};
\node (tr2) at (2.5,4.95) [] {$\mathbb{C}^2$};
\node (tr1) at (2.5,5.65) [] {$\mathbb{C}$};
\node (tr0) at (2.5,6.35) [] {$0$};
\draw [->] (bl0) to [out=150,in=-150] (bl1);
\draw [->] (bl1) to [out=150,in=-150] (bl2);
\draw [->] (bl2) to [out=150,in=-150] (l3);
\draw [->] (l3) to [out=150,in=-150] (tl2);
\draw [->] (tl2) to [out=150,in=-150] (tl1);
\draw [->] (tl1) to [out=150,in=-150] (tl0);
\draw [->] (l3) -- node [above] {$h$} (r4);
\draw [->] (br0) to [out=30,in=-30] (br1);
\draw [->] (br1) to [out=30,in=-30] (br2);
\draw [->] (br2) to [out=30,in=-30] (br3);
\draw [->] (br3) to [out=30,in=-30] (r4);
\draw [->] (r4) to [out=30,in=-30] (tr3);
\draw [->] (tr3) to [out=30,in=-30] (tr2);
\draw [->] (tr2) to [out=30,in=-30] (tr1);
\draw [->] (tr1) to [out=30,in=-30] (tr0);
\node at (-1,1) [] {$\SL_0$};
\node at (-1,5) [] {$\SL^0$};
\node at (4,1) [] {$\SR_0$};
\node at (4,5) [] {$\SR^0$};
\end{tikzpicture}\quad \quad \quad \quad  \quad \quad
\begin{tikzpicture}[scale=0.9]
\foreach \i in {0,1,3}
{
\draw [lightgray] (-1,\i*0.7) -- (3,\i*0.7);
\draw [lightgray] (-1,6-\i*0.7) -- (3,6-\i*0.7);
}
\draw [lightgray] (3,4.6) -- (2,4.6) -- (0,6.7) -- (2,1.4) -- (3,1.4);
\draw [blue] (0,6.7) -- (2,1.4);
\node (bl0) at (-0.5,-0.35) [] {$0$};
\node (bl1) at (-0.5,0.35) [] {$\mathbb{C}$};
\node (bl2) at (-0.5,1.4) [] {$\mathbb{C}^2$};
\node (l3) at (-0.5,3) [] {$\mathbb{C}^3$};
\node (tl2) at (-0.5,4.6) [] {$\mathbb{C}^2$};
\node (tl1) at (-0.5,5.65) [] {$\mathbb{C}$};
\node (tl0) at (-0.5,6.35) [] {$0$};
\node (br0) at (2.5,-0.35) [] {$0$};
\node (br1) at (2.5,0.35) [] {$\mathbb{C}$};
\node (br2) at (2.5,1.05) [] {$\mathbb{C}^2$};
\node (br3) at (2.5,1.8) [] {$\mathbb{C}^3$};
\node (r4) at (2.5,3) [] {$\mathbb{C}^4$};
\node (tr3) at (2.5,4.25) [] {$\mathbb{C}^3$};
\node (tr2) at (2.5,4.95) [] {$\mathbb{C}^2$};
\node (tr1) at (2.5,5.65) [] {$\mathbb{C}$};
\node (tr0) at (2.5,6.35) [] {$0$};
\node (m3) at (1.5,4.6) [] {$\mathbb{C}^3$};
\node (m2) at (0.8,5.65) [] {$\mathbb{C}^2$};
\node (m1) at (0.25,6.35) [] {$\mathbb{C}$};
\node (m0) at (1,6.5) [] {$0$};
\draw [teal,->] (bl0) to [out=150,in=-150] (bl1);
\draw [teal,->] (bl1) to [out=150,in=-150] (bl2);
\draw [teal,->] (bl2) to [out=150,in=-150] (l3);
\draw [->] (l3) to [out=150,in=-150] (tl2);
\draw [->] (tl2) to [out=150,in=-150] (tl1);
\draw [->] (tl1) to [out=150,in=-150] (tl0);
\draw [teal,->] (l3) -- node [above] {$h$} (r4);
\draw [->] (br0) to [out=30,in=-30] (br1);
\draw [->] (br1) to [out=30,in=-30] (br2);
\draw [->] (br2) to [out=30,in=-30] (br3);
\draw [->] (br3) to [out=30,in=-30] (r4);
\draw [->] (r4) to [out=30,in=-30] (tr3);
\draw [->] (tr3) to [out=30,in=-30] (tr2);
\draw [->] (tr2) to [out=30,in=-30] (tr1);
\draw [->] (tr1) to [out=30,in=-30] (tr0);
\draw [red, ->] (r4) -- (m3);
\draw [red, ->] (m3) -- (m2);
\draw [red, ->] (m2) -- (m1);
\draw [red, ->] (m1) -- (m0);
\draw [->] (tl2) -- (m3);
\draw [->] (tl1) -- (m2);
\draw [->] (tl0) -- (m1);
\node at (-2,1) [] {$\SL_0$};
\node at (-2,5) [] {$\SL^0$};
\node at (4,1) [] {$\SR_0$};
\node at (4,5) [] {$\SR^0$};
\end{tikzpicture}
\caption{Pulling up a left cusp}
\label{fig: pulling up a left cusp}
\end{figure}
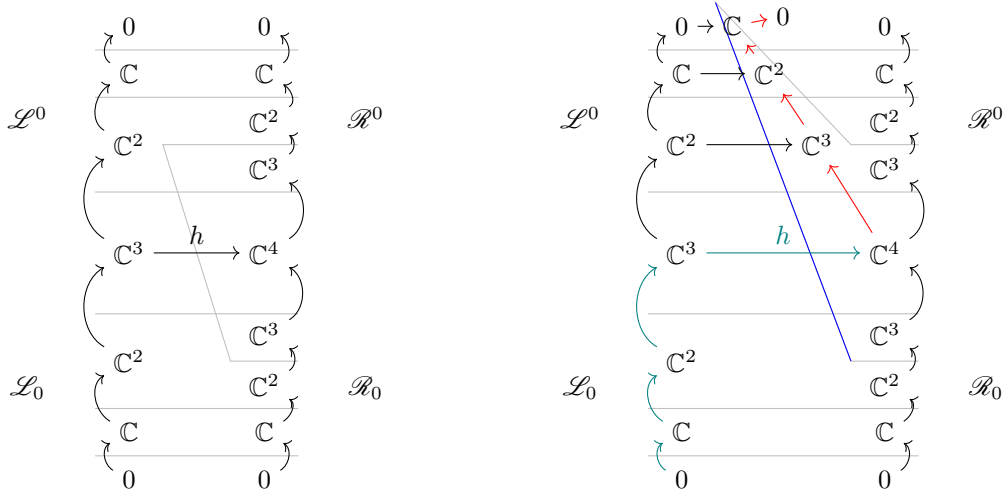


The green maps in the bottom region define the extension $h(\SL_0)$. By the exactness of the quadrilaterals in the top region, the red maps define the extension $h(\SL^0)$. In particular, the extensions $h(\SL_0)$ and $h(\SL^0)$ are completely determined by the original data of the quiver representations. 

\noindent Now, with these extensions defined, it follows that if there is a white lollipop emerging in the $i$th gap with $0\leq i\leq n-1$, counting from below in the GP-graph\footnote{The case $i=0$ is a lollipop at the bottom, and $i=n-1$ is a lollipop at the top.}, then the corresponding relative position conditions are
\begin{equation}\label{eq: type 3 white lollipop relative position}
h(\SL_0)\overset{s_{n-1}\cdots s_{i+1}}{\sim}\SR_0 \quad \text{and} \quad h(\SL^0) \overset{s_{n-1}\cdots s_{n-i}}{\sim} \SR^0.
\end{equation}

Suppose that there is a black lollipop, and thus the linear map between the two ambient vector spaces $h:\mathbb{C}^{n}\rightarrow \mathbb{C}^{n-1}$ is surjective. Then, given any flag $\SF=(0\subset \SF_1\subset \cdots\subset \SF_{n-1}=\mathbb{C}^{n-1})$ in $\mathbb{C}^{n-1}$, we consider $h^{-1}(\SF_i)$ and insert $\ker(h)$ in front of it so as to form a flag in $\mathbb{C}^n$:
\[
h^{-1}(\SF) := (0\subset \ker(h)\subset h^{-1}(\SF_1)\subset \cdots \subset h^{-1}(\SF_{n-1})=\mathbb{C}^{n}).
\]

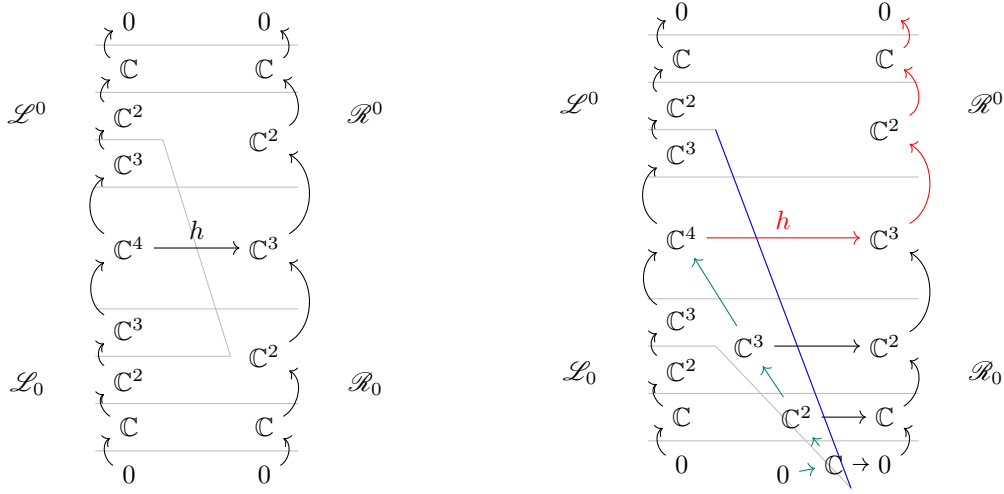
\begin{figure}[H]
\centering
\begin{tikzpicture}[scale=0.9]
\foreach \i in {0,1,3}
{
\draw [lightgray] (0,\i*0.7) -- (3,\i*0.7);
\draw [lightgray] (0,6-\i*0.7) -- (3,6-\i*0.7);
}
\draw [lightgray] (0,4.6) -- (1,4.6) -- (2,1.4) -- (0,1.4);
\node (bl0) at (2.5,-0.35) [] {$0$};
\node (bl1) at (2.5,0.35) [] {$\mathbb{C}$};
\node (bl2) at (2.5,1.4) [] {$\mathbb{C}^2$};
\node (l3) at (2.5,3) [] {$\mathbb{C}^3$};
\node (tl2) at (2.5,4.6) [] {$\mathbb{C}^2$};
\node (tl1) at (2.5,5.65) [] {$\mathbb{C}$};
\node (tl0) at (2.5,6.35) [] {$0$};
\node (br0) at (0.5,-0.35) [] {$0$};
\node (br1) at (0.5,0.35) [] {$\mathbb{C}$};
\node (br2) at (0.5,1.05) [] {$\mathbb{C}^2$};
\node (br3) at (0.5,1.8) [] {$\mathbb{C}^3$};
\node (r4) at (0.5,3) [] {$\mathbb{C}^4$};
\node (tr3) at (0.5,4.25) [] {$\mathbb{C}^3$};
\node (tr2) at (0.5,4.95) [] {$\mathbb{C}^2$};
\node (tr1) at (0.5,5.65) [] {$\mathbb{C}$};
\node (tr0) at (0.5,6.35) [] {$0$};
\draw [->] (bl0) to [out=30,in=-30] (bl1);
\draw [->] (bl1) to [out=30,in=-30] (bl2);
\draw [->] (bl2) to [out=30,in=-30] (l3);
\draw [->] (l3) to [out=30,in=-30] (tl2);
\draw [->] (tl2) to [out=30,in=-30] (tl1);
\draw [->] (tl1) to [out=30,in=-30] (tl0);
\draw [->] (r4) -- node [above] {$h$} (l3);
\draw [->] (br0) to [out=150,in=-150] (br1);
\draw [->] (br1) to [out=150,in=-150] (br2);
\draw [->] (br2) to [out=150,in=-150] (br3);
\draw [->] (br3) to [out=150,in=-150] (r4);
\draw [->] (r4) to [out=150,in=-150] (tr3);
\draw [->] (tr3) to [out=150,in=-150] (tr2);
\draw [->] (tr2) to [out=150,in=-150] (tr1);
\draw [->] (tr1) to [out=150,in=-150] (tr0);
\node at (-1,1) [] {$\SL_0$};
\node at (-1,5) [] {$\SL^0$};
\node at (4,1) [] {$\SR_0$};
\node at (4,5) [] {$\SR^0$};
\end{tikzpicture}\quad \quad \quad \quad  \quad \quad
\begin{tikzpicture}[scale=0.9]
\foreach \i in {0,1,3}
{
\draw [lightgray] (0,\i*0.7) -- (4,\i*0.7);
\draw [lightgray] (0,6-\i*0.7) -- (4,6-\i*0.7);
}
\draw [lightgray] (0,4.6) -- (1,4.6) -- (3,-0.7) -- (1,1.4) -- (0,1.4);
\draw [blue] (3,-0.7) -- (1,4.6);
\node (bl0) at (3.5,-0.35) [] {$0$};
\node (bl1) at (3.5,0.35) [] {$\mathbb{C}$};
\node (bl2) at (3.5,1.4) [] {$\mathbb{C}^2$};
\node (l3) at (3.5,3) [] {$\mathbb{C}^3$};
\node (tl2) at (3.5,4.6) [] {$\mathbb{C}^2$};
\node (tl1) at (3.5,5.65) [] {$\mathbb{C}$};
\node (tl0) at (3.5,6.35) [] {$0$};
\node (br0) at (0.5,-0.35) [] {$0$};
\node (br1) at (0.5,0.35) [] {$\mathbb{C}$};
\node (br2) at (0.5,1.05) [] {$\mathbb{C}^2$};
\node (br3) at (0.5,1.8) [] {$\mathbb{C}^3$};
\node (r4) at (0.5,3) [] {$\mathbb{C}^4$};
\node (tr3) at (0.5,4.25) [] {$\mathbb{C}^3$};
\node (tr2) at (0.5,4.95) [] {$\mathbb{C}^2$};
\node (tr1) at (0.5,5.65) [] {$\mathbb{C}$};
\node (tr0) at (0.5,6.35) [] {$0$};
\node (m3) at (1.5,1.4) [] {$\mathbb{C}^3$};
\node (m2) at (2.2,0.35) [] {$\mathbb{C}^2$};
\node (m1) at (2.75,-0.35) [] {$\mathbb{C}$};
\node (m0) at (2,-0.5)[] {$0$};
\draw [->] (bl0) to [out=30,in=-30] (bl1);
\draw [->] (bl1) to [out=30,in=-30] (bl2);
\draw [->] (bl2) to [out=30,in=-30] (l3);
\draw [->, red] (l3) to [out=30,in=-30] (tl2);
\draw [->, red] (tl2) to [out=30,in=-30] (tl1);
\draw [->, red] (tl1) to [out=30,in=-30] (tl0);
\draw [->, red] (r4) -- node [above] {$h$} (l3);
\draw [->] (br0) to [out=150,in=-150] (br1);
\draw [->] (br1) to [out=150,in=-150] (br2);
\draw [->] (br2) to [out=150,in=-150] (br3);
\draw [->] (br3) to [out=150,in=-150] (r4);
\draw [->] (r4) to [out=150,in=-150] (tr3);
\draw [->] (tr3) to [out=150,in=-150] (tr2);
\draw [->] (tr2) to [out=150,in=-150] (tr1);
\draw [->] (tr1) to [out=150,in=-150] (tr0);
\node at (-1,1) [] {$\SL_0$};
\node at (-1,5) [] {$\SL^0$};
\node at (5,1) [] {$\SR_0$};
\node at (5,5) [] {$\SR^0$};
\draw [teal, <-] (r4) -- (m3);
\draw [teal, <-] (m3) -- (m2);
\draw [teal, <-] (m2) -- (m1);
\draw [teal, <-] (m1) -- (m0);
\draw [<-] (bl2) -- (m3);
\draw [<-] (bl1) -- (m2);
\draw [<-] (bl0) -- (m1);
\end{tikzpicture}
\caption{Pulling down a right cusp}
\label{fig: pulling down a right cusp}
\end{figure}

Similar to the white lollipop case, the extension of $(\SR_0,\SR^0)$ to $(h^{-1}(\SR_0),h^{-1}(\SR^0))$ can be achieved geometrically by a sequence of RII moves that pulls the right cusp downward in the front projection, as depicted in Figure \ref{fig: pulling down a right cusp}. Note that the green maps in bottom region define the extension $h^{-1}(\SR_0)$, whereas the red maps in the top region define the extension $h^{-1}(\SR^0)$. 

\noindent It follows that if there is a black lollipop occurring in the $i$th gap with $0\leq i\leq n$, counting from below in the GP-graph, then the corresponding relative position conditions are
\begin{equation}\label{eq: type 3 black lollipop relative position}
\SL_0\overset{s_{i-1}\cdots s_1}{\sim}h^{-1}(\SR_0) \quad \text{and} \quad  \SL^0\overset{s_{n-i+1}\cdots s_1}{\sim}h^{-1}(\SR^0).
\end{equation}

In summary, given a GP-graph $\bG$, we can divide $\bG$ into columns of three types such that every consecutive pair of non-Type 1 columns is separated by a Type 1 column and every consecutive pair of Type 1 columns is separated by a non-Type 1 column. In between Type 3 columns there is a unique ambient vector space $V_i= \mathbb{C}^n$ for some $n$, and they are linked by linear maps $h_i:V_{i-1}\rightarrow V_i$ that are either injective with a 1-dimensional cokernel or surjective with a 1-dimensional kernel. The above discussion proves the following lemma:

\begin{lemma}\label{lem:flag_description} For a GP-graph $\bG$ with a decomposition into columns as above, the moduli space $\M_1(\Lambda)$ can be described by the following data:
\begin{enumerate}
    \item a pair of flags in $V_i$ for each Type 1 column contained in the $V_i$ part of $\bG$;
    \item for each Type 2 column, the neighboring flags satisfy the relative position condition according to \eqref{eq: type 2 bottom relative position} and \eqref{eq: type 2 top relative position};
    \item for each Type 3 column, the neighboring flags satisfy the relative position condition according to \eqref{eq: type 3 white lollipop relative position} and \eqref{eq: type 3 black lollipop relative position};
\end{enumerate}
where we quotient this data by the equivalence relation $(\SF,h)\sim (\SF',h')$ for a collection of elements $g_i\in \GL(V_i)$ such that $h_i\circ g_{i-1}=g_i\circ h'_i$.
\end{lemma}

In the flag description of Lemma \ref{lem:flag_description}, the rank-1 local system $\Phi$ on $\Lambda$ can be constructed by taking quotients of consecutive vector subspaces in each flag and then gluing them along strands of $\Lambda$ at crossings and cusps in the same manner as before. Note that in this context, only \eqref{eq:gluing with inclusion maps} is used when gluing these rank-1 local systems at crossings because all linear maps near a crossing are now inclusions of vector subspaces. In particular, the surjective maps in the top region of the front projection are now turned into inclusions of kernels.


\subsubsection{Description for (-1)-closures.}\label{sssec:description_closure} Finally, there is another description of $\M_1(\Lambda)$ and $\FM(\La,\T)$ as moduli space of configurations of flags, which aligns better when comparing with the flag moduli of the weaves $\ww(\bG)$. In that latter case, there will be only one ambient vector space. The description in Lemma \ref{lem:flag_description}, which is associated to the specific front $\mathfrak{f}(\bG)$, after using RII and RIII moves, can be shown to be equivalent to a description with a unique ambient (top-dimensional) vector space. Indeed, rather than using flags from different ambient spaces with varying dimensions, we can perform additional RII and RIII moves to push strands as the blue one in Figure \ref{fig: pulling up a left cusp} all the way to the left and push strands as the blue one in Figure \ref{fig: pulling down a right cusp} all the way to the right (see also Lemma \ref{lem:linkweave}). This will extend all flags from all Type 1 columns to flags in $\mathbb{C}^h$ where $h$ is the total number of horizontal lines in the GP-graph $\bG$. Moreover, these flags will satisfy the relative position conditions imposed by the external weave lines of the initial weave $\ww=\ww(\bG)$, or equivalently, the cyclic positive braid word $\beta=\beta(\bG)$ for which $\Lambda$ is its $(-1)$ closure. In this context, \cite[Proposition 1.5]{STZ_ConstrSheaves}, or Lemma \ref{def: working def for M1}, reads:

\begin{lemma}\label{lem:1closure} Let $\beta=(i_1,i_2,\dots,i_l)\in\mbox{Br}_h^+$ be a positive braid word on $h$  strands and $\Lambda$ be the Legendrian link associated to the front given by the $(-1)$-closure of $\beta$. Then
\[
\M_1(\Lambda)\cong\left.\left\{(\SF_0,\SF_1,\SF_2,\dots, \SF_l)\ \middle| \ \begin{array}{l}\text{$\SF_i$ is a flag in $\mathbb{C}^h$ for all $i$} \\ 
\SF_0\overset{s_{i_1}}{\sim}\SF_1\overset{s_{i_2}}{\sim}\cdots \overset{s_{i_l}}{\sim}\SF_l=\SF_0\end{array}\right\}\right/\PGL_h.
\]
\end{lemma}

\subsection{Factoriality Property}\label{ssec:delta-complete} In the upcoming construction of a cluster $\mathcal{A}$-structure for the moduli space $\FM(\La,T)$, we shall need that the coordinate ring $\mathcal{O}(\FM(\La,T))$ is a unique factorization domain (a.k.a.~factorial). This can be a subtle condition to verify and thus we provide in this section an argument that the condition of $\Delta$-completeness of the braid $\beta(\bG)$, as introduced in Section \ref{sssec:examplesGPgraph}, is sufficient for factoriality. Note that all shuffle graphs have $\beta(\bG)$ be a $\Delta$-complete braid, and thus the rings $\mathcal{O}(\FM(\La(\bG),T))$ are factorial if $\bG$ is shuffle.

\begin{prop}\label{prop:UFD} Let $\bG$ be a GP-graph with $\beta(\bG)$ a $\Delta$-complete braid. Then the moduli space $\FM(\Lambda(\bG),T)$ is an affine variety whose coordinate ring is factorial.
\end{prop}
\begin{proof} Since the moduli space $\FM(\Lambda,T)$ is an Legendrian invariant, without loss of generality we can turn $\Lambda$ into the $(-1)$-closure of an $n$-stranded positive braid $\beta(\bG)=\Delta \gamma$ and use the description from Section \ref{sssec:description_closure}. Let us first consider the case where the set $T$ of marked points can be arranged into a configuration with one marked point per level along a vertical line between $\Delta$ and $\gamma$. (It follows that $|T|=n$.) This case is depicted as follows:
\[
\begin{tikzpicture}[scale=0.7]
\draw [dashed] (0,0) rectangle node [] {$\Delta$} (2,2);
\draw [dashed] (3,0) rectangle node [] {$\gamma$} (5,2);
\draw [dashed] (-1,0) -- (-1,2);
\draw [dashed] (6,0) -- (6,2);
\foreach \i in {0,1,3} 
{
\draw (2,\i*0.5+0.25) -- (3,\i*0.5+0.25);
\draw (-1,\i*0.5+0.25) -- (0,\i*0.5+0.25);
\draw (5,\i*0.5+0.25) -- (6,\i*0.5+0.25);
\node at (2.5,\i*0.5+0.25) [] {$\bullet$};
}
\node at (2.5,1.25) [] {$\vdots$};
\node at (-0.5,1.25) [] {$\vdots$};
\node at (5.5,1.25) [] {$\vdots$};
\end{tikzpicture}
\]
Let $B_+$ and $B_-$ be the Borel subgroups of $\mbox{PGL}_n$ of upper triangular and lower triangular matrices, respectively. We can exhaust the $\PGL_n$-action on flag configurations by fixing the two flags at the two ends of $\Delta$ to be the two unique flags stabilized by $B_+$ and $B_-$, respectively, while requiring that the decoration on the flag $\SF_l$ at the dashed line (after $\gamma$ on the right or before $\Delta$ on the left) to be the standard one, i.e., mapping $\overline{e}_i$ to $1$ for each consecutive quotient $\Span\{e_1,\dots, e_i\}/\Span\{e_1,\dots, e_{i-1}\}\cong \Span(\overline{e}_i)$. 

Let $(i_1,\dots, i_l)$ be a positive word for the positive braid $\gamma$ such that $\beta(\bG)=\Delta\gamma$. Let us record a flag as a matrix with row vectors, such that the span of the last $k$ row vectors give the $k$-dimensional subspace in the flag. Then $\SF_l$ can be recorded by the permutation matrix $w_0$. Starting from the flag $\SF_l$, the flags $\SF_{l-1}, \SF_{l-2},\dots$ to the left of $\SF_l$ can then be given by
\[
\SF_k=B_{i_{k+1}}(z_{k+1})B_{i_{k+2}}(z_{k+2})\cdots B_{i_l}(z_l)w_0.
\]
In the end, we need $\SF_0$ to be the standard flag 
\[
0\subset \Span\{e_n\}\subset \Span\{e_{n-1},e_n\}\subset \cdots \subset \Span\{e_2,\dots, e_n\}\subset \mathbb{C}^n,
\]
which is equivalent to requiring that $B_{i_1}(z_1)B_{i_2}(z_2)\cdots B_{i_k}(z_l)w_0$ to be upper triangular. This shows that $\FM(\Lambda(\bG),T)$ is isomorphic to the braid variety $X(\beta(\bG), w_0)$ from \cite{CGGS}, which has a factorial coordinate ring.\footnote{We thank Eugene Gorsky for an explanation of why this is the case. See also upcoming work of the first author with E.~Gorsky and co-authors where this is written in detail.}

Now let us consider the case of an arbitrary number of marked points. Let us start with the set $T$ having one marked point per level, as in the case above. Suppose $m$ of the marked points share the same link component, then we can move these marked points along that link component until they get inside an horizontal interval with no crossings or cusps. Then these marked points are just changing decorations on the same underlying 1-dimensional quotient of consecutive vector spaces of the same flag. Thus we can extract out a $(\mathbb{C}^\times)^{m-1}$-torus factor and replace these marked points with one marked point. By doing this for each link component, we can reduce $T$ to a set $T'$ with one marked point per link component, and conclude that
\[
\FM(\Lambda, T)\cong \FM(\Lambda,T')\times (\mathbb{C}^\times)^{n-N}
\]
as affine varieties, where $N$ is the number of link components in $\La=\Lambda(\bG)$. This implies that
\[
\mathcal{O}(\FM(\Lambda,T))\cong \mathcal{O}(\FM(\Lambda, T'))\otimes \mathbb{C}[t_i^{\pm 1}]_{i=1}^{n-N}.
\]
If there is an element in $\mathcal{O}(\FM(\Lambda, T'))$ admitting two non-equivalent factorizations, then these two factorizations are still valid and non-equivalent in $\mathcal{O}(\FM(\Lambda,T))$, contradicting the fact that $\mathcal{O}(\FM(\Lambda,T))$ is factorial. Thus, we can conclude that $\mathcal{O}(\FM(\Lambda, T'))$ is factorial when $T'$ consists of one marked point per link component. In general, for any set $T''$ with at least one marked point per link component, we can implement the same argument above and write
\[
\FM(\Lambda, T'')\cong \FM(\Lambda, T')\times (\mathbb{C}^\times)^{|T''|-N}.
\]
Algebraically this implies that 
\[
\mathcal{O}(\FM(\Lambda, T''))\cong \mathcal{O}(\FM(\Lambda, T'))\otimes \mathbb{C}[t_i^{\pm 1}]_{i=1}^{|T''|-N}.
\]
Again, since $\mathcal{O}(\FM(\Lambda, T'))$ is factorial, so is the tensor product $\mathcal{O}(\FM(\Lambda, T'))\otimes \mathbb{C}[t_i]_{i=1}^{|T''|-N}$. Given that $\mathcal{O}(\FM(\Lambda, T'))\otimes \mathbb{C}[t_i^{\pm 1}]_{i=1}^{|T''|-N}$ is a localization of this factorial tensor product, it is factorial as well.
\end{proof}

\subsection{Moduli spaces for the Lagrangian \texorpdfstring{$L(\ww(\bG))$}{}}\label{ssec:moduli_Lagrangian} The moduli spaces $\M_1(\Lambda)$ and $\FM(\Lambda, T)$ depend only on the Legendrian isotopy type of $\La$. In particular, if $\Lambda=\Lambda(\bG)$ is a GP-link, then these moduli spaces are invariant under square moves and other combinatorial equivalences of the GP-graph $\bG$ which preserve the Legendrian isotopy class of $\La$. The GP-graph also provides the information of an embedded exact Lagrangian filling for $\La(\bG)$. Namely, the exact Lagrangian filling $L=L(\bG)$ described by the initial weave $\ww=\ww(\bG)$. The Guillermou-Jin-Treumann map \cite{JinTreumann}, or \cite{EHK,CasalsNg}, imply that there are open embeddings
$$H^1(L;\C^\times)\lr\M_1(\Lambda),\quad H^1(L,T;\C^\times)\lr\FM(\Lambda,T),$$
where the domains of these embeddings parametrize (decorated) $\C$-local systems on $L$ (with decoration $T$), and the map is essentially the microlocalization functor. These open torus charts $(\C^\times)^{b_1(L)}$ and $(\C^\times)^{b_1(L,T)}$ can be described in terms of flags if the Lagrangian filling $L$ is obtained from a weave, as explained in \cite{CasalsZaslow}; we shall use it in the proof of Theorem \ref{thm:main}. The definition of $\M_1(\ww)$ from \cite{CasalsZaslow} is as follows:

\begin{definition}\label{def:flagmoduli} Let $\ww\sse\R^2$ be a weave. By definition, the {\it total flag moduli space} $\tilde{\mathcal{M}}_1(\ww)$ associated to $\ww$ is comprised of tuples of flags, as follows:  
	
	\begin{itemize}
		\item[i)] There is a flag $\SF^\bullet(F)$ assigned to each face $F$ of the weave $\ww$, i.e. to each connected component of $\R^2\setminus\ww$.\\
		
		\item[ii)] 
		\label{def:flagmoduli-general-ii}
		For each pair of adjacent faces $F_1,F_2\sse \R^2\setminus \ww$, sharing an $s_i$-edge, their two associated flags $\SF^\bullet(F_1),\SF^\bullet(F_2)$ are in relative position $s_i\in S_n$, i.e.~ they must satisfy
		$$\SF_j(F_1)=\SF_j(F_2),\quad 0\leq j\leq N, j\neq i,\mbox{ and }\SF_i(F_1)\neq\SF_i(F_2).$$
	\end{itemize}
	The group $\PGL_n$ acts on the space $\tilde{\mathcal{M}}_1(\ww)$ simultaneously. By definition, the {\it flag moduli space} of the weave $\ww$ is the quotient stack $\mathcal{M}_1(\ww):=\tilde{\mathcal{M}}(\ww)/\PGL_n$.
\end{definition}

By Subsection \ref{sssec:flag_description}, $\mathcal{M}_1(\ww)$ is an open subspace of $\mathcal{M}_1(\Lambda)$ via restriction to the boundary. Indeed, since the weaves $\ww$ are free weaves \cite[Section 7.1.2]{CasalsZaslow}, the data of flags at the boundary of the initial weave, uniquely determines the flags at each face of $\ww$. (This fact can also be verified combinatorially.) It follows from \cite{CasalsZaslow} that $\mathcal{M}_1(\ww)$ are complex tori $\mathcal{M}_1(\ww)\cong(\C^\times)^{\dim\mathcal{M}_1(\Lambda)}$, and thus these moduli spaces of flags associated to the initial weave $\ww$ are natural candidates for an initial cluster chart in the moduli space $\mathcal{M}_1(\Lambda)$ for a GP-link $\Lambda$. (These complex tori are indeed the images of the Guillermou-Jin-Treumann maps.)  The definition of candidate cluster $\cX$-variables will be the subject of the next subsection.

The decorated version of the flag moduli $\M_1(\ww)$, which we denote as $\FM(\ww, T)$, is naturally defined by adding a framing away from $T$ along the boundary $\partial L(\ww)=\Lambda$. It also follows that $\FM(\ww,T)$ is naturally an open torus chart in $\FM(\Lambda,T)$. The corresponding definition of the candidate cluster $\cA$-variables is undertaken in Subsection \ref{ssec:clusterAvars} below.


\subsection{Microlocal monodromies: unsigned candidate \texorpdfstring{$\cX$-variables}{}}\label{ssec:clusterXvars}

Let us consider the open  toric chart $\M_1(\ww)\sse\M_1(\Lambda)$ from Subsection \ref{ssec:moduli_Lagrangian}. We now build a function $$X_\gamma:\M_1(\ww)\lr\C$$ associated to each $\sf Y$-cycle $\gamma$, generalizing our previous work in \cite[Section 7]{CasalsZaslow}, see also \cite[Section 5.1]{STZ_ConstrSheaves}, to our context. First we observe that the data of $\M_1(\ww)$ associates a flag $\SF$ in each connected component of the complement of $\ww$ in $\mathbb{R}^2$. We associate the 1-dimensional vector space $\SF_i/\SF_{i-1}$ to the $i$th sheet in the lift of each connected component. Then across the lifts of each weave line, we define two linear isomorphisms
\begin{equation}\label{eq: parallel transport}
\begin{tikzpicture}[scale=0.7, baseline=20]
\draw (0,0) node [left] {$\SL_i/\SL_{i-1}$} to [out=0,in=-135] (2,0.75) to [out=45,in=180] (4,1.5) node [right] {$\SR_{i+1}/\SR_i$};
\draw (0,1.5) node [left] {$\SL_{i+1}/\SL_i$} to [out=0,in=135] (2,0.75) to [out=-45,in=180] (4,0) node [right] {$\SR_i/\SR_{i-1}$};
\node at (0.5,0.75) [] {$\SL_i$};
\node at (3.5,0.75) [] {$\SR_i$};
\node at (2,-0.5) [] {$\SL_{i-1}=\SR_{i-1}$};
\node at (2,2) [] {$\SL_{i+1}=\SR_{i+1}$};
\node at (-4,2) [left] {$\displaystyle \psi_+:\frac{\SL_i}{\SL_{i-1}}\hookrightarrow \frac{\SL_{i+1}}{\SL_{i-1}}=\frac{\SR_{i+1}}{\SR_{i-1}}\twoheadrightarrow \frac{\SR_{i+1}}{\SR_i}.
$};
\node[black] at (-4,0) [left] {$\displaystyle \psi_-:\frac{\SR_i}{\SR_{i-1}}\hookrightarrow \frac{\SR_{i+1}}{\SR_{i-1}}=\frac{\SL_{i+1}}{\SL_{i-1}}\twoheadrightarrow \frac{\SL_{i+1}}{\SL_i}.$};
\end{tikzpicture}
\end{equation}

Note that $\psi_{\pm}$ are isomorphisms because $\SL$ and $\SR$ are in $s_i$-transverse position, as they are separated by a weave line labeled with $s_i$. Now, given a loop $\gamma$ on $L$, we may perturb it so that it intersects with any lifts of weave lines transversely. Then by composing several of the isomorphisms $\psi_\pm$ above and their inverses, we obtain a linear automorphism for each generic fiber along $\gamma$. Since each generic fiber is a 1-dimensional vector space, we can represent this linear automorphism by a non-zero scalar $\psi_\gamma$. This non-zero scalar $\psi_\gamma$ is also known as the \emph{microlocal monodromy} of the sheaf moduli space $\M_1(\ww)$ along $\gamma$. However, the microlocal monodromies $\psi_\gamma$ do not naturally give rise to a local system on $L$\footnote{They give a {\it twisted} local system as in \cite[Part 13]{Guillermou19_SheafSummary}, or a twisted flat connection as in \cite[Part 10]{GMN_SpecNet13}.}, as the following illustrates:

\begin{example}
Consider the weave with a unique trivalent vertex, which depicts a Lagrangian 2-disk filling, as drawn in blue in Figure \ref{fig:parallel_transport_on_unknot} (left). According to the definition of $\cM_1(\ww)$, there is a flag $l_i\subset \mathbb{C}^2$ in each of the three sectors, and they are pairwise transverse from each other. Let $\gamma$ be a curve on $L(\ww)\cong \Lambda(\ww)$, which, under the front projection, goes from the lower sheet to the upper sheet and then back to the lower sheet; see again Figure \ref{fig:parallel_transport_on_unknot} (left). By definition, the microlocal parallel transport $\psi_\gamma$ should be the map drawn in Figure \ref{fig:parallel_transport_on_unknot} (center), which is the linear map that projects parallel to the line $l_2$.
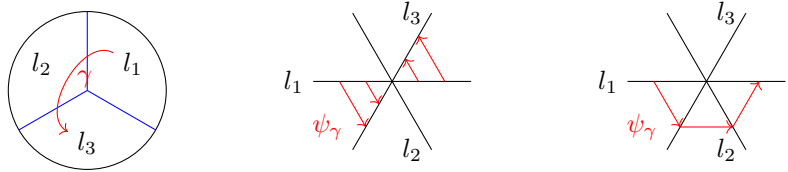
\begin{figure}[H]
    \centering
    \begin{tikzpicture}[scale=0.7]
    \draw (0,0) circle [radius=1.5];
    \foreach \i in {1,2,3}
    {
    \draw [blue] (0,0) -- (90+\i*120:1.5);
    \node (\i) at (-90+\i*120:1) [] {$l_\i$};
    }
    \draw [->,red] (1) to [out=150,in=150] node [right] {$\gamma$}  (3);
    \end{tikzpicture}\quad \quad \quad \quad
    \begin{tikzpicture}[scale=0.7]
    \foreach \i in {1,2,3}
    {
    \draw (60+\i*120:1.5) node [left] {$l_\i$} -- (-120+\i*120:1.5);
    }
    \draw [->,red] (-1,0) -- node [below left] {$\psi_\gamma$} (-0.5,-0.866);
    \draw [->,red] (-0.5,0) -- (-0.25,-0.433);
    \draw [<-,red] (0.25,0.433) -- (0.5,0);
    \draw [<-,red] (0.5,0.866) -- (1,0);
    \end{tikzpicture}
    \quad \quad \quad \quad
    \begin{tikzpicture}[scale=0.7]
    \foreach \i in {1,2,3}
    {
    \draw (60+\i*120:1.5) node [left] {$l_\i$} -- (-120+\i*120:1.5);
    }
    \draw [->,red] (-1,0) -- node [below left] {$\psi_\gamma$} (-0.5,-0.866);
    \draw [->,red] (-0.5,-0.866) -- (0.5,-0.866);
    \draw [->,red] (0.5,-0.866) -- (1,0);
    \end{tikzpicture}
    \caption{A weave for the unique filling of the max tb unknot and microlocal parallel transports from its sheaf quantization.}\label{fig:parallel_transport_on_unknot}
\end{figure}
\noindent Consider the lift $\xi$ of a loop that goes around a trivalent weave vertex in $\mathbb{R}^2$, which is a double cover for the projection onto the weave plane. Without loss of generality, let us suppose $\xi$ starts at the lower sheet in the sector containing $l_1$. The parallel transport along $\xi$ is then the composition of the three linear projections, as drawn in Figure \ref{fig:parallel_transport_on_unknot} (right), which is equal to the linear map $v\mapsto -v$ on $l_1$. In other words, $\psi_\xi=-1$. However, $\xi$ is a contractible cycle on $L(\ww)$ and thus the microlocal monodromy assignment $\xi\mapsto \psi_\xi$ cannot be a local system on $L(\ww)$.
\end{example}


Let us now specialize to our situation, with $\bG$ a GP-graph and $\ww=\ww(\bG)$ its initial weave. By Section \ref{sec:weaves}, there is a distinguished linearly independent subset $\fS(\bG)\subset H_1(L(\ww(\bG)))$ of $\mathbb{L}$-compressible cycles parametrized by the sugar-free hulls in the GP-graph. For each element in $\fS(\bG)$, we choose a $\sf Y$-tree representative $\gamma$, which exists by Lemma \ref{lem: Y representability}, and define
\[
X_\gamma:=-\psi_\gamma.
\]
These functions shall become our cluster $\cX$-variables, once signs are fixed and Theorem \ref{thm:main} is proven. Note that, since we can isotope the $\sf Y$-tree $\gamma$ to a short $\sf I$-cycle, i.e. an equivalent monochromatic edge, we may use it to compute $X_\gamma$ explicitly, as follows. In a neighborhood of a short $\sf I$-cycle, labeled with the permutation $s_i$, a point in the flag moduli $\M_1(\ww)$ is specified by the data of a quadruple of flags.  Each of these flags has the same subspaces $\SF^j$ in each region for $j\neq i$, and for $j = i$ we additionally require the data in each region of a line $l$ in the two-dimensional space $V := \SF^{i+1}/\SF^{i-1}.$  This is the data of four lines $a, b, c, d\sse V$. The function $X_\gamma$ is then equal to the cross-ratio
\[X_\gamma=\langle a, b, c, d\rangle = -\frac{a\wedge b}{b\wedge c}\cdot\frac{c\wedge d}{d\wedge a}.\quad \quad \quad \quad \quad 
\begin{tikzpicture}[baseline=0,scale=0.8]
\draw [blue] (-1,-0.5) -- (-0.25,0) -- (-1,0.5);
\draw [blue] (1,-0.5) -- (0.25,0) -- (1,0.5);
\draw [blue] (-0.25,0) -- (0.25,0);
\node at (-1,0) [] {$a$};
\node at (0,0.5) [] {$b$};
\node at (1,0) [] {$c$};
\node at (0,-0.5) [] {$d$};
\end{tikzpicture}
\]

The definition of $X_\gamma$, following \cite{FockGoncharov_ModuliLocSys} and \cite{STWZ,CasalsZaslow}, is not particularly new. It is also possible to define $X_\gamma$ directly and combinatorially from the $\sf Y$-trees, in line with \cite[Section 7]{CasalsZaslow}. The fact that these functions $\{X_\gamma\}$ transform according to an $\mathcal{X}$-mutation formula under a square-face mutation is due to \cite{STWZ}, and under the more general weave mutation due to \cite{CasalsZaslow}. Indeed, let $\Gamma=\{\gamma_i\}$ be a maximal collection of $\sf Y$-trees in $\ww(\bG)$ which are linearly independent in $H_1(L(\ww(\bG)))$, $Q(\Gamma)$ be their (algebraic) intersection quiver, and $X_\Gamma=\{X_{\gamma_i}\}$ be a labeling of each vertex of the quiver. Then, it is shown in \cite[Section 7.2.2]{CasalsZaslow} that weave mutation at one such $\sf Y$-tree $\gamma\in\Gamma$ induces a quiver mutation of $Q(\Gamma)$ at the vertex associated to $\gamma$, and the set of variables $X_\Gamma$ changes according to a cluster $\cX$-mutation.


Defining these candidate cluster $\cX$-variables is relatively useless for the purpose of proving existence of cluster structures: the variables $X_\gamma$ do {\it not} extend to global in $\M_1(\Lambda)$ in general and we cannot deduce the existence of a cluster $\cX$-structure merely from constructing this initial seed $(Q(\Gamma),X_\Gamma)$. Moreover, in general there could be many choices of $\Gamma$ for a fixed general weave $\ww$, and it is not known whether different choices yields equivalent, or even quasi-equivalent, cluster seeds. It thus becomes crucial to construct cluster $\cA$-variables for $\FM(\Lambda, T)$, ideally in a symplectic invariant manner, as we will momentarily do. By \cite{BFZ05}, a cluster $\cA$-structure can be shown to exist, once the necessary properties of the candidate $\mathcal{A}$-variables are proven. As a byproduct, Corollary \ref{cor:Xstructure} then deduces the existence of the cluster $\cX$-structure on $\M_1(\Lambda)$ where the variables are microlocal monodromies.

\subsection{Collections of sign curves: fixing signs}\label{ssec:fixingsigns} Let $\bG$ be a GP-graph, $\ww=\ww(\bG)$ its initial weave and $L:=L(\ww)$ its initial filling, and $T$ a set of marked points in $\La(\bG)=\dd L$. Let us denote the set of lifts of trivalent weave vertices on $L$ by $P\sse L$. It follows from Subsections \ref{ssec:moduli_Lagrangian} and \ref{ssec:clusterXvars} that each point of the flag moduli $\M_1(\ww)$ defines a rank 1 local system on $L\setminus P$ with $-1$ monodromy around each point in $P$. In this section, we describe a way to add signs to monodromies to obtain a (non-canonical) isomorphism between $\M_1(\ww)$ and $\Loc_1(L)$. This is a combinatorial expression of the fact that, in our case, global sections of the Kashiwara-Schapira stack are (canonically) isomorphic to the category of twisted local systems and (non-canonically) also isomorphic to the category of local systems. In terms of weave combinatorics, we proceed as follows:

\begin{definition} A \emph{sign curve} is an unoriented curve on the weave surface $L$ that intersects the lifts of weave lines transversely and whose endpoints lie in the set $P\sqcup T$. By definition, a collection $C$ of sign curves on $L$ is \emph{coherent} if each point in $P$ is incident to one and only one sign curve in $C$, and all curves in $C$ intersect transversely.
\end{definition}

\noindent We record sign curves on $L$ by drawing dotted curves on $\mathbb{R}^2$ in juxtaposition with the weave $\ww$ and labeling the indices of the sheets they are on.

Fix a coherent set $C$ of sign curves on $L$. For any path $\gamma$ on $L$, we may perturb $\gamma$ so that it intersects elements of $C$ transversely. Then, we redefine the parallel transport along $\gamma$ to be the microlocal parallel transport $\psi_\gamma$ multiplied by a factor of $-1$ whenever the curve $\gamma$ passes through a sign curve in $C$. Since each branch point of $L$ is incident to one and only one sign curve, this new parallel transport corrects the monodromy around each point in $P$ to be $1$, defining an isomorphism
\[
\Phi_C: \M_1(\ww)\longrightarrow \Loc_1(L)\cong H^1(L;\C^\times)\cong (\C^\times)^{b_1(L)}.
\]

In fact, we can do better than an arbitrary isomorphism $\M_1(\ww)\xrightarrow{\cong} \Loc_1(L)$. From Subsection \ref{ssec:clusterXvars}, our candidates for cluster $\mathcal{X}$-variables are of the form $-\psi_\gamma$ for initial absolute cycles $\gamma\in \fS(\bG)$, and we can in fact incorporate this extra sign in front of $\psi_\gamma$ into the set of coherent sign curves.

\begin{definition}\label{defn:compatible sign curves} A coherent set $C$ of sign curves on $L$ is said to be {\it compatible} if for all initial absolute cycles $\gamma\in \fS(\bG)$:
$$\Phi_C(p)(\gamma)=X_\gamma(p):=-\psi_\gamma(p),\quad \forall p\in \cM_1(\ww).$$
\end{definition}



For the initial free weave $\ww=\ww(\bG)$ constructed from a GP-graph $\bG$, we can find a compatible set of sign curves as follows. First, we observe that all trivalent weave vertices of $\ww$ occur near the boundary of the weave. Thus, at each trivalent weave vertex, two of the three adjacent sectors are facing away from the weave: we will draw our sign curves inside these two sectors. Next, we break the weave $\ww$ down into weave columns, and by Section \ref{sec:weaves} trivalent weave vertices only occur inside Type 2 columns.

\noindent Let us further classify Type 2 columns into two types: a Type 2 column is said to be \emph{critical} if it is the rightmost Type 2 column that contains part of an initial cycle $\gamma\in \fS(\bG)$; it is said to be \emph{non-critical} otherwise. By construction, each critical Type 2 column has a unique initial cycle $\gamma$ that ends there.

If a Type 2 column is non-critical, we draw a sign curve in either sector on either sheet, and then lead it towards the boundary of $L$; once it gets within a collar neighborhood of the boundary $\partial L=\Lambda$, the sign curve will follow along $\Lambda$ until it reaches a marked point. (Such a marked point exists because we have at least one marked point per link component.)

\noindent If a Type 2 column is critical, we consider the unique initial cycle $\gamma$ that ends at this Type 2 column. We compute the product of all the signs $\gamma$ has picked up along all the previous trivalent weave vertices. If the product is $1$, we add a sign curve $c$ on the appropriate sheet of either of the two sectors so that $\gamma$ intersects with $c$ non-trivially. If the product is $-1$, we add a sign curve $c$ on the appropriate sheet of either of the two sectors so that $\gamma$ intersects with $c$ trivially. By doing so, we guarantee that $\Phi_C$ maps $\gamma$ to $X_\gamma=-\psi_\gamma$, as desired.

Thus, a compatible set $C$ of sign curves exists in our setting, and we can explicitly identify $\M_1(\ww)$ with the moduli space $\Loc_1(L)$ of rank 1 local systems on $L$. This identification allows us to interpret the cluster $\mathcal{X}$-variables $X_\gamma$ as actual monodromies of local systems along the initial cycles $\gamma$, and also fixing the necessary signs for the upcoming constructions. 

\begin{prop} If $\ww$ admits a compatible set of sign curves and $\ww'$ is mutation equivalent to $\ww$, then $\ww'$ also admits a (non-canonical) compatible set of sign curves.
\end{prop}
\begin{proof} Both weave equivalences and weave mutations are local operations on the weave. Therefore, it suffices to verify that compatible sets of sign curves can be constructed locally, before and after such local operations. That is, locally in a neighborhood where the weave equivalence or mutation is going to be performed, we want to argue that any given compatible set of sign curves on that piece of the initial weave -- before an equivalence or mutation -- we can construct a compatible set of sign curves afterwards, locally on that piece of the weave after the operation.

For weave equivalences, Moves I, IV, and V in Definition \ref{def:EquivalentWeaves} do not involve any trivalent weave vertices. Thus, sign curves that pass through any of these local pictures can be carried through these equivalences using planar homotopies. In contrast, Move II (the push-through move), III, and VI do involve trivalent weave vertices. In the case of Move VI, the weave lines lift to sheets that are not adjacent to each other, therefore the weave line with no trivalent vertex (yellow in Figure \ref{fig:ReidemeisterWeave}) can be ignored when studying the set of compatible sign curves, reducing to the constant case of a trivalent vertex. By Remark \ref{rmk:non-independence of weave equivalences}, Move III is a concatenation of Moves I and II. Thus, the only weave equivalence move that remains to be studied is Move II, which will be discussed momentarily. For weave mutations, it suffices to check mutations along short $\sf I$-cycles since any $\sf Y$-tree is weave equivalent to a short $\sf I$-cycle by Proposition \ref{prop:Ytreebounds}. In conclusion, we need to study compatible sets of sign curves locally near a push-through and a weave mutation.

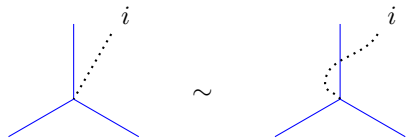
\begin{figure}[H]
    \centering
    \begin{tikzpicture}[baseline=0]
        \foreach \i in {0,1,2}
        {
        \draw [blue] (0,0) -- (90+\i*120:1);
        }
        \draw[dotted, thick] (60:1) node [above right] {$i$} -- (0,0);
    \end{tikzpicture}\quad \quad $\sim$ \quad \quad 
    \begin{tikzpicture}[baseline=0]
        \foreach \i in {0,1,2}
        {
        \draw [blue] (0,0) -- (90+\i*120:1);
        }
        \draw[dotted, thick] (60:1) node [above right] {$i$} to [out=-120,in=90] (-0.2,0.2) to [out=-90,in=120] (0,0);
    \end{tikzpicture}
    \caption{Applying a planar homotopy to a sign curve so that it arrives at the trivalent vertex from a different face. (The index $i$ can be either $1$ or $2$.)}
    \label{fig: adjacent sector}
\end{figure}

A priori, we must study sign curves that arrive to the trivalent vertices from different faces of the weave; a face being any connected component of the complement of the weave lines. That said, if a sign curve is incident to a trivalent weave vertex, we can apply a planar homotopy to the sign curve so that it arrives to the trivalent vertex from any of the another faces near the trivalent vertex. This is depicted in Figure \ref{fig: adjacent sector}. Therefore, it suffices to study the case that a sign curve arrives at a trivalent only from one of the three faces near the trivalent vertex. (If the curve arrived from another face, we can homotope the curve and, locally in a small neighborhood around the vertex, have it arrive from another face.) 

With this reduction, it suffices to study compatible sets of sign curves locally near a push-through and a weave mutation which arrive from one face (of our choice) at each trivalent vertex. Up to symmetry, there are only three cases to check, shown below in Figure \ref{fig:sign_curves_under_weave_mutations}. The figure illustrates how to resolve the problem at hand: for each set of compatible sign curves before the equivalence or mutation (on the left of each diagram), we can construct a set of compatible sign curves afterwards (on the right of each diagram).
%
\end{proof}

\begin{figure}[H]
    \centering
    \begin{tikzpicture}[baseline=0,scale=0.8]
    \draw [blue] (-1.5,-0.5) -- (-1,0) -- (-1.5,0.5);
    \foreach \i in {0,1,2}
    {
        \draw [blue] (0,0) -- (180+\i*120:1);
        \draw [red] (0,0) -- (120+\i*120:1);
    }
    \draw [dotted, thick] (-2,0) node [left] {\footnotesize{$i$}} -- (-1,0);
    \end{tikzpicture} \quad $\sim$ \quad 
    \begin{tikzpicture}[baseline=0,scale=0.8]
    \draw [blue] (-1,-0.5) to [out=30, in=180] (0,-0.25) -- (-60:1);
    \draw [blue] (-1,0.5) to [out=-30,in=180] (0,0.25) -- (60:1);
    \draw [red] (120:1) -- (0,0.25) to [out=-120,in=120] (0,-0.25) -- (-120:1);
    \draw [blue] (0,0.25) to [out=-60,in=60] (0,-0.25);
    \draw [red] (0,0.25) -- (0.75,0) -- (0,-0.25);
    \draw [red] (0.75,0) -- (1.5,0);
    \draw [dotted, thick] (-1,0) node [left] {\footnotesize{$i$}} -- (0.75,0);
    \end{tikzpicture}\\
    \begin{tikzpicture}[baseline=0,scale=0.8]
    \draw [blue] (45:1) -- (0.3,0) -- (-45:1);
    \draw [blue] (135:1) -- (-0.3,0) -- (-135:1);
    \draw [blue] (-0.3,0) -- (0.3,0);
    \draw [dotted,thick] (-1,0) node [left] {\footnotesize{$i$}} -- (-0.3,0);
    \draw [dotted,thick] (0,1) node [above]{\footnotesize{$i$}} -- (0.3,0);
    \end{tikzpicture} $\quad \leftrightarrow \quad $
    \begin{tikzpicture}[baseline=0,scale=0.8]
    \draw [blue] (45:1) -- (0,0.3) -- (135:1);
    \draw [blue] (-45:1) -- (0,-0.3) -- (-135:1);
    \draw [blue] (0,-0.3) -- (0,0.3);
    \draw [dotted,thick] (-1,0) node [left] {\footnotesize{$i$}} -- (0,-0.3);
    \draw [dotted,thick] (0,1) node [above]{\footnotesize{$i$}} -- (0,0.3);
    \end{tikzpicture}\quad \quad \quad \quad\quad \quad \quad \quad
    \begin{tikzpicture}[baseline=0,scale=0.8]
    \draw [blue] (45:1) -- (0.3,0) -- (-45:1);
    \draw [blue] (135:1) -- (-0.3,0) -- (-135:1);
    \draw [blue] (-0.3,0) -- (0.3,0);
    \draw [dotted,thick] (-0.3,0) arc (180:0:0.3) ;
    \node at (0,0.3) [above] {\footnotesize{$i$}};
    \end{tikzpicture} $\quad \leftrightarrow \quad $
    \begin{tikzpicture}[baseline=0,scale=0.8]
    \draw [blue] (45:1) -- (0,0.3) -- (135:1);
    \draw [blue] (-45:1) -- (0,-0.3) -- (-135:1);
    \draw [blue] (0,-0.3) -- (0,0.3);
    \draw [dotted,thick] (0,0.3) arc (90:-90:0.3) ;
    \node at (0.3,0) [right] {\footnotesize{$i$}};
    \end{tikzpicture}
    \caption{Existence of compatible sets of sign curves before and after weave equivalences (Top) and weave mutations (Bottom). (The index $i$ can be either $1$ or $2$.)}
    \label{fig:sign_curves_under_weave_mutations}
\end{figure}
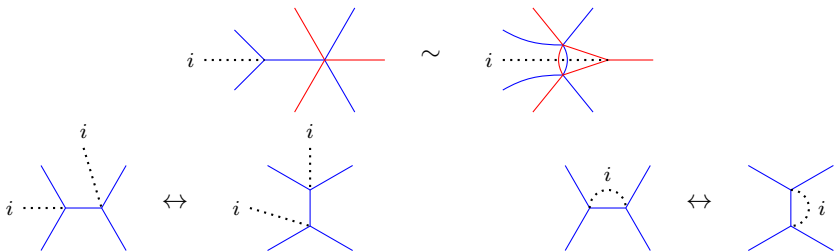

\subsection{Microlocal merodromies: candidate cluster \texorpdfstring{$\cA$-variables}{}}\label{ssec:clusterAvars}

This subsection addresses the construction of what shall become the cluster $\cA$-variables on the moduli space $\FM(\Lambda,T)$ for a GP-link $\Lambda=\Lambda(\bG)$. In the previous section, we explained that cluster $\cX$-variables were indexed by certain {\it absolute} cycles $\gamma\in H_1(L)$ in the Lagrangian filling $L=L(\ww(\bG))$ and $X_\gamma$ was a natural rational function with a symplectic origin: the microlocal monodromy along $\gamma$ of the sheaf associated with the weave $\ww=\ww(\bG)$.

Now, the new idea is that cluster $\cA$-variables $\{A_\eta\}$ will be indexed by certain {\it relative} cycles $\eta\in H_1(L\setminus T,\La\setminus T)$ and the functions $A_\eta:\FM(\Lambda,T)\lr\C$ will be defined by what we call the {\it microlocal merodromy} along $\eta$. Intuitively, this merodromy along $\eta$ is constructed as a microlocal parallel transport along $\eta$. Here are the details.

\subsubsection{Microlocal merodromy}\label{ssec:microlocal_merodromies}

Let $\bG$ be an GP-graph, $\Lambda=\Lambda(\bG)$ be its GP-link, and $T$ be a collection of marked points on $\Lambda$ with at least one marked point per link component. Fix a compatible set $C$ of sign curves and let $\ww=\ww(\bG)$ be the initial weave. The flag moduli $\FM(\ww,T)$ is an open subset of the moduli space $\FM(\Lambda, T)$, and every point in this open subset defines a local system, via the identification $\Phi_C:\cM_1(\ww)\xrightarrow{\cong} \Loc_1(L)$ in Subsection \ref{ssec:fixingsigns}, together with a framing (trivialization) of the rank-1 local system $\Phi$ on the connected components of $\Lambda\setminus T=(\partial L)\setminus T$.

The framing data defines a special vector $\phi_x\in \Phi_x$ at any point $x\in \Lambda\setminus T$. Given an oriented curve $\eta\sse L$ with both the source point $s=\dd_- \eta$ and the target point $t=\dd_+ \eta$ contained inside $\Lambda\setminus T$, we can parallel transport $\phi_s$ from the source $s$ to the target $t$ along $\eta$, obtaining a non-zero vector in $\eta(\phi_s)\in\Phi_t$. The ratio $\frac{\eta(\phi_s)}{\phi_t}$ is a non-zero number $A_\eta$, which defines a $\mathbb{C}^\times$-valued function on $\FM(\ww,T)$. This can be naturally generalized to relative 1-cycles $\eta\in H_1(L\setminus T,\Lambda\setminus T)$.

\begin{definition}\label{def:merodromy} The function $A_\eta:\FM(\ww, T)\longrightarrow \mathbb{C}^\times$ is said to be the \emph{microlocal merodromy} along the oriented curve $\eta$.
\end{definition}

\noindent Since $\FM(\ww,T)$ is an open subset of $\FM(\Lambda, T)$, $A_\eta$ can also be viewed as a rational function on $\FM(\La,T)$. Note that, a priori, $A_\eta$ might not extend to a regular function on $\FM(\La,T)$. We emphasize that the decoration in $\FM(\La,T)$ are needed in order to define $A_\eta$, and thus microlocal merodromies cannot be defined in $\mathcal{M}_1(\La)$.

The microlocal merodromies associated to relative cycles coming from marked points are non-vanishing. Indeed, for each marked point $t\in T$, we pick a small half-disk neighborhood $U_t$ of $t$, as drawn in Figure \ref{fig:merodromy_markedpoint}, and define $\xi_t:=\partial U_t$. By definition, we have that
\[
A_t:=A_{\xi_t}=\frac{\xi_t(\lambda)}{\rho}\neq 0.
\]
\begin{figure}[H]
\begin{tikzpicture}[scale=0.6]
\draw (0,0) -- node [above] {$\lambda$} (2,0) -- node [above] {$\rho$} (4,0);
\node at (2,0) [] {$\ast$};
\draw [blue,decoration={markings,mark=at position 0.5 with {\arrow{stealth}}},postaction={decorate}] (1.5,0) arc (-180:0:0.5);
\end{tikzpicture}
\caption{A marked point $t\in T$, with decorations $\la$ and $\rho$ to the left and right, and the boundary of a half-disk neighborhood $U_t$.}\label{fig:merodromy_markedpoint}
\end{figure}
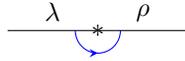
In particular, by using the ratio $\frac{\xi_t(\lambda)}{\rho}$, we can extend $A_t$ to a global invertible function on the entire moduli space $\FM(\Lambda,T)$. (This property does not in general hold for $A_\eta$ if $\eta$ is an arbitrary relative cycle in $H_1(L\setminus T,\Lambda\setminus T)$.) Now consider the following exact sequence of lattices:
\[
0\longrightarrow \mathbb{Z}\overset{i}{\longrightarrow} \bigoplus_{t\in T}\mathbb{Z}\xi_t\longrightarrow H_1(L\setminus T,\Lambda\setminus T)\overset{\pi}{\longrightarrow} H_1(L,\Lambda)\longrightarrow 0,
\]
where $i(1):=\sum_{t\in T}\xi_t$. This exact sequence implies the following two corollaries.

\begin{cor}\label{cor:prod A_t=1} $\prod_{t\in T} A_t=1$ and hence $A_t$ is a unit in $\mathcal{O}(\FM(\Lambda,T))$ for every $t\in T$.
\end{cor}
\begin{proof} It follows from the fact that $\sum_{t\in T}\xi_t=0$ in $H_1(L\setminus T,\Lambda\setminus T)$.
\end{proof}




\begin{cor}\label{cor: merodromies of the same relative cycle only differ by units} If $\eta_1,\eta_2\in H_1(\Sigma\setminus T,\Lambda\setminus T)$ satisfy $\pi(\eta_1)=\pi(\eta_2)$, then $A_{\eta_1}$ and $A_{\eta_2}$ are related to each other by a Laurent monomial in the variables $A_t$, $t\in T$.
\end{cor}
\begin{proof} It follows from the fact that $\ker \pi = \Span\{\xi_t\mid t\in T\}$.
\end{proof}

\noindent The later corollary starts to hint at the quasi-cluster equivalence that appears if different basis completions in $H_1(\Sigma\setminus T,\Lambda\setminus T)$ are chosen, as the former corollary indeed hints at the fact that $A_t$, $t\in T$, are a type of frozen variables.

\subsubsection{Crossing values}\label{sssec:} The next aim is to compute $A_\eta$ for curves whose support is transverse to a weave $\ww$. Consider the following commutative diagram of vector space inclusions 
\[
\begin{tikzpicture}[scale=0.9]
\node (s) at (0,-1) [] {$V_s$};
\node (w) at (-1,0) [] {$V_w$};
\node (e) at (1,0) [] {$V_e$};
\node (n) at (0,1) [] {$V_n$};
\draw [->] (s) -- (w);
\draw [->] (s) -- (e);
\draw [->] (w) -- (n);
\draw [->] (e) -- (n);
\end{tikzpicture}
\]
and assume that $0\rightarrow V_s\rightarrow V_w\oplus V_e\rightarrow V_n\rightarrow 0$ is exact. Let $\alpha_s,\alpha_w,\alpha_e,\alpha_n$ be non-zero top-dimensional (volume) forms in $V_s,V_w,V_e,V_n$, respectively. Then we write $\alpha_w=\beta_w\wedge\alpha_s$ and $\alpha_e=\beta_e\wedge \alpha_s$ for some forms $\beta_w$ and $\beta_e$. 

\begin{definition}\label{def:crossingvalue} In the context of a diagram as above, we define 
\[
\alpha_w\overset{\alpha_s}{\wedge}\alpha_e:=\beta_w \wedge \beta_e\wedge \alpha_s.
\]
The top form $\alpha_w\overset{\alpha_s}{\wedge}\alpha_e$ is a non-zero on $V_n$ and does not depend on the choice of $\beta_w$ or $\beta_e$. By definition, the ratio $\frac{\alpha_w\overset{\alpha_s}{\wedge}\alpha_e}{\alpha_n}$ is said to be the \emph{crossing value} of the quadruple of top forms $\alpha_s,\alpha_w,\alpha_e,\alpha_n$.
\end{definition}

Let us describe how to use crossing values to compute merodromies along planar relative cycles. Consider a flag $\SF=(0\subset \SF_1\subset \SF_2\subset \cdots \subset \SF_n=\bC^n)$ with a choice of $\phi_i\neq 0\in \SF_i/\SF_{i-1}$ for all $i\in [1,n]$. The choice of such $\phi=(\phi_i)$, $i\in[1,n]$ is said to be a \emph{framing} for the flag $\SF$. Given such \emph{framed flag} $(\SF,\phi)$, we can construct top forms $\alpha_i\in \bigwedge^i\SF_i$, $i\in[1,n]$, by first lifting each $\phi_j$ to a vector in $\tilde{\phi}_j\in \SF_j$, $j\in[1,n]$, and then taking ordered wedges, leading to the forms
\begin{equation}\label{eq: definition of decoration}
\alpha_i:=\tilde{\phi}_i\wedge \tilde{\phi}_{i-1}\wedge \cdots \wedge\tilde{\phi_1},\quad i\in[1,n].
\end{equation}
Note that each form $\alpha_i$ is independent of the choice of lifts. 

\begin{definition} Given a flag $\SF$, a collection $\alpha=(\alpha_1,\alpha_2,\dots, \alpha_n)$ of non-vanishing forms $\alpha_i\in\bigwedge^i\SF_i$, $i\in[1,n]$, is said to be a \emph{decoration} on the flag $\SF$. A flag with a decoration is referred to as a \emph{decorated flag}.
\end{definition}

\noindent Note that we can reverse the construction above and recover a framing from a decoration. Thus, framings $(\phi_1,\ldots,\phi_n)$, and decorations $(\alpha_1,\alpha_2,\dots, \alpha_n)$ of a flag $\SF$ are equivalent pieces of data. By definition, two decorated (or framed) flags $(\SL,\alpha)$ and $(\SR,\beta)$ are in $s_i$-transverse position if the underlying flags $\SL\overset{s_i}{\sim} \SR$ are in $s_i$-transverse position, i.e. $s_i$-transversality does not see decorations (or framings).


\color{black}


\noindent {Suppose $(\SL,\lambda)$ and $(\SR,\rho)$ are two framed flags such that $\SL\overset{s_i}{\sim} \SR$. Let $\alpha$ and $\beta$ be the decorations constructed from $\lambda$ and $\rho$, respectively. Consider the parallel transport maps $\psi_\pm$ defined in \eqref{eq: parallel transport}. The images $\psi_+(\lambda_i)$ and $\psi_-(\rho_i)$ are readily computed in terms of decorations as follows:

\begin{lemma}\label{lem: local crossing value} $\psi_+(\lambda_i)= \dfrac{\alpha_i\overset{\alpha_{i-1}}{\wedge}\beta_i}{\beta_{i+1}}\rho_{i+1}$ and $\psi_-(\rho_i)=\dfrac{\beta_i\overset{\beta_{i-1}}{\wedge}\alpha_i}{\alpha_{i+1}}\lambda_{i+1}$.
\end{lemma}
\begin{proof} Let $\tilde{\lambda}_i$ and $\tilde{\rho}_{i+1}$ be lifts of $\lambda_i$ and $\rho_{i+1}$. By construction, the framing $\psi_+(\lambda_i)$ can obtained as follows. First, lift $\la_i\in\SL_i/\SL_{i-1}$ to a vector $\tilde\la_i\in\SL_i$ and consider this vector as $\tilde\la_i\in\SL_{i+1}$ via the inclusion $\SL_i\sse\SL_{i+1}$. Then, using $\SL_{i+1}=\SR_{i+1}$, we can view $\tilde\la_i\in\SR_{i+1}$ and thus finally $\psi_+(\la_i)=\pi(\tilde\la_i)$, where $\pi:\SR_{i+1}\lr\SR_{i+1}/\SR_i$ is the quotient map. Each of $\psi_+(\lambda_i)$ and $\rho_{i+1}$ is a (volume) 1-form on $\SR_{i+1}/\SR_i$, and can be pulled-back via $\pi$ to 1-forms in $\SR_{i+1}$. By wedging these forms with (any) top form in $\SR_{i}$, such as $\beta_i$, we obtain the top forms $\tilde{\lambda}_i\wedge\beta_i$ and $\tilde{\rho}_{i+1}\wedge \beta_i$. Since we wedged both $\psi_+(\lambda_i)$ and $\rho_{i+1}$ with the same form $\beta_i$, their ratios are equal:

$$\frac{\psi_+(\lambda_i)}{\rho_{i+1}}=\frac{\tilde{\lambda}_i\wedge\beta_i }{ \tilde{\rho}_{i+1}\wedge \beta_i}.$$

 By \eqref{eq: definition of decoration}, we also have $\alpha_i=  \tilde{\lambda}_i\wedge \alpha_{i-1}$ and $\beta_{i+1}= \tilde{\rho}_{i+1}\wedge \beta_i$. Therefore

$$\frac{\psi_+(\lambda_i)}{\rho_{i+1}}=\frac{\tilde{\lambda}_i\wedge\beta_i }{ \tilde{\rho}_{i+1}\wedge \beta_i}=\frac{\alpha_i\overset{\alpha_{i-1}}{\wedge}\beta_i}{\beta_{i+1}}.$$

The equality for $\psi_-$ is obtained similarly.
\end{proof}

\begin{remark}
By Lemma \ref{lem: local crossing value}, the inverses of $\psi_\pm$ are computed analogously. Namely:
$$\psi_+^{-1}(\rho_{i+1})= \dfrac{\beta_{i+1}}{\alpha_i\overset{\alpha_{i-1}}{\wedge}\beta_i}\lambda_i,\qquad \mbox{and} \qquad\psi_-^{-1}(\lambda_{i+1})=\dfrac{\alpha_{i+1}}{\beta_i\overset{\beta_{i-1}}{\wedge}\alpha_i}\rho_i.$$
\hfill$\Box$
\end{remark}

Now suppose $\eta\sse\Sigma(\ww)$ is a lift of a planar curve in $\mathbb{R}^2$ to the weave front. Then, it defines a partial cross-section of the weave surface, where $\eta$ passes through a collection of (framed) flags $\SL=\SF_0, \SF_1,\SF_2,\dots, \SF_l=\SR$. For each flag $\SF_i$, $0<i<l$, we choose a sequence of top forms $\alpha_j=(\alpha_{i,j})$. Since the parallel transport along $\eta$ consists of compositions of linear isomorphisms like the maps $\psi_\pm$ in Lemma \ref{lem: local crossing value}, or their inverses, Lemma \ref{lem: local crossing value} allows us to compute $A_\eta$.

\begin{example}\label{exmp:crossing value} Consider the cross-section of a weave surface depicted in Figure \ref{fig:merodromy_cross_section}, and let $\eta$ be the blue relative cycle. The sequences of top forms $\alpha$ and $\delta$ are determined by the decorations $\lambda$ and $\rho$, respectively. The tuples of top forms $\beta$ and $\gamma$ are chosen arbitrary. Note that $\alpha_0,\beta_0,\gamma_0$ and $\delta_0$ are trivial. Let us denote the $\psi_\pm$ maps associated to each of the three crossings by $_{i}\psi_\pm$, where $i\in[1,3]$, $i=1$ being associated to the leftmost crossing, $i=2$ to the center crossing and $i=3$ to the rightmost crossing.\\

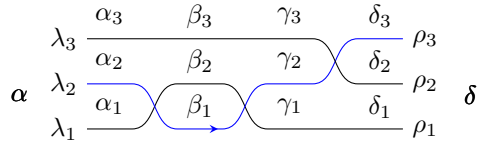
\begin{figure}[H]
    \centering
    \begin{tikzpicture}[scale=0.6]
\draw (0,0) node [left] {$\lambda_1$} -- (1,0) to [out=0,in=180] (2,1) -- (3,1) to [out=0,in=180] (4,0) -- (7,0) node [right] {$\rho_1$};
\draw [blue, decoration={markings,mark=at position 0.4 with {\arrow{stealth}}},postaction={decorate}] (0,1) node [left, black] {$\lambda_2$} -- (1,1) to [out=0,in=180] (2,0) -- (3,0) to [out=0,in=180] (4,1) -- (5,1) to [out=0,in=180] (6,2) -- (7,2) node [right, black] {$\rho_3$};
\draw (0,2) node [left] {$\lambda_3$} -- (5,2) to [out=0,in=180] (6,1) -- (7,1) node [right] {$\rho_2$};
\foreach \i in {1,2,3} {
\node at (0.5,\i-0.5) [] {$\alpha_\i$};
\node at (2.5,\i-0.5) [] {$\beta_\i$};
\node at (4.5,\i-0.5) [] {$\gamma_\i$};
\node at (6.5,\i-0.5) [] {$\delta_\i$};
\node at (-1.5,0.75) [] {$\alpha$};
\node at (8.5,0.75) [] {$\delta$};
}
\end{tikzpicture}
    \caption{Computation of a merodromy}\label{fig:merodromy_cross_section}
\end{figure}

By definition, the microlocal merodromy along $\eta$ is
$$A_\eta=\frac{\psi_\eta(\lambda_2)}{\rho_3}=\dfrac{({_3\psi_+}\circ {_2\psi_+}\circ {_1\psi_-^{-1})}(\la_2)}{\rho_3}.$$

By Lemma \ref{lem: local crossing value}, each of the microlocal merodromy $_i\psi_\pm$ and their inverses are computed as
\[
A_\eta=\frac{\psi_\eta(\lambda_2)}{\rho_3}=\frac{\gamma_2\overset{\gamma_1}{\wedge}\delta_2}{\delta_3}\cdot\frac{\beta_1\wedge \gamma_1}{\gamma_2}\cdot\frac{\alpha_2}{\beta_1\wedge \alpha_1}=\frac{\beta_1\wedge \delta_2}{\delta_3}\cdot \frac{\alpha_2}{\beta_1\wedge \alpha_1}.
\]
\noindent  Two observations based on this computation. First, the right-hand side of this expression shows that $A_\eta$ depends on the underlying undecorated flag associated with the $\beta$, but it is invariant under any non-zero rescaling of the decoration $\beta$. This is a general fact. Namely, the function $A_\eta$ does depend on the intermediate flags between the two flags at the endpoints; nevertheless, it does not depend on the decorations of these intermediate flags.

Second, a reason for the decoration $\gamma$ not appearing in the computation of $A_\eta$ above is that the flag associated to $\gamma$ is uniquely determined by the flags associated to $\beta$ and $\delta$. Observe that in the case that the slice along $\eta$ yields a reduced braid word, the intermediate flags are uniquely determined by the flags at the endpoints and thus the microlocal merodromy only depends on the decorated flags at the endpoints. For more general computations, the study of microlocal merodromies involves understanding properties, such as regularity, of products of crossing values and inverses thereof.\hfill$\Box$ 
\end{example}

In general, a microlocal merodromy $A_\eta$ will be expressed in terms of ratios of crossing values, and is only a {\it rational} function. Nevertheless, certain choices of $\eta$ within the initial weave $\ww=\ww(\bG)$ yield {\it regular} funtions. Indeed, let us consider the following special family of merodromies. By Subsection \ref{ssec:naiverelativecyles}, each face $f$ of $\bG$ has an associated naive relative cycle $\eta_f$ in $H_1(L\setminus T,\Lambda\setminus T)$. Let $A_f:=A_{\eta_f}$ be the microlocal merodromy of this naive relative cycle. Since $\{\partial f\}$ is a basis of $H_1(L)$, it follows from Poincar\'{e} duality that $\{\pi(\eta_f)\}$ is a basis of $H_1(L,\Lambda)$, where $\pi:H_1(L\setminus T,\Lambda\setminus T)\rightarrow H_1(L,\Lambda)$ is the natural projection map. By Corollary \ref{cor: merodromies of the same relative cycle only differ by units}, we conclude that different choices of $\eta_f$ only change $A_f$ by a multiple of units.

\begin{prop}\label{prop: A_f are regular functions} Let $f$ be a face in a GP-graph $\bG$. Then the microlocal merodromy $$A_f:\FM(\La(\bG),T)\lr\C$$
is a regular function.
\end{prop}
\begin{proof}
Suppose that the face $f\in\bG$ corresponds to the $k$th gap in $\bG$, counting from the bottom. Then the associated relative cycle $\eta_f$ can be written as
$$\eta_f=\sum_{i=1}^k \xi_i,$$
where the relative cycles $\xi_i$ are described as follows. Consider the braid word $\w_{0,n}$, then $\xi_i$ is the relative cycle given by the $i$th strand in $\w_{0,n}$, counting from the bottom on the left, when considered as a slice of the weave $\ww(\bG)$ along $\eta_f$. The following figure depicts such $\xi_i$ for $n=4$:
\[
\begin{tikzpicture}
\draw [decoration={markings,mark=at position 0.3 with {\arrow{stealth}}},postaction={decorate}] (0,0) node [left] {$\xi_1$} -- (0.5,0) -- (2,1.5) -- (3.5,1.5);
\draw [decoration={markings,mark=at position 0.3 with {\arrow{stealth}}},postaction={decorate}] (0,0.5) node [left] {$\xi_2$}-- (0.5,0.5) -- (1,0) -- (1.5,0) -- (2.5,1) -- (3.5,1);
\draw [decoration={markings,mark=at position 0.3 with {\arrow{stealth}}},postaction={decorate}] (0,1) node [left] {$\xi_3$} -- (1,1) -- (2,0) -- (2.5,0) -- (3,0.5) -- (3.5,0.5);
\draw [decoration={markings,mark=at position 0.3 with {\arrow{stealth}}},postaction={decorate}] (0,1.5) node [left] {$\xi_4$} -- (1.5,1.5) -- (3,0) -- (3.5,0);
\end{tikzpicture}
\]

Now, the microlocal parallel transport along $\xi_i$, from the $i$th strand on the left to the $(n-i+1)$th strand on the right, can be computed via Lemma \ref{lem: local crossing value}. In particular, each microlocal merodromy $A_{\xi_i}$ has contributions from $(i-1)$ crossings because $\w_{0,n}$ is a half-twist, and thus it is obtained after composing $(i-1)$ instances of $\psi^{-1}_\pm$. In such a slice spelling $\w_{0,n}$, let $(\SL,\alpha)$ be the decorated flag at the left endpoints and $(\SR,\beta)$ the decorated flag at the left endpoints. In line with Example \ref{exmp:crossing value}, we obtain
\[
A_{\xi_i} = \frac{\gamma\wedge \beta_{n-i}}{\beta_{n-i+1}}\cdot \frac{\alpha_{i}}{\gamma\wedge \alpha_{i-1}},\quad i\in[1,k],
\]
where $\gamma$ is a non-zero vector in the line $\SL_{i}\cap \SR_{n-i+1}$. Note that the formula for $A_{\xi_i}$ does not actually depend on $\gamma$. In fact, since $\w_{0,n}$ is a reduced word, we have $\alpha_{i-1}\wedge \beta_{n-i+1}\neq 0$. Thus, after wedging both $\gamma\wedge \beta_{n-i}$ and $\beta_{n-i+1}$ with $\alpha_{i-1}$, the expression for $A_{\xi_i}$ above reads
$$A_{\xi_i} =\frac{\alpha_{i-1}\wedge\gamma\wedge \beta_{n-i}}{\alpha_{i-1} \wedge\beta_{n-i+1}}\cdot \frac{\alpha_{i}}{\gamma\wedge \alpha_{i-1}}=(-1)^{i-1}\frac{\alpha_{i-1}\wedge\gamma\wedge \beta_{n-i}}{\alpha_{i-1} \wedge\beta_{n-i+1}}\cdot \frac{\alpha_{i}}{\alpha_{i-1}\wedge \gamma}=$$
$$=(-1)^{i-1}\frac{\alpha_{i-1}\wedge\gamma\wedge \beta_{n-i}}{\alpha_{i-1} \wedge\beta_{n-i+1}}\cdot \vartheta=(-1)^{i-1}\frac{\vartheta\cdot (\alpha_{i-1}\wedge\gamma)\wedge \beta_{n-i}}{\alpha_{i-1} \wedge\beta_{n-i+1}}=(-1)^{i-1}\dfrac{\alpha_{i}\wedge \beta_{n-i}}{\alpha_{i-1}\wedge \beta_{n-i+1}},\qquad i\in[1,k],$$
where we have denoted by $\vartheta\in\mathbb{C}^\times$ the unique non-zero scalar such that $\alpha_i=\vartheta\cdot (\alpha_{i-1}\wedge\gamma)$.
In conclusion, we obtain
\[
A_f=\prod_{i=1}^kA_{\xi_i}=(-1)^{k(k-1)/2}\cdot \frac{\alpha_k\wedge \beta_{n-k}}{\beta_n}.
\]
By definition $\beta_n\neq 0$ is non-zero and therefore $A_f$ is a regular function for each face $f$.
\end{proof}

\begin{remark}
Microlocal merodromies can also be used to define the frozen cluster $\mathcal{X}$-variables associated to the relative cycles in $H_1(L,T)$ that are not in the image of $H_1(L)$. In the moduli space $\mathcal{M}_1(\La,T)$, the microlocal merodromy allows one to compare framings at the endpoints of the relative cycles, which are marked points $T$ where the (stalk of the microlocal) local system has been trivialized.
\end{remark}


\subsection{Vanishing of microlocal merodromies and flag relative positions}\label{ssec:relative_position_flags}
In this section we study the vanishing loci of the microlocal merodromies $A_f:\FM(\La(\bG),T)\lr\C$ associated to faces $f\sse\bG$ of a GP-graph $\bG$. The key technical result, Proposition \ref{prop:relative position scanning}, relates the vanishing loci of microlocal merodromies associated to different faces of the graph $\bG$. This result is crucial to deduce the necessary properties of these candidate cluster $\mathcal{A}$-variables, such as regularity, and conclude Theorem \ref{thm:main}.

Thus far, we have parametrized the relative position of a pair of flags in $\mathbb{C}^n$ by the symmetric group $S_n$. This relative position is invariant under the diagonal $\GL_n$ action, and hence is also in bijection with $\GL_n$-orbits in $\mathcal{B}(n)\times \mathcal{B}(n)$. The inclusion relation on closures of these orbits defines a partial order called the \emph{Bruhat order} on $S_n$, i.e.~ $u\leq v$ if $\mathcal{O}_u\subset \overline{\mathcal{O}}_v$. Combinatorially, the Bruhat order can be computed through set comparison. 

\begin{definition} For two equal-sized subsets $I=\{i_1<i_2<\cdots<i_m\}$ and $J=\{j_1<j_2<\cdots<j_m\}$ of $\{1,\dots, n\}$, we define $I\leq J$ if $i_k\leq j_k$ for all $1\leq k\leq m$. By definition, for two permutations $u$ and $v$ of $S_n$, $u\leq v$ in the \emph{Bruhat order} if and only if $\{u(1),\dots, u(m)\}\leq \{v(1),\dots, v(m)\}$ for all $1\leq m< n$.
\end{definition}

By Subsection \ref{ssec:moduli_Legendrian}, the moduli $\FM(\La(\bG),T)$ can be understood in terms of tuples of flags, with maps between them and incidence constraints. The flags can be read directly from the front $\mathfrak{G}$. In particular, for a Type 1 column of $\bG$, there exists a unique pair of (decorated) flags $\SF_0$ and $\SF^0$, see Figure \ref{fig:Flags_Type1}. In the points of the open torus chart $\FM(\ww,T)\sse\FM(\La(\bG),T)$, these two flags $\SF_0$ and $\SF^0$ are in $w_0$-relative position, but in general the relative position between $\SF_0$ and $\SF^0$, at another point of $\FM(\La(\bG),T)$ might vary. The dependence of $\SF_0$ and $\SF^0$ on the point $p\in \FM(\La(\bG),T)$ will be denoted by $\SF_0(p)$ and $\SF^0(p)$.

By Proposition \ref{prop: A_f are regular functions} the microlocal merodromy $A_i$ associated to the $i$th gap of a Type 1 column, counting from below in the GP-graph, is a regular function on $\FM(\Lambda, T)$. Moreover, it can be expressed as $\dfrac{\alpha_i\wedge \beta_{n-i}}{\beta_n}$, up to a multiple of units, where $\alpha$ and $\beta$ are decorations on the pair of flags $\SF_0$ and $\SF^0$ placed at the bottom and the top of that Type 1 column, respectively. In particular, the restriction of $A_i|_{\FM(\ww,T)}$ to the open torus chart $\FM(\ww,T)\sse\FM(\La(\bG),T)$ is a non-vanishing function. The following lemma shows that we can describe the vanishing locus of this microlocal merodromy $A_i:\FM(\La(\bG),T)\lr\C$ in terms of the relative position between the two flags $\SF_0$ and $\SF^0$:

\begin{lemma}\label{lem: vanishing due to bruhat order} Let $\bG$ be a GP-graph, $(\SF_0,\alpha),(\SF^0,\beta)$ the pair of decorated flags associated with a Type 1 column $C$ and $A_i$ the $i$th microlocal merodromy associated to $C$, $i\in[1,n]$. Consider a point $p\in\FM(\La(\bG),T)$ and the permutation $w\in S_n$ such that $\SF_0(p)\overset{w}{\sim}\SF^0(p)$. Then $A_i(p)=0$ if and only if $w\leq s_iw_0$ in the Bruhat order.
\end{lemma}
\begin{proof} Without loss of generality, we may assume that the decoration $\alpha$ is proportional to $(e_{w(1)},e_{w(1)}\wedge e_{w(2)},\dots, e_{w(1)}\wedge e_{w(2)}\wedge \cdots e_{w(n)})$ and the decoration $\beta$ is proportional to $(e_1,e_1\wedge e_2,\dots, e_1\wedge e_2\wedge \cdots e_n)$. From this we know that $A_i=0$ if and only if 
\[
e_{w(1)}\wedge e_{w(2)}\wedge \cdots \wedge e_{w(i)}\wedge e_1\wedge e_2\wedge \cdots \wedge e_{n-i}=0,
\]
which is equivalent to saying that the following intersection is non-empty:
\[
\{w(1),w(2),\dots, w(i)\}\cap\{1,2,\dots, n-i\}\neq \emptyset.
\]

\noindent Now note that for the permutation $v=s_iw_0$, all $\{v(1),\dots, v(m)\}$ are maximal sets with respect to the linear order on $\{1,\dots, n\}$ except when $m=i$, where 
\[
\{v(1),\dots, v(i)\}=\{n,n-1,\dots, n-i+2,n-i\}.
\]

\noindent If $w\leq v$, then $\{w(1),w(2),\dots, w(i)\}\leq \{n,n-1,\dots, n-i+2, n-i\}$. This implies that among $w(1),w(2),\dots, w(i)$, some index no greater than $n-i$ must have appeared. Therefore we have $\{w(1),w(2),\dots, w(i)\}\cap\{1,2,\dots, n-i\}\neq \emptyset$ and hence $A_i=0$.

Conversely, if $w\not\leq v$, then we must have $$\{w(1),w(2),\dots, w(i)\}=\{n,n-1,\dots, n-i+1\},$$
\noindent which implies that $\{w(1),w(2),\dots, w(i)\}\cap\{1,2,\dots, n-i\}=\emptyset$ and hence $A_i\neq 0$.
\end{proof}

Lemma \ref{lem: vanishing due to bruhat order} shows that in order to study whether the microlocal merodromies $A_f$ vanish or not, it suffices to consider the relative position between the pair of flags in a Type 1 column that contains part of the face $f$. We use the following simple lemma in the proof of Proposition \ref{prop:relative position scanning}, through Lemma \ref{lem: Type 3 black lollipop Bruhat inequality}:

\begin{lemma}\label{lem: relative position} Let $u,v\in S_n$ and consider three flags $\SF,\SF'$ and $\SF''$ such that $\SF\overset{u}{\sim}\SF'\overset{v}{\sim}\SF''$. Let $l$ denote the length function on $S_n$ and for any $w\in S_n$ we denote by $\underline{w}$ the positive braid represented by a (equivalently any) reduced word of $w$. Then the following holds:
\begin{enumerate}
    \item If $l(uv)=l(u)+l(v)$, then $\SF\overset{uv}{\sim} \SF''$.
    \item In general, if $\SF\overset{w}{\sim}\SF'$, then $w\leq \Dem(\underline{u} \ \underline{v})$ in the Bruhat order.\footnote{
$\Dem$ denotes the Demazure product on positive braids.}
\end{enumerate}
\end{lemma}
\begin{proof} (1) Follows from the standard fact that, for Bruhat cells, if $l(uv)=l(u)+l(v)$ then $(BuB)(BvB)=BuvB$. For (2), at the level of Bruhat cells we know that $(Bs_iB)(Bs_iB)=Bs_iB\sqcup B$; this is equivalent to saying that if $\SF\overset{s_i}{\sim}\SF'\overset{s_i}{\sim}\SF''$, then there are two possible relative position relations between $\SF$ and $\SF''$: either $\SF\overset{s_i}{\sim}\SF''$ or $\SF=\SF''$. For both cases, the relative position is at most $\Dem(\underline{s_i} \ \underline{s_i})=s_i$. The general statement follows from the well-definedness of the Demazure product. 
\end{proof}

For notational convenience, for $1\leq i\leq j<n$, let us define the following permutation:
\[
w_{[i,j]}:=\begin{tikzpicture}[baseline=15]
\node at (1,2) [] {$s_{n-1}$};
\foreach \i in {0,1}
{
\node at (0.5+\i,1.5) [] {$s_{n-2}$};
}
\node at (0,1) [] {$\iddots$};
\node at (2,1) [] {$\ddots$};
\foreach \i in {0,1,2}
{
\node at (-0.5+\i*0.75,0.5) [] {$s_2$};
}
\node at (1.75,0.5) [] {$\cdots$};
\node at (2.5,0.5) [] {$s_2$};
\foreach \i in {0,1,2}
{
\node at (-1+\i*0.75,0) [] {$s_1$};
}
\node at (1.75,0) [below] {$\underbrace{\quad \quad \quad \quad}_\text{delete all $s_1$'s from the $i$th to $j$th copy}$};
\node at (3,0) [] {$s_1$};
\end{tikzpicture}
\quad \quad 
\text{and} \quad \quad 
\overline{w}:=w_0w^{-1}w_0.
\]
In this notation, $s_iw_0=w_{[i,i]}=\overline{w}_{[i,i]}$. In terms of set comparison, $u\leq w_{[i,j]}$ in the Bruhat order if and only if for all $n-j\leq l\leq n-i$,
\begin{equation}\label{eq: w_[i,j] inequality}
\{u(1),u(2),\dots, u(l)\}\leq \{i\leq \cdots\},
\end{equation}
where $\{i\leq \cdots\}$ means the set of the appropriate size (say of size $l$) consisting of the greatest $l-1$ elements in $\{1,\dots, n\}$ together with the element $i$.

Finally, the core of this subsection is the following result, which states that we can use lollipop reactions (Definition \ref{def:lollipop reaction}) to keep track of the relative position conditions on flags and, in turn, understand vanishing conditions for microlocal merodromies associated to faces. The goal of the remaining part of this subsection is to prove:

\begin{prop}\label{prop:relative position scanning} Let $\bG$ be a GP-graph and $f,g\sse \bG$ two faces. Suppose that $g$ is selected in a lollipop reaction initiated from a lollipop in $f$. Then $A_f=0$ implies $A_g=0$.
\end{prop}

As discussed in Subsection \ref{ssec:lollipop chain reaction}, the scanning wall in in a lollipop reaction moves to the right if the lollipop is white and to the left if the lollipop is black. By Lemma \ref{lem: vanishing due to bruhat order}, at the starting point, $A_f=0$ implies that the flags at the two ends of the wall are at most $w_{[i,i]}=\overline{w}_{[i,i]}$ apart, where $i$ is the index of the gap between the two adjacent horizontal lines, counting from below in the GP-graph. This is schematically illustrated in the following picture:
\[
\begin{tikzpicture}[scale=0.7]
\draw (0,0) -- (2,0);
\draw (0,2) -- (2,2);
\draw (1,1) -- (2,1);
\draw [fill=white] (1,1) circle [radius=0.1];
\draw [red, <->] (0.5,0) -- node [left] {$\leq w_{[i,i]}$} (0.5,2);
\end{tikzpicture}\quad \quad \quad \quad 
\begin{tikzpicture}[scale=0.7]
\draw (0,0) -- (2,0);
\draw (0,2) -- (2,2);
\draw (1,1) -- (0,1);
\draw [fill=black] (1,1) circle [radius=0.1];
\draw [red, <->] (1.5,0) -- node [right] {$\leq \overline{w}_{[i,i]}$} (1.5,2);
\end{tikzpicture}
\]

The heart of the argument is proving that in the case of a white (resp. black) lollipop reaction, as the wall scans to the right (resp. left), the flags at the two ends of the wall will be at most $w_{[i,j]}$ (resp. $\overline{w}_{[i,j]}$) apart, where $[i,j]$ is the interval containing the indices of the gaps that the wall crosses (counting from below in the GP-graph). Due to symmetry, we will only prove Proposition \ref{prop:relative position scanning} for white lollipop reactions; the proof for the case of black lollipop reactions is completely symmetric. Let us start with the following lemma:

\begin{lemma}\label{lem: Type 3 black lollipop Bruhat inequality} Let $\bG$ be a GP-graph and consider a Type 3 column with a black lollipop (we shift the indices on the right to match the Coxeter generators), as follows: 
\[
\begin{tikzpicture}[scale=0.7]
\foreach \i in {-1,0,1,3,4,5} 
{
\draw (0,\i*0.5) -- (2,\i*0.5);
}
\draw (0,1) -- (1,1);
\draw [fill=black] (1,1) circle [radius=0.1];
\node at (1,0.3) [] {$\vdots$};
\node at (1,1.8) [] {$\vdots$};
\node at (0,-0.25) [left] {$1$};
\node at (0,0.75) [left] {$k-1$};
\node at (0,1.25) [left] {$k$};
\node at (0,2.25) [left] {$n-1$};
\node at (2,-0.25) [right] {$2$};
\node at (2,1) [right] {$k$};
\node at (2,2.25) [right] {$n-1$};
\end{tikzpicture}
\]
Let $(\SL_0,\SL^0)$ be the pair of flags to the left and $(\SR_0,\SR^0)$ be the pair of flags to the right. Suppose that $\SL_0\overset{u}{\sim}\SL^0$ and $\SR_0\overset{v}{\sim}\SR^0$. $($Hence $h^{-1}(\SR_0)\overset{v}{\sim}h^{-1}(\SR^0)$.$)$ Then, in the Bruhat order we have that
\[
u\geq s_{k-1}s_{k-2}\cdots s_1 vs_1s_2\cdots s_{n-k}.
\]
\end{lemma}
\begin{proof} By construction, $h^{-1}(\SR_0)$ and $h^{-1}(\SR^0)$ share the same 1-dimensional subspace. Therefore $v(1)=1$, which implies that $vs_1s_2\cdots s_{n-k}$ is reduced. Lemma \ref{lem: relative position} (1) implies that
\[
\SL_0\overset{s_{k-1}\cdots s_1}{\sim}h^{-1}(\SR_0)\overset{vs_1\cdots s_{n-k}}{\sim} \SL^0,
\]
where the first relative position is given by the Type 3 column requirement.

Let us record a permutation $w\in S_n$ as an $n$-tuple $(w(1),w(2),\dots, w(n))$. Since $v(1)=1$, we may assume that $v=(1,v(2),v(3),\dots, v(n))$. Then 
\[
vs_1s_2\cdots s_{n-k}=(v(2),v(3),\dots, v(n-k+1),1,v(n-k+2),\dots, v(n)).
\]
Note that left multiplication by $s_i$ interchanges the entries $i$ and $i+1$. From (the proof of) Lemma \ref{lem: relative position} (2), we know that when multiplying $s_i$ on the left of a permutation $w$, there is only one possible relative position $s_iw$ if $i$ is on the left of $i+1$, and there can be two possible relative positions $w$ and $s_iw$ if $i$ is on the right of $i+1$, in which case $s_iw<w$. Thus, performing all the left multiplications $s_{k-1}\cdots s_1$ to $vs_1\cdots s_{n-k}$ yield the smallest relative position relation, and hence $u\geq s_{k-1}\cdots s_1vs_1\cdots s_{n-k}$ as claimed. 
\end{proof}

\medskip

\noindent\textit{Proof of Proposition \ref{prop:relative position scanning}.} It suffices to argue the case of a white lollipop reaction, by symmetry. We inductively verify the claim that, as the wall scans from left to right, the relative position between the pair of flags is kept to be at most $w_{[i,j]}$.

Suppose that the wall scanning is passing through a column of Type 2 or 3. Let $(\SL_0,\SL^0)$ be the pair of flags on the left of this column and $(\SR_0,\SR^0)$ be the pair of flags on the right of this Type 3 column. Suppose the that wall on the left goes across the interval $[i,j]$. Then by assumption, $\SL_0\overset{u}{\sim}\SL^0$ for $u\leq w_{[i,j]}$. Let $v$ be the permutation such that $\SR_0\overset{v}{\sim}\SR^0$.

Let us start with the hardest case, namely a Type 3 column with a black lollipop at the $k$th gap with $i< k\leq j$, as depicted in Figure \ref{fig:prop_scanning_cases} (left). By a shift of indices, we can view $v$ as an element in $S_{[2,n]}$, the permutation group that acts on the set $[2,n]=\{2,3,\dots, n\}$. We want to prove that $v\leq w_{[i,j-1]}\in S_{[2,n]}$. By shifting the indices of the Coxeter generators in \eqref{eq: w_[i,j] inequality}, we have that $v\leq w_{[i,j-1]}\in S_{[2,n]}$ if and only if for all $n-j\leq l\leq n-i-1$,
\[
\{v(2),v(3),\dots, v(l+1)\}\leq \{i+1\leq \cdots\}.
\]
Let us proceed by contradiction: suppose $v\not\leq w_{[i,j-1]}\in S_{[2,n]}$; then there must exist some $l$ with $n-j\leq l\leq n-i-1$ such that 
\[
\{v(2),v(3),\dots, v(l+1)\}= \{a,a+\cdots, \dots\},
\]
where $a>i+1$ is the smallest element in the set on the right. To deduce a contradiction, it suffices to prove that $\tilde{v}:=s_{k-1}\cdots s_1vs_1\cdots s_{n-k}\not\leq w_{[i,j]}\in S_{[1,n]}$, since Lemma \ref{lem: Type 3 black lollipop Bruhat inequality} states that $\tilde{v}$ is the smallest possible relative position between $\SL_0$ and $\SL^0$. Direct computation yields that 
\[
\tilde{v} = (v(2)',v(3)',\dots, v(n-k+1)',k,v(n-k+2)',\dots, v(n)'),
\]
where
\[
m':=\left\{\begin{array}{ll} m-1 & \text{if $m\leq k$},\\
m &\text{if $m> k$}.\end{array}\right.
\]

\noindent If $l\leq n-k$, then
\[
\{\tilde{v}(1),\tilde{v}(2),\dots, \tilde{v}(l)\}=\{v(2)',v(3)',\dots, v(l+1)'\}=\{a', a'+\cdots, \dots\}\not\leq \{i\leq \cdots\},
\]
where the last $\not\leq$ relation is because $a'\geq a-1>i$. This shows that $\tilde{v}\not\leq w_{[i,j]}$. 

\noindent Else, if $l>n-k$, then
\begin{align*}
\{\tilde{v}(1),\tilde{v}(2),\dots, \tilde{v}(l+1)\}=&\{v(2)',v(3)',\dots, v(n-k+1)',k,v(n-k+2)',v(l+1)'\}\\
=&\{k,a',a'+\cdots, \dots\}\not\leq \{i\leq \cdots\},
\end{align*}
where the last $\not\leq$ relation is because both $a'$ and $k$ are greater than $i$. 
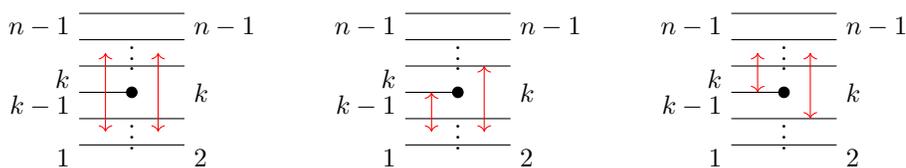
\begin{figure}[H]
\begin{tikzpicture}[scale=0.7]
\foreach \i in {-1,0,1,3,4,5} 
{
\draw (0,\i*0.5) -- (2,\i*0.5);
}
\draw (0,1) -- (1,1);
\draw [fill=black] (1,1) circle [radius=0.1];
\node at (1,0.3) [] {$\vdots$};
\node at (1,1.8) [] {$\vdots$};
\node at (0,-0.25) [left] {$1$};
\node at (0,0.75) [left] {$k-1$};
\node at (0,1.25) [left] {$k$};
\node at (0,2.25) [left] {$n-1$};
\node at (2,-0.25) [right] {$2$};
\node at (2,1) [right] {$k$};
\node at (2,2.25) [right] {$n-1$};
\draw [red,<->] (0.5,0.25) -- (0.5,1.75);
\draw [red,<->] (1.5,0.25) -- (1.5,1.75);
\end{tikzpicture} \quad \quad 
\begin{tikzpicture}[scale=0.7]
\foreach \i in {-1,0,1,3,4,5} 
{
\draw (0,\i*0.5) -- (2,\i*0.5);
}
\draw (0,1) -- (1,1);
\draw [fill=black] (1,1) circle [radius=0.1];
\node at (1,0.3) [] {$\vdots$};
\node at (1,1.8) [] {$\vdots$};
\node at (0,-0.25) [left] {$1$};
\node at (0,0.75) [left] {$k-1$};
\node at (0,1.25) [left] {$k$};
\node at (0,2.25) [left] {$n-1$};
\node at (2,-0.25) [right] {$2$};
\node at (2,1) [right] {$k$};
\node at (2,2.25) [right] {$n-1$};
\draw [red,<->] (0.5,0.25) -- (0.5,1);
\draw [red,<->] (1.5,0.25) -- (1.5,1.5);
\end{tikzpicture} \quad \quad 
\begin{tikzpicture}[scale=0.7]
\foreach \i in {-1,0,1,3,4,5} 
{
\draw (0,\i*0.5) -- (2,\i*0.5);
}
\draw (0,1) -- (1,1);
\draw [fill=black] (1,1) circle [radius=0.1];
\node at (1,0.3) [] {$\vdots$};
\node at (1,1.8) [] {$\vdots$};
\node at (0,-0.25) [left] {$1$};
\node at (0,0.75) [left] {$k-1$};
\node at (0,1.25) [left] {$k$};
\node at (0,2.25) [left] {$n-1$};
\node at (2,-0.25) [right] {$2$};
\node at (2,1) [right] {$k$};
\node at (2,2.25) [right] {$n-1$};
\draw [red,<->] (0.5,1) -- (0.5,1.75);
\draw [red,<->] (1.5,0.5) -- (1.5,1.75);
\end{tikzpicture}
\caption{Three of the cases in the proof of Proposition \ref{prop:relative position scanning}.}\label{fig:prop_scanning_cases}
\end{figure}

There are three more special cases to consider for Type 3 columns with black lollipops, two of them depicted at the center and right of Figure \ref{fig:prop_scanning_cases}, as well as the case where neither $k-1$ nor $k$ is not contained in $[i,j]$. 

For the case of Figure \ref{fig:prop_scanning_cases} (center), we want to prove that if $u\leq w_{[i,k-1]}\in S_{[1,n]}$, then $v\leq w_{[i,k-1]}\in S_{[2,n]}$. We proceed by contradiction again. Suppose $v\not\leq w_{[i,k-1]}$; then there must exist some $l$ with $n-k\leq l\leq n-i-1$ such that
\[
\{v(2),v(3),\dots, v(l+1)\}=\{a,a+\cdots, \dots\},
\]
where $a>i+1$ is the smallest element on the right. By the same argument, we see that 
\begin{align*}
\{\tilde{v}(1),\tilde{v}(2),\dots, \tilde{v}(l+1)\}=&\{v(2)',v(3)',\dots, v(n-k+1)',k,v(n-k+2)',v(l+1)'\}\\
=&\{k,a',a'+\cdots, \dots\}\not\leq \{i\leq \cdots\}.
\end{align*}
This shows that $\tilde{v}\not\leq w_{[i,k-1]}\in S_{[1,n]}$.

For the case of Figure \ref{fig:prop_scanning_cases} (right), we want to prove that if $u\leq w_{[k,j]}\in S_{[1,n]}$, then $v\leq w_{[k-1,j-1]}\in S_{[2,n]}$; a proof by contradiction works again, as follows. Suppose $v\not\leq w_{[k-1,j-1]}$; then there must exist some $l$ with $n-j\leq l\leq n-k$ such that
\[
\{v(2),v(3),\dots, v(l+1)\}=\{a,a+\cdots, \dots\},
\]
where $a>k$ is the smallest element on the right. Since $a>k$, we know that $a'=a$, and hence
\[
\{\tilde{v}(1),\tilde{v}(2),\dots, \tilde{v}(l)\}=\{a',a'+\cdots, \dots\}=\{a,a+\cdots,\dots\}\not\leq \{k\leq \cdots\}.
\]
This shows that $\tilde{v}\not\leq w_{[k,j]}$.

The remaining case, where neither $k-1$ nor $k$ is contained in $[i,j]$, can be proved similarly, and it is left as an exercise. This covers all the cases with a Type 3 column with a black lollipop.

Next we consider the case of a Type 3 column with a white lollipop in the $k$th gap, counting from below in the GP-graph, as depicted below. By the Type 3 column requirement, we know that
\[
\begin{tikzpicture}[scale=0.7,baseline=15]
\foreach \i in {-1,0,1,3,4,5} 
{
\draw (0,\i*0.5) -- (2,\i*0.5);
}
\draw (2,1) -- (1,1);
\draw [fill=white] (1,1) circle [radius=0.1];
\node at (1,0.3) [] {$\vdots$};
\node at (1,1.8) [] {$\vdots$};
\node at (0,1) [left] {$k-1$};
\node at (0,-0.25) [left] {$1$};
\node at (0,2.25) [left] {$n-2$};
\node at (2,-0.25) [right] {$1$};
\node at (2,0.75) [right] {$k-1$};
\node at (2,1.25) [right] {$k$};
\node at (2,2.25) [right] {$n-1$};
\end{tikzpicture}, \quad \quad \quad \quad 
\SR_0 \overset{s_{k}s_{k+1}\cdots s_n}{\sim} h(\SL_0) \overset{u}{\sim} h(\SL^0) \overset{s_ns_{n-1}\cdots s_{n-k}}{\sim} \SR^0,
\]
where by assumption $u\leq w_{[i,j]}$. By Lemma \ref{lem: relative position} (2) and a direct computation, we conclude that $\SR_0\overset{v}{\sim}\SR^0$ for 
\[
v\leq \Dem(\underline{s_k\cdots s_n} \ \underline{u}\ \underline{s_n\cdots s_{n-k}})\leq \Dem(\underline{s_k\cdots s_n} \ \underline{w_{[i,j]}}\ \underline{s_n\cdots s_{n-k}}) =\left\{\begin{array}{ll} w_{[i,j]} & \text{if $j<k-1$}, \\
w_{[i,j+1]} & \text{if $i\leq k-1\leq j$},\\
w_{[i+1,j+1]} &\text{if $i\geq k$}.
\end{array}\right.
\]

Lastly, we consider Type 2 columns. By using Lemma \ref{lem: relative position}, we compute that, unless we are in one of the two wall shrinking situations depicted below, we have the logical implication $u\leq w_{[i,j]}\implies v\leq w_{[i,j]}$.
\[
\begin{tikzpicture}[scale=0.7]
\draw (0,2) -- (2,2);
\draw (0,1) -- (2,1);
\draw (0,0) -- (2,0);
\node at (0,1.5) [left] {$j$};
\vertbar[](1,1,2);
\draw[<->, red] (0.5,-0.5) -- (0.5,2);
\draw[<->, red] (1.5,-0.5) -- (1.5,1);
\node[red] at (1,-0.5) [] {$\rightsquigarrow$};
\end{tikzpicture}\quad \quad \quad \quad \quad \quad 
\begin{tikzpicture}[scale=0.7]
\draw (0,2) -- (2,2);
\draw (0,1) -- (2,1);
\draw (0,0) -- (2,0);
\node at (0,0.5) [left] {$i$};
\vertbar[](1,1,0);
\draw[<->, red] (0.5,0) -- (0.5,2.5);
\draw[<->, red] (1.5,1) -- (1.5,2.5);
\node[red] at (1,2.5) [] {$\rightsquigarrow$};
\end{tikzpicture}
\]
For the leftmost of the two cases depicted above, we directly have $u\leq w_{[i,j]}\implies v\leq w_{[i,j-1]}$, and for the rightmost one, we obtain $u\leq w_{[i,j]}\implies v\leq w_{[i+1,j]}$, as required.
\qed

\subsection{Completeness of GP-graphs}\label{ssec:completenessGP} This brief subsection discusses the concept of {\it complete} GP-graphs, for which the argument we present gives a complete proof of Theorem \ref{thm:main}. As prefaced in Subsection \ref{ssec:delta-complete}, the factoriality of the coordinate ring $\mathcal{O}(\FM(\La(\bG),T))$ is a requirement. Following the results from Subsection \ref{ssec:relative_position_flags}, we add an additional hypothesis, as follows:

\begin{definition}\label{def:completeGPgraph} A grid plabic graph $\bG$ is said to be \emph{complete} if the moduli stack $\fM=\fM(\La(\bG),T)$ satisfies the following properties
\begin{itemize}\setlength{\itemsep}{0.5em}
    \item[-] The coordinate ring $\mathcal{O}(\fM)$ is a unique factorization domain (UFD).
    \item[-] For any face $f\sse\bG$ with a sugar-free hull $\S_f$, the vanishing locus of the microlocal merodromy $A_f$ is contained in the vanishing loci of $A_g$, for all faces $g\in \S_f$.
\end{itemize}
\end{definition}
These two conditions are technical, and are only trying to capture the most general type of GP-graph $\bG$ for which the argument works. In practice, if a reasonable example of a $\bG$ is given, it is possible to verify the second condition by direction computation (e.g.~using Gr\"obner basis), whereas the factoriality condition is, to our knowledge, more subtle. That said, as explained in Subsection \ref{ssec:delta-complete}, we have developed combinatorial criteria to ensure the first condition, e.g.~$\Delta$-completeness of $\beta(\bG)$, or even more combinatorially, $\bG$ being a shuffle graph. In fact, shuffle graphs $\bG$ also satisfy the second condition, as can be seen by examining the following combinatorial property:

\begin{definition}\label{def:complete_GP}
A GP-graph $\bG$ is said to be $\mathbb{S}$-complete if every sugar-free hull of $\bG$ can be obtained via some lollipop chain reaction.
\end{definition}

\noindent By Propositions \ref{prop:UFD} and \ref{prop:relative position scanning}, $\mathbb{S}$-complete GP-graphs satisfy the second condition in Definition \ref{def:complete_GP}. By Proposition \ref{prop:lollipop chain reaction for shuffles}, shuffle graphs are $\mathbb{S}$-complete, and therefore complete. The schematics of implications are: $\mbox{Shuffle graphs } \bG \Longrightarrow \left(\beta(\bG)\mbox{ }\Delta\mbox{-complete}\right)\mbox{ + }\left(\mathbb{S}\mbox{-complete }\bG\right)\mbox{ } \Longrightarrow \mbox{ complete }\bG$.

In summary, though Theorem \ref{thm:main} is proven for complete grid plabic graphs, there are large classes of $\bG$-graphs that can be proven combinatorially to be complete, either because they are shuffle or because $\mathbb{S}$-completeness and $\Delta$-completeness are directly verified. Note that shuffle graphs include all plabic fences, so all open Bott-Samelson varieties at the level of $\FM(\La,T)$, several families of interesting links (such as the twist knots), many braid varieties (e.g.~all 3-stranded ones), and more.


\subsection{Proof of the Main Theorem}\label{ssec:codim 2 argument}


In this section, we conclude the proof Theorem \ref{thm:main}. At this stage, we can consider the initial open torus chart in $\FM(\La,T)$ given by $\FM(\ww(\bG),T)$, as built in Section \ref{sec:weaves} and Subsection \ref{ssec:moduli_Lagrangian}, with the candidate cluster $\mathcal{A}$-variables being the microlocal merodromies (constructed in Subsection \ref{ssec:clusterAvars}) along a set of initial relative cycles (built in Section \ref{sec:weaves}). Namely, given a GP-graph $\bG$ and an initial set of relative cycles $\{\eta_1,\ldots,\eta_s\}$ for the pair $(L(\bG),T)$ associated to $\ww=\ww(\bG)$, we have an isomorphism $\C[A_{\eta_1}^{\pm1},\ldots,A_{\eta_s}^{\pm1}]\cong \mathcal{O}(\FM(\ww,T))$, and thus $\{A_{\eta_1},\ldots,A_{\eta_s}\}$ and their inverses naturally coordinatize the open torus chart $\FM(\ww(\bG),T)\sse\FM(\La,T)$. This defines an initial open toric chart $U_0$, candidate for an initial cluster $\mathcal{A}$-chart, with the quiver $Q(\bG,\eta)$ being the intersection quiver associated to the (duals of the) relative cycles $\{\eta_1,\ldots,\eta_s\}$.

We shall now show that the algebra $\mathcal{O}(\FM(\La,T))$ coincides with an upper cluster algebra, along with the remaining items of Theorem \ref{thm:main}. In geometric terms, the key ingredient for the former claim will be to prove that the initial cluster chart $U_0\sse\FM(\La,T)$ together with all the once-mutated charts cover the moduli space $\FM(\Lambda,T)$ up to codimension 2, i.e.~ $$U_0\bigcup_{\eta\mbox{ \footnotesize mutable\normalsize}} \mu_{\eta}(U_0)\stackrel{\mbox{\tiny up to codim.~2\normalsize}}{=}\FM(\La,T).$$
By Hartogs's extension theorem, any two normal varieties that differ at most in codimension 2 have the same algebra of regular functions. It thus follows from this codimension 2 isomorphism that $\mathcal{O}(\FM(\Lambda, T))$ is equal to the coordinate ring of the union of the initial chart $U_0$ and all the once-mutated charts. Then \cite[Corollary 1.9]{BFZ05} is used to conclude that the coordinate ring of such a union is an upper cluster algebra.

\subsubsection{Erasing $\sf Y$-trees on weaves and vanishing loci of face merodromies.} In order to establish the above covering of $\FM(\Lambda, T)$, up to codimension 2, by  $U_0$ and the once-mutated charts $\mu_\eta(U_0)$, we need to gain understanding of the codimension 1 strata in $\FM(\Lambda, T)$ that appear as vanishing loci of certain microlocal merodromies. These loci can be explicitly described via non-free weaves that are obtained by erasing $\sf Y$-cycle in $\ww=\ww(\bG)$.

First, we begin with the diagrammatic process on the weave that erases $\sf Y$-trees. Let $\ww$ be a free weave and $\sf Y$-tree $\gamma\sse\ww$. By definition, the weave $\ww_{\hat{\gamma}}$  is the weave obtained by erasing $\gamma\sse\ww$ from $\ww$ according to the following local models:
\begin{itemize}\setlength{\itemsep}{0.5em}
    \item If $\gamma$ ends at a trivalent vertex, then we simply erase the weave line contained in $\gamma$ together with the trivalent vertex, turning the local picture into a single weave line. This is depicted in Figure \ref{fig:erasing_tree} (left).
    \item If $\gamma$ goes straight through a hexavalent vertex, then we erase the two weave lines contained in $\gamma$ and pull the remaining weave lines apart according to their colors, turning the local picture into two weave lines. We draw this in Figure \ref{fig:erasing_tree} (center).
    \item If $\gamma$ branches off at a hexavalent vertex, erase the three weave lines contained in $\gamma$, turning the local picture into a trivalent weave vertex picture. See Figure \ref{fig:erasing_tree} (left).
\end{itemize}
\begin{figure}[H]
    \centering
    \begin{tikzpicture}[baseline=0,scale=0.8]
    \draw [lime, line width=5] (0,0) -- (1,0);
    \foreach \i in {0,1,2} 
    {
        \draw[blue] (0,0) -- (120*\i:1);
    }
    \end{tikzpicture} $\rightsquigarrow$ 
    \begin{tikzpicture}[baseline=0,scale=0.8]
    \foreach \i in {1,2} 
    {
        \draw[blue] (0,0) -- (120*\i:1);
    }
    \end{tikzpicture}\quad \quad \quad \quad 
    \begin{tikzpicture}[baseline=0,scale=0.8]
    \draw [lime,line width=5] (-1,0) -- (1,0);
    \foreach \i in {0,1,2} 
    {
        \draw[blue] (0,0) -- (120*\i:1);
        \draw [red] (0,0) -- (60+120*\i:1);
    }
    \end{tikzpicture}$\rightsquigarrow$
    \begin{tikzpicture}[baseline=0,scale=0.8]
    \draw [blue] (120:1) -- (-0.1,0) -- (-120:1);
    \draw [red] (60:1) -- (0.1,0) -- (-60:1);
    \end{tikzpicture}\quad \quad \quad \quad
    \begin{tikzpicture}[baseline=0,scale=0.8]
    \foreach \i in {0,1,2} 
    {
        \draw [lime, line width=5] (0,0) -- (120*\i:1);
        \draw[blue] (0,0) -- (120*\i:1);
    }
    \foreach \i in {0,1,2}
    {
        \draw [red] (0,0) -- (60+120*\i:1);
    }
    \end{tikzpicture}$\rightsquigarrow$
    \begin{tikzpicture}[baseline=0,scale=0.8]
    \foreach \i in {0,1,2} 
    {
        \draw [red] (0,0) -- (60+120*\i:1);
    }
    \end{tikzpicture}
    \caption{The local models for erasing $\sf Y$-trees.}\label{fig:erasing_tree}
\end{figure}
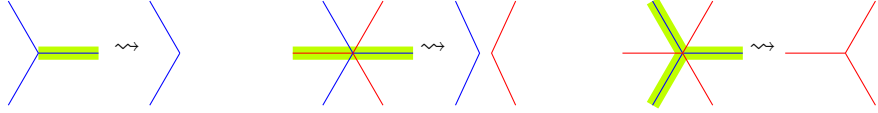

Alternatively, it is possible to first shorten the $\sf Y$-tree to a short $\sf I$-cycle, and erase the $\sf I$-cycle, which only requires applying Figure \ref{fig:erasing_tree} (left) twice. The following lemma verifies that this resulting weave is equivalent to the weave obtained by erasing the $\sf Y$-tree directly:

\begin{lemma}[$\sf Y$-tree erasing]\label{lem: erasing a Y tree} Let $\ww$ be a free weave and $\gamma\sse\ww$ a $\sf Y$-tree. Let $\ww'\sim \ww$ be a weave obtained from $\ww$ by the double track trick that shortens $\gamma$ into a short $\sf I$-cycle (see Proposition \ref{prop:Ytreebounds}), which we still denote $\gamma\sse\ww'$. Then there exists a weave equivalence $\ww'_{\hat{\gamma}}\sim \ww_{\hat{\gamma}}$. In particular, since $\ww'_{\hat{\gamma}}$ is not free, neither is $\ww_{\hat{\gamma}}$.
\end{lemma}
\begin{proof} Let us start with the short $\sf I$-cycle $\gamma$ in $\ww'$. Erasing $\gamma$ in $\ww'$ leaves two weave lines with a Reeb chord in the middle, see Figure \ref{fig:lemma_erasing} (left). Follow the rest of the double tracks and contract these two weave lines inductively. 
\begin{figure}[H]
\begin{tikzpicture}[baseline=0]
    \draw [lime, line width=5] (-0.25,0) -- (0.25,0);
    \draw [blue] (-0.25,0) -- (0.25,0);
    \draw [blue] (-0.75,-0.25) -- (-0.5,-0.25) -- (-0.25,0) -- (-0.5,0.25) -- (-0.75,0.25);
    \draw [blue] (0.75,-0.25) -- (0.5,-0.25) -- (0.25,0) -- (0.5,0.25) -- (0.75,0.25);
    \end{tikzpicture} $\rightsquigarrow$ 
    \begin{tikzpicture}[baseline=0]
    \draw [blue] (-0.75,-0.25) -- (-0.5,-0.25) -- (-0.25,0) -- (-0.5,0.25) -- (-0.75,0.25);
    \draw [blue] (0.75,-0.25) -- (0.5,-0.25) -- (0.25,0) -- (0.5,0.25) -- (0.75,0.25);
    \end{tikzpicture}\quad \quad \quad \quad 
    \begin{tikzpicture}[baseline=0,scale=0.7]
    \draw [blue] (-0.75,1) -- (0,0.25);
    \draw [red] (0,0.25) -- (0.25,0) -- (0,-0.25);
    \draw [blue] (-0.75,-1) -- (0,-0.25);
    \draw [red] (0.75,1) -- (0,0.25);
    \draw [blue] (0,0.25) -- (-0.25,0) -- (0,-0.25); 
    \draw [red] (0,-0.25) -- (0.75,-1);
    \draw [blue] (0,0.25) -- (0.5,0.25) -- (0.75,0) -- (0.5,-0.25) -- (0,-0.25);
    \draw [red] (0,0.25) -- (-1,0.25);
    \draw [red] (0,-0.25) -- (-1,-0.25);
    \end{tikzpicture}$\rightsquigarrow$
    \begin{tikzpicture}[baseline=0,scale=0.7]
    \draw [blue] (-0.25,1) -- (0,0) -- (-0.25,-1);
    \draw [red] (0.5,1) -- (0.25,0) -- (0.5,-1);
    \draw [red] (-1,0.25) -- (-0.5,0.25) -- (-0.25,0) -- (-0.5,-0.25) -- (-1,-0.25);
    \end{tikzpicture}\quad \quad \quad \quad
    \begin{tikzpicture}[baseline=0,scale=0.8]
    \foreach \i in {0,1,2}
    {   
        \draw [blue] (60+120*\i:0.29) -- (15+120*\i:1);
        \draw [red] (60+120*\i:0.29) -- (60+120*\i:1);
        \draw [blue] (60+120*\i:0.29) -- (105+120*\i:1);
        \draw [blue] (0,0) -- (60+120*\i:0.29);
    }
    \draw [red] (0,0) circle [radius=0.29];
    \draw [blue] (15:1) -- (1.2,0) -- (-15:1);
    \end{tikzpicture}$\rightsquigarrow$
    \begin{tikzpicture}[baseline=0,scale=0.8]
    \foreach \i in {0,1,2} 
    {
        \draw [red] (0,0) -- (60+120*\i:1);
    }
    \foreach \i in {1,2}
    {
    \draw [blue] (15+120*\i:1) -- (30+120*\i:0.5) -- (120*\i:0.29) -- (-30+120*\i:0.5) -- (-15+120*\i:1);
    }
    \end{tikzpicture}
    \caption{Three types of removals of (pieces of) $\sf Y$-tree cycles.}\label{fig:lemma_erasing}
\end{figure}
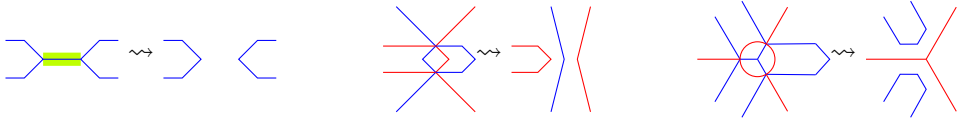
At the part where the double track goes straight though, the contracting weave line can be pulled through this part by undoing a candy twist, as depicted in Figure \ref{fig:lemma_erasing} (center). At the part where the double tracks branch off, the contracting weave line can be pulled through it by using weave equivalences, thus becoming two contracting weave lines, as illustrated Figure \ref{fig:lemma_erasing} (right). In the end, we recover the weave $\ww_{\hat{\gamma}}$, as required.
\end{proof}

Let us now study the vanishing loci of microlocal merodromies $A_f$ associated to faces $f\sse\bG$ by using Lemma \ref{lem: erasing a Y tree}. For that, fix a compatible set $C$ of sign curves on the initial weave $\ww=\ww(\bG)$ so that the chart $\cM_1(\ww)$ can be identified with $\Loc_1(L)$. For each relative 1-cycle $\eta\in H_1(L\setminus T,\Lambda\setminus T)$, its microlocal merodromy $A_\eta$ (Definition \ref{def:merodromy}) is well-defined and, in particular, we can associate a naive microlocal merodromy $A_f$ with each face $f$ of $\bG$. By Corollary \ref{cor: merodromies of the same relative cycle only differ by units} and Proposition \ref{prop: A_f are regular functions}, these $A_f$'s are regular functions on $\FM(\Lambda,T)$ and they are unique up to multiples of units.

\begin{prop}\label{prop: non-empty vanishing locus for sugar-free hulls} Let $\bG$ be a complete GP-graph and $f\sse\bG$ a face. Suppose that the lollipop chain reaction initiated from $f$ is complete. Then for any microlocal merodromy $A_f$ associated with $f$, the vanishing locus $\{A_f=0\}\sse\FM(\La(\bG),T)$ is non-empty.
\end{prop}

\begin{proof} By assumption, $f$ admits a sugar-free hull $\mathbb{S}_f$, which, by Lemma \ref{lem: Y representability}, gives rise to a $\sf Y$-tree $\gamma\sse\ww$ in the initial weave $\ww=\ww(\bG)$. Apply Proposition \ref{prop:Ytreebounds} to $\gamma$, making it a short $\sf I$-cycle, and place this short $\sf I$-cycle near the end of the original $\sf Y$-tree, so that it lies inside some Type 1 weave column. If we delete this short $\sf I$-cycle, we obtain a weave $\ww'=\ww_{\hat{\gamma}}$ whose associated weave surface is immersed, with a single interior Reeb chord at the midpoint of the short $\sf I$-cycle. We claim that its associated stratum $\M_1(\ww')$ in $\FM(\La(\bG),T)$ is non-empty.

To prove this claim, cut the weave $\ww'$ open vertically across the column into two weaves $\ww_1$ and $\ww_2$, so that both $\ww_1$ and $\ww_2$ are free weaves again. By \cite{JinTreumann}, or \cite{EHK,NRSSZ}, the two strata $\M_1(\ww_1)$ and $\M_1(\ww_2)$ are non-empty. Now, since this column used to be a Type 1 column, the vertical slice along the cut is a reduced word of $w=s_iw_0$ for some $i$. Given that any two pairs of flags of relative position $w$ are related to each other by a (non-unique) general linear group element, we use this action to line up the two pairs of flags of relative position $w$ and glue any point in $\M_1(\ww_1)$ with any point in $\M_1(\ww_2)$ and get a point in $\M_1(\ww')$. This shows that $\M_1(\ww')$ is non-empty.

\noindent Since $\FM(\ww',T)$ fibers over $\M_1(\ww')$, it follows that $\FM(\ww',T)$ is non-empty as well. Let $p$ be a point in $\FM(\ww',T)$. By Lemma \ref{lem: erasing a Y tree}, the weave $\ww'$ is equivalent to the weave $\ww_{\hat{\gamma}}$. Let $\eta_f$ be a relative 1-chain associated with $f$ and, without loss of generality, we may assume that $\eta_f$ is contained in some Type 1 column. By construction, the cross-section of the initial weave $\ww$ at a Type 1 column is always the positive (half-twist) braid $\Delta$, the positive lift of $w_0$. The erasing of the $\sf Y$-tree $\gamma$ turns the cross-sectional positive braid into a positive braid $\beta$ whose Demazure product satisfies $D(\beta)\leq s_iw_0$, for all $i$ such that the $i$th gap, counting from below in the GP-graph $\bG$, is contained in the face $f$. Moreover, the two flags at the two ends of $\eta_f$ are of relative position $w\leq D(\beta)$. Thus we conclude that $w\leq s_iw_0$ and hence $A_f(p)=0$ by Lemma \ref{lem: vanishing due to bruhat order}.
\end{proof}

We also establish the converse of Proposition \ref{prop: non-empty vanishing locus for sugar-free hulls}, i.e.~ if the lollipop chain reaction initiated from $f$ is incomplete, then the microlocal merodromy $A_f$ must be a unit in $\mathcal{O}(\FM(\La,T))$.

\begin{prop}\label{prop: unit iff chain reaction completes} Let $\bG$ be a GP-graph and $f\sse\bG$ a face. Then the microlocal merodromy $A_f$ is a unit if and only if the lollipop chain reaction initiated at $f$ is incomplete.
\end{prop}
\begin{proof} Indeed, if the lollipop chain reaction initiated from $f$ is complete, then Proposition \ref{prop: non-empty vanishing locus for sugar-free hulls} implies that $\{A_f=0\}$ is non-empty and hence $A_f$ cannot a unit. For the converse implication, note that the lollipop chain reaction initiated from $f$ being incomplete implies that during a certain lollipop reaction in the chain, the selection process runs out of faces to select. This is equivalent to saying that the selection process selects an unbounded face $g$ of the GP-graph. Now, if there exists a point $p\in \FM(\La,T)$ with $A_f(p)=0$, then by Proposition \ref{prop:relative position scanning}, $A_g(p)=0$ as well. Nevertheless, this is impossible because any relative 1-chain  $\eta_g$ associated with $g$ must map to the identity under the projection map $H_1(L\setminus T, \Lambda\setminus T)\rightarrow H_1(L,\Lambda)$, and hence $A_g$ is a unit by Corollary \ref{cor: merodromies of the same relative cycle only differ by units}. Therefore $\{A_f=0\}$ is empty and $A_f$ is a unit.
\end{proof}

Finally, let us discuss the ratios of microlocal merodromies that appear when two faces share a sugar-free hull . Recall that the dual basis $\fB^\vee$ of $H_1(L, \Lambda)$ is constructed by starting with the set $\fS(\bG)$ of initial absolute cycles, which is a linearly independent subset of $H_1(L)$ and is in bijection with the sugar-free hulls of the GP-graph $\bG$. Each element of $\fS(\bG)$ is a linear combination of the naive absolute cycles $\gamma_f$, which are in bijection with the faces of $\bG$. Then, we complete $\fS(\bG)$ to a basis $\fB$ of $H_1(L)$ via a replacement process from bottom to top along the Hasse diagram $\mathcal{H}$ of the sugar-free hulls with respect to inclusion. On the dual side, this replacement process is performed from the top down along the Hasse diagram $\mathcal{H}$, replacing the naive relative cycles $\eta_f$ one-by-one and thus obtaining a basis $\fB^\vee$ of $H_1(L,\Lambda)$ dual to $\fB$. By choosing a representative $A_f$ for each naive relative cycle $\eta_f$, we construct a microlocal merodromy function $A_i$ for each $i\in \fB^\vee$.

\noindent As there can be multiple faces sharing the same sugar-free hull, some faces (naive absolute cycles) are set aside during the replacement process. Suppose $\bS_f=\bS_g$ for two different faces $f,g\sse\bG$ and suppose that we set aside $\gamma_f$, while selecting $\gamma_g$. Then, on the dual side, the set-aside naive absolute cycle $\gamma_f$ will correspond to the relative cycle $\eta_f-\eta_g$, which in turn gives rise to the microlocal merodromy function $A_f/A_g$.

\begin{prop}\label{prop:A_f/A_g is a unit} Let $\bG$ be a GP-graph and $f,g\sse\bG$ two faces with equal sugar-free hulls. Then, the microlocal merodromy $A_f/A_g$ corresponding to a set-aside naive absolute cycle $\gamma_f$ is a unit. 
\end{prop}
\begin{proof} Since the faces $f$ and $g$ have the same sugar-free hull, the lollipop chain reaction initiated from one of them must contain the other. Therefore, by Proposition \ref{prop:relative position scanning}, $A_f(p)=0$ if and only if $A_g(p)=0$ for any $p\in\FM(\La,T)$. Since $\mathcal{O}(\FM(\Lambda, T))$ is a UFD, this implies that $A_f$ and $A_g$ are associates of each other, and thus $A_f/A_g$ is a unit.
\end{proof}

\subsubsection{Rank of exchange matrix and mutation formulae for Lagrangian surgeries} Recall the notation $U_0=\FM(\ww, T)\sse\FM(\Lambda, T)$ for the open toric chart associated to the (Lagrangian filling for the) weave $\ww=\ww(\bG)$. Let us denote the naive microlocal monodromies as $\{A_f\}$, where $f$ runs through the faces of $\bG$, and the microlocal monodromies associated with marked points as $\{A_t\}_{t\in T}$. The microlocal merodromies associated with the dual basis $\fB^\vee$ will be denoted by $\{A_i\}$, where $i\in[1,b_1(L)]$. By construction,
\[
\mathcal{O}(U_0)=\mathbb{C}\left[A_{f_1}^{\pm 1},\ldots,A_{f_{b_1(L)}}^{\pm 1} A_t^{\pm 1}\right]=\mathbb{C}\left[A_1^{\pm 1},\ldots,A_{b_1(L)}^{\pm 1},A_t^{\pm 1}\right],
\]
where the variable $t$ runs through the set of marked points $T$.
Let us also fix $\tilde{\fB}$ to be a completion of the basis $\fB\subset H_1(L)$ to a basis of $H_1(L, T)$. Then, by possibly adding relative 1-chains $\{\xi_t\}_{t\in T}$, we can modify elements of $\fB^\vee$ so that $\fB^\vee\sqcup\{\xi_t\}_{t\in T}$ becomes a dual basis of $\tilde{\fB}$. In this modification, each microlocal merodromy $A_i$ is multiplied by some Laurent monomial in the $A_t$'s. To ease notation, we rename the microlocal merodromies $A_i$'s to include these Laurent monomials as well. Since $\tilde{\fB}$ and $\tilde{\fB}^\vee$ are dual bases, there is a natural bijection between them, and we will use them interchangeably as index sets for microlocal monodromies and microlocal merodromies.

Now, the intersection form  $\{\cdot,\cdot\}$ on $H_1(L)$ can be extended to a skew-symmetric form on $H_1(L,T)$ by imposing a half integer value for intersections at $T$.
\[
\begin{tikzpicture}[baseline=10,scale=0.3]
\draw (-2,2) arc (-180:0:2 and 0.5);
\draw (-1,0) -- (0,1.5);
\draw (1,0) -- (0,1.5);
\node at (0,1.5) [] {$\ast$};
\end{tikzpicture}\quad =\quad \frac{1}{2}\quad \begin{tikzpicture}[baseline=10,scale=0.3]
\draw [dashed] (-2,2) arc (-180:0:2 and 0.5);
\draw (-1,0) -- (0,1.5);
\draw (0,1.5) -- (0.5,2.25);
\draw (1,0) -- (0,1.5);
\draw (0,1.5) -- (-0.5,2.25);
\node at (0,1.5) [] {$\ast$};
\end{tikzpicture}
\]
With respect to the basis $\tilde{\fB}$, the intersection form on $H_1(L,T)$ is then encoded by a $\tilde{\fB}\times \tilde{\fB}$ skew-symmetric matrix $\epsilon$, where
\[
\epsilon_{ij}=\{\gamma_i,\gamma_j\}
\]
for any pair $\gamma_i,\gamma_j\in \tilde{\fB}$.
For any absolute 1-cycle $\gamma$ in $H_1(L)$, Subsection \ref{ssec:clusterXvars} constructs the microlocal monodromy function $\psi_\gamma$ on $\M_1(\ww)$. After the correction by sign curves, we have $X_i:=\psi_{\gamma_i}$, and the collection $\{X_i\}_{i\in \fB}$ are our candidates for the initial  cluster $\mathcal{X}$-variables. 

Let $p:\FM(\Lambda,T)\lr\M_1(\Lambda)$ be the forgetful map. Then, by restricting to the respective tori supported on the initial weave $\ww$, we also obtain $p:\FM(\ww, T)\rightarrow \M_1(\ww)$.

\begin{prop}\label{prop: p map pull-back} Consider the forgetful map $p:\FM(\ww, T)\rightarrow \M_1(\ww)$. For an initial absolute cycle $\gamma_i\in \fS(\bG)$,
\[
p^*(X_i)=\prod_{j\in \tilde{\fB}} A_j^{\epsilon_{ij}}.
\]
\end{prop}
\begin{proof} Let us denote the relative 1-chain dual to $\gamma_i\in \tilde{\fB}$ by $\eta_i$. It suffices to prove that 
\[
\gamma_i=\sum_j \epsilon_{ij}\eta_j
\]
under the inclusion map $H_1(L)\cong H_1(L\setminus T)\hookrightarrow H_1(L\setminus T, \Lambda\setminus T)$. Note that since we have at least one marked point per link component in $\Lambda$, we can lift elements from $H_1(L,\Lambda)$ to $H_1(L,T)$. Now, for any element $\theta=\sum_{k\in \tilde{\fB}}c_k\gamma_k\in H_1(L,T)$ that is a lift of a relative 1-cycle in $H_1(L, \Lambda)$, we have
\[
\inprod{\sum_j\epsilon_{ij}\eta_j}{\theta}=\inprod{\sum_j\epsilon_{ij}\eta_j}{\sum_kc_k\gamma_k}=\sum_j \epsilon_{ik}c_k=\inprod{\gamma}{\theta}.
\]
Since the intersection form is non-degenerate on the tensor product $H_1(L)\otimes H_1(L, \Lambda)$, it indeed follows that $\gamma=\sum_j\epsilon_{ij}\eta_j$.
\end{proof}

\begin{cor}\label{cor: full-ranked} The rectangular exchange matrix $\epsilon|_{\fS(\bG)\times\tilde{\fB}}$ is full-ranked.
\end{cor}
\begin{proof} Since $\epsilon|_{\fS(\bG)\times \tilde{\fB}}$ is a submatrix of $\epsilon|_{\fB\times \tilde{\fB}}$, it suffices to prove that $\epsilon|_{\fB\times \tilde{\fB}}$ is full-ranked. The latter follows from the surjectivity of the map $p:\FM(\ww, T)\lr \M_1(\ww)$.
\end{proof}

Let us now focus towards on the effect that a Lagrangian surgery, in the form of weave mutation, has on microlocal merodromies. For that we need to understand how relative cycles change under such an operation, as follows. Let $\gamma_k$ be an initial absolute cycle which, by Lemma \ref{lem: Y representability}, we represent $\gamma_k$ as a $\sf Y$-tree on the initial weave $\ww=\ww(\bG)$. By Proposition \ref{prop:Ytreebounds}, there exists a weave equivalence that isotopes $\gamma_k$ to a short $\sf I$-cycle. In this weave equivalence, the dual basis element $\eta_k$ of $\gamma_k$ must also be isotoped to a curve that cuts through this short $\sf I$-cycle in the middle. Thus, near the short $\sf I$-cycle $\gamma_k$, the local model is the one depicted in Figure \ref{fig:mutation of relative cycles} (left). Note that each of the four weave lines extending out of this local picture may be part of multiple initial absolute cycles. However, we may assume without loss of generality that all other initial relative cycles are outside this local picture.
\begin{figure}[H]
\centering
\begin{tikzpicture}[scale=0.6]
\draw [line width = 8, lime] (-1,0)-- (1,0);
\draw (-1,0) -- (1,0);
\draw (-1,0) -- (-1.5,1.5);
\draw (-1,0) -- (-1.5,-1.5);
\draw (1,0) -- (1.5,1.5);
\draw (1,0) -- (1.5,-1.5);
\draw [dashed, orange] (0,-1.5) -- (0,1.5);
\node at (0,0) [above right] {$1$};
\node at (0.5,0) [below] {$\gamma_k$};
\node at (0,-1.5) [below] {$\eta_k$};
\end{tikzpicture}
\quad \quad \quad \quad 
\begin{tikzpicture}[scale=0.6]
\draw [line width = 8, lime] (0,-1)-- (0,1);
\draw (0,-1) -- (0,1);
\draw (0,-1) -- (-1.5,-1.5);
\draw (0,-1) -- (1.5,-1.5);
\draw (0,1) -- (-1.5,1.5);
\draw (0,1) -- (1.5,1.5);
\draw [dashed, orange] (0,-1.5) -- (-1,0) -- (1,0) -- (0,1.5);
\node at (0,0) [above right] {$1$};
\node at (0,-0.5) [right] {$\gamma'_k$};
\node at (-0.5,-0.9) [left] {$-1$};
\node at (0.5,0.9) [right] {$-1$};
\node at (0,-1.5) [below] {$\eta'$};
\end{tikzpicture}
\caption{Mutation of initial relative cycles.}
\label{fig:mutation of relative cycles}
\end{figure}
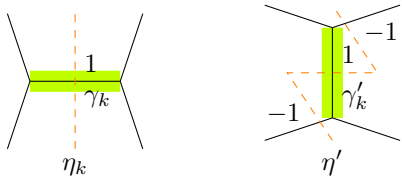

\noindent By performing a weave mutation at $\gamma_k$, we obtain a new weave $\ww_k$, which is mostly identical to $\ww$ except the local picture is replaced by Figure \ref{fig:mutation of relative cycles}. Note that the initial absolute cycle $\gamma_k$ in $\ww$ is replaced by the new absolute cycle $\gamma'_k$ in $\ww_k$.  

In the local model in Figure \ref{fig:mutation of relative cycles}, we have also drawn a relative 1-chain $\eta'$, which connects to the rest of $-\eta_k$ outside of the local picture. Nevertheless, this relative 1-chain $\eta'$ is not the correct replacement for $\eta_k$ after the surgery, because in addition to having intersection $1$ with $\gamma'_k$, the new relative cycle $\eta'_k$ also needs to have trivial intersection with all other absolute cycles in $\fB$. At this stage, $\eta'$ could possibly have non-trivial intersections with absolute cycles that come into the local picture from the northeast and the southeast. Thus, the correct replacement for the relative cycle $\eta_k$ after the weave mutation is the linear combination
\begin{equation}\label{eq:mutaion formula for relative cycles}
\eta'_k=\eta'+\sum_{j\in \tilde{\fB}}[-\epsilon_{kj}]_+\eta_j.
\end{equation}

This explains how to keep track of relative cycles after a weave mutation, and thus a Lagrangian surgery in our context. Note that the moduli space $\FM(\ww_k, T)$ also defines an open toric chart $\FM(\ww_k, T)\sse\FM(\Lambda,T)$, as $\ww_k$ defines an embedded exact Lagrangian filling as well. Let us denote this chart, where we have performed a Lagrangian surgery at the $k$th disk of the $\mathbb{L}$-compressing system, by $U_k\sse\FM(\Lambda,T)$, and denote the microlocal merodromy associated with the relative cycle $\eta'_k$ by $A'_k$.  In order to understand the change of the microlocal merodromies under surgery, we have the following result.

\begin{prop}\label{prop:mutation1} At any point $u\in U_0\cap U_k$, 
\[
A'_k=A_k^{-1}(1+p^*(X_k))\prod_{j\in \tilde{\fB}}A_j^{[-\epsilon_{kj}]_+},
\]
where $p:U_0\longrightarrow \M_1(\ww)$ is the forgetful map restricted to $U_0\cap U_k$.
\end{prop}
\begin{proof} It suffices to prove that $A_{\eta'}=A_k^{-1}(1+p^*(X_k))$. Since $\gamma_k$ is a short $\sf I$-cycle, we can assume that the four neighboring flags are four lines $l_e$, $l_s$, $l_w$, and $l_n$ in $\mathbb{C}^2$. Let $v_i$ be a non-zero vector in each line $l_i$ and let $\det$ denote the dual of the non-zero 2-form associated with $\mathbb{C}^2$. Following the crossing-value formula, we obtain
$$
A_k=\det(v_s\wedge v_n),\qquad 
A_{\eta'}=\frac{\det(v_e\wedge v_w)}{\det(v_n\wedge v_e)\det(v_s\wedge v_w)}.
$$
Therefore the product can be computed as
\begin{align*}
A_kA_{\eta'}=&\frac{\det(v_s\wedge v_n)\det(v_e\wedge v_w)}{\det(v_n\wedge v_e)\det(v_s\wedge v_w)}=\frac{\det(v_n\wedge v_e)\det(v_s\wedge v_w)+\det(v_n\wedge v_w)\det(v_e\wedge v_s)}{\det(v_n\wedge v_e)\det(v_s\wedge v_w)}\\
=&1+p^*(X_k). \qedhere
\end{align*}
\end{proof}

\begin{remark} Instead of the the relative 1-chain $\eta'$ depicted in Figure \ref{fig:mutation of relative cycles} (right), we could have chosen $\eta'$ to be have support the other zig-zag $\begin{tikzpicture}[baseline=-5,scale=0.6]\draw [dashed, orange] (0,0.5) -- (-0.3,0) -- (0.3,0) -- (0,-0.5);
\end{tikzpicture}$ with appropriate intersection numbers. In this other choice, Equation \eqref{eq:mutaion formula for relative cycles} would need to be modified to $\eta'_k=\eta'+\sum_{j \in \tilde{\fB}}[\epsilon_{kj}]_+\eta_j$ and the equation in Proposition \ref{prop:mutation1} would be modified to $A'_k=A_k^{-1}(1+p^*(X_k^{-1}))\prod_{j\in \tilde{\fB}}A_j^{[\epsilon_{kj}]_+}$. These compatible changes of signs $\epsilon_{kj}\to -\epsilon_{kj}$ define the chiral dual cluster structure on the same variety, which is the cluster structure defined by the opposite quiver. This chiral dual is discussed in \cite[Section 1.2]{FockGoncharov_ensemble}. Since the quiver is given by the intersection pairing between absolute cycles, the choice in Figure \ref{fig:mutation of relative cycles} is, in a sense, naturally dictated by the chosen orientation on the filling $L(\ww)$.\hfill$\Box$
\end{remark}

Proposition \ref{prop:mutation1} and Proposition \ref{prop: p map pull-back}, yield the desired cluster $\mathcal{A}$-mutation formula for the microlocal merodromies under Lagrangian surgeries on the set of initial $\mathbb{L}$-compressing disks:

\begin{cor}\label{cor: cluster mutation formula} Let $\bG$ be a GP-graph, $\{\eta_1,\ldots,\eta_s\}$ the set of naive relative cycles and $\{A_i\}$ the associated set of naive microlocal merodromies. Consider the microlocal merodromy $A_k'$ along the relative 1-chain $\eta_k'$ obtained from $\eta_k$ by weave mutation at the dual 1-cycle $\gamma_k$. Then

$$\displaystyle A'_k=\dfrac{\prod_{j\in \tilde{\fB} }A_j^{[\epsilon_{kj}]_+} +\prod_{j\in \tilde{\fB} }A_j^{[-\epsilon_{kj}]_+}}{A_k}.$$
\end{cor}

\subsubsection{Regularity of initial microlocal merodromies}
By Proposition \ref{prop: A_f are regular functions}, the naive microlocal merodromies $A_f$ are regular functions on the moduli space $\FM(\La,T)$. However, since the adjusted microlocal merodromies $\{A_i\}_{i\in \tilde{\fB}}$, corresponding to the initial basis are ratios of the naive microlocal merodromies, the initial merodromies $\{A_i\}$ are only rational functions a priori. Our next goal is to prove that for all $i\in \fB$, $A_i$ are actually global regular functions, and that they are either irreducible if $i\in \fS(\bG)$, or units otherwise. Let us start with the following lemma:

\begin{lemma}\label{lem: factorization} Let $U_0\sse\FM(\La,T)$ be the initial open toric chart and $f$ a unit in $\mathcal{O}(U_0)$, resp. $\mathcal{O}(U_k)$. Suppose that  $f=gh$ in $\mathcal{O}(\FM(\Lambda,T))$ for some $g,h\in\mathcal{O}(\FM(\Lambda,T))$. Then $g$ and $h$ are also units in $\mathcal{O}(U_0)$, resp. $\mathcal{O}(U_k)$.
\end{lemma}
\begin{proof} Indeed, if $f=gh$ in $\mathcal{O}(\FM(\Lambda,T))$, then $f=gh$ in $\mathcal{O}(U_0)$ as well, and if $f$ is a Laurent monomial in $\mathcal{O}(U_0)$, then each of $g$ and $h$ must be a Laurent monomial, too. The proof for the case where $f$ is a unit in $\mathcal{O}(U_k)$ is analogous.
\end{proof}

\noindent We first show that the initial merodromies are irreducible assuming they are regular functions:

\begin{lemma}\label{lem: regular implies irreducible} Let $\bG$ be a GP-graph, $\gamma_k\in \fS(\bG)$ an initial absolute cycle and consider $A_k:\FM(\La,T)\dashrightarrow\C$ an associated microlocal merodromy. Suppose that $A_k$ is a regular function, i.e., an element of $\mathcal{O}(\FM(\Lambda,T))$. Then $A_k$ is irreducible.
\end{lemma}
\begin{proof} 
Suppose $A_k=gh$ in $\mathcal{O}(\FM(\Lambda,T))$ with neither $g$ nor $h$ being a unit. Then Lemma \ref{lem: factorization} implies that $g$ and $h$ must be Laurent monomials , and hence can be expressed as $\prod_{i\in \tilde{\fB}} A_i^{m_i}$ and $\prod_{i\in \tilde{\fB}} A_i^{n_i}$, respectively, up to a multiple of units in $\mathcal{O}(\FM(\Lambda,T))$. Since $A_k=\prod_j A_i^{m_i+n_i}$, then at least one of $m_k$ and $n_k$ must be positive. Without loss of generality, let us assume that $m_k>0$. Then, since $h$ is not a unit, there must be some $j\neq k$ such that $A_j$ is not a unit and $n_j>0$. If $A_j$ is not a unit, then by Proposition \ref{prop: unit iff chain reaction completes}, $j$ must correspond to an initial absolute cycle $\gamma_j$. By mutating along the initial absolute cycle $\gamma_j$, we obtain a new weave $\ww_j$. By Lemma \ref{lem: factorization}, $h$ is also a Laurent monomial in the new chart $\mathcal{O}(U_j)$ associated to $\ww_j$ and hence we can write $g$ and $h$ as
\[
g=A'^{p_j}_j\prod_{i\neq j} A_i^{p_i} \quad \quad \text{and} \quad \quad h=A'^{q_j}_j\prod_{i\neq j} A_i^{q_i}.
\]
Note that since $A_k=gh$, we must have $p_j+q_j=0$. If $p_j=q_j=0$, then we have a contradiction because $\prod_{i\neq j}A_i^{q_i}=h=\prod_i A_i^{n_i}$ with $n_j>0$. That said, if $p_j$ and $q_j$ are non-zero, then one of them must be positive; suppose $p_j>0$. By Corollary \ref{cor: cluster mutation formula}, we have $A'_j=M_1+M_2$ where $M_1$ and $M_2$ are two algebraically independent Laurent monomials in $\{A_i\}_{i\in \tilde{\fB}}$, up to units. It then follows that
\[
g=(M_1+M_2)^{p_j}\prod_{i\neq j} A_i^{p_i},
\]
which shows that $g$ is not a Laurent monomial in $\{A_i\}_{i\in \tilde{\fB}}$. This is again a contradiction, and therefore $A_k$ must be an irreducible element in $\mathcal{O}(\FM(\Lambda,T))$.
\end{proof}

\noindent We are ready to conclude regularity, and thus irreducibility, of initial merodromies:

\begin{prop} Let $\bG$ be a GP-graph, $\gamma_k\in \fS(\bG)$ an initial absolute cycle and $A_k$ an associated microlocal merodromy. Then $A_k$ is a regular function and an irreducible element in $\mathcal{O}(\FM(\Lambda,T))$.  
\end{prop}
\begin{proof} By Lemma \ref{lem: regular implies irreducible}, it suffices to prove that $A_k$ is a regular function. We proceed by induction from top down along the Hasse diagram $\mathcal{H}$; recall that vertices of $\mathcal{H}$ are sugar-free hulls and hence they are naturally indexed by the set of initial absolute cycles $\fS(\bG)$. For the base case, suppose $k$ is a maximal vertex in the Hasse diagram. Then $A_k=A_f$ for some naive relative cycle $\eta_f$. Then Proposition \ref{prop: A_f are regular functions} implies that $A_k=A_f$ is a regular function, as required. Inductively, suppose for all $i>k$ in the Hasse diagram $\mathcal{H}$, $A_i$ is a regular function on $\FM(\La,T)$. By Lemma \ref{lem: regular implies irreducible}, $A_i$ are irreducible elements in $\mathcal{O}(\FM(\Lambda,T))$ as well. Let $f_i$ be the face selected for each vertex $i$ of $\mathcal{H}$. Then, for each vertex $i$ of $\mathcal{H}$, we have
\[
A_{f_i}=\prod_{j\geq i}A_j.
\]
In particular, if $i>k$, then the above is the unique factorization of the naive microlocal merodromies $A_{f_i}$ in $\mathcal{O}(\FM(\Lambda,T))$, up to mutliple of units. This also implies that the irreducible elements $\{A_i\}_{i>k}$ are not associates of each other because $A_{f_i}$ are not units by Corollary \ref{cor: merodromies of the same relative cycle only differ by units}.

In addition to the above, if $i>k$, then $f_i$ is contained in the sugar-free hull $\S_{f_i}=\S_i$. Proposition \ref{prop:relative position scanning} implies $\{A_{f_i}=0\}\subset \{A_{f_k}=0\}$, but we obtain the inclusion $\{A_i=0\}\subset \{A_{f_i}=0\}$ as well because $A_i$ is an irreducible factor of $A_{f_i}$. Therefore $\{A_i=0\}\subset \{A_{f_k}=0\}$, which is equivalent to $A_{f_k}$ being divisible by $A_i$ for all $i>k$. Since $A_i$ are distinct irreducible elements of $\mathcal{O}(\FM(\Lambda, T))$, it follows that the quotient
\[
A_k=A_{f_k}\prod_{i>k}A_i^{-1}
\]
is also a regular function on $\FM(\Lambda,T)$ as well. The induction is now complete.
\end{proof}

\subsubsection{Conclusion of the argument} We finalize the proof of the covering of $\mathcal{O}(\FM(\Lambda,T))$ by the initial and adjacent charts, up to codimension 2. For each initial absolute cycle $\gamma_k\in \fS(\bG)$, we denote the vanishing locus of the associatied microlocal merodromy by $D_k:=\{A_k=0\}\sse\FM(\Lambda,T)$. Since $A_k$ is an irreducible element of $\mathcal{O}(\FM(\Lambda,T))$, $D_k$ is irreducible as a codimension 1 subvariety in $\FM(\Lambda,T)$.

\begin{prop}\label{prop:intersection_open} Let $\bG$ be a GP-graph, $\gamma_k\in \fS(\bG)$ initial absolute cycle, $U_k\sse\FM(\Lambda,T)$ the open torus chart associated to the Lagrangian surgery of $L(\ww(\bG))$ at $\gamma_k$, and $D_k\sse\FM(\Lambda,T)$ the vanishing locus of its associated microlocal merodromy. Then the intersection $U_k\cap D_k\sse\FM(\Lambda,T)$ is a non-empty open subset of the vanishing locus $D_k$.
\end{prop}
\begin{proof} It suffices to prove that $U_k\cap D_k$ is non-empty. Similar to the proof of Proposition \ref{prop: non-empty vanishing locus for sugar-free hulls}, we apply Proposition \ref{prop:Ytreebounds} to move $\gamma_k$ to a short $\sf I$-cycle near the end of the original $\sf Y$-tree, so that it lies inside some Type 1 weave column. By deleting this short $\sf I$-cycle, we obtain a weave $\ww'$ whose moduli space $\FM(\ww',T)$ is a subset of $D_k$. It suffices to show that $ \FM(\ww',T)\cap U_k\neq \emptyset$, but this clear: for instance, in the case where the weave looks like the one on the left below,
\[
\begin{tikzpicture}[scale=0.7]
\draw [blue] (0,0) -- (4,0);
\draw [red] (0,1) -- (4,1);
\draw [blue] (0,1.75) -- (1,1.75) -- (1.25,2) -- (1,2.25) -- (0,2.25);
\draw [blue] (2.5,3.5) -- (2.5,2) -- (4,2);
\node at (-1,1) [] {$\ww_1$};
\node at (5,1) [] {$\ww_2$};
\node at (1,-0.5) [] {$\SL_0$};
\node at (1,0.5) [] {$\SL_1$};
\node at (1,1.25) [] {$\SL_2$};
\node at (0.5,2) [] {$\SL_3$};
\node at (3.5,-0.5) [] {$\SR_0$};
\node at (3.5,0.5) [] {$\SR_1$};
\node at (3.5,1.5) [] {$\SR_2$};
\node at (3.5,2.5) [] {$\SR_3$};
\end{tikzpicture}
\quad \quad \quad \quad \quad \quad \quad
\begin{tikzpicture}[scale=0.7]
\draw [blue] (0,0) -- (4,0);
\draw [red] (0,1) -- (4,1);
\draw [blue] (0,1.75) -- (4,1.75);
\draw [blue] (0,2.25) -- (2.5,2.25) -- (2.5,4);
\draw [blue] (2,1.75) -- (2,2.25);
\node at (1,-0.5) [] {$\SL_0$};
\node at (1,0.5) [] {$\SL_1$};
\node at (1,1.25) [] {$\SL_2$};
\node at (0.5,2) [] {$\SL_3$};
\node at (3.5,-0.5) [] {$\SR_0$};
\node at (3.5,0.5) [] {$\SR_1$};
\node at (3.5,1.25) [] {$\SR_2$};
\node at (3.5,2) [] {$\SR_3$};
\node at (2.5,-1) [] {$\ww_v$};
\node at (1.5,3) [] {$\SL_2$};
\end{tikzpicture}
\]
we first fix a point in $\M_1(\ww_1)$ and, then based on the flags $\SL_0,\SL_1,\SL_2, \SL_3$, we choose a point in $\M_1(\ww_2)$ with flags $\SR_0=\SL_0$, $\SR_1=\SL_1$, $\SR_2=\SL_2$, but $\SR_3\neq \SL_3$, and then glue them together. This gives a point in $\M_1(\ww')$ which is also in $\M_1(\ww_k)$, where $\ww_k$ is the mutated weave, which is also shown on the right above).
\end{proof}

\noindent At this stage, the covering property, up to codimension 2, readily follows:

\begin{prop}\label{prop: codim 2} Let $\bG$ be a GP-graph, $\gamma_k\in \fS(\bG)$ the initial absolute cycles and $U_k\sse\FM(\Lambda,T)$ the open torus charts associated to the Lagrangian surgery of $L(\ww(\bG))$ at each $\gamma_k$, where $U_0$ is the initial chart associated to $L(\ww(\bG))$. Then
$$\codim\left(U_0\cup \bigcup_{k\in \fS(\bG)}U_k\right)\geq 2,$$
i.e.~ the union of $U_0$ and all the adjacent charts $U_k$ covers $\FM(\Lambda,T)$ up to codimension 2.
\end{prop}
\begin{proof} By Proposition \ref{prop:intersection_open}, the intersection $U_j\cap D_j$ is open in $D_j$, for all $j$, and thus we have $\codim (D_j\cap U_j^c)\geq 2$ for each $j$. Thus $\codim\left(U_0\cup \bigcup_kU_k\right)\geq 2$ by the inclusions:
\[
    \left(U_0\cup \bigcup_kU_k\right)^c=U_0^c\cap \bigcap_kU_k^c
    =\left(\bigcup_j D_j\right)\cap \bigcap_k U_k^c
    =\bigcup_j\left(D_j\cap \bigcap_k U_k^c\right)
    \subset  \bigcup_j\left(D_j\cap U_j^c\right). \qedhere
\]
\end{proof}


Theorem \ref{thm:main} and Corollary \ref{cor:Xstructure} are now concluded as follows:

\begin{thm}\label{thm: upper cluster algebra} Let $\bG$ be a complete GP-graph. Then the coordinate ring of regular functions  $\mathcal{O}(\FM(\Lambda(\bG),T))$ has the structure of an upper cluster algebra.

\end{thm}
\begin{proof} Consider the open subset $U_0\cup\bigcup_{k\in \fS(\bG)}U_k\subset \FM(\Lambda,T)$. Proposition \ref{prop: codim 2} shows the equality of coordinate rings $\mathcal{O}(\FM(\Lambda,T))=\mathcal{O}(U_0\cup\bigcup_{k\in \fS(\bG)}U_k)$. Corollary \ref{cor: cluster mutation formula} implies that $\mathcal{O}(U_0\cup\bigcup_{k\in \fS(\bG)}U_k)$ is an upper bound of a cluster algebra. In addition, since $T$ has at least one marked point per link component, Corollary \ref{cor: full-ranked} shows that the rectangular exchange matrix $\epsilon|_{\fS(\bG)\times \tilde{\fB}}$ is full-ranked. Then \cite[Corollary 1.9]{BFZ05} implies that this upper bound coincides with its upper cluster algebra and therefore we conclude that $\mathcal{O}(\FM(\Lambda,T))$ is an upper cluster algebra.
\end{proof}

\begin{cor} Let $\bG$ be a complete GP-graph. Then $\mathcal{O}(\M_1(\Lambda(\bG)))$ has the structure of a cluster Poisson algebra. 
\end{cor}
\begin{proof} Let us temporarily denote the cluster $\mathcal{A}$-variety defined by the $\tilde{\fB}\times \tilde{\fB}$ exchange matrix $\epsilon$ by $\mathcal{A}$ and denote the cluster $\mathcal{X}$-variety associated with the submatrix $\epsilon|_{B\times B}$ by $\mathcal{X}$. Since $\epsilon|_{B\times \tilde{B}}$ is full-ranked, which follows from the surjectivity of $p:\FM(\Lambda,T)\rightarrow \M_1(\Lambda)$, the cluster theoretical map $p:\mathcal{A}\rightarrow \mathcal{X}$ is also surjective. Both $\mathcal{O}(\mathcal{A})$ and $\mathcal{O}(\mathcal{X})$ are intersections of Laurent polynomial rings, and thus a rational function $f$ on $\mathcal{X}$ is regular if and only if $p^*(f)$ is regular on $\mathcal{A}$, see \cite[Lemma A.1]{ShenWeng_cyclic_sieving}. That said, given that $p:\FM(\Lambda,T)\rightarrow \M_1(\Lambda)$ is surjective, a rational function on $\M_1(\Lambda)$ is regular if and only if $p^*(g)$ is regular on $\FM(\Lambda,T)$. Now consider the following commutative diagram
\[
\xymatrix{\FM(\Lambda,T) \ar@{-->}[r]_\cong^\alpha \ar[d]_p & \mathcal{A} \ar[d]^p \\
\M_1(\Lambda) \ar@{-->}[r]_\chi & \mathcal{X}.}
\]
Both horizontal maps are birational because $\FM(\Lambda,T)$ (resp. $\M_1(\Lambda)$) and $\mathcal{A}$ (resp. $\mathcal{X}$) share an open torus chart $\FM(\ww,T)$ (resp. $\M_1(\ww)$). In addition, Theorem \ref{thm: upper cluster algebra} implies that the top map induces an isomorphism between $\mathcal{O}(\FM(\Lambda,T))$ and $\mathcal{O}(\mathcal{A})$. Now, given a regular function $f\in \mathcal{O}(\mathcal{X})$, the pull-back $\chi^*(f)$ is a rational function on $\M_1(\Lambda)$ by birationality; but since $p^*\circ \chi^*(f)=\alpha^*\circ p^*(f)$ is regular on $\FM(\Lambda,T)$, it follows that $\chi^*(f)$ is regular on $\M_1(\Lambda)$. Conversely, if we are given a regular function $f\in \mathcal{O}(\M_1(\Lambda))$, we know that $(\chi^{-1})^*(f)$ is a rational function on $\mathcal{X}$ by birationality; but since $p^*\circ (\chi^{-1})^*(f)=(\alpha^{-1})^*\circ p^*(f)$ is regular on $\mathcal{A}$, it follows that $(\chi^{-1})^*(f)$ is regular on $\mathcal{X}$ as well. Therefore we conclude that $\chi$ induces an algebra isomorphism between $\chi^*:\mathcal{O}(\mathcal{X})\rightarrow \mathcal{O}(\M_1(\Lambda))$, and hence $\mathcal{O}(\M_1(\Lambda))$ is a cluster Poisson algebra.
\end{proof}

\section{Cluster DT Transformations for Shuffle Graphs}\label{sec:DT}

The cluster Donaldson-Thomas (DT) transformation is a cluster variety automorphism that manifests the Donaldson-Thomas invariants of a 3d Calabi-Yau category associated with the cluster ensemble \cite{KontsevichSoibelman_MotivicDT,KelDT,GoncharovLinhui_DT}. 
In this section we prove Corollary \ref{cor:DT}, i.e.~we focus on the cluster varieties $\cM_1(\Lambda)$ associated with shuffle graphs and in particular show that their cluster DT transformation is the composition of a Legendrian isotopy and a contactomorphism of $(\mathbb{R}^3,\xi_\st)$.

\subsection{Initial Quivers of Shuffle Graphs} Let us first prove features of the initial quivers associated with shuffle graphs. From now onward, we assume without loss of generality that shuffle graphs have all vertical edges with a black vertex on top.

\begin{prop} Let $\bG$ be a shuffle graph and $f\sse \bG$ a face. Then the sugar-free hull $\S_f$ can have a staircase pattern on at most one of its sides.
\end{prop}
\begin{proof} If a sugar-free hull has staircase patterns on more than one of its sides, then somewhere in this sugar-free hull we must have two opposing staircases that look like 
\[
\begin{tikzpicture}[scale=0.7,baseline=10]\draw (0,0) -- (0,0.7) -- (-0.3,0.7) -- (-0.3,1); \draw (1,0) -- (1,0.3) -- (0.7,0.3) -- (0.7,1);\draw [fill=black] (0,0.7) circle [radius=0.1];\draw [fill=white] (-0.3,0.7) circle [radius=0.1];\draw [fill=black] (0,0.7) circle [radius=0.1];\draw [fill=white] (-0.3,0.7) circle [radius=0.1];\draw [fill=black] (1,0.3) circle [radius=0.1];\draw [fill=white] (0.7,0.3) circle [radius=0.1];\end{tikzpicture} \quad \quad \quad \quad \text{or} \quad \quad \quad \quad \begin{tikzpicture}[scale=0.7,baseline=10]\draw (0,0) -- (0,0.3) -- (-0.3,0.3) -- (-0.3,1); \draw (1,0) -- (1,0.7) -- (0.7,0.7) -- (0.7,1);\draw [fill=black] (0,0.3) circle [radius=0.1];\draw [fill=white] (-0.3,0.3) circle [radius=0.1];\draw [fill=black] (0,0.3) circle [radius=0.1];\draw [fill=white] (-0.3,0.3) circle [radius=0.1];\draw [fill=black] (1,0.7) circle [radius=0.1];\draw [fill=white] (0.7,0.7) circle [radius=0.1];\end{tikzpicture}.
\]
In either case, the horizontal lines containing the horizontal edges above violate Definition \ref{def:shufflegraph}. 
\end{proof}

\begin{cor}\label{cor:shape of SF hulls} Let $\bG$ be a shuffle graph. Then all its sugar-free hulls must be in one of the following three shapes:
\[
\begin{tikzpicture}[scale=0.6]
\draw (0,0) -- (2,0);
\draw(0,0.5) -- (-0.3,0.5);
\draw(-1,2) -- (2,2);
\vertbar[](0,0,0.5);
\vertbar[](-0.3,0.5,1);
\vertbar[](-1,1.5,2);
\vertbar[](2,0,2);
\node at (-0.6,1.25) [] {$\ddots$};
\end{tikzpicture} \quad \quad \quad \quad \quad \quad
\begin{tikzpicture}[scale=0.6]
\draw (0,0) -- (3,0);
\draw (0,2) -- (3,2);
\vertbar[](0,0,2);
\vertbar[](3,0,2);
\end{tikzpicture}\quad \quad \quad \quad \quad \quad
\begin{tikzpicture}[scale=0.6]
\draw (0,0) -- (-3,0);
\draw(0,0.5) -- (-0.3,0.5);
\draw(-1,2) -- (-3,2);
\vertbar[](0,0,0.5);
\vertbar[](-0.3,0.5,1);
\vertbar[](-1,1.5,2);
\vertbar[](-3,0,2);
\node at (-0.6,1.25) [] {$\ddots$};
\end{tikzpicture}
\]
\end{cor}

\begin{prop}\label{prop: no arrow between sugar-free hulls and their subsets} Let $\bG$ be a shuffle graph. If $\S_f$ and $\S_g$ are two sugar-free hulls and $\S_f\subset \S_g$, then there is no arrow between their corresponding quiver vertices $Q(\bG)$, i.e., $\inprod{\partial \S_f}{\partial \S_g}=0$.
\end{prop}
\begin{proof} If $\S_f$ and $\S_g$ do not share boundaries, then $\inprod{\partial \S_f}{\partial \S_g}=0$. If $(\partial \S_f)\cap (\partial \S_g)\neq \emptyset$, then based on their possible shapes listed in Corollary \ref{cor:shape of SF hulls}, we see that $(\partial \S_f)\cap (\partial \S_g)$ must be the union of a consecutive sequence of edges. By going over all possibilities of having opposite colors at the two end points, we deduce that each possibility will always cut $\S_g$ into smaller sugar-free regions, making $\S_g$ no longer a sugar-free hull. Thus, the two end points of this union must be of the same color. Note that the pairing $\inprod{\partial \S_f}{\partial \S_g}$ can be computed by summing over contributions from the bipartite edges in $(\partial \S_f)\setminus(\partial \S_g)$: since the two end points of $(\partial \S_f)\setminus(\partial \S_g)$ are the same as the two end points of $(\partial \S_f)\cap (\partial \S_g)$, we can conclude that the contributions from the bipartite edges must cancel each other out, leaving $\inprod{\partial \S_f}{\partial \S_g}=0$ as a result.
\end{proof}

\noindent By Definition \ref{def:shufflegraph}, a shuffle graph $\bG$ with $n$ horizontal lines is equipped with a permutation $\sigma\in S_n$. Based on the permutation $\sigma$, we decompose the vertex set of the initial quiver $Q(\bG)$ as follows. 

\begin{definition} For each integer $m$ with $1\leq m<n$, we define $\sigma^{-1}[m,n]$ to be the preimage of the $(n-m+1)$-element set $[m,n]$. We order elements in $\sigma^{-1}[m,n]$ according to the ordinary linear order on natural numbers. We say that $(i,j)$ form a \emph{level} if $i<j$ in $\sigma^{-1}[m,n]$ for some $m$ and there is no $k\in \sigma^{-1}(m,n)$ such that $i<k<j$. We say a quiver vertex is on level $(i,j)$ if its corresponding sugar-free hull is sandwiched between the $i$th and the $j$th horizontal lines.
\end{definition}

\noindent If $\sigma^{-1}(m)=k$ and there exists $i<j$ in $\sigma^{-1}[m+1,n]$ such that $i<k<j$, then there can be sugar-free hulls on level $(i,j)$ containing sugar-free hulls on levels $(i,k)$ and $(k,j)$. It is possible to visualize this phenomenon as a branching on the quiver $Q(\bG)$: the \emph{main branch} contains sugar-free hulls on levels $(i,k)$ and $(k,j)$ and the \emph{side branch} contains sugar-free hulls on level $(i,j)$. See Figure \ref{fig:multiple branches} (right) and \ref{fig:branching quiver}. Note that such a branching may happen multiple times, with the side branch of the former branching becoming the main branch of the next. Figure \ref{fig:branching quiver} illustrates this for a shuffle graph associated with the permutation $\sigma=[4 1 2 3]$, with two branchings on each side.
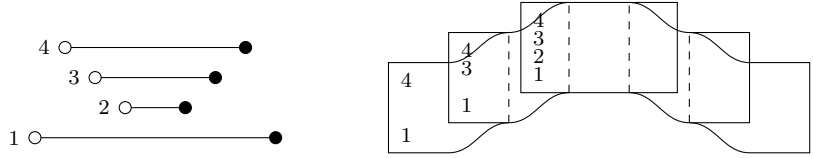
\begin{figure}[H]
    \centering
    \begin{tikzpicture}[scale=0.8]
    \node at (-0.1,0) [left] {\footnotesize{$1$}};
    \node at (0.4,1.5) [left] {\footnotesize{$4$}};
    \node at (0.9,1) [left] {\footnotesize{$3$}};
    \node at (1.4,0.5) [left] {\footnotesize{$2$}};
    \horline[](0,4,0);
    \horline[](1.5,2.5,0.5);
    \horline[](1,3,1);
    \horline[](0.5,3.5,1.5);
    \end{tikzpicture}\quad \quad \quad \quad
    \begin{tikzpicture}[scale=0.8]
    \draw (0.2,0) rectangle (2.8,1.5);
    \draw (1,1.5) to [out=180,in=0] (0,1) -- (-1,1) -- (-1,-0.5) -- (0,-0.5) to [out=0,in=180] (1,0);
    \draw (2,1.5) to [out=0,in=180] (3,1) -- (4,1) -- (4,-0.5) -- (3,-0.5) to [out=180,in=0] (2,0);
    \draw (0,1) to [out=180,in=0] (-1,0.5) -- (-2,0.5) --(-2,-1) -- (-1,-1) to [out=0,in=180] (0,-0.5);
    \draw (3,1) to [out=0,in=180] (4,0.5) -- (5,0.5) -- (5,-1) -- (4,-1) to [out=180,in=0] (3,-0.5);
    \draw [dashed] (1,0) -- (1,1.5);
    \draw [dashed] (2,0) -- (2,1.5);
    \draw [dashed] (0,-0.5) -- (0,1);
    \draw [dashed] (3,-0.5) -- (3,1);
    \foreach \i in {1,...,4}
    {
    \node at (0.5,0.3*\i) [] {\footnotesize{$\i$}};
    }
    \foreach \i in {1,3,4}
    {
    \node at (-0.7,0.3*\i-0.5) [] {\footnotesize{$\i$}};
    }
    \foreach \i in {1,4}
    {
    \node at (-1.7,0.3*\i-1) [] {\footnotesize{$\i$}};
    }
    \end{tikzpicture}
    \caption{The left picture shows the relative lengths of horizontal lines in a shuffle graph $\bG$ associated with $\sigma$. The right picture describes the branching of the quiver $Q(\bG)$; the collection of numbers on each branching records the horizontal lines that define the levels for quiver vertices on that branch.}
    \label{fig:multiple branches}
\end{figure}

\begin{figure}[H]
    \centering
    \begin{tikzpicture}[scale=0.9]
    \horline[](0,5.5,0);
    \horline[](-0.2,5.7,2);
    \horline[](1,4.5,1);
    \vertbar[](0.5,0,2);
    \vertbar[](5,0,2);
    \foreach \i in {0,3,4}
    {
    \vertbar[](1.5+\i*0.5,0,1);
    \vertbar[](4-\i*0.5,1,2);
    }
    \node at (1,1.5) [] {$1$};
    \node at (2.25,0.5) [] {$2$};
    \node at (2.25,1.5) [] {$3$};
    \node at (3.25,1.5) [] {$4$};
    \node at (3.25,0.5) [] {$5$};
    \node at (4.5,0.5) [] {$6$};
    \end{tikzpicture}\quad \quad \quad \quad 
    \begin{tikzpicture}
    \draw (-0.5,0) rectangle (5.5,2);
    \draw (1.5,2) to [out=180,in=60] (0.5,1.6) -- (0.5,-0.4) to [out=60,in=180] (1.7,0);
    \draw [lightgray] (1.5,2) -- (1.5,1) -- (1.7,1) -- (1.7,0);
    \draw (3.3,2) to [out=0,in=120] (4.5,1.6) -- (4.5,-0.4) to [out=120,in=0] (3.5,0);
    \draw [lightgray] (3.3,2) -- (3.3,1) -- (3.5,1) -- (3.5,0);
    \node (2) at (0,0.5) [] {$2$};
    \node (3) at (2,1.5) [] {$3$};
    \node (5) at (3,0.5) [] {$5$};
    \node (4) at (5,1.5) [] {$4$};
    \node (1) [yslant=0.5, anchor=south] at (1,1) [] {$1+2$};
    \node (6) [yslant=-0.5, anchor=south] at (4,0) [] {$4+6$};
    \draw [blue, decoration={markings, mark=at position 0.5 with {\arrow{>}}}, postaction={decorate}] (3) -- (4);
    \draw [blue, decoration={markings, mark=at position 0.6 with {\arrow{>}}}, postaction={decorate}] (4) -- (2);
    \draw [blue, decoration={markings, mark=at position 0.4 with {\arrow{>}}}, postaction={decorate}] (2) -- (5);
    \draw [red, decoration={markings, mark=at position 0.6 with {\arrow{>}}}, postaction={decorate}] (1) -- (3);
    \draw [red, decoration={markings, mark=at position 0.6 with {\arrow{>}}}, postaction={decorate}] (5) -- (6);
    \draw [red, decoration={markings, mark=at position 0.5 with {\arrow{>}}}, postaction={decorate}] (4) to [out=170,in=20] (1);
    \draw [red, decoration={markings, mark=at position 0.5 with {\arrow{>}}}, postaction={decorate}] (6) to [out=190,in=-10] (2);
    \draw [red, decoration={markings, mark=at position 0.6 with {\arrow{>}}}, postaction={decorate}] (6) to [out=150,in=-10] (1);
    \draw [red, decoration={markings, mark=at position 0.5 with {\arrow{>}}}, postaction={decorate}] (3) to [out=-10,in=120] (6);
    \draw [red, decoration={markings, mark=at position 0.5 with {\arrow{>}}}, postaction={decorate}] (1) to [out=-60,in=160] (5);
    \end{tikzpicture}
    \caption{Example of the branching phenomenon of the quiver of a shuffle graph. The blue arrows lie on the main branch (the plabic fence part). The red arrows go between the side branches and the main branch}
    \label{fig:branching quiver}
\end{figure}
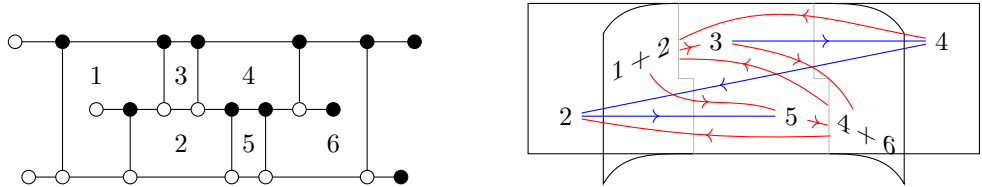

\subsection{Reflection Moves} In order to geometrically construct the DT-transformations for general shuffle graphs, we now generalize the left and right reflection moves introduced in \cite{ShenWeng}. 

Consider a Type 2 weave column with an outgoing weave line $s_i$ on one side (top or bottom). By using weave equivalences, we can extend this outgoing weave line inward, penetrating through the weave column and forming a trivalent weave vertex on the other side with color $s_{n-i}$. If in addition, either of the two horizontal weave lines incident to the new trivalent weave vertex happens to be outgoing as well, then we can homotope the weave locally so that the local picture becomes a Type 2 weave column with an outgoing weave line $s_{n-i}$ on the other side.

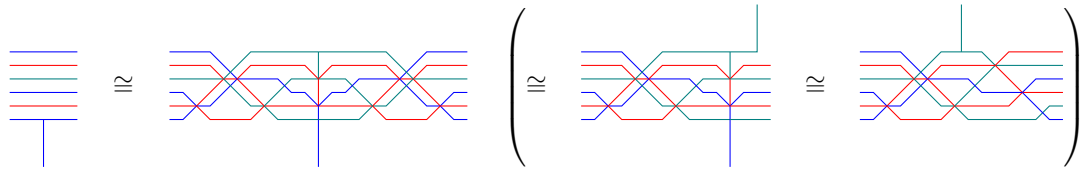
\begin{figure}[H]
    \centering
    \begin{tikzpicture}[baseline=15,scale=0.9]
    \foreach \i in {1,3,6}
    {
        \draw [blue] (0,\i*0.2) -- (1,\i*0.2);
    }
    \foreach \i in {2,5}
    {
        \draw [red] (0,\i*0.2) -- (1,\i*0.2);
    }
    \draw [teal] (0,0.8) -- (1,0.8);
    \draw [blue] (0.5,0.2) -- (0.5,-0.5);
    \end{tikzpicture} \quad $\cong$ \quad 
    \begin{tikzpicture}[baseline=15,scale=0.9]
    \draw [blue] (0,0.2) -- (0.2,0.2) -- (0.4,0.4) -- (0.2,0.6) -- (0,0.6);
    \draw [red] (0,0.4) -- (0.4,0.4);
    \draw [red] (0.4,0.4) -- (0.8,0.8);
    \draw [teal] (0,0.8) -- (0.8,0.8);
    \draw [red] (0,1) -- (0.6,1) -- (0.8,0.8);
    \draw [blue] (0.4,0.4) -- (0.6,0.4) -- (1,0.8);
    \draw [teal] (0.8,0.8) -- (1.2,0.4) -- (1.4,0.4);
    \draw [red] (0.4,0.4) -- (0.6,0.2) -- (1.2,0.2) -- (1.4,0.4);
    \draw [red] (0.8,0.8) -- (1,0.8) -- (1.4,0.4);
    \draw [blue] (0,1.2) -- (0.6,1.2) -- (1,0.8);
    \draw [teal] (1.4,0.4) -- (1.6,0.2) -- (2.2,0.2);
    \draw [red] (1.4,0.4) -- (2.2,0.4);
    \draw [teal] (1.4,0.4) -- (1.8,0.8) -- (2.2,0.8);
    \draw [blue] (1,0.8) -- (1.6,0.8) -- (1.8,0.6) -- (2,0.6) -- (2.2,0.4);
    \draw [teal] (0.8,0.8) -- (1.2,1.2) -- (2.2,1.2);
    \draw [red] (1,0.8) -- (1.2,1) -- (2,1) -- (2.2,0.8);
    \draw [blue] (2.2,-0.5) -- (2.2,0.4);
    \draw [red] (2.2,0.4) -- (2.2,0.8);
    \draw [teal] (2.2,0.8) -- (2.2,1.2);
    \draw [blue] (4.4,0.2) -- (4.2,0.2) -- (4,0.4) -- (4.2,0.6) -- (4.4,0.6);
    \draw [red] (4.4,0.4) -- (4,0.4);
    \draw [red] (4,0.4) -- (3.6,0.8);
    \draw [teal] (4.4,0.8) -- (3.6,0.8);
    \draw [red] (4.4,1) -- (3.8,1) -- (3.6,0.8);
    \draw [blue] (4,0.4) -- (3.8,0.4) -- (3.4,0.8);
    \draw [teal] (3.6,0.8) -- (3.2,0.4) -- (3,0.4);
    \draw [red] (4,0.4) -- (3.8,0.2) -- (3.2,0.2) -- (3,0.4);
    \draw [red] (3.6,0.8) -- (3.4,0.8) -- (3,0.4);
    \draw [blue] (4.4,1.2) -- (3.8,1.2) -- (3.4,0.8);
    \draw [teal] (3,0.4) -- (2.8,0.2) -- (2.2,0.2);
    \draw [red] (3,0.4) -- (2.2,0.4);
    \draw [teal] (3,0.4) -- (2.6,0.8) -- (2.2,0.8);
    \draw [blue] (3.4,0.8) -- (2.8,0.8) -- (2.6,0.6) -- (2.4,0.6) -- (2.2,0.4);
    \draw [teal] (3.6,0.8) -- (3.2,1.2) -- (2.2,1.2);
    \draw [red] (3.4,0.8) -- (3.2,1) -- (2.4,1) -- (2.2,0.8);
    \end{tikzpicture} \quad $\left(\cong \quad 
    \begin{tikzpicture}[baseline=15,scale=0.9]
    \draw [blue] (0,0.2) -- (0.2,0.2) -- (0.4,0.4) -- (0.2,0.6) -- (0,0.6);
    \draw [red] (0,0.4) -- (0.4,0.4);
    \draw [red] (0.4,0.4) -- (0.8,0.8);
    \draw [teal] (0,0.8) -- (0.8,0.8);
    \draw [red] (0,1) -- (0.6,1) -- (0.8,0.8);
    \draw [blue] (0.4,0.4) -- (0.6,0.4) -- (1,0.8);
    \draw [teal] (0.8,0.8) -- (1.2,0.4) -- (1.4,0.4);
    \draw [red] (0.4,0.4) -- (0.6,0.2) -- (1.2,0.2) -- (1.4,0.4);
    \draw [red] (0.8,0.8) -- (1,0.8) -- (1.4,0.4);
    \draw [blue] (0,1.2) -- (0.6,1.2) -- (1,0.8);
    \draw [teal] (1.4,0.4) -- (1.6,0.2) -- (2.2,0.2);
    \draw [red] (1.4,0.4) -- (2.2,0.4);
    \draw [teal] (1.4,0.4) -- (1.8,0.8) -- (2.2,0.8);
    \draw [blue] (1,0.8) -- (1.6,0.8) -- (1.8,0.6) -- (2,0.6) -- (2.2,0.4);
    \draw [teal] (0.8,0.8) -- (1.2,1.2) -- (2.2,1.2);
    \draw [red] (1,0.8) -- (1.2,1) -- (2,1) -- (2.2,0.8);
    \draw [blue] (2.2,-0.5) -- (2.2,0.4);
    \draw [red] (2.2,0.4) -- (2.2,0.8);
    \draw [teal] (2.2,0.8) -- (2.2,1.2);
    \draw [teal] (2.2,1.2) -- (2.6,1.2) -- (2.6,1.9);
    \draw [red] (2.2,0.8) -- (2.4,1) -- (2.8,1);
    \draw [teal] (2.2,0.8) -- (2.8,0.8);
    \draw [blue] (2.2,0.4) -- (2.4,0.6) -- (2.8,0.6);
    \draw [red] (2.2,0.4) -- (2.8,0.4);
    \draw [teal] (2.2,0.2) -- (2.8,0.2);
    \end{tikzpicture} \quad \cong \quad 
    \begin{tikzpicture}[baseline=15,scale=0.9]
    \draw [blue] (0,0.2) -- (0.2,0.2) -- (0.4,0.4) -- (0.2,0.6) -- (0,0.6);
    \draw [red] (0,0.4) -- (0.4,0.4);
    \draw [red] (0.4,0.4) -- (0.8,0.8);
    \draw [teal] (0,0.8) -- (0.8,0.8);
    \draw [red] (0,1) -- (0.6,1) -- (0.8,0.8);
    \draw [blue] (0.4,0.4) -- (0.6,0.4) -- (1,0.8);
    \draw [teal] (0.8,0.8) -- (1.2,0.4) -- (1.4,0.4);
    \draw [red] (0.4,0.4) -- (0.6,0.2) -- (1.2,0.2) -- (1.4,0.4);
    \draw [red] (0.8,0.8) -- (1,0.8) -- (1.4,0.4);
    \draw [blue] (0,1.2) -- (0.6,1.2) -- (1,0.8);
    \draw [teal] (1.4,0.4) -- (1.6,0.2) -- (2.6,0.2) -- (2.8,0.4) -- (3,0.4);
    \draw [red] (1.4,0.4) -- (2.2,0.4) -- (2.4,0.6);
    \draw [teal] (1.4,0.4) -- (2,1) -- (3,1);
    \draw [teal] (0.8,0.8) -- (1.2,1.2) -- (1.8,1.2) -- (2,1);
    \draw [red] (1,0.8) -- (1.2,1) -- (2,1) -- (2.4,0.6);
    \draw [blue] (1,0.8) -- (1.6,0.8) -- (1.8,0.6) -- (2.4,0.6) -- (2.8,0.2) -- (3,0.2);
    \draw [red] (2.4,0.6) -- (3,0.6);
    \draw [blue] (2.4,0.6) -- (2.6,0.8) -- (3,0.8);
    \draw [red] (2,1) -- (2.2,1.2) -- (3,1.2);
    \draw [teal] (1.5,1.2) -- (1.5,1.9);
    \end{tikzpicture}
    \right)$
    \caption{The first two pictures are an example of the local penetration move. If there is an additional weave equivalence that turns the $2$nd picture into the $3$rd one with an outgoing weave line on the top, then we can use it to turn the $3$rd one back to a Type 2 weave column again.}
    \label{fig:weave penetration}
\end{figure}

These weave equivalences are more general than reflection moves for rainbow closures in \textit{loc. cit.}. In our upcoming construction of cluster DT transformations, we will apply this weave equivalence to lollipops. For example, suppose $b$ is a black lollipop on the $i$th horizontal line in a grid plabic graph $\bG$. Take the first vertical edge $e'$ with a black vertex on the $i$th horizontal line as we search from right to left starting from $b$. Suppose the other vertex of $e'$ lies on the $j$th horizontal line whose right end point lies to the right of $b$, and suppose there are no more vertical edges (of either pattern) between the $i$th and $j$th horizontal lines to the right of $e'$. Then the reflection move can be used to turn $e'$ into its opposite pattern; a side-effect is that this move would also switch the portions of the $i$th and $j$th horizontal lines on the right side of $e'$, resulting in a possibly non-planar bicolor graph. Below is an example of such a move done on a plabic graph with three horizontal lines. A similar move can be applied to a white lollipop $w$, and the vertical edge is found by scanning rightward from $w$.

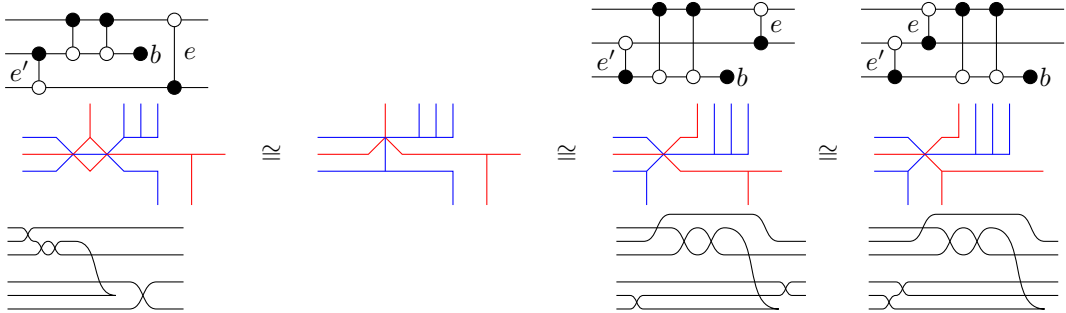
\begin{figure}[H]
    \centering
    \begin{tikzpicture}[scale=0.9]
    \draw (0,0) -- (3,0);
    \draw (0,0.5) -- (2,0.5);
    \draw [fill=black] (2,0.5) circle[radius=0.1];
    \draw (0,1) -- (3,1);
    \vertbar[](0.5,0,0.5);
    \vertbar[](1,0.5,1);
    \vertbar[](1.5,0.5,1);
    \vertbar[](2.5,1,0);
    \node at (2.5,0.5) [right] {$e$};
    \node at (0.5,0.25) [left] {$e'$};
    \node at (2,0.5) [right] {$b$};
    \end{tikzpicture}
    \quad \quad  \quad \quad  \quad \quad  \quad \quad  \quad \quad  \quad \quad  \quad \quad
    \begin{tikzpicture}[scale=0.9]
    \draw (0,0.5) -- (3,0.5);
    \draw (0,0) -- (2,0);
    \draw [fill=black] (2,0) circle[radius=0.1];
    \draw (0,1) -- (3,1);
    \vertbar[](0.5,0.5,0);
    \vertbar[](1,0,1);
    \vertbar[](1.5,0,1);
    \vertbar[](2.5,1,0.5);
    \node at (2.5,0.75) [right] {$e$};
    \node at (0.5,0.25) [left] {$e'$};
    \node at (2,0) [right] {$b$};
    \end{tikzpicture} \quad \quad  
    \begin{tikzpicture}[scale=0.9]
    \draw (0,0.5) -- (3,0.5);
    \draw (0,0) -- (2.5,0);
    \draw [fill=black] (2.5,0) circle[radius=0.1];
    \draw (0,1) -- (3,1);
    \vertbar[](0.5,0.5,0);
    \vertbar[](2,0,1);
    \vertbar[](1.5,0,1);
    \vertbar[](1,1,0.5);
    \node at (0.5,0.25) [left] {$e'$};
    \node at (2.5,0) [right] {$b$};
    \node at (1,0.75) [left] {$e$};
    \end{tikzpicture} \\ 
    \smallskip
    \begin{tikzpicture}[baseline=10,scale=0.9]
    \draw[blue](-0.5,0.25) --(0,0.25) -- (0.25,0.5) -- (0,0.75)-- (-0.5,0.75);
    \draw [red] (-0.5,0.5) -- (0.25,0.5);
    \draw [red] (0.25,0.5) -- (0.5,0.25) -- (0.75,0.5) -- (0.5,0.75) -- cycle;
    \draw [red] (0.5,0.75) -- (0.5,1.25);
    \draw [blue] (0.25,0.5) -- (0.75,0.5);
    \draw [blue] (0.75,0.5) -- (1,0.25) -- (1.5,0.25) -- (1.5,-0.25);
    \draw [blue] (0.75,0.5) -- (1,0.75) -- (1.5,0.75) -- (1.5,1.2 5);
    \draw [blue] (1,0.75) -- (1,1.25);
    \draw [blue] (1.25,0.75) -- (1.25,1.25);
    \draw [red] (0.75,0.5) -- (2.5,0.5);
    \draw [red] (2,0.5) -- (2,-0.25);
    \end{tikzpicture} \quad $\cong$  \quad 
    \begin{tikzpicture}[baseline=10,scale=0.9]
    \draw [red] (-0.5,0.5) -- (0.25,0.5) -- (0.5,0.75);
    \draw [red] (0.5,0.75) -- (0.5,1.25);
    \draw [blue] (-0.5,0.25) -- (1,0.25) -- (1.5,0.25) -- (1.5,-0.25);
    \draw [blue] (-0.5,0.75)  -- (1,0.75) -- (1.5,0.75) -- (1.5,1.2 5);
    \draw [blue] (1,0.75) -- (1,1.25);
    \draw [blue] (1.25,0.75) -- (1.25,1.25);
    \draw [red] (0.5,0.75) -- (0.75,0.5) -- (2.5,0.5);
    \draw [red] (2,0.5) -- (2,-0.25);
    \draw [blue] (0.5,0.75) -- (0.5,0.25);
    \end{tikzpicture} \quad $\cong$ \quad 
    \begin{tikzpicture}[baseline=10,scale=0.9]
    \draw[blue](-0.5,0.25) --(0,0.25) -- (0.25,0.5) -- (0,0.75)-- (-0.5,0.75);
    \draw [red] (-0.5,0.5) -- (0.25,0.5);
    \draw [red] (0.25,0.5) -- (0.5,0.75) -- (0.75,0.75) -- (0.75,1.25);
    \draw [red] (0.25,0.5) -- (0.5,0.25) -- (2,0.25);
    \draw [red] (1.5,0.25) -- (1.5,-0.25);
    \draw [blue] (0.25,0.5) -- (1.5,0.5) -- (1.5,1.25);
    \draw [blue] (1,0.5) -- (1,1.25);
    \draw [blue] (1.25,0.5) -- (1.25,1.25);
    \draw [blue] (0,0.25) -- (0,-0.25);
    \end{tikzpicture} \quad $\cong$ \quad
    \begin{tikzpicture}[baseline=10,scale=0.9]
    \draw[blue](-0.5,0.25) --(0,0.25) -- (0.25,0.5) -- (0,0.75)-- (-0.5,0.75);
    \draw [red] (-0.5,0.5) -- (0.25,0.5);
    \draw [red] (0.25,0.5) -- (0.5,0.75) -- (0.75,0.75) -- (0.75,1.25);
    \draw [red] (0.25,0.5) -- (0.5,0.25) -- (2,0.25);
    \draw [red] (0.5,0.25) -- (0.5,-0.25);
    \draw [blue] (0.25,0.5) -- (1.5,0.5) -- (1.5,1.25);
    \draw [blue] (1,0.5) -- (1,1.25);
    \draw [blue] (1.25,0.5) -- (1.25,1.25);
    \draw [blue] (0,0.25) -- (0,-0.25);
    \end{tikzpicture} \\
    \smallskip
    \begin{tikzpicture}[baseline=0,scale=0.9]
    \draw (0,1.2) -- (0.2,1.2) to [out=0,in=180] (0.4,1.4)-- (2.6,1.4);
    \draw (0,1) -- (0.4,1) to [out=0,in=180] (0.6,1.2) to [out=0,in=180] (0.8,1) -- (2.6,1);
    \draw (0,1.4) -- (0.2,1.4) to [out=0,in=180] (0.4,1.2) to [out=0,in=180] (0.6,1) to [out=0,in=180] (0.8,1.2) -- (1,1.2) to [out=0,in=180] (1.6,0.4) -- (0,0.4);
    \draw (0,0.2) -- (1.8,0.2) to [out=0,in=180] (2.2,0.6) -- (2.6,0.6);
    \draw (0,0.6) -- (1.8,0.6) to [out=0,in=180] (2.2,0.2) -- (2.6,0.2);
    \end{tikzpicture}
    \quad \quad  \quad \quad  \quad \quad  \quad \quad  \quad \quad  \quad \quad   \quad \quad  \quad \quad
    \begin{tikzpicture} [baseline=0,scale=0.9]
    \draw (-0.4,1.2) -- (0,1.2)  to [out=0,in=180] (0.4,1.6)-- (1.6,1.6) to [out=0,in=180] (2,1.2) -- (2.4,1.2);
    \draw (-0.4,1) -- (0.4,1) to [out=0,in=180] (0.8,1.4) to [out=0,in=180] (1.2,1) -- (2.4,1);
    \draw (-0.4,1.4) --  (0.4,1.4) to [out=0,in=180] (0.8,1) to [out=0,in=180] (1.2,1.4) to [out=0,in=180] (2,0.2) -- (0,0.2) to [out=180,in=0] (-0.2,0.4) -- (-0.4,0.4);
    \draw (-0.4,0.2) -- (-0.2,0.2) to [out=0,in=180] (0,0.4) -- (2,0.4)  to [out=0,in=180] (2.2,0.6) -- (2.4,0.6);
    \draw (-0.4,0.6) -- (2,0.6) to [out=0,in=180] (2.2,0.4) -- (2.4,0.4);
    \end{tikzpicture} \quad \quad 
    \begin{tikzpicture} [baseline=0,scale=0.9]
    \draw (-0.6,1.2)  -- (0,1.2) to [out=0,in=180] (0.4,1.6)-- (1.6,1.6) to [out=0,in=180] (2,1.2) -- (2.2,1.2);
    \draw (-0.6,1) -- (0.4,1) to [out=0,in=180] (0.8,1.4) to [out=0,in=180] (1.2,1) -- (2.2,1);
    \draw (-0.6,1.4) --  (0.4,1.4) to [out=0,in=180] (0.8,1) to [out=0,in=180] (1.2,1.4) to [out=0,in=180] (2,0.2) -- (-0.2,0.2) to [out=180,in=0] (-0.4,0.4) -- (-0.6,0.4);
    \draw (-0.6,0.2) -- (-0.4,0.2) to [out=0,in=180] (-0.2,0.4)to [out=0,in=180] (0,0.6) -- (2.2,0.6);
    \draw (-0.6,0.6) -- (-0.2,0.6) to [out=0,in=180] (0,0.4) -- (2.2,0.4);
    \end{tikzpicture}
    \caption{Example of a reflection move near a black lollipop.}
    \label{fig:right_reflection_move}
\end{figure}

Since such a move can potentially destroy planarity to the right of edge $e'$, sugar-free hulls no longer make sense there. However, if the part of the quiver corresponding to the region on the right of edge $e'$ does not get involved in the current iterative step, then this does not affect the construction of the cluster DT transformations. Moreover, the reflection move can enable us to pass a vertical edge through an obstructing lollipop. In Figure \ref{fig:right_reflection_move}, the lollipop $b$ is preventing the vertical edge $e$ from moving to the left in the left picture; after the reflection move, we can now move $e$ through the lollipop $b$; even better, the edge $e$ is now contained within a subgraph that is a plabic fence, for which we know a recursive procedure to construct the cluster DT transformation \cite{ShenWeng}.

In the front projection, the reflection move is a Legendrian RII move that pulls out a cusp, which is a Legendrian isotopy. Since the reflection move is a weave equivalence, the quiver does not change under such a move. Note that if $e'$ is the only the vertical edge present in each of the bicolor graphs in Figure \ref{fig:right_reflection_move}, then this reflection move recovers the right reflection move in \cite{ShenWeng}. (A similar picture can be drawn for left reflection moves.)

By realizing the reflection moves as Legendrian RII moves on the front projection, we can define a Legendrian isotopy on Legendrian links associated with shuffle graphs. By construction, between the top region and the bottom region of the initial weave associated with a shuffle graph, the outgoing weave lines inside one of them is always just a half twist. In the front projection, this region can be untangled into a collection of parallel horizontal lines. Thus, we can clockwise rotate every crossing in the other region one-by-one to this region using just reflection moves (and homotopy) on the front projection. We call this Legendrian isotopy the \emph{half K\'{a}lm\'{a}n loop} $K^{1/2}$. It is not a Legendrian loop and $K^{1/2}$ does not automatically give rise to an automorphism on $\cM_1(\Lambda)$: it only gives rise to an isomorphism $K^{1/2}:\cM_1(\Lambda)\rightarrow \cM_1(\Lambda')$, where $\Lambda'$ is the image of $\La$ under the Legendrian isotopy $K^{1/2}$. In order to make this into an automorphism, we need the involution $t$ induced from the strict contactomorphism $t:(x,y,z)\mapsto (-x,y,-z)$ on $\mathbb{R}^3$. By Proposition \ref{def: working def for M1}, we see that this strict contactomorphism reverse all maps in the quiver representation, which implies that we need to dualize all vector spaces and take transpositions of all the maps. Note that this coincides with the definition of the transposition map $t:\cM_1(\Lambda')\rightarrow \cM_1(\Lambda)$ in \cite{ShenWeng}. All parallel transportation maps are now dualized as well but the microlocal monodromies and microlocal merodromies remain unchanged and therefore $t$ preserves the cluster structure and is a cluster isomorphism. We denote $\DT:=t\circ K^{1/2}=K^{1/2}\circ t$ as our candidate for the cluster Donaldson-Thomas transformation for shuffle graphs.

\subsection{Edge Migration in a Plabic Fence} Besides the reflection moves, we also need to move vertical edges through regions that locally look like plabic fences, and cluster mutations are needed for this process. In this subsection, we will discuss these moves and prove some basic result about the color change of quiver vertices (green vs. red). We begin with a quick review of the meaning of vertex colors, green and red, in a quiver. Fix an initial quiver $Q$ with no frozen vertices. We construct a framed quiver $\tilde{Q}$ from $Q$ by adding a frozen vertex $i'$ for every vertex $i$ of $Q$, together with a single arrow pointing from $i$ to $i'$. Note that by construction, the exchange matrix of $\tilde{Q}$ is $\tilde{\epsilon}=\begin{pmatrix} \epsilon & \id \\ -\id & 0 \end{pmatrix}$ where $\epsilon$ is the exchange matrix of $Q$. 

For any mutation sequence $\mu_\mathbf{i}$ on the quiver $Q$, we can apply the same mutation sequence $\mu_\mathbf{i}$ to $\tilde{Q}$ and get a new quiver $\tilde{Q}':=\mu_\mathbf{i}(Q)$. Note that the unfrozen part of $\tilde{Q}'$ is identical to $Q':=\mu_\mathbf{i}(Q)$. A remarkable property of $\tilde{Q}'$ is that for any unfrozen vertex $i$, $\tilde{\epsilon}'_{ij'}$ is either non-negative or non-positive for all framing frozen vertices $j'$: this is known as the \emph{sign coherence} phenomenon of $c$-vectors in cluster theory \cite{DWZ, GHKK}. We say a vertex $i$ in $Q'$ is \emph{green} if $\tilde{\epsilon}'_{ij'}$ is non-negative for all framing frozen vertices $j'$ and a vertex $i$ in $Q'$ is \emph{red} if $\tilde{\epsilon}'_{ij'}$ is non-positive for all framing frozen vertices $j'$. For a given initial quiver $Q$ (all of whose vertices are green), if a mutation sequence $\mu_\mathbf{i}$ turns every quiver vertex red, then we say $\mu_\mathbf{i}$ is a \emph{reddening sequence}. If additionally a reddening sequence $\mu_\mathbf{i}$ only mutates at green vertices, then we say that $\mu_\mathbf{i}$ is a \emph{maximal green sequence}. For any fixed initial seed, the cluster Donaldson-Thomas transformation can be captured combinatorially by a reddening sequence \cite{KelDT,GoncharovLinhui_DT}. Thus, it is important to keep track of color change of quiver vertices as we perform cluster mutations.

Let us now consider a plabic fence $\bG$. By construction, the initial quiver $Q=Q(\bG)$ is a planar quiver with one unfrozen vertex for each face in $\bG$, and the arrows in $Q$ are drawn in a way such that they form a clockwise cycle around a neighboring group of white vertices and form a counterclockwise cycle around a neighboring group of black vertices. If we have two adjacent vertical edges of opposite patterns on the same level, we can exchange them by doing a mutation at the quiver vertex corresponding to the face they bound: this is just the square move. If we have two adjacent vertical edges of opposite patterns not on the same level, then we can slide them through each other without doing any mutation on the quiver. See Figure \ref{fig:sq_move_and_sliding}.
\begin{figure}[H]
    \begin{tikzpicture}[baseline=10,scale=0.8]
    \draw (0,0) -- (2,0);
    \draw (0,1) -- (2,1);
    \vertbar[](0.5,0,1);
    \vertbar[](1.5,1,0);
    \end{tikzpicture} \quad $\leftrightarrow$ \quad
    \begin{tikzpicture}[baseline=10,scale=0.8]
    \draw (0,0) -- (2,0);
    \draw (0,1) -- (2,1);
    \vertbar[](0.5,1,0);
    \vertbar[](1.5,0,1);
    \end{tikzpicture} \quad \quad \quad 
    \begin{tikzpicture} [baseline=10,scale=0.8]
    \foreach \i in {0,1,2}
    {
    \draw (0,\i*0.5) -- (1.5,\i*0.5);
    }
    \vertbar[](0.5,0,0.5);
    \vertbar[](1,1,0.5);
    \end{tikzpicture}\quad $\leftrightarrow$ \quad
    \begin{tikzpicture} [baseline=10,scale=0.8]
    \foreach \i in {0,1,2}
    {
    \draw (0,\i*0.5) -- (1.5,\i*0.5);
    }
    \vertbar[](1,0,0.5);
    \vertbar[](0.5,1,0.5);
    \end{tikzpicture}
    \quad \quad \quad 
    \begin{tikzpicture} [baseline=10,scale=0.8]
    \foreach \i in {0,1,2}
    {
    \draw (0,\i*0.5) -- (1.5,\i*0.5);
    }
    \vertbar[](0.5,0.5,0);
    \vertbar[](1,0.5,1);
    \end{tikzpicture}\quad $\leftrightarrow$ \quad
    \begin{tikzpicture} [baseline=10,scale=0.8]
    \foreach \i in {0,1,2}
    {
    \draw (0,\i*0.5) -- (1.5,\i*0.5);
    }
    \vertbar[](1,0.5,0);
    \vertbar[](0.5,0.5,1);
    \end{tikzpicture}
    \caption{Left: the square move in a plabic fence. Middle and right: sliding vertical edges of opposite patterns on different levels through each other.}
    \label{fig:sq_move_and_sliding}
\end{figure}
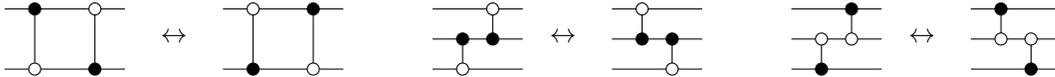

On a Legendrian weave, the sliding of edges corresponds to a weave equivalence, whereas the square move can be described by a weave mutation along a long $\sf I$-cycle, which can be locally described by the following movie.
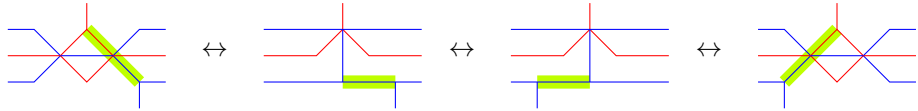
\begin{figure}[H]
    \centering
    \begin{tikzpicture}[baseline=10,scale=0.7]
    \draw [line width = 5, lime] (1.5,1) -- (2.5,0);
    \draw [blue] (0,0) -- (0.5,0) -- (1,0.5) -- (0.5,1) -- (0,1);
    \draw [red] (0,0.5) -- (1,0.5);
    \draw [red] (1,0.5) -- (1.5,1) -- (2,0.5) -- (1.5,0) -- cycle;
    \draw [red] (1.5,1.5) -- (1.5,1);
    \draw [blue] (1,0.5) -- (2,0.5);
    \draw [blue] (3,0) -- (2.5,0) -- (2,0.5) -- (2.5,1) -- (3,1);
    \draw [blue] (2.5,-0.5) -- (2.5,0);
    \draw [red] (2,0.5) -- (3,0.5);
    \end{tikzpicture} \quad $\leftrightarrow$ \quad
    \begin{tikzpicture}[baseline=10,scale=0.7]
    \draw [line width = 5, lime] (1.5,0) -- (2.5,0);
    \draw [blue] (0,0) -- (3,0);
    \draw [red] (0,0.5) -- (1,0.5) -- (1.5,1) -- (2,0.5) -- (3,0.5);
    \draw [red] (1.5,1.5) -- (1.5,1);
    \draw [blue] (1.5,1) -- (1.5,0);
    \draw [blue] (0,1) -- (3,1);
    \draw [blue] (2.5,-0.5) -- (2.5,0);
    \end{tikzpicture}\quad $\leftrightarrow$ \quad
    \begin{tikzpicture}[baseline=10,scale=0.7]
    \draw [line width = 5, lime] (1.5,0) -- (0.5,0);
    \draw [blue] (0,0) -- (3,0);
    \draw [red] (0,0.5) -- (1,0.5) -- (1.5,1) -- (2,0.5) -- (3,0.5);
    \draw [red] (1.5,1.5) -- (1.5,1);
    \draw [blue] (1.5,1) -- (1.5,0);
    \draw [blue] (0,1) -- (3,1);
    \draw [blue] (0.5,-0.5) -- (0.5,0);
    \end{tikzpicture}\quad $\leftrightarrow$ \quad
    \begin{tikzpicture}[baseline=10,scale=0.7]
    \draw [line width = 5, lime] (1.5,1) -- (0.5,0);
    \draw [blue] (0,0) -- (0.5,0) -- (1,0.5) -- (0.5,1) -- (0,1);
    \draw [red] (0,0.5) -- (1,0.5);
    \draw [red] (1,0.5) -- (1.5,1) -- (2,0.5) -- (1.5,0) -- cycle;
    \draw [red] (1.5,1.5) -- (1.5,1);
    \draw [blue] (1,0.5) -- (2,0.5);
    \draw [blue] (3,0) -- (2.5,0) -- (2,0.5) -- (2.5,1) -- (3,1);
    \draw [red] (2,0.5) -- (3,0.5);
    \draw [blue] (0.5,-0.5) -- (0.5,0);
    \end{tikzpicture}
    \caption{Square move in terms of Legendrian weaves: the first and last moves are weave equivalences; the middle move is a weave mutation.}
\end{figure}

A maximal green sequence on $Q(\bG)$ can be constructed recursively as follows:
\begin{itemize}
    \item[-] Take the right most vertical edge $e$ of the $\begin{tikzpicture}[baseline=5]\vertbar[](0,0,0.5);\end{tikzpicture}$ pattern and change it to the opposite pattern.
    \item[-] Move this newly changed vertical edge $e$ to the left, passing all remaining $\begin{tikzpicture}[baseline=5]\vertbar[](0,0,0.5);\end{tikzpicture}$ edges.
\end{itemize}
This iterative process terminates when we run out of vertical edges with a black vertex on top. Note that all mutations occur in this maximal green sequence come from square moves; this will not be the case for general shuffle graphs. Lastly, here is a result that we will need in the next subsection:

\begin{lemma}[{\cite[Proposition 4.6]{ShenWeng}}]\label{lem:moving edges} Each square move turns the mutating vertex from green to red and turns the vertex directly to the right of the mutating vertex (if such a vertex exists) from red back to green. As a result, at the end of each iteration of moving a vertical edge $e$ to the left, the left most quiver vertex on the level of $e$ turns red, while the color of every other vertex remain the same.
\end{lemma}

\subsection{DT Transformations for Shuffle Graphs} Let us construct the cluster DT transformations for shuffle graphs. By Condition (1) of Definition \ref{def:shufflegraph}, if a vertical edge can be placed between the $i$th and $j$th horizontal lines with $|i-j|>1$, then there must be disjoint two continuous regions we can place vertical edges between them, with one on the left and the other one on the right. Let us call them the \emph{left region} and the \emph{right region}, respectively. If $|i-j|=1$, then there is only one continuous region where we can place vertical edges between the two horizontal lines. The main strategy is to go through all vertical edges of $\bG$ one by one from right to left. For each vertical edge $e$ we do one of the following depending, on its location in $\bG$.

\smallskip

 \noindent\textbf{(I) If $e$ lies on level $(i,i+1)$:}
 \begin{enumerate}
     \item[(I.1)] Apply a reflection move to change the pattern of $e$ to $\begin{tikzpicture}[baseline=5]\vertbar[](0,0.5,0);\end{tikzpicture}$.
     \item[(I.2)] Move $e$ to the left through all $\begin{tikzpicture}[baseline=5]\vertbar[](0,0,0.5);\end{tikzpicture}$ edges incident to the $i$th or $(i+1)$st horizontal lines. Note that a cluster mutation occurs whenever we exchange $e$ and a $\begin{tikzpicture}[baseline=5]\vertbar[](0,0,0.5);\end{tikzpicture}$ edge at the same horizontal level.
 \end{enumerate}

 \noindent\textbf{(II) If $e$ lies in the right region of level $(i,j)$ with $j-i>1$:}
\begin{enumerate}
    \item[(II.1)] Apply a reflection move to change the pattern of $e$ to $\begin{tikzpicture}[baseline=5]\vertbar[](0,0.5,0);\end{tikzpicture}$.
    \item[(II.2)] Move $e$ all the way into the left region between the $i$th and the $j$th horizontal lines, and through all $\begin{tikzpicture}[baseline=5]\vertbar[](0,0,0.5);\end{tikzpicture}$ edges incident to the $i$th or $j$th horizontal lines.
\end{enumerate}

\begin{remark}\label{rem:mutation needed} There is the following subtlety in step (II.2). Since $j-i>1$, when moving $e$ to the left, we will encounter $j-i-1$ many black lollipops. Each time when we encounter a black lollipop, we will try to apply a move similar to Figure \ref{fig:right_reflection_move}. If such a move can be applied, then we will get another vertical edge of the same pattern as $e$, and we need to move this newly changed edge along with $e$ as a group to the left. Moreover, this newly changed edge itself may encounter a black lollipop, too, and consequently introduce another vertical edge into the moving group. This moving group eventually will reach the other side, and we need a way to recover $e$ back as a vertical edge on level $(i,j)$. This can be done by a move mirror to that of Figure \ref{fig:right_reflection_move}:

\begin{figure}[H]
    \centering
    \begin{tikzpicture}[baseline=10]
    \draw (0.5,0) -- (2.5,0);
    \draw (0.5,1) -- (2.5,1);
    \draw (1,0.5) -- (2.5,0.5);
    \draw [fill=white] (1,0.5) circle [radius=0.1];
    \vertbar[](1.5,0.5,0);
    \vertbar[](2,1,0.5);
    \node at (2,0.75) [right] {$e$};
    \node at (1.5,0.25) [right] {$e'$};
    \end{tikzpicture} \quad $\rightsquigarrow$ \quad 
    \begin{tikzpicture}[baseline=10]
    \draw (0.5,0.5) -- (2.5,0.5);
    \draw (0.5,1) -- (2.5,1);
    \draw (1,0) -- (2.5,0);
    \draw [fill=white] (1,0) circle [radius=0.1];
    \vertbar[](1.5,0,0.5);
    \vertbar[](2,1,0.5);
    \node at (2,0.75) [right] {$e$};
    \node at (1.5,0.25) [right] {$e'$};
    \end{tikzpicture}\quad $\rightsquigarrow$ \quad 
    \begin{tikzpicture}[baseline=10]
    \draw (0.5,0.5) -- (2.5,0.5);
    \draw (0.5,1) -- (2.5,1);
    \draw (1.5,0) -- (2.5,0);
    \draw [fill=white] (1.5,0) circle [radius=0.1];
    \vertbar[](2,0,0.5);
    \vertbar[](1,1,0.5);
    \node at (1,0.75) [left] {$e$};
    \node at (2,0.25) [right] {$e'$};
    \end{tikzpicture} \quad $\rightsquigarrow$ \quad 
    \begin{tikzpicture}[baseline=10]
    \draw (0.5,0) -- (2.5,0);
    \draw (0.5,1) -- (2.5,1);
    \draw (1.5,0.5) -- (2.5,0.5);
    \draw [fill=white] (1.5,0.5) circle [radius=0.1];
    \vertbar[](2,0.5,0);
    \vertbar[](1,1,0);
    \node at (1,0.5) [left] {$e$};
    \node at (2,0.25) [right] {$e'$};
    \end{tikzpicture} 
    \caption{Two reflection moves to cover the vertical edge $e$.}
    \label{fig:left reflection move}
\end{figure}
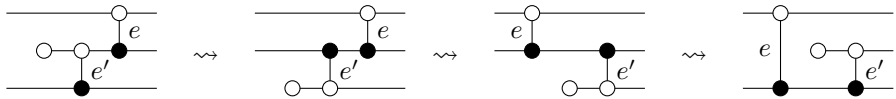
Note that all moves in Figure \ref{fig:left reflection move} correspond to weave equivalences and hence no cluster mutations occur. After the vertical edge $e$ is recovered, we can send the auxiliary vertical edge $e'$ back to where it was before, and then do another reflection move to restore the pattern of $e'$. Note that upon restoring the location and the pattern of $e'$, the non-planarity caused by the earlier reflection move (right picture of Figure \ref{fig:right_reflection_move}) will be cancelled, and we return to a grid plabic graph after the iteration.
\end{remark}

\begin{remark}\label{rem:no mutation needed} There is also a possibility that although there is a black lollipop $b$ in the way, no vertical edge $e'$ is found and hence it is not possible to perform the move in Figure \ref{fig:right_reflection_move}. We claim that in this case we can directly move $e$ through the obstructing horizontal line directly without the need of any weave (cluster) mutations. This follows from the fact that if no vertical edge $e'$ is present, then the incoming weave line corresponding to the gap where $e'$ should have been does not need to be tangled in the weave, and hence we can perform a weave equivalence to move the vertical edge $e$ through. See Figure \ref{fig:going through a horizontal line without mutations}.
\begin{figure}[H]
    \centering
    \begin{tikzpicture}[scale=0.9]
    \draw (0.5,0) -- (3.5,0);
    \draw (0.5,1) -- (3.5,1);
    \horline[](1,2.5,0.5);
    \vertbar[](1.5,0.5,1);
    \vertbar[](2,0.5,1);
    \vertbar[](3,1,0);
    \end{tikzpicture}  \quad \quad \quad \quad \quad \quad \quad \quad  \quad \quad \quad \quad  \quad \quad 
    \begin{tikzpicture}[scale=0.9]
    \draw (0,0) -- (3,0);
    \draw (0,1) -- (3,1);
    \horline[](1,2.5,0.5);
    \vertbar[](1.5,0.5,1);
    \vertbar[](2,0.5,1);
    \vertbar[](0.5,1,0);
    \end{tikzpicture}\\ 
    \begin{tikzpicture}[baseline=5,scale=0.9]
    \draw [blue] (0,0.25) -- (1,0.25);
    \draw [red] (0.5,-0.5) -- (0.5,0) -- (0.75,0) -- (1,0.25) -- (0.75,0.5) -- (0.5,0.5) -- (0.5,1);
    \draw [blue] (2,-0.5) -- (2,0) -- (1.25,0) -- (1,0.25) -- (1.25,0.5) -- (2,0.5) -- (2,1);
    \draw [blue] (1.5,0.5) -- (1.5,1);
    \draw [blue] (1.75,0.5) -- (1.75,1);
    \draw [red] (1,0.25) -- (3,0.25);
    \draw [red] (2.5,0.25) -- (2.5,-0.5);
    \end{tikzpicture} \quad $\cong$ \quad 
    \begin{tikzpicture}[baseline=2,scale=0.9]
    \draw [blue] (0,0.125) -- (0.625,0.125) -- (0.75,0.25)-- (1,0.25);
    \draw [blue] (0.625,0.125) -- (0.75,0) -- (1,0);
    \draw [red] (0.5,-0.75) -- (0.5,-0.25) -- (0.75,-0.25) -- (1,0) to [out=30,in=-30] (1,0.25) -- (0.75,0.5) -- (0.5,0.5) -- (0.5,1);
    \draw [blue] (2,-0.75) -- (2,-0.25) -- (1.25,-0.25) -- (1,0) to [out=150,in=-150] (1,0.25) -- (1.25,0.5) -- (2,0.5) -- (2,1);
    \draw [blue] (1.5,0.5) -- (1.5,1);
    \draw [blue] (1.75,0.5) -- (1.75,1);
    \draw [red] (1,0.25) -- (3,0.25);
    \draw [red] (1,0) -- (2.5,0) -- (2.5,-0.75);
    \end{tikzpicture} \quad $\cong$ \quad 
    \begin{tikzpicture}[baseline=5,scale=0.9]
    \draw [blue] (-0.5,0.25) -- (1,0.25);
    \draw [blue] (0,0.25) -- (0,-0.5);
    \draw [red] (0.5,-0.5) -- (0.5,0) -- (0.75,0) -- (1,0.25) -- (0.75,0.5) -- (0.5,0.5) -- (0.5,1);
    \draw [blue] (2,-0.5) -- (2,0) -- (1.25,0) -- (1,0.25) -- (1.25,0.5) -- (2,0.5) -- (2,1);
    \draw [blue] (1.5,0.5) -- (1.5,1);
    \draw [blue] (1.75,0.5) -- (1.75,1);
    \draw [red] (1,0.25) -- (2.5,0.25);
    \end{tikzpicture}
    \caption{Example of a special case where the edge $e'$ is absent}
    \label{fig:going through a horizontal line without mutations}
\end{figure}
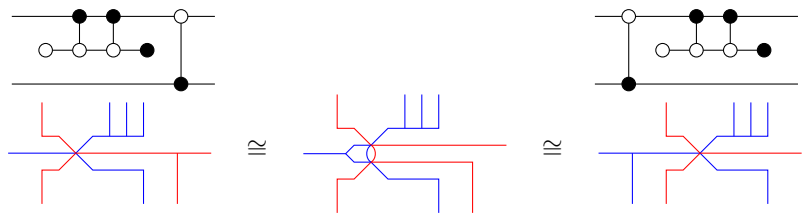
\end{remark}

\noindent\textbf{(III) If $e$ lies in the left region between the $i$th and the $j$th horizontal lines with $|i-j|>1$:}
\begin{enumerate}
    \item[(III.1)] Move $e$ all the way to the right region between the $i$th and the $j$th horizontal lines so that it becomes the right most vertical edge in the plabic graph.
    \item[(III.2)] Apply a reflection move to change the pattern of $e$ to $\begin{tikzpicture}[baseline=5]\vertbar[](0,0.5,0);\end{tikzpicture}$.
\end{enumerate}

Note that in this case, we need to move the vertical edge $e$ to the right before changing its pattern. But since we are going through vertical edges in $\bG$ one by one from right to left, by the time we get to $e$, all vertical edges to its right must be of the $\begin{tikzpicture}[baseline=5]\vertbar[](0,0.5,0);\end{tikzpicture}$ pattern already. Thus, moving $e$, which is of the $\begin{tikzpicture}[baseline=5]\vertbar[](0,0,0.5);\end{tikzpicture}$ pattern, to the right through $\begin{tikzpicture}[baseline=5]\vertbar[](0,0.5,0);\end{tikzpicture}$ edges, is completely mirror to step (II.1) before. Below are pictures depicting the reflection moves we need to perform during this process.

\begin{figure}[H]
    \centering
    \begin{tikzpicture}[baseline=10]
    \draw (0,0) -- (3,0);
    \draw (0,1) -- (3,1);
    \vertbar[](0.5,0,1);
    \draw (1,0.5) -- (3,0.5);
    \draw [fill=white] (1,0.5) circle [radius=0.1];
    \vertbar[](1.5,1,0.5);
    \vertbar[](2,1,0.5);
    \vertbar[](2.5,0.5,0);
    \node at (0.5,0.5) [left] {$e$};
    \node at (2.5,0.25) [right] {$e'$};
    \end{tikzpicture}\quad $\rightsquigarrow$\quad
    \begin{tikzpicture}[baseline=10]
    \draw (0,0.5) -- (3,0.5);
    \draw (0,1) -- (3,1);
    \vertbar[](0.5,0.5,1);
    \draw (1,0) -- (3,0);
    \draw [fill=white] (1,0) circle [radius=0.1];
    \vertbar[](1.5,1,0);
    \vertbar[](2,1,0);
    \vertbar[](2.5,0,0.5);
    \node at (0.5,0.75) [left] {$e$};
    \node at (2.5,0.25) [right] {$e'$};
    \end{tikzpicture}\quad $\rightsquigarrow$\quad
    \begin{tikzpicture}[baseline=10]
    \draw (0,0.5) -- (3,0.5);
    \draw (0,1) -- (3,1);
    \vertbar[](2,0.5,1);
    \draw (0.5,0) -- (3,0);
    \draw [fill=white] (0.5,0) circle [radius=0.1];
    \vertbar[](1,1,0);
    \vertbar[](1.5,1,0);
    \vertbar[](2.5,0,0.5);
    \node at (2,0.75) [right] {$e$};
    \node at (2.5,0.25) [right] {$e'$};
    \end{tikzpicture}\\
    \quad \\
    \quad \\
    \begin{tikzpicture}[baseline=10]
    \draw (0,0) -- (2,0);
    \draw (0,1) -- (2,1);
    \draw (0,0.5) -- (1.5,0.5);
    \draw [fill=black] (1.5,0.5) circle [radius=0.1];
    \vertbar[](0.5,0.5,1);
    \vertbar[](1,0,0.5);
    \node at (0.5,0.75) [left] {$e$};
    \node at (1,0.25) [left] {$e'$};
    \end{tikzpicture}\quad $\rightsquigarrow$\quad
    \begin{tikzpicture}[baseline=10]
    \draw (0,0.5) -- (2,0.5);
    \draw (0,1) -- (2,1);
    \draw (0,0) -- (1.5,0);
    \draw [fill=black] (1.5,0) circle [radius=0.1];
    \vertbar[](0.5,0.5,1);
    \vertbar[](1,0.5,0);
    \node at (0.5,0.75) [left] {$e$};
    \node at (1,0.25) [left] {$e'$};
    \end{tikzpicture}\quad $\rightsquigarrow$\quad
    \begin{tikzpicture}[baseline=10]
    \draw (0,0.5) -- (2,0.5);
    \draw (0,1) -- (2,1);
    \draw (0,0) -- (1,0);
    \draw [fill=black] (1,0) circle [radius=0.1];
    \vertbar[](1.5,0.5,1);
    \vertbar[](0.5,0.5,0);
    \node at (1.5,0.75) [left] {$e$};
    \node at (0.5,0.25) [left] {$e'$};
    \end{tikzpicture}\quad $\rightsquigarrow$\quad
    \begin{tikzpicture}[baseline=10]
    \draw (0,0) -- (2,0);
    \draw (0,1) -- (2,1);
    \draw (0,0.5) -- (1,0.5);
    \draw [fill=black] (1,0.5) circle [radius=0.1];
    \vertbar[](1.5,0,1);
    \vertbar[](0.5,0,0.5);
    \node at (1.5,0.5) [right] {$e$};
    \node at (0.5,0.25) [left] {$e'$};
    \end{tikzpicture}
    \caption{Reflection moves in the (III.1) step: the top row is an example of how to shrink $e$ to a shorter vertical edge, and the bottom row is an example of how to recover $e$ after moving through a horizontal line.}
    \label{fig:pushing vertical edge to the right}
\end{figure}
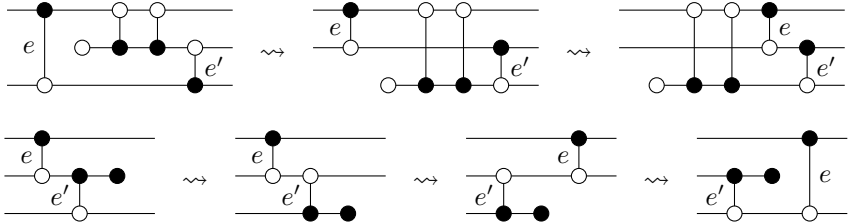

We can now conclude the following result:

\begin{thm}\label{thm:DT} Let $\bG$ be a shuffle graph. Then $\DT=t\circ K^{1/2}$ is the cluster Donaldson-Thomas transformation on $\cM_1(\Lambda)$.
\end{thm}
\begin{proof} Since $t$ is a cluster isomorphism, it suffices to prove that $K^{1/2}$ gives rise to a reddening sequence. The vertices of the quiver $Q(\bG)$ are grouped into regions (Figures \ref{fig:branching quiver} and \ref{fig:multiple branches}): we claim that after each iterative step (I) and step (II) of moving an edge $e$ on level $(i,j)$, the left most green vertex of level $(i,j)$ turns red, and after each iterative step (III) of moving an edge $e$ on level $(i,j)$, the right most green vertex of level $(i,j)$ turns red. Indeed, the case (I) follows from Lemma \ref{lem:moving edges} directly; so it remains to consider (II) and (III).

Let us consider (II) first. In the process of moving a $\begin{tikzpicture}[baseline=5]\vertbar[](0,0.5,0);\end{tikzpicture}$ edge $e$ on level $(i,j)$ to the left, before we encounter any lollipop, the quiver vertices would change according to Lemma \ref{lem:moving edges}, turning from green to red when a square move occurs and then turning back to green in the next square move. Let us now consider what happens when we encounter a black lollipop at the $k$th horizontal line (with $i<k<j$). If we are in the situation of Remark \ref{rem:no mutation needed}, then we can directly jump through the whole $k$th horizontal line without any cluster mutations, and Lemma \ref{lem:moving edges} will continue to take care of the rest. It thus remains to consider what happens when we need to do moves according to Remark \ref{rem:mutation needed}. Note that since the moves in Figure \ref{fig:right_reflection_move} do not induce any cluster mutations, there is no change to the quiver itself. However, the way we branch the quiver is different: the sugar-free hull to the left of edge $e$ was non-rectangular before the moves in Figure \ref{fig:right_reflection_move} but it becomes rectangular after the moves; thus, the corresponding quiver vertex was on the side branch before the moves and relocates itself to the main branch after the moves.

\noindent After this quiver vertex is relocated to the main branch (which is a quiver of a plabic fence), we can make use of Lemma \ref{lem:moving edges} again. Note that we need to move $e$ as well as the auxiliary edge $e'$ together to the left as a group, and in that process, there is still a possibility of introducing more edges to that left-moving group. Nevertheless, by induction it is enough to consider what happens when the two-member group $e'$ and $e$ reach the left white lollipop of the $k$th horizontal line. By Lemma \ref{lem:moving edges} we know that both the quiver vertex $v$ to the right of $e$ and the quiver vertex $v'$ to the right of $e'$ have turned red. Next, under the moves in Figure \ref{fig:left reflection move}, we restore $e$ to a vertical edge on level $(i,j)$ without any cluster mutations. Finally, we need to send $e'$ back to where it was before, and this process reverses the mutations we did on the level of $e'$: the quiver vertices on that level will turn from green to red and then back to green again one-by-one\footnote{Note that we are mutating at red quiver vertices in this process; thus, we are not claiming that the whole mutation sequence is maximal green. On the other hand, such moves are not needed in the case of plabic fences, which is why we can obtain a maximal green sequence.}; in the end all quiver vertices on the same level as $e'$ are restored back to green. The iterative step can now continue further to the left on level $(i,j)$, and the color change will again follow Lemma \ref{lem:moving edges}.

The case (III) is essentially (II) in reverse. Suppose we are moving a $\begin{tikzpicture}[baseline=5]\vertbar[](0,0,0.5);\end{tikzpicture}$ edge $e$ on level $(i,j)$, and suppose the first white lollipop it encounters is on the $k$th horizontal line. Let $v$ be the quiver vertex to the right of $e$, which is the right-most green on level $(i,j)$ at the beginning of the iterative step. Under the moves in the top row of Figure \ref{fig:pushing vertical edge to the right}, we move $v$ into level $(k,j)$, and obtain another vertical edge $e'$ of the $\begin{tikzpicture}[baseline=5]\vertbar[](0,0.5,0);\end{tikzpicture}$ pattern. Note that the vertex $v'$ to the right of $e'$ is red at this moment. Now we need to move the vertical edges $e$ and $e'$ to the right: first $e'$, then $e$. For each square move on level $(i,k)$ (the level of $e'$), the mutating quiver vertex changes from red to green\footnote{These mutations are also not green.} and stays green afterward. On the other hand, for each square move on level $(k,j)$ (the level of $e$), the mutating quiver vertex turns from green to red and the one in the next mutation (if exists) turns from red to green. As a result, when $e$ and $e'$ gets to the right end of the $k$th horizontal line, all quiver vertices on the level of $e$ are red and all quiver vertices on the level of $e'$ are green. After the second row of moves in Figure \ref{fig:pushing vertical edge to the right} (which do not involve mutations), we need to send $e'$ back to where it was, and that will make the quiver vertices on level of $e'$ undergo another the reverse sequence of color change again, restoring all of them back to red. Of course, if there are move vertices further to the right of $e$, we need to continue moving $e$ rightward, which will make the remaining quiver vertices to the right on level $(i,j)$ turn from red to green and then back to red again one-by-one. In the end, we see that precisely the quiver vertex $v$ turns red after this iterative step.

In conclusion, since all quiver vertices in $Q(\bG)$ are either inside the middle region, to the left of a vertical edge in the right region, or to the right of a vertical edge in the left region, we see that after all vertical edges in $\bG$ change from $\begin{tikzpicture}[baseline=5]\vertbar[](0,0,0.5);\end{tikzpicture}$ to $\begin{tikzpicture}[baseline=5]\vertbar[](0,0.5,0);\end{tikzpicture}$, all quiver vertices would turn red. Therefore $K^{1/2}$ is indeed a reddening sequence.
\end{proof}

The statement in Corollary \ref{cor:DT} about the cluster duality conjecture now follows from \cite{GHKK}, as our quivers are full-ranked and a DT-transformation exists. Finally, we remark that  the same argument used in \cite{GSW2} to distinguish infinitely many Lagrangian fillings also works for any shuffle graph whose quiver is mutation equivalent to an acyclic quiver of infinite type. Indeed, the DT-transformation will be of infinite order and so will be its square,   the Legendrian K\'alm\'an loop.

\begin{example} Consider following shuffle graph $\bG$, with its quiver depicted and $\La(\bG)$ at its right: 
\[
\begin{tikzpicture}[scale=0.7]
\horline[](1,6.5,0);
\horline[](1,6.5,2);
\horline[](2,5.5,1);
\foreach \i in {1.5,6}
{
\vertbar[](\i,0,2);
}
\foreach \i in {2.5,4,4.5}
{
\vertbar[](\i,0,1);
}
\foreach \i in {3,3.5,5}
{
\vertbar[](\i,1,2);
}
\node at (2,1.5) [] {\footnotesize{$1$}};
\node at (3.25,0.5) [] {\footnotesize{$2$}};
\node at (3.25,1.5) [] {\footnotesize{$3$}};
\node at (4.25,1.5) [] {\footnotesize{$4$}};
\node at (4.25,0.5) [] {\footnotesize{$5$}};
\node at (5.5,0.5) [] {\footnotesize{$6$}};
\end{tikzpicture}  \quad
\begin{tikzpicture}[scale=0.9]
\node[teal] (3) at (2,0.5) [] {\footnotesize{$1+2$}};
\node[teal] (4) at (3,-0.2) [] {\footnotesize{$2$}};
\node[teal] (5) at (3,1.2) [] {\footnotesize{$3$}};
\node[teal] (6) at (4,1.2) [] {\footnotesize{$4$}};
\node[teal] (7) at (4,-0.2) [] {\footnotesize{$5$}};
\node[teal] (8) at (5,0.5) [] {\footnotesize{$4+6$}};
\draw[decoration={markings, mark=at position 0.5 with {\arrow{>}}}, postaction={decorate}] (3) -- (5);
\draw[decoration={markings, mark=at position 0.5 with {\arrow{>}}}, postaction={decorate}] (3) -- (7);
\draw[decoration={markings, mark=at position 0.5 with {\arrow{>}}}, postaction={decorate}] (4) -- (7);
\draw[decoration={markings, mark=at position 0.5 with {\arrow{>}}}, postaction={decorate}] (5) -- (6);
\draw[decoration={markings, mark=at position 0.7 with {\arrow{>}}}, postaction={decorate}] (5) -- (8);
\draw[decoration={markings, mark=at position 0.7 with {\arrow{>}}}, postaction={decorate}] (6) -- (4);
\draw[decoration={markings, mark=at position 0.5 with {\arrow{>}}}, postaction={decorate}] (6) -- (3);
\draw[decoration={markings, mark=at position 0.5 with {\arrow{>}}}, postaction={decorate}] (7) -- (8);
\draw[decoration={markings, mark=at position 0.5 with {\arrow{>}}}, postaction={decorate}] (8) -- (4);
\draw[decoration={markings, mark=at position 0.7 with {\arrow{>}}}, postaction={decorate}] (8) -- (3);
\end{tikzpicture}   \quad
\begin{tikzpicture}
\draw (1.2,1.2) -- (0.4,1.2) -- (1,0) -- (4.6,0) -- (4,1.2) -- (3.6,1.2);
\draw (1.2,0.8) -- (1,0.8) -- (1.2,0.4) -- (4,0.4) -- (3.8,0.8) -- (3.6,0.8);
\foreach \i in {1.2,3.4}
{
\crossing(\i,0.8)(\i+0.2,1.2);
}
\foreach \i in {2,2.2,2.8}
{
\crossing(\i,0.8)(\i+0.2,1);
}
\foreach \i in {1.8,2.4,2.6}
{
\crossing(\i,1)(\i+0.2,1.2);
}
\draw (1.4,1.2) -- (1.8,1.2);
\draw (1.4,0.8) -- (2,0.8);
\draw (2,1.2) -- (2.4,1.2);
\draw (2.4,0.8) -- (2.8,0.8);
\draw (2.8,1.2) -- (3.4,1.2);
\draw (3,0.8) -- (3.4,0.8);
\draw (1.8,1) -- (1.6,1) -- (2,0.2) -- (3.6,0.2) -- (3.2,1) -- (3,1);
\end{tikzpicture}
\]
is mutation equivalent to an acyclic quiver of infinite type, e.g.~consider the mutation sequence $\mu_5\circ \mu_{4+6}\circ \mu_2\circ \mu_4\circ \mu_3\circ \mu_{4+6}$. Thus $\Lambda(\bG)$, which is a max-tb representative in the smooth knot type $10_{161}$, admits infinitely many non-Hamiltonian isotopic embedded exact fillings. Note that the smooth knot type $10_{161}$ is not a rainbow closure of a positive braid. The reddening sequence realizing DT is 
\begin{align*}
&(\mu_2\circ \mu_5\circ \mu_4\circ \mu_3\circ \mu_5\circ \mu_2)\circ ()\circ ()\circ (\mu_4)\circ (\mu_5)\circ (\mu_2\circ \mu_5)\circ (\mu_{4+6}\circ \mu_4)\\
&\circ(\mu_5\circ \mu_2\circ \mu_{1+2}\circ  \mu_3\circ \mu_{4+6}\circ \mu_2\circ \mu_5),
\end{align*}
where we have grouped the mutations of each iterative step inside a pair of parentheses, and empty parenthesis means no mutations at that step. The first mutation is $\mu_5$, then $\mu_2$, then $\mu_{4+6}$ and so on until the last three mutations are $\mu_4$, $\mu_5$ and lastly $\mu_2$.
\end{example}




\bibliographystyle{alpha}
\bibliography{CW}

\end{document}